\documentclass{article}

\usepackage{Arxiv}

\usepackage[utf8]{inputenc} 
\usepackage[T1]{fontenc}    

\usepackage{etoolbox}
\usepackage[bookmarks=true]{hyperref}       
\makeatletter
\patchcmd{\@bibitem}{\ignorespaces}{\label{bib-#1}\ignorespaces}{}{}

\usepackage{url}            
\usepackage{booktabs}       
\usepackage{amsfonts}       
\usepackage{nicefrac}       
\usepackage{microtype}      
\usepackage{xcolor}         
\usepackage{amsmath}
\usepackage{amsthm}
\usepackage{amssymb}
\usepackage{algpseudocode}
\usepackage[english]{babel}

\usepackage{aliascnt}
\usepackage{cleveref}
\crefname{assumption}{Assumption}{Assumptions}
\crefname{theorem}{Theorem}{Theorems}
\crefname{lemma}{Lemma}{Lemmas}

\usepackage{algorithm}

\usepackage{graphicx}
\usepackage{subfigure}
\usepackage{stackengine}

\usepackage{multicol}
\usepackage{lipsum}
\usepackage{enumitem}
\usepackage{makecell}
\usepackage{arydshln}
\usepackage{bbding}
\usepackage{tablefootnote}
\usepackage{titlesec}

\usepackage[bottom]{footmisc}

\usepackage{wrapfig}

\def\prox#1{\mathbf{prox}_{#1}}
\newtheorem{assumption}{Assumption}
\newtheorem{theorem}{Theorem}
\newtheorem{lemma}{Lemma}
\newtheorem{definition}{Definition}
\newtheorem{remark}{Remark}

\newtheorem{corollary}{Corollary}
\newcommand{\sM}{M_{\tau,\sigma,\theta}}
\newcommand{\sXi}{\Xi_{\tau,\sigma,\theta}}
\newcommand{\ED}[1]{\mathbb{E}\left[ \| #1\|^2 \right]}

\newcommand{\fixspp}{\tfrac{M_{\tau,\sigma,\theta}}{N_t}}
\def\fprod#1{\left\langle#1\right\rangle}
\def\prox#1{\mathbf{prox}_{#1}}
\DeclareMathOperator*{\argmax}{argmax}
\DeclareMathOperator*{\argmin}{argmin}
\DeclareMathOperator*{\mo}{mod}

\DeclareMathOperator{\diag}{diag}

\def\grad{\nabla}

\def\bx{\mathbf{x}}  

\def\cB{\mathcal{B}}
\def\cC{\mathcal{C}}
\def\cD{\mathcal{D}}

\def\cF{\mathcal{F}}
\def\cG{\mathcal{G}}

\def\cI{\mathcal{I}}

\def\cL{\mathcal{L}}

\def\cO{\mathcal{O}}

\def\cX{\mathcal{X}}
\def\cY{\mathcal{Y}}

\def\mE{\mathbb{E}}

\def\smskip{\smallskip}

\def\texitem#1{\par\smskip\noindent\hangindent 25pt
               \hbox to 25pt {\hss #1 ~}\ignorespaces}

\def\norm#1{\|#1\|}

\newcommand{\BEAS}{\begin{eqnarray*}}
\newcommand{\EEAS}{\end{eqnarray*}}
\newcommand{\BEA}{\begin{eqnarray}}
\newcommand{\EEA}{\end{eqnarray}}
\newcommand{\BEQ}{\begin{eqnarray}}
\newcommand{\EEQ}{\end{eqnarray}}
\newcommand{\BIT}{\begin{itemize}}
\newcommand{\EIT}{\end{itemize}}
\newcommand{\BNUM}{\begin{enumerate}}
\newcommand{\ENUM}{\end{enumerate}}

\newcommand{\BA}{\begin{array}}
\newcommand{\EA}{\end{array}}

\newcommand{\reals}{\mathbb{R}}
\newcommand{\integers}{\mathbb{Z}}

\newcommand{\dom}{\mathop{\bf dom}}

\def\red#1{\textcolor{black}{#1}}

\newif\ifpagenumbering
\pagenumberingtrue

\pagenumberingfalse

\newsavebox{\theorembox}
\newsavebox{\lemmabox}
\newsavebox{\defnbox}
\newsavebox{\assbox}
\savebox{\theorembox}{\noindent\bf Theorem}
\savebox{\lemmabox}{\noindent\bf Lemma}
\savebox{\defnbox}{\noindent Definition}

\usepackage{soul}

\usepackage[normalem]{ulem}
\DeclareUnicodeCharacter{2212}{~}

\def\sa#1{\textcolor{red}{#1}}
\def\xz#1{\textcolor{black}{#1}}
\def\xzd#1{\textcolor{green}{#1}}
\def\xzh#1{\textcolor{blue}{#1}}
\usepackage{todonotes}
\def\nsa#1{\todo[size=footnotesize]{NSA:~#1}}

\usepackage{bbm}
\usepackage{accents}
\newlength{\dhatheight}

\def\nsa#1{\textcolor{black}{#1}}
\def\sa#1{\textcolor{black}{#1}}
\def\xz#1{\textcolor{black}{#1}}
\def\xzh#1{\textcolor{black}{#1}}
\def\xzd#1{\textcolor{black}{#1}}
\def\mg#1{\textcolor{black}{#1}}
\def\mgrev#1{\textcolor{black}{#1}}

\def\xzrev#1{\textcolor{black}{#1}}
\def\na#1{\textcolor{black}{#1}}
\def\xzf#1{\textcolor{black}{#1}}

\definecolor{caribbeangreen}{rgb}{0.0, 0.8, 0.6}
\definecolor{plum}{rgb}{0.3,0,0.7}
\definecolor{teal}{rgb}{0.0, 0.5, 0.5}

\usepackage{tikz}
\usepackage{calc}

\title{\sa{SAPD+~: An Accelerated Stochastic Method 
for 
Nonconvex-Concave Minimax Problems}}

\author{
Xuan Zhang\\
Department of Industrial and Manufacturing Engineering\\
Pennsylvania State University
\\
University Park, PA,USA.\\
\texttt{xxz358@psu.edu}
\And
Necdet Serhat Aybat\\
Department of Industrial and Manufacturing Engineering\\
Pennsylvania State University
\\
University Park, PA,USA.\\
\texttt{nsa10@psu.edu}
\And
Mert Gürbüzbalaban\\
Department of Management Science and Information Systems\\
Rutgers University\\
Piscataway, NJ, USA\\
\texttt{mg1366@rutgers.edu}
}

\begin{document}
\maketitle
\begin{abstract}
We propose a new stochastic 
method \texttt{SAPD+} 
for solving nonconvex-concave minimax problems of the form $\min\max\cL(x,y)=f(x)+\Phi(x,y)-g(y)$, where $f,g$ are closed convex 
and $\Phi(x,y)$ is a smooth function that is weakly convex in $x$, (strongly) concave in $y$. \red{Let $\delta^2$ denote the variance bound for the unbiased stochastic oracle used within \texttt{SAPD+} to estimate $\grad\Phi$. When $\delta>0$,} for both strongly concave and merely concave settings, 
\texttt{SAPD+} achieves the best known oracle complexities: \red{$\cO\Big(\kappa_y\max\Big\{1,\frac{\delta^2}{\epsilon^2}\Big\}\frac{L\cG_0}{\epsilon^{2}}\Big)$ for the strongly concave case \textit{without} assuming compactness of the problem domain, and $\cO\Big(\frac{L^3\cD_y^2\cG_0}{\epsilon^{4}}\Big(1+\frac{\delta^2}{\epsilon^2}\Big)\Big)$ for the merely concave case, where $\kappa_y\geq 1$ is the condition number, $L$ is the Lipschitz constant of $\grad \Phi$, $\cG_0$ is the primal-dual gap of the initial point, and $\cD_y=\sup\{\norm{y}:\ y\in\dom g\}$.} 
We also propose \texttt{SAPD+} with \textit{variance reduction}, 
which enjoys \red{$\cO\Big(\max\Big\{\kappa_y,\sqrt{\frac{\delta}{\epsilon}}\Big\}\cdot (1+\kappa_y\frac{\delta}{\epsilon})\frac{L\cG_0}{\epsilon^2}\Big)$} oracle complexity  
for weakly convex-strongly concave setting \red{--this is the best known upper complexity bound in the literature for this setting and 
our paper establishes it for the first time.} We demonstrate the efficiency of \texttt{SAPD+} on a distributionally robust learning problem with
a \nsa{nonconvex regularizer} and also on a multi-class classification problem in deep learning. 
\end{abstract}

\section{Introduction}\label{section: weakly covex algorithms}
\vspace*{-2mm}
We consider the following saddle-point~\sa{(SP) problem}:
\begin{equation}\label{eq:main problem}
    \min_{x\in \mathcal{X}} \max_{y\in \mathcal{Y}} \mathcal{L}(x,y)\triangleq f(x)+\Phi(x,y)-g(y),
\end{equation}
where $\mathcal{X}$ and $\mathcal{Y}$ are, \sa{$n$ and $m$ dimensional 
Euclidean spaces,} 
the function $\nsa{\Phi} 
:\mathcal{X}\times\cY \rightarrow \mathbb{R}$ is smooth and \nsa{possibly} 
\sa{nonconvex} in $x\in\cX$ and $\mu_y$-strongly
concave in $y\in\cY$ \sa{for some $\mu_y\geq 0$} --\mg{with the convention that for $\mu_y=0$, $\Phi$ is merely 
\sa{concave} (MC) in $y$}, and the functions $f$ and $g$ are \sa{closed}, convex and possibly nonsmooth. In this \xz{paper}, we consider a particular case of nonconvexity, i.e.,  we assume that \sa{$\Phi(\cdot,y)$} is weakly convex~(WC) for any fixed \sa{$y\in\dom g\subset\cY$}. Weakly convex functions 
\sa{constitute} a rich class of non-convex functions and arise 
\sa{naturally} in many \sa{practical} settings 
\sa{for} machine learning~\sa{(ML)} applications \cite{davis2019stochastic,rafique1810non}, precise definitions will be given later in Section \ref{sec-preliminaries}. \sa{In practice, WC assumption is widely satisfied, e.g., under smoothness --see~\cref{rem:smooth-WC}\xzrev{;}
\xzrev{most of the}
work \sa{in related literature} considering nonconvex-(strongly) concave 
SP problems provide their analyses under the premise of weak convexity.} 
\mg{The problem \eqref{eq:main problem} with $\mu_y> 0$ is called a weakly convex-strongly concave (WCSC) saddle-point
problem, whereas for $\mu_y=0$, it is called a weakly convex-merely concave (WCMC) saddle-point problem. 
Both problems arise frequently in many 
\sa{ML} 
\xzrev{settings}
including constrained optimization of 
\sa{WC} objectives based on Lagrangian duality \cite{li2021augmented}, Generative Adversarial Networks (GAN)
(where 
\sa{$x$ 
denotes} the parameters of the \xzrev{\emph{generator}} network whereas \sa{$y$ 
represents} the parameters of the \xzrev{\emph{discriminator}} network \cite{goodfellow2014generative}), distributional robust learning with weakly convex loss functions such as those arising in deep learning \cite{gurbuzbalaban2020stochastic,rafique1810non} and learning problems with non-decomposable losses \cite{rafique1810non}.}

\mg{There are 
\nsa{two} important settings 
for \eqref{eq:main problem}: (i) the \emph{deterministic setting}, where the \sa{partial} gradients of 
$\Phi$ 
are exactly available, (ii) the \emph{stochastic setting}, where 
only stochastic estimates of the gradients are available. Although, recent years have witnessed significant advances in the deterministic setting \cite{chen2021accelerated,huang2021efficient,jin2020local,lin2020gradient,lin-near-optimal,lu2020hybrid,ostrovskii2021efficient, thekumparampil2019efficient,xu2020unified}; our focus in this paper will be mainly \xzrev{on} the \emph{stochastic setting}\xzrev{,} which is more relevant and more applicable to 
\sa{ML} problems. Indeed, due to large-dimensions and the \sa{sheer} size of the modern datasets, computing gradients exactly is either infeasible or impractical in 
\sa{ML} practice\xzrev{,} and gradients are often estimated stochastically based on mini-batches (randomly sampled subset of data points) as in the case of stochastic gradient-type algorithms.}

\mg{There is a growing literature on the WCSC and WCMC problems in the stochastic setting. Several metrics for quantifying the quality of an approximate solution to \eqref{eq:main problem} have been proposed in the literature. A common way to assess the performance is to define the \emph{primal function} $\phi (\cdot) \sa{\triangleq} \max_{y\in \cY} \mathcal{L}(\cdot,y)$ and measure the \sa{violation of} first-order
\sa{necessary conditions for}} 
the non-convex problem $\min_{x\in \cX} \phi (x)$.
\sa{Given the primal iterate sequence $\{x_k\}_{k\geq 0}$ of a stochastic SP algorithm and a threshold $\epsilon>0$, a commonly used \sa{metric} is the \emph{gradient norm of
the Moreau envelope}~\sa{(GNME)}; indeed, the objective is to provide a bound $K_\epsilon$ such that $\mathbb{E}[\norm{\nsa{\grad}\phi_\lambda(x_k)}]\leq\epsilon$ for all $k\geq K_\epsilon$,} where $\phi_\lambda$ denotes the Moreau envelope of the primal function $\phi$ \sa{--see Definitions}~\ref{Def: Moreau envelope},~\ref{def:primal-function} and \ref{Def: stationary point}. Another commonly used natural metric is the \emph{gradient norm of the primal function} $\phi(\cdot)$~\cite{boct2020alternating,huang2021efficient,huang2022accelerated,luo2020stochastic,xu2020enhanced}, abbreviated as GNP, 
\sa{where the aim is to derive $K_\epsilon$ such that $\mathbb{E}[ \|\nabla\phi(x_k)\|] \leq \epsilon$ for all $k\geq K_\epsilon$.}  Other metrics such as the notion of $\epsilon$-first-order Nash equilibrium (FNE) and its generalized versions also exist in the literature \cite{nouiehed2019solving,ostrovskii2021efficient}.

\sa{When using any of the \nsa{aforementioned metrics}, the ultimate goal is to establish a bound on the oracle (sampling) complexity, i.e., $\sum_{k=0}^{K_\epsilon}b_k$, where $b_k$ denotes the batch-size for iteration $k\geq 0$.} 
\sa{\nsa{For the WCSC setting,} it crucial to note that GNME, GNP and FNE metrics are 
all} equivalent in the sense that convergence in either of them implies convergence in \nsa{the other two metric} for WCSC problems~\cite{lin2020gradient}. In this paper, \sa{for the WCSC setting,} we adopt \sa{both} GNME \mg{and GNP} as \mg{the main performance metrics} to analyze our algorithms; \sa{indeed, in \cref{thm:metric-equivalance} we show that}, {\color{black} when the non-smooth part $f(\cdot)=0$}, 
we can convert a GNME guarantee to a GNP guarantee by incurring only little additional cost compared to the \nsa{computational cost required for the} GNME guarantee, \nsa{and} \sa{the overall worst-case complexity \mg{(in terms of worst-case dependency to the target accuracy $\epsilon$)} remains the same for both metrics}. {\color{black} When the non-smooth part $f(\cdot)\neq 0$, we also obtain similar guarantees and \sa{show equivalence between the metrics based on GNME and \xzrev{\emph{the generalized gradient mapping}}.}}
\sa{On the other hand, for the WCMC setting, we provide our guarantees in GNME metric as $\phi$ is not necessarily differentiable for this scenario.} 
Moreover, our work accounts for the individual effects of $L_{xx}$, $L_{xy}$, $L_{yx}$ and $L_{yy}$, i.e., the Lipschitz constants of $\nabla_x\Phi(\cdot, y)$, $\nabla_x\Phi(x, \cdot)$, $\nabla_y\Phi(x, \cdot)$ and $\nabla_y\Phi(\cdot, y)$~(see \cref{ASPT: lipshiz gradient}), respectively, instead of using the worst-case 
{parameters} $L\triangleq\max\{L_{xx}, L_{xy}, L_{yx}, L_{yy}\}$, while 
\sa{the majority of related work} ignore the influence of 
\sa{these block Lipschitz constants \nsa{in their analyses.}} 
We emphasize that using the worst-case parameters will lead to a theoretically conservative step sizes, and this phenomenon has been validated in the work~\cite{zhang2021robust}.\looseness=-10
\vspace*{-2mm}
\paragraph{Contributions.} Table~\ref{table-one} summarizes the relevant existing work for WCSC and WCMC problems closest \nsa{to our setting.} 
\sa{More specifically, in Table \ref{table-one}, for the stochastic setting, we report the (oracle) complexity with respect to the GNP and GNME as the performance metrics for WCSC and WCMC problems, respectively, and the batch-size (number of data points in the mini-batches) required at every iteration.}
We also report whether the method is based on a variance-reduction (VR) technique. VR-based methods mentioned in Table~\ref{table-one} use a small batch-size \sa{$b'$}
\xzrev{all iterations except for few, where they}
need a large batch-size \sa{$b\geq b'$ once in} every $q$ iterations. 
\nsa{The period $q$ is equal to the number of times small batches are sampled consecutively plus one, and it} is \sa{also an 
algorithm} parameter. Therefore, for VR-methods, we report the batch size as a triplet $(b', b, q)$. 
In the column \nsa{``Compactness"}, we list whether achieving the specific complexity requires assuming compactness of the primal \sa{and/or} dual domains. 

\sa{To make the
comparison of our results with the existing work easier, we provide the results in the table for the worst-case setting, where} $\kappa_y \triangleq \frac{L}{\mu_y}$, and we 
\sa{report the $\epsilon$-, $\kappa_y$- and $L$-dependency 
of the complexity results for the existing} algorithms. That being said, our results have finer granularity in terms of their dependence to the individual effects of $L_{xx}$, $L_{xy}$, $L_{yx}$ and $L_{yy}$ as we mentioned earlier. \red{In the discussion below, $\cG_0$ denotes the primal-dual gap of the initial point, $\delta^2$ denotes the variance bound of the stochastic oracle for $\grad \Phi$, and $\cD_y\triangleq\sup_{y\in\dom g}\norm{y}$.} 

Our contributions (\sa{also} summarized in \cref{table-one}) are as follows:
\begin{itemize}[noitemsep,topsep=0pt,leftmargin=*,align=left]
    \item We propose a new stochastic 
    method, \texttt{SAPD+}, 
    based on the inexact proximal point method (iPPM). \sa{In this framework, one inexactly solves strongly convex-strongly concave (SCSC) saddle point sub-problems
    using an accelerated primal-dual method, SAPD~\cite{zhang2021robust}.} 
    In \cref{thm:unbounded-domains}, \sa{we 
    establish an oracle complexity of $\cO(L\kappa_y\epsilon^{-4})$} for WCSC problems, and unlike the majority of  existing work we do not require \sa{compactness for neither the primal nor the dual domain}. To our knowledge, 
    our bound has \sa{the best \mgrev{$\kappa_y$} 
    dependence 
    in the literature; indeed, prior to this work, without using variance reduction, the best known complexity was \xzrev{$\cO(L{\kappa_y^2}\epsilon^{-4})$}~shown in~\cite{yang2022faster}; hence, we establish a \xzrev{$\cO({\kappa_y})$} improvement.} \red{More precisely, the 
    complexity bound for \texttt{SAPD+} is $\cO\Big(\kappa_y\max\Big\{1,\frac{\delta^2}{\epsilon^2}\Big\}\frac{L\cG_0}{\epsilon^{2}}\Big)$ for all $\epsilon>0$ --for details, see \cref{rem:SAPD+-complexity}.}
    \item We propose a variance-reduced version of 
    \texttt{SAPD+} in \cref{cor:complexity VR}. \sa{For the WCSC setting, \texttt{SAPD+} \textit{with} variance reduction has improved the existing best known bound for a first-order method employing variance reduction.} Indeed, prior to our work the best bound was $\cO(L\kappa_y^3\red{(1+\delta^2)}\epsilon^{-3})$~\cite{luo2020stochastic}; \red{the  oracle complexity bound for \texttt{SAPD+} is $\mathcal{O}\Big(\nsa{\max\Big\{\kappa_y,\sqrt{\frac{\delta}{\epsilon}}\Big\}\cdot (\kappa_y\tfrac{\delta}{\epsilon}+1)}~\cdot\tfrac{L\cG_0}{\epsilon^2}\Big)$ for all $\epsilon>0$ --for details, see \cref{rem:SAPD+VR-complexity}.}
    \item \red{One should choose between \texttt{SAPD+} with or without variance reduction depending on the level of accuracy $\epsilon>0$ desired. Indeed, based on the previous two items, it is apparent that \texttt{SAPD+} (\textit{without} VR) has a better complexity bound than \texttt{SAPD+} \textit{with} VR whenever $\epsilon>\frac{\delta}{\kappa_y}$. More precisely, for $\epsilon>\delta$, \texttt{SAPD+} complexity is $\cO(L\kappa_y\cG_0/\epsilon^2)$ which is the same with the deterministic case; on the other hand, for $\epsilon\in [\frac{\delta}{\kappa_y},\delta]$, \texttt{SAPD+} complexity is $\cO(L\kappa_y\cG_0\delta^2/\epsilon^4)$. Furthermore, for smaller values of $\epsilon>0$, \texttt{SAPD+} \textit{with} VR becomes the method of choice, of which complexity is $\cO(L\kappa_y^2\cG_0\delta/\epsilon^3)$ for $\epsilon\in [\frac{\delta}{\kappa_y^2}, \frac{\delta}{\kappa_y}]$, and $\cO(L\kappa_y\cG_0\delta^{1.5}/\epsilon^{3.5})$ for $\epsilon<\frac{\delta}{\kappa_y^2}$.} 
    \footnote{\red{Compared to other VR-based methods with complexity $\cO(L\kappa_y^3/\epsilon^3)$, \texttt{SAPD+} \textit{with} VR
    has a better bound of $\cO(L\kappa_y^2/\epsilon^3)$ for $\epsilon>\frac{\delta}{\kappa_y^2}$; on the other hand, for $\epsilon\leq \frac{\delta}{\kappa_y^2}$, \texttt{SAPD+} \textit{with} VR still has a better bound as long as  $\kappa_y=\Omega(\epsilon^{-1/4})$. On a different but related note, in statistical learning, the regularization
parameter is usually 
$\cO(1/\sqrt{n})$ or $\cO(1/n)$ \cite{bousquet2002stability} where $n$ is the number of samples; thus, the condition number $\kappa_y$ is on the order of $\sqrt{n}$ or $n$.}}
    \item For the WCMC case, \red{our proposed algorithm \texttt{SAPD+} has $\cO\Big(\frac{L^3\cD_y^2}{\epsilon^4}\Big(1+\frac{\delta^2}{\epsilon^2}\Big)\Big)$ complexity,} which is the best 
    to our knowledge, \sa{improving the best known complexity by $\log^3(L/\epsilon^2)$ factor.}
    \item Finally, we demonstrate the efficiency of 
    \texttt{SAPD+} on a distributionally robust learning problem 
    and \sa{also} on a (worst-case) multi-class classification problem in deep learning. 
\end{itemize} 


\renewcommand{\arraystretch}{1.3}
\begin{table}[h]
\label{table-one}
{\scriptsize
\begin{center}
\begin{tabular}{c l c c c }
\hline
Ref. &  Complexity & Compactness & VR-based & \multicolumn{1}{c}{Batchsize}  \\
 \cline{2-5}
 \multicolumn{5}{c}{\textbf{Weakly Convex-Strongly Concave (WCSC) problems}} 
\\
\hline
$\mathbf{{}^*}$Rafique \emph{et al.}~\cite{rafique1810non} & $\cO(\epsilon^{-4}\log(\epsilon^{-1}))$ & (n, n)  &
\XSolidBrush  & $\cO(1)$ \\
$\mathbf{{}^{\dag}}$Yan \emph{et al.}~\cite{yan2020optimal} & $\cO(\epsilon^{-4}\log(\epsilon^{-1}))$ &   (y, y) & \XSolidBrush & $\cO(1)$ \\

$\mathbf{{}^{\dag}}$Yang \emph{et al.}~\cite{yang2022faster} &$\cO(L\kappa^2_y\epsilon^{-4})$ &  (n, n)  & \XSolidBrush & $\cO(1)$ \\

Lin \emph{et al.}~\cite{lin2020gradient} &  $\cO(L\kappa^3_y\epsilon^{-4})$ &  (n, y)  & \XSolidBrush & $\cO(\kappa_y\epsilon^{-2})$ \\

Bot and \"{B}ohm~~\cite{boct2020alternating} &  $\cO(L\kappa^3_y\epsilon^{-4})$ &  (n, n)   & \XSolidBrush & $\cO(\kappa_y\epsilon^{-2})$ \\

$\mathbf{{}^\ddagger}$Huang \emph{et al.}~\cite{huang2021efficient}&    $\cO(\kappa_y^5\mu_y^{-1}\epsilon^{-3})$ & (n, n)   &
\Checkmark & $\cO(\kappa_y\epsilon^{-1}),\;\cO(\kappa_y^2\epsilon^{-2}),\; \cO(\kappa_y\epsilon^{-1})$ \\

$\mathbf{{}^\S}$Huang \emph{et al.}~\cite{huang2022accelerated}& $\tilde{\cO}(L^{1.5}\kappa_y^{3.5}\epsilon^{-3})$ & (y, y)  &
\Checkmark & $\cO(\sqrt{\kappa_y})$ \\

Luo \emph{et al.}~\cite{luo2020stochastic}& $\cO(L\kappa_y^{3}\epsilon^{-3})$ &   (y, y) & \Checkmark & $\cO(\kappa_y\epsilon^{-1}),\;\cO(\kappa_y^2\epsilon^{-2}),\;\cO(\kappa_y\epsilon^{-1}) $ \\

Xu \emph{et al.}~\cite{xu2020enhanced}& $\cO(L\kappa_y^{3}\epsilon^{-3})$ &  (y, y)  & \Checkmark & $\cO(\kappa_y\epsilon^{-1}),\;\cO(\kappa_y^2\epsilon^{-2}),\; \cO(\kappa_y\epsilon^{-1})$  \\

\hdashline
\texttt{SAPD+}, \cref{thm:unbounded-domains} & $\cO(L\kappa_y\epsilon^{-4})$ &  (n, n) & \XSolidBrush &  $\cO(1) $ \\ 

\texttt{SAPD+}, \cref{cor:complexity VR} &  {$\cO(\max\{\kappa_y,\epsilon^{-1/2}\}\cdot L\kappa_y\epsilon^{-3})$} &  (n, n)   &  \Checkmark &  $\cO(\kappa_y\epsilon^{-1}),\; \cO(\kappa_y\epsilon^{-2}),\; \cO(\epsilon^{-1})$  \\

\hline
 \multicolumn{5}{c}{\textbf{Weakly Convex-Merely Concave (WCMC) problems}}\\ 
\hline
Rafique \emph{et al.}~\cite{rafique1810non} &  $\sa{\cO(L^3\epsilon^{-6}\log^3(L\epsilon^{-2}))}$  &  (y, y)  &
- & $\cO(1)$ \\

Bot and \"{B}ohm~~\cite{boct2020alternating} &  $\cO(L^5\epsilon^{-8})$ &  (n, y)  & - & $\cO(1)$\\

Lin \emph{et al.}~\cite{lin2020gradient} &  $\cO(L^{3}\epsilon^{-8})$ &  (n, y)  & - & $\cO(1)$ \\
\hdashline
\texttt{SAPD+}, \cref{Thm:WCMC} &  $\cO(\sa{L^3}\epsilon^{-6})$ &  (n, y) &  - &  $\cO(1)$ \\
\hline
\end{tabular}
\vspace*{2mm}
\caption{{\footnotesize Summary of relevant work \sa{for WCSC and WCMC} problems. For \mg{the column} ``Compactness'', we use \nsa{\texttt{y} and \texttt{n} to indicate when the results require compactness and when do not require it,} respectively; the first argument is for primal domain and the second is for dual domain.
For batchsize, \sa{we use \sa{$(b',b,q)$} format for VR-based methods 
\nsa{to state \emph{small batch} $(b')$, \emph{large batch} $(b)$, and \emph{frequency} $(q)$ employed within the algorithm}}.\\[1mm]
\textbf{Table notes:} \sa{$\mathbf{{}^*}$For WCSC setting, \cite{rafique1810non} assumes 
$\Phi(\cdot,y) \triangleq c^{\top}(\cdot)y$ is weakly convex and $g(\cdot)$ is strongly convex. ${}^\dag$
In~\cite{yan2020optimal}, $\cL=\Phi$ and $\Phi$ need not be smooth, rather second moment of stochastic subgradients is assumed to be uniformly bounded. \nsa{When $\Phi$ is $L$-smooth, $\Phi(\cdot,y)$ and $\Phi(x,\cdot)$ are $L_\Phi$-Lipschitz}, the results in~\cite{yan2020optimal} imply $\cO(L_\Phi^2\kappa_y^2\epsilon^{-4}\log^2(\sqrt{\kappa_y}L_\Phi/\epsilon))$ complexity. 
${}^{\ddagger,\S}$The complexity \mgrev{results} 
reported here are different than those in~\cite{huang2021efficient,huang2022accelerated}. The issues in their proofs leading to the wrong complexity results are explained in Appendix~I. \mgrev{The notation $\tilde{\cO}$ ignores logarithmic factors}.\looseness=-1}
}} 
\end{center}}
\vspace*{-2mm}
\end{table}
\vspace*{-2mm}
\paragraph{Notation.} 
Throughout the paper, $\|\cdot\|$ denotes the Euclidean norm. 
Given $f:\reals^n\to\reals\cup\{\infty\}$ a closed convex function,
$\prox{\lambda f}(x)\triangleq\argmin_{w} f(w)+\frac{1}{2\lambda}\|w-x\|^2$ denotes the proximal map of $f$. Given \mg{random} $\omega$, let $\tilde{\nabla}_x\Phi(x,y;\omega)$ and $\tilde{\nabla}_y\Phi(x,y;\omega)$ \nsa{denote} \mg{unbiased} estimators of $\nabla\Phi_x(x,y)$ and $\nabla\Phi_y(x,y)$. Moreover,
given a random mini-batch $\cB=\{\omega_i\}_{i=1}^b$, we let $\tilde{\nabla}_x\Phi_{\cB}(x,y)\xzrev{\triangleq}\frac{1}{b}\sum_{i=1}^b\tilde{\nabla}_x\Phi(x,y;\omega_i)$ to denote the stochastic gradient estimate based on the batch $\mathcal{B}$, \sa{and we define $\tilde{\nabla}_y\Phi_{\cB}(\cdot,\cdot)$ similarly.}\vspace*{-2mm}

\section{Preliminaries}\label{sec-preliminaries}
\vspace*{-2mm}
\sa{ \mg{We start with describing the notion of weak convexity.}}
\begin{definition}\label{def: weakly convex}
\sa{$\nsa{h}:\reals^d\to\reals\cup\{+\infty\}$ is $\gamma$-weakly convex if $x\mapsto \nsa{h}(x)+\frac{\gamma}{2}\|x\|^2$ is convex.}
\end{definition}
\begin{definition}\label{def: l-smooth}
A \sa{differentiable function $\nsa{h}:\reals^d\to\reals\cup\{+\infty\}$} is $L$-smooth if \sa{$\exists L>0$ such that} for $\forall x,x'\in\dom \nsa{h}$, $\|\nabla \nsa{h}(x) - \nabla \nsa{h}(x')\|\leq L\|x-x'\|$. 
\end{definition}
\begin{remark}
\label{rem:smooth-WC}
\sa{If a function is $L$-smooth, then it is also $L$-weakly convex.}
\end{remark}
\vspace*{-2mm}
Remark \ref{rem:smooth-WC} shows that weak convexity is a rich class containing the class of smooth functions. 
\mg{In the rest of the paper, we consider the SP problem in~\eqref{eq:main problem}. Next, we introduce our assumptions.}

\begin{assumption}
\label{ASPT: fg}
$f:\mathcal{X}\rightarrow \reals\cup\{\xzrev{+}\infty\}$ and $g:\mathcal{Y}\rightarrow \reals\cup\{\xzrev{+}\infty\}$ are proper, closed, convex functions. 
Let $\Phi:\cX\times\cY\to\reals$ be such that 
    (i) for any $y\in \dom g\subset \mathcal{Y}$,
    $\Phi(\cdot,y)$ is $\gamma$-weakly convex and bounded from below;
    (ii) for any $x\in\dom f\subset \mathcal{X}$,
    $\Phi(x, \cdot)$ is $\mu_y$-strongly concave \sa{for some $\mu_y\geq 0$}; \sa{(iii) $\Phi$ is differentiable on an open set containing $\dom f\times \dom g$.}
\end{assumption}

\begin{assumption}\label{ASPT: lipshiz gradient}
$\exists L_{xx},L_{yy} \geq 0$, $L_{xy},L_{yx} > 0$ such that
        $\| \nabla_x \Phi(x,y) - \nabla_x \Phi(\bar{x},\bar{y})\|
        \leq 
        L_{xx}\|x- \bar{x}\|+L_{xy}\|y-\bar{y}\|$, and $\| \nabla_y \Phi(x,y) - \nabla_y \Phi(\bar{x},\bar{y})\|
        \leq 
        L_{yx}\|x- \bar{x}\|
        + L_{yy}\|y- \bar{y}\|$ for all $x,\bar{x}\in \dom f\subset \mathcal{X}$, and $y,\bar{y}\in \dom g\subset \mathcal{Y}$.
\end{assumption}

{Assumption \ref{ASPT: fg} allows non-convexity in $x$ while requiring \sa{(strong)} concavity in the $y$ variable.} 
Assumption~\ref{ASPT: lipshiz gradient} is standard in the analysis of first-order methods \sa{for solving SP problems}. It should be noticed that when $L_{yx}=L_{xy}=0$, the problem \sa{in \eqref{eq:main problem}} can be solved separately for the primal and dual variables; hence, it is natural to assume $L_{yx}, L_{xy}>0$.

\sa{Suppose that we implement \texttt{SAPD}, stated in \mg{Algorithm} \ref{Alg: SAPD}, 
on the following SCSC problem: 
{
\begin{align}
\label{eq:WCSC-subproblem-generic}
\min_{x\in\cX}\max_{y\in\cY}\cL(x,y)+\frac{\mu_x+\gamma}{2}\norm{x-x_0}^2    
\end{align}}%
for some given $\mu_x>0$ and $x_0\in\cX$ --strong convexity follows from $\cL(\cdot,y)$ being $\gamma$-weakly convex.} 
\begin{algorithm}[H]
\caption{\texttt{SAPD} Algorithm}
{\footnotesize
\begin{algorithmic}[1]
\State {\textbf{Input:} $\tau,\sigma,\theta,\mu_x,x_0, y_0,N$}
\State \sa{$\bar{\Phi}(x,y)\gets\Phi(x,y)+\frac{\mu_x+\gamma}{2}\norm{x-x_0}^2$}
\State {$\tilde q_0\gets 0$}
\For{$k=0,1,2,...,N$}
\State $\tilde{s}_k \leftarrow \sa{\tilde\nabla_y} \sa{{\Phi}}(x_k,y_k;\sa{\omega_k^y}) + \theta \tilde{q}_k$
\State $y_{k+1}\leftarrow\prox{\sigma g}(y_k+\sigma \tilde{s}_k)$
\State $x_{k+1}\leftarrow\prox{\tau f}(x_k-\tau\sa{\tilde\nabla_x}\bar{\Phi}(x_k,y_{k+1};\sa{\omega_k^x}))$
\State \sa{$\tilde{q}_{k+1}\leftarrow\sa{\tilde\nabla_y} {\Phi}(x_{k+1},y_{k+1};\sa{\omega_{k+1}^y}) -\sa{\tilde\nabla_y}{\Phi}(x_{k},y_{k};\sa{\omega_{k}^y})$}
\EndFor
\State {\textbf{Output}:}{\sa{$(\bar{x}_N,\bar y_N)=\frac{1}{N}\sum_{k=0}^{N-1}(x_{k+1},y_{k+1})$}}
\end{algorithmic}}%
\label{Alg: SAPD}
\end{algorithm}

\sa{
We 
make the following assumption on the statistical nature of the gradient noise as in, e.g., \cite{can2022,fallah2020optimal,zhang2021robust}}.

\begin{assumption}\label{ASPT: unbiased noise assumption}
\sa{Given arbitrary $x_0\in\cX$ and $\mu_x>0$, let $\{x_k,y_k\}$ \nsa{sequence} be generated by \texttt{SAPD}, stated in \mgrev{Algorithm~\ref{Alg: SAPD},} running on \eqref{eq:WCSC-subproblem-generic}.} There exist $\delta_x,\delta_y\geq 0$ such that for
all \sa{$k\geq 0$}, 
the stochastic gradients $\tilde{\nabla}_x\Phi(x_k,y_{\sa{k+1}};\omega_k^x)$, $\tilde{\nabla}_y\Phi(x_k,y_k;\omega_k^y)$ and random sequences $\{\omega_k^x\}_k$, $\{\omega_k^y\}_k$ 
satisfy the following conditions:\\[-4mm]
\begin{enumerate}[label=(\roman*),itemsep=-1ex,topsep=0pt]
    \item 
    $\mathbb{E}[\tilde{\nabla}_x\Phi(x_k,y_{k+1};\omega_k^x)|x_k,y_{k+1} ] = \nabla_x\Phi(x_k,y_{k+1})$;
    \item 
    $\mathbb{E}[\tilde{\nabla}_y\Phi(x_k,y_{k};\omega_k^y)|x_k,y_k ] = \nabla_y\Phi(x_k,y_k)$;
    \item 
    $\mathbb{E}[\|\tilde{\nabla}_x\Phi(x_k,y_{k+1};\omega_k^x) - \nabla_x\Phi(x_k,y_{k+1})\|^2 |x_k,y_{k+1} ]\leq \delta_x^2$;
    \item 
    $\mathbb{E}[\|\tilde{\nabla}_y\Phi(x_k,y_{k};\omega_k^y) - \nabla_y\Phi(x_k,y_k)\|^2 |x_k,y_k ]\leq \delta_y^2$.
\end{enumerate}
\end{assumption}
{\cref{ASPT: unbiased noise assumption} says that the gradient noise conditioned 
\sa{on} the iterates is unbiased with a finite variance\footnote{\nsa{When we run \texttt{SAPD}, stated in {Algorithm~\ref{Alg: SAPD},} on \eqref{eq:WCSC-subproblem-generic}, we use the convention that $\tilde \nabla_x \bar \Phi(x_k,y_{k+1};\omega_k^x) \triangleq \tilde \nabla_x \Phi(x_k,y_{k+1};\omega_k^x)  + (\mu_x + \gamma)(x_k-x_0)$.}}. Such assumptions are common in the literature, e.g., \cite{can2022,fallah2020optimal,zhang2021robust}, and are satisfied when gradients are estimated from randomly sampled data points with replacement.} 

For \sa{WCSC minimax problems}, \sa{a commonly adopted definition for $\epsilon$-stationary is based} on Moreau envelope, \sa{e.g., see~\cite{lin2020gradient,yan2020optimal}}. It is inspired by Davis and \sa{Drusvyatskiy's work}~\cite{davis2019stochastic} for solving weakly convex minimization problems. \sa{For the sake of completeness, we briefly review this idea below.}
\begin{definition}\label{Def: Moreau envelope}
\sa{Let $\phi:\mathbb{R}^d\rightarrow\mathbb{R}\cup\{+\infty\}$ be $\gamma$-weakly convex. Then, for any  $\lambda\in(0,\gamma^{-1})$, Moreau envelope of $\phi$ is defined as $\phi_{\lambda}:\mathbb{R}^d\rightarrow\mathbb{R}$ \sa{such that} $\phi_{\lambda}(x)\triangleq\min_{w\in\cX} \phi(w) + \frac{1}{2\lambda}\|w-x \|^2$.} 
\end{definition}
\begin{lemma}\label{Lemma: graident of ME}
\sa{Let $\phi:\mathbb{R}^d\rightarrow\mathbb{R}\cup\{+\infty\}$ be a $\gamma$-weakly convex function. For any given $\lambda\in(0,\gamma^{-1})$, $\phi_{\lambda}(\cdot)$} is {well-defined} on $\cX$. Moreover, \sa{$\nabla \phi_{\lambda}(x) = \lambda^{-1}(x - \prox{\lambda\phi}(x))$ for $x\in\cX$; hence, $\phi_\lambda$ is $\lambda^{-1}$-smooth, where $\prox{\lambda\phi}(\sa{{x}}) \xzrev{\triangleq} \argmin_{\sa{w}\in \cX} \{ \phi(w)+\tfrac{1}{2\lambda}\|\sa{w} - \sa{{x}}\|^2\}$.}%
\end{lemma}
\begin{definition}\label{def:primal-function}
\sa{Under \cref{ASPT: fg}, let $\phi,\phi^s:\reals^d\to\reals\cup\{+\infty\}$ such that $\phi(x) \triangleq \max_{y\in\cY}\mathcal{L}(x,y)$
and $\phi^s(x)=\phi(x)-f(x)$ for \xzrev{$x\in\dom f$}, i.e., $\phi^s(x)\triangleq\max_{y\in\cY}\Phi(x,y)-g(y)$ for \xzrev{$x\in\dom f$}.}
\end{definition}
\begin{remark}
\sa{Under \cref{ASPT: fg},  since \sa{$\Phi(\cdot,y)$ is $\gamma$-weakly convex for any $y\in\dom g$, $\phi^s$ is $\gamma$-weakly convex}}\footnote{One can argue that $\phi^s(\cdot) + \frac{\gamma}{2}\|\cdot\|^2$ is \sa{a pointwise supremum} of convex functions.}; hence, $\phi$ is also $\gamma$-weakly convex. Note that
\begin{equation}
\label{eq:prox-SCSC}
    \prox{\lambda\phi}(\sa{{x}}) = \argmin_{\sa{w}\in \cX} \{ \phi(w)+\tfrac{1}{2\lambda}\|\sa{w} - \sa{{x}}\|^2\}=\argmin_{\sa{w}\in \cX} \max_{y\in\cY}  \mathcal{L}(\sa{w},y)+\tfrac{1}{2\lambda}\|\sa{w} - \sa{{x}}\|^2.
\end{equation}
Furthermore, when $\mu_y>0$, $\phi^s$ is differentiable \sa{on $\dom f$}.
\end{remark}

\mg{In the following definition, we introduce the notion of $\epsilon$-stationary with respect to the GNME metric.} 
\begin{definition}\label{Def: stationary point}
A point \sa{$x_\epsilon$ is an $\epsilon$-stationary point of a $\gamma$-weakly convex function $\phi$} if $\|\nabla\phi_{\lambda}(x_\epsilon)\|\leq \epsilon$ for some $\lambda\in(0,\gamma^{-1})$. If $\epsilon=0$, then \sa{$x_\epsilon$} is a stationary point \sa{of $\phi$.}
\end{definition}
\sa{Thus, \mg{according to Lemma \ref{Lemma: graident of ME},} computing an $\epsilon$-stationary point $x_\epsilon$ for $\phi$ is equivalent to searching for $x_\epsilon$ such that $\|x_\epsilon-\prox{\lambda\phi}(x_\epsilon) \|$ is small.} 
\sa{Recall that for any $\lambda\in(0,\gamma^{-1})$, 
$\prox{\lambda\phi}(x)$ is well-defined and unique. 
\mg{We also \sa{observe} from} \eqref{eq:prox-SCSC} that  $\prox{\lambda\phi}(\cdot)$ computation is indeed an SCSC SP problem.
To compute $x_\epsilon$ such that $\|x_\epsilon-\prox{\lambda\phi}(x_\epsilon) \|$ is small,
it is natural to consider the \mg{iPPM algorithm} -- \xzrev{e.g., see}~\cite{iusem2003inexact}. A generic iPPM generates $\{x_0^t\}_{t\geq 0}$ such that $x^{t+1}_{0}\approx \prox{\lambda\phi}(x^{t}_{0})$\mg{, i.e.\xzrev{,} proximal steps are} ``inexactly'' computed for $t\geq 0$, starting from an arbitrary given point $x^0_0\in\cX$.} 

\nsa{In the next section,} we 
\sa{describe} 
\nsa{the proposed
\texttt{SAPD+} method}, \sa{an iPPM algorithm employing \texttt{SAPD} to \sa{\emph{inexactly}} solve the SCSC subproblems arising in the iPPM iterations.}
 

\begin{algorithm}[H]
\caption{\sa{\texttt{SAPD+}} Algorithm} 
{\footnotesize
\begin{algorithmic}[1]
\State \textbf{Input:} $\{\tau,\sigma,\theta,\mu_x\}$, $(x^0_0, y^0_0) \in \mathcal{X} \times \mathcal{Y}$, $\{N_t\}_{t\geq 0}\in \mathbb{Z}^+$
\For{$t=0,{1},2,...,T$}
\If{\texttt{VR-flag} == \textbf{false}}
\State $(x^{t+1}_{0},y^{t+1}_{0})\gets \texttt{SAPD}(\tau,\sigma,\theta,\mu_x,x^{t}_{0},y^{t}_{0},N_t)$ \label{SAPD step chpt 2}
\Else
\State $(x^{t+1}_{0},y^{t+1}_{0})\gets \texttt{VR-SAPD}(\tau,\sigma,\theta,\mu_x,x^{t}_{0},y^{t}_{0},N_t)$
\EndIf
\EndFor
\end{algorithmic}}%
\label{Alg: SAPD-plus}
\end{algorithm}

\section{\nsa{The proposed algorithm {\texttt{SAPD+}} and its analysis}}\label{subsection: Algorithm and analysis}
\sa{The convergence and robustness properties of \texttt{SAPD} for SCSC 
SP problems are analyzed in~\cite{zhang2021robust}. For the WCSC 
SP~problems, as we explained in the previous \nsa{section}, 
\xzrev{the} main idea is to apply \sa{the iPPM framework as stated in~\texttt{SAPD+}~(see~Algorithm~\ref{Alg: SAPD-plus})} which requires successively solving SCSC SP problems. In the rest, the counter for iPPM \nsa{outer iterations} is denoted with $t\in\integers_+$. At each outer iteration $t\geq 0$, we inexactly compute the prox map, i.e., $x^{t+1}_{0}\approx \prox{\lambda\phi}(x^{t}_{0})$, which is well-defined for $\lambda\in(0,\gamma^{-1})$; hence, to derive our preliminary results, we fix $\lambda=(\mu_x+\gamma)^{-1}$ for some given $\mu_x>0$ -- thus, $\cL(x,y)+\frac{\mu_x+\gamma}{2}\norm{x-x^{t}_{0}}^2$ is SCSC in $(x,y)$ with moduli $(\mu_x,\mu_y)$ and has a unique saddle point. Consider the following SCSC SP problem:
{
\begin{equation}\label{eq: SCSC problem chpt2}
  \min_{x\in\cX} \max_{y\in\cY}
  \cL^{t}(x,y)\triangleq f(x)+\Phi^{t}(x,y)-g(y),~\text{where}~\Phi^{t}(x,y) \xzrev{\triangleq}\Phi(x,y) + \frac{\mu_x+\gamma}{2}\|x-x^{t}_{0}\|^2.
\end{equation}}}%

\sa{
We will construct $\{x^{t}_{0}\}_{t=1}^{T}\subset \dom f$ 
by \emph{inexactly} solving 
\eqref{eq: SCSC problem chpt2} at each outer iteration $t\in\integers_+$ through running \texttt{SAPD} for $N_t\in\integers_+$ iterations} --we will specify $N_t\in\integers_+$ later. 
\sa{Next, we 
briefly explain the main step of \texttt{SAPD+} with \texttt{VR-flag}=\textbf{false}.}
\nsa{The statement 
in line \ref{SAPD step chpt 2} of \xzrev{Algorithm}~\ref{Alg: SAPD-plus} means that} $(x^{t+1}_{0},y^{t+1}_{0})$ is generated using \texttt{SAPD}, which is displayed in \xzrev{Algorithm}~\ref{Alg: SAPD} --indeed, \texttt{SAPD} is run on \eqref{eq: SCSC problem chpt2} for $N_t$ iterations with \texttt{SAPD} parameters $(\tau,\sigma,\theta)$ and starting from the initial point $(x^{t}_{0},y^{t}_{0})$. 
To analyze the convergence of \xz{\texttt{SAPD+}}, we first define the gap function $\cG^{t}$ for $t$-th outer iteration:
\begin{equation}
    \cG^{t}(x,y) \triangleq \max_{y'\in\cY} \cL^{t}(x,y') -  \min_{x'\in\cX} \cL^{t}(x',y).
\end{equation}
\sa{Recall that $\cL^{t}$ is an SCSC function; therefore, \emph{i)} it has a unique saddle point denoted by $(x^{t}_{*},y^{t}_{*})$, and \sa{it is important to note that $x^{t}_{*}=\prox{\lambda\phi}(x^{t}_{0})$ for $\phi(x)=\max_{y\in\cY}\cL(x,y)$ and $\lambda=(\gamma+\mu_x)^{-1}$;} \emph{ii)} {for any $(x,y)\in\dom f\times\dom g$},
the following quantities are well-defined:}
\begin{equation}\label{eq: x* y*}
    x^{t}_{*}(y) \triangleq\argmin_{x'\in\cX} \cL^{t}(x',y),\quad y_{*}(x) \triangleq \argmax_{y'\in\cY} \cL^{t}(x,y')\sa{=\argmax_{y'\in\cY} \mathcal{L}(x,y').}
\end{equation}
Thus, it follows that
$\cG^{t}(x,y) = \cL^{t}(x,y_{*}(x)) - \cL^{t}(x^{t}_{*}(y), y)$.
Moreover, for $(x,y)\in\dom f\times\dom g$, we also define 
    $\cG(x,y) \triangleq \sup_{y'\in\cY} \mathcal{L}(x,y') -  \inf_{x'\in\cX} \mathcal{L}(x',y)$.
\sa{\cref{ASPT: fg}} ensures that $\cG$ 
is well defined. 

\sa{\nsa{Next,} we first provide our oracle complexity 
\mg{in the GNME metric} under the compactness assumption of the primal-dual domains; later, in 
\cref{sec:compactness}, we show that under \mg{a} particular subdifferentiability assumption, \na{in particular when $f$ and $g$ are both Lipschitz on their domains,} compactness requirement can be avoided.}
\begin{assumption}
\label{assump:compact}
\sa{$\dom f$ and $\dom g$ are compact sets.}
\end{assumption}
\begin{theorem}
\label{cor:complexity-fg}
Suppose Assumptions~\ref{ASPT: fg},~\ref{ASPT: lipshiz gradient},~\ref{ASPT: unbiased noise assumption}, and~\ref{assump:compact} hold. Let $\mu_x=\gamma$, $\theta=1$, $\tau,\sigma$ and $N$ be chosen as
{\small
\begin{equation}
\label{Condition: SP solution to noisy LMI label-gap}
N = 33\max\{\tfrac{4}{\gamma\tau},\tfrac{8}{\mu_y\sigma}\},\quad \tau = \min\{\tfrac{1}{L_{yx}+L_{xx}+2\gamma},\tfrac{1}{L_{xy}},\tfrac{1}{480\gamma}\cdot\tfrac{\epsilon^2}{\delta_x^2} \},\quad \sigma = \min\{\tfrac{1}{L_{yx}+2L_{yy}},\tfrac{1}{4512\gamma}\cdot\tfrac{\epsilon^2}{\delta_y^2}\}.
\end{equation}}%
\nsa{Then, for any} $\epsilon>0$, \xz{when \texttt{VR-flag}=\textbf{false}},
\sa{\xz{\texttt{SAPD+}} \nsa{guarantees that} 
\begin{align}
\label{eq:main-bound}
    \na{\frac{1}{T+1}\sum_{t=0}^{T} \mathbb{E}\left[\|\nabla\phi_{\lambda}(x^t_0) \|^2 \right]
        \leq\epsilon^2,}
\end{align}
for \na{all} $T \geq \sa{96}\cG(x_0^0,y_0^0)  \cdot\tfrac{\gamma}{\epsilon^2}+1$, and 
\na{computing a point $x_\epsilon$ such that $\mathbb{E}[\|\nabla\phi_\lambda(x_\epsilon)\|]\leq \epsilon$} requires $C_\epsilon$ \nsa{stochastic first-order} oracle calls in total where}
{\small
\begin{align*}
\sa{C_\epsilon=}\mathcal{O}\Big(  
\Big(\tfrac{\max\{L_{xx},L_{yx},L_{xy}\}}{\gamma}  + \tfrac{\max\{L_{yy},L_{yx}\}}{\mu_y}\Big)
\gamma\cdot\epsilon^{-2}+ \Big(
\tfrac{\delta_x^2}{\gamma} +  
\tfrac{\delta_y^2}{\mu_y} \Big)
\sa{\gamma^2}\cdot\epsilon^{-4}\Big)\cG(x_0^0,y_0^0).
\end{align*}}%
\end{theorem}
\begin{proof}
\nsa{See~\cref{sec:construction} for the proof.}
\end{proof}
\begin{remark}\label{remark-exchange-min and expectation}
\na{Choosing $t_*$ uniformly at random among $\{0,1,\ldots,T\}$ and setting $x_\epsilon=x_0^{t_*}$ implies that $\mathbb{E}\left[\|\nabla\phi_{\lambda}(x_\epsilon) \| \right]\leq \epsilon$. One disadvantage of this commonly used simple randomization approach is that the variance of $\|\nabla\phi_{\lambda}(x_\epsilon) \|$ might be large.}
\nsa{\na{Alternatively,} since $\mathbb{E}\left[\min_{t=0,\ldots,T}\|\nabla\phi_{\lambda}(x^t_0) \| \right] \leq \min_{t=0,\ldots,T} \mathbb{E}\left[\|\nabla\phi_{\lambda}(x^t_0) \| \right]$, the guarantees given in \cref{cor:complexity-fg} also hold for achieving  $\mathbb{E}\big[\min_{t=0,\ldots,T}\|\nabla\phi_{\lambda}(x^t_0) \|\big]\leq \epsilon$.} \na{Naively, the random vector $x_\varepsilon$ in Theorem \ref{cor:complexity-fg} can also be chosen as $x_0^{t_*}$ where $t_*=\argmin_{0\leq t \leq T}\|\nabla\phi_{\lambda}(x^t_0) \|$. However, this choice of $t_*$ can not be computed trivially since computing $\nabla\phi_{\lambda}(\cdot)$ requires solving an SCSC minimax problem. That said, we give the details of such a method in the appendix to generate a point $x_\epsilon$ such that $\mathbb{E}[\|\nabla\phi_\lambda(x_\epsilon)\|]\leq \epsilon$ within $\tilde{\cO}\left(\frac{L\kappa_y\cG(x_0^0,y_0^0) }{\epsilon^2}+ \frac{ L\kappa_y\delta^2\cG(x_0^0,y_0^0) }{\epsilon^4} \right)$ stochastic first-order oracle calls -- see \cref{thm:find-min-GNME} in \cref{sec:eps-stationary}. It is plausible to expect that the variability of this approach would be smaller that the randomization approach stated at the beginning of the remark.} 
\end{remark}

\begin{remark}
\label{rem:SAPD+-complexity}
\sa{For any $y\in\dom g$, since $\Phi(\cdot,y)$ $L_{xx}$-smooth, it is necessarily $L_{xx}$-weakly convex; hence, $\gamma\leq L_{xx}$. To get a worst-case complexity, let}
\begin{equation}\label{eq:uniform setting}
   L \triangleq \max\{L_{xy},L_{yx},L_{xx},L_{yy}\},~\kappa_y \triangleq L/\mu_y,~\delta \triangleq \max\{\delta_x,\delta_y\},~\gamma = L. 
\end{equation}
\sa{Our oracle complexity \small{$C_\epsilon$} in \cref{cor:complexity-fg} can be simplified as} 
\small{$\sa{C_\epsilon}=\mathcal{O}\left(\max\{1, \tfrac{\delta^2}{\epsilon^2}\}\tfrac{\kappa_y L\cG(x_0^0,y_0^0)}{\epsilon^2} \right)$}.

In fact, Li et al.~\cite{li2021complexity} (see also~\cite{zhang2021complexity}) provide a lower complexity bound for a class of first-order stochastic algorithms that do not use variance reduction. The lower bound for finding $\epsilon$-stationary \sa{points of smooth WCSC problems}
\sa{in GNP metric} is  $\Omega(L\Delta_{\phi}(\sqrt{\kappa_y}\epsilon^{-2}+\kappa_y^{\frac{1}{3}}\epsilon^{-4}))$, where $\Delta_{\phi} \triangleq \phi(x_0) - \min_{x\in\mathcal{X}}\phi(x)$ \nsa{and $x_0$ is \mgrev{an arbitrary} 
initial point.}
\end{remark}
\sa{Consider $\phi=f+\phi^s$ as given in \cref{def:primal-function}.} \mg{For $\lambda>0$, the map $G_\lambda: \mathbb{R}^d \to \mathbb{R}^d$ defined as
{\small
\begin{equation}
\label{eq:generalized_GM}
G_\lambda(\tilde{x})\triangleq \frac{1}{\lambda} [\tilde{x} - 
\prox{\lambda f} 
\big(
\tilde{x} - \lambda\nabla\phi^s(\xz{\tilde{x}})
\big)]
\end{equation}}%
is called the\ \emph{generalized gradient mapping} and its norm is frequently used in optimization for assessing stationarity (see e.g. \cite{drusvyatskiy2018error}). 
\cref{cor:complexity-fg} provides guarantees in the GNME metric. \sa{
\cref{thm:metric-equivalance} shows that given $x_\epsilon$, an $\epsilon$-stationary point in GNME metric (see~\cref{Def: stationary point}) \nsa{in expectation}, we can generate $\tilde{x}$ such that $\nsa{\mathbb{E}[\norm{G_\lambda(\tilde{x})}]}\leq \epsilon$ for some $\lambda>0$, i.e.,
an $\epsilon$-stationary point in 
\emph{generalized gradient mapping} metric, 
within $\tilde{\cO}(1/\epsilon^2)$ \texttt{SAPD} 
iterations. \nsa{Indeed,} when $f(\cdot)=0$, \nsa{this metric 
and the GNP metric are the same.}}} 
\begin{theorem}
\label{thm:metric-equivalance}
\xz{Suppose Assumptions~\ref{ASPT: fg},~\ref{ASPT: lipshiz gradient},~\ref{ASPT: unbiased noise assumption} hold}, and \sa{$x_\epsilon$}, an $\epsilon$-stationary point for the $\gamma$-weakly convex function $\phi(\cdot)=\max_{y\in\mathcal{Y}} \cL(\cdot,y)$ \sa{in expectation}, i.e., $\sa{\mathbb{E}[\|\nabla\phi_{\lambda}(\sa{x_\epsilon})\|]}\leq\frac{\epsilon}{2}$ for 
\sa{some fixed} $\lambda\in(0,\gamma^{-1})$ is given. Then, 
\xzrev{there exists some $\tau,\sigma,\theta$ -- see \cref{Condition: SP solution to noisy LMI} in \cref{sec:Thm2_proof}, such that}
initialized from $x_\epsilon$, 
\texttt{SAPD}, stated in~Algorithm~\ref{Alg: SAPD}, can generate \sa{$\tilde x$} satisfying
\mg{$\mathbb{E}\big[ \norm{G_\lambda(\tilde{x})} \leq \epsilon$} within 
$\tilde{\mathcal{O}}(\frac{1}{\xz{\epsilon^2}})$ stochastic first-order oracle calls, where $\phi^s(\cdot)=\max_{y\in\cY}\Phi(\cdot,y)-g(y)$ so that $\phi=f+\phi^s$ as in Definition \ref{def:primal-function}.\vspace*{-2mm}
\end{theorem}
\begin{proof}
\nsa{See \cref{sec:Thm2_proof} for the proof.}
\end{proof}
\subsection{\sa{Relaxing the compactness assumption}}
\label{sec:compactness}
\nsa{In 
\cref{cor:complexity-fg},} we assume that $\dom f$ and $\dom g$ are compact sets, e.g., $f(\cdot)=\mathbbm{1}_X(\cdot)$ and $g(\cdot)=\mathbbm{1}_Y(\cdot)$, where $X\subset\cX$ and $Y\subset\cY$ are compact convex sets. In this section, we show that \texttt{SAPD+} can also handle unbounded domains under the following assumption.
\begin{assumption}
\label{assump:bounded-subdifferential}
\sa{For $f$ and $g$ closed convex, suppose $\exists B_f,~B_g>0$ such that $\inf\{\norm{s_f}:\ s_f\in\partial f(x)\}\leq B_f$ for all $x\in\dom f$ and $\inf\{\norm{s_g}:\ s_g\in\partial g(y)\}\leq B_g$ for all $y\in\dom g$.}
\end{assumption}
\begin{remark}
Assumption~\ref{assump:bounded-subdifferential} holds when $f$ is an indicator function of a closed convex set (not necessarily bounded). 
Two important examples for this scenario \xzrev{are}: (i) $f(\cdot)=0$, 
(ii) $f$ is a norm, e.g., $\ell_1$-, $\ell_2$-, or the Nuclear norms.
\end{remark}
\na{The next result is an important one showing that \cref{assump:bounded-subdifferential} is equivalent to $f$ and $g$ being Lipschitz on their domains.}
\begin{lemma} \cite[Lemma A.2]{kong2023iteration}
    \na{Let $f:\reals^n\to\reals\cup\{+\infty\}$ be a proper, closed, convex function. Then, for some $B_f>0$, the function $f$ satisfies \cref{assump:bounded-subdifferential} if and only if $|f(x_1)-f(x_2)|\leq B_f\|x_1-x_2\|$ for any $x_1,x_2\in \dom f$.}
\end{lemma}
\sa{The existing work based on iPPM framework either require compactness, e.g.,~\cite{yan2020optimal}, or some special structure on $\cL$, e.g.,~\cite{rafique1810non}. This is also true for VR-based methods, e.g.,\cite{huang2022accelerated,luo2020stochastic,xu2020enhanced}. To our knowledge, ours is the first one to overcome this difficulty and strictly improve \nsa{the best known} complexity bound for \nsa{the WCSC setting} without compactness assumption; moreover, the same idea also works simultaneously with a variance reduction technique that will be discussed later (see \cref{sec:VR}). Finally, the same trick for removing compactness assumption for \nsa{the} WCSC setting also helps removing \nsa{the} compactness assumption for the primal domain in WCMC setting \nsa{and we still improve the best known complexity for this setting as well} (see~\cref{sec:wcmc}).}
\xz{
\begin{remark}
In \cite{lin2020gradient}, when $f=g=0$, 
\mg{boundedness} of dual space is required while
\cref{assump:bounded-subdifferential} is a weaker 
requirement. \sa{Furthermore, based on the discussion with the authors of \cite{yan2020optimal}, 
compactness of the domain is needed for their proof to hold.} 
In~
\cite{huang2021efficient}, 
the sub-level set $\{x:\phi(x)+f(x)\leq \alpha\}$ is required \mg{to be} 
compact for all $\alpha>0$. 
\mg{There are simple convex functions that do not satisfy this condition} such as 
$f(x)=\max\{0,x\}$. 
Bot and \"{B}ohm~~\cite{boct2020alternating} use \mg{milder assumptions than \cite{lin2020gradient}} without requiring compactness; however, \mg{their complexity is the same as the complexity} of \cite{lin2020gradient}.
\end{remark}}
\begin{theorem}
\label{thm:unbounded-domains}
\sa{The result of \cref{cor:complexity-fg} continues to hold, if one replaces the compact domain assumption, i.e., \cref{assump:compact}, with \cref{assump:bounded-subdifferential}.} 
\end{theorem}
\vspace*{-2mm}
\begin{proof}
\nsa{See \cref{sec:Thm3_proof} for the proof.}
\end{proof}\vspace*{-2mm}
\vspace*{-2mm}
\section{Variance reduction}
\label{sec:VR}\vspace*{-2mm}
\nsa{Variance reduction techniques have been found useful for solving SCSC problems in finite sum form, e.g., \cite{palaniappan2016stochastic} --see also~\cite{can2022} using Richardson-Romberg extrapolation in solving SCSC problems with noisy gradients to obtain improved practical performance.}
In this section, 
we equip \texttt{SAPD+} with \sa{SPIDER variance reduction  technique~\cite{fang2018spider}, a variant of SARAH~\cite{nguyen2017sarah,nguyen2017sarah}} More precisely, for inexactly solving SCSC subproblems given in~\eqref{eq: SCSC problem chpt2}, we propose using \texttt{VR-SAPD} as stated in Algorithm~\ref{Alg: SAPD-VR}. Note \texttt{VR-SAPD} employs a large batchsize \sa{of $b$ in} every $q$ iterations and use small \sa{batchsizes of $\xzrev{b'_{x}}$ and ${\xzrev{b'_{y}}}$ for the rest.}
We prove that \texttt{SAPD+} 
using variance reduction, \sa{i.e., with \texttt{VR-flag}=\textbf{true}, achieves an oracle complexity of \red{${\cO(\max\{\kappa_y,\sqrt{\delta/\epsilon}\}\cdot (\kappa_y\delta/\epsilon+1)L\cG_0\epsilon^{-2})}$, which is the best known 
bound in the literature to our knowledge.}
}

\begin{algorithm}[H]
\caption{\texttt{VR-SAPD} Algorithm}
\label{Alg: SAPD-VR}
{\small
\begin{algorithmic}[1]
\State {\textbf{Input:} $\tau,\sigma,\theta,\mu_x,x_0, y_0,N,\sa{b,{\xzrev{b'_{x}}},{\xzrev{b'_{y}}},q}$}
\State \sa{$\bar{\Phi}(x,y)\gets\Phi(x,y)+\frac{\mu_x+\gamma}{2}\norm{x-x_0}^2$}
\State \sa{Let \xzrev{$\cB^x_{0}$},\xzrev{$\cB^y_{0}$} be random mini-batch samples with  \xzrev{$|\cB^x_{0}|=|\cB^y_{0}|=\xzf{b_0}$}}
\State \sa{$w_0\gets \tilde\nabla_y \Phi_{\xzrev{\cB^y_0}}(x_0,y_0)$,\quad $\tilde s_0\gets w_0$}
\For{$k\geq 0$}
\State $y_{k+1} \leftarrow \prox{\sigma g}(y_k+\sigma \tilde{s}_k)$
\If{$\mo(k,q)$ == $0$} 
\State $v_k\leftarrow\tilde\nabla_x \bar{\Phi}_{\xzrev{\cB^x_k}}(x_k,y_{k+1})$
\Else
\State Let $\cI_k^x$ be random mini-batch sample  with $|\cI_k^x|=\sa{\xzrev{b'_{x}}}$
\State $v_k\leftarrow\tilde\nabla_x \bar{\Phi}_{\cI_k^x}(x_k,y_{k+1})-\tilde\nabla_x\bar{\Phi}_{\cI_k^x}(x_{k-1},y_{k}) + v_{k-1}$
\EndIf
\State $x_{k+1} \leftarrow \prox{\tau f}(x_k-\tau v_k)$
\State \xzrev{Let $\cB^x_{k+1}$,$\cB^y_{k+1}$ be random mini-batch samples  with  \hspace*{4.4mm}$|\cB^x_{k+1}|=|\cB^y_{k+1}|=b$}
\If{$\mo(\sa{k+1},q)$ == $0$} 
\State \sa{$w_{k+1}\leftarrow\tilde\nabla_y \Phi_{\xzrev{\cB^y_{k+1}}}(x_{k+1},y_{k+1})$}
\Else
\State Let $\cI_{k+1}^y$ be mini-batch sample  with $|\cI_{k+1}^y|=\sa{{\xzrev{b'_{y}}}}$
\State $\tilde{q}_{k+1}\gets\tilde\nabla_y \Phi_{\cI_{k+1}^y}(x_{k+1},y_{k+1})-\tilde\nabla_y\Phi_{\cI_{k+1}^y}(x_{k},y_{k})$
\State \sa{$w_{k+1}\gets w_{k}+\tilde{q}_{k+1}$}
\EndIf
\State \sa{$\tilde{s}_{k+1} \leftarrow (1+\theta) w_{k+1} - \theta w_{k}$}
\EndFor
\State {\textbf{Output}:} $(\bar{x}_N,\bar{y}_N)=\frac{1}{N}\sum_{k=0}^{N-1}(x_{k+1},y_{k+1})$
\end{algorithmic}}
\end{algorithm}    

Here, we use
{\footnotesize $\tilde\nabla_y\Phi_{\xzrev{\cB^y_k}}^t(x_k,y_{k})$} to represent {\footnotesize $\frac{1}{|\xzrev{\cB^y_k}|} \sum_{\omega_k^i\in  \xzrev{\cB^y_k}} \tilde\nabla_y\Phi(x_k,y_y;\xzrev{\vartheta^{y,i}_k)}$}, where {\footnotesize$\xzrev{\cB^y_k}=\{\xzrev{\vartheta^{y,i}_k}\}_{i=1}^b$} is the mini-batch with {\footnotesize$|\xzrev{\cB^y_k}|=b$ and we define $\tilde\nabla_x\Phi_{\xzrev{\cB^x_k}}^t(x_k,y_{k+1})$} similarly. \mgrev{In} addition, {\footnotesize$\cI^x_k=\{\omega_k^{x,i}\}$} and {\footnotesize$\cI^y_k=\{\omega_k^{y,i}\}$} with {\footnotesize$|\cI_k^x|=\xzrev{b'_{x}}$} and {\footnotesize$|\cI_k^y|={\xzrev{b'_{y}}}$} denote the small mini-batches for generating {\footnotesize$\tilde\nabla_y\Phi_{\cI^y_k}^t(x_k,y_{k})$} and {\footnotesize$\tilde\nabla_x\Phi_{\cI^x_k}^t(x_k,y_{k+1})$}. \nsa{When we run \texttt{VR-SAPD} on a generic subproblem as in~\eqref{eq:WCSC-subproblem-generic}, we use the convention that {\footnotesize$\tilde \nabla_x \bar \Phi_{\cB^x_k}(x_k,y_{k+1}) \triangleq \tilde \nabla_x \Phi_{\cB^x_k}(x_k,y_{k+1})  + (\mu_x + \gamma)(x_k-x_0)$}.}

Throughout this section we make 
\nsa{a continuity assumption} on the stochastic \nsa{first-order oracles} similar
 to~\cite{huang2021efficient,huang2022accelerated,luo2020stochastic,xu2020enhanced}.
\begin{assumption}\label{ASPT: lipshiz gradient VR}
     $\exists L_{xx}, L_{xy},L_{yx},L_{yy} \sa{\geq} 0$ such that $\forall x,\bar{x}\in \dom f\subset\cX$ and $\forall y,\bar{y}\in \dom g\subset \mathcal{Y}$,
   {\small \begin{equation}\label{eq:LGX1-VR}
    \begin{aligned}
           \| \tilde\nabla_y \Phi(x,y;\omega) -\tilde \nabla_y \Phi(\bar{x},\bar{y};\omega)\|
        \leq 
        L_{yx}\|x- \bar{x}\|
        + L_{yy}\|y- \bar{y}\|,\quad \sa{w.p.~1}, \\                \| \tilde\nabla_x \Phi(x,y;\omega) - \tilde\nabla_x \Phi(\bar{x},\bar{y};\omega)\|
        \leq 
        L_{xx}\|x- \bar{x}\|+L_{xy}\|y-\bar{y}\|,\quad \sa{w.p.~1}.
    \end{aligned}
    \end{equation}}
\end{assumption}
\begin{assumption}\label{ASPT: independence} \sa{Consider $\texttt{SAPD+}$ with $\texttt{VR-flag}=\textbf{true}$.}
\xzrev{We assume} (i) \sa{for any $k\geq 0$, the random mini-batches \xzrev{$\cB_k^x$,\;$\cB_k^x$,\;}\nsa{$\cI_k^x$ and $\cI_k^y$ consist of 
independent elements, and \xzrev{$\cB^k_x$ is independent from $\cB^y_k$}; (ii) for $i\in\{k-1,k\}$} $\xzrev{\cB_k^x}$, $\cI_k^x$ are independent of $(x_i,y_{i+1})$, and $\xzrev{\cB_k^y}$, $\cI_k^y$ are independent of $(x_i,y_{i})$.}
\end{assumption}
\begin{remark}
\sa{For finite-sum type problems of the form $\min_x\max_y \mg{\frac{1}{n}}\sum_{i=1}^n\Phi_i(x,y)$, \mg{we can set the stochastic gradient according to 
$\tilde \nabla_x \Phi(x,y;\omega) =  \nabla_x \Phi_\omega(x,y)$ and 
$\tilde \nabla_y \Phi(x,y;\omega) =  \nabla_y \Phi_\omega(x,y)$ where $\omega$ is uniformly drawn at random from $\{1,\ldots,n\}$}. Therefore, if \xzrev{mini-batch samples} 
are drawn from $\{1,\ldots,n\}$ uniformly at random with replacement; batches will be independent of the past iterates satisfying \cref{ASPT: independence}.}
\end{remark}


\begin{theorem}\label{cor:complexity VR}
\sa{Suppose Assumptions~\ref{ASPT: fg},\ref{ASPT: unbiased noise assumption},\ref{ASPT: lipshiz gradient VR} and~\ref{ASPT: independence} hold. Moreover, either \cref{assump:compact} or \cref{assump:bounded-subdifferential} holds. Let \nsa{$\mu_x=\gamma$}, $\theta=1$, and 
$\tau$, $\sigma$, 
$b$ and $N$ be chosen as follows:}
{\footnotesize
\begin{equation}
\label{eq: speical paramater VR innner loop}
\begin{aligned}
    &\tau = \Big(L_{yx}+ 
    L_{xx}+2\gamma+ 2(q-1)\Big(\frac{(L_{xx}+2\gamma)^2}{\gamma \xzrev{b'_{x}}} + \frac{10L^2_{yx}}{\mu_y {\xzrev{b'_{y}}}}\Big)\Big)^{-1},\;
    \sigma = \left(2L_{yy}+L_{yx} + 2(q-1)\Big(\frac{L^2_{xy}}{\gamma \xzrev{b'_{x}}} + \frac{10L^2_{yy}}{\mu_y {\xzrev{b'_{y}}}}\Big)\right)^{-1},
    \\
    & 
    N = \xzf{\max\Big\{2(1+\zeta)\max\Big\{\frac{1}{\gamma\tau}-1,~\frac{1}{\mu_y\sigma}\Big\},\frac{288}{b_0}\frac{\delta_x^2}{\epsilon^2}, \frac{720}{b_0}\frac{\gamma}{\mu_y}\frac{\delta_y^2}{\epsilon^2 }\Big\}},
    \qquad 
    b \geq\Big\lceil \max\Big\{\xzrev{\frac{\xzf{288}\delta_x^2}{\gamma}},~\frac{\xzf{720}\delta_y^2}{\mu_y}\Big\}\frac{\gamma}{\epsilon^2}\Big\rceil.
\end{aligned}
\end{equation}
}%
\nsa{For any $\epsilon>0$ and parameters $\xzrev{b_0,b'_{x}},{\xzrev{b'_{y}}},q\in \mathbb{N}^+$,} 
\nsa{when $\texttt{VR-flag}=\textbf{true}$, $\texttt{SAPD+}$} \na{guarantees that \eqref{eq:main-bound} holds
for all} \sa{$T \geq 288\cG(x_0^0,y_0^0)  \cdot\frac{\gamma}{\epsilon^2}$}, and \na{computing a point $x_\epsilon$ such that $\mathbb{E}[\|\nabla\phi_\lambda(x_\epsilon)\|]\leq \epsilon$} requires $T(\xzf{b_0}+Nb/q+N(\xzrev{b'_{x}}+{\xzrev{b'_{y}}}))$
\nsa{stochastic first-order} oracle calls in total, where
{\footnotesize
\begin{equation}
\label{eq:N-bound-VR}
N = \mathcal{O}\Big(\max\Big\{  
\frac{L_{yx}+L_{xx}}{\gamma} +\frac{q}{\xzrev{b'_{x}}}\frac{L_{xx}^2}{\gamma^2} + \frac{q}{{\xzrev{b'_{y}}}}\frac{L^2_{yx}}{\gamma\mu_y},\quad 
\frac{L_{yy}+L_{yx}}{\mu_y}+
\frac{q}{{\xzrev{b'_{y}}}}\frac{L^2_{yy}}{\mu_y^2} + \frac{q}{\xzrev{b'_{x}}}\frac{L^2_{xy}}{\gamma\mu_y},\quad
\xzf{\frac{\delta_x^2}{b_0\epsilon^2}},\quad \frac{\gamma}{\mu_y}\xzf{\frac{\delta_y^2}{b_0\epsilon^2}}
\Big\}
\Big).
\end{equation}}%
\end{theorem}
\begin{proof}
\nsa{See~\cref{sec:Thm4_proof} for the proof.}
\end{proof}\vspace*{-4mm}
\xzf{\begin{remark}
\label{rem:SAPD+VR-complexity}
For any $y\in\dom g$, since $\Phi(\cdot,y)$ $L_{xx}$-smooth, it is necessarily $L_{xx}$-weakly convex; thus, $\gamma\leq L_{xx}$. \sa{For the worst-case complexity, consider the setting in \eqref{eq:uniform setting},} 
and let 
$\xzrev{b'_{x}} = {\xzrev{b'_{y}}}=\sa{b'}$. Then, \cref{cor:complexity VR} implies that 
{\small$b = \mathcal{O}\Big(\kappa_y\tfrac{\delta^2}{\epsilon^2}\Big)$}, {\small$\xzf{N = \mathcal{O}\big(\max\{\kappa_y+\kappa_y^2\tfrac{q}{\sa{b'}},\tfrac{\kappa_y\delta^2}{b_0\epsilon^2}\}\big)}$}, and {\small$T = \mathcal{O}\Big(\tfrac{L\cG(x_0^0,y_0^0)}{\epsilon^2}\Big)$}; hence, if we set $b'=\sqrt{b\kappa_y}$, $q=\sqrt{\frac{b}{\kappa_y}}$, and $b_0 = \kappa_y\big(\frac{\delta}{\epsilon}\big)^{1.5}$, the total complexity is given by $T\Big(\xzf{b_0}+ N(b/q+\sa{b'}+1)\Big)=\mathcal{O}\Big(\nsa{\max\Big\{\kappa_y,\sqrt{\frac{\delta}{\epsilon}}\Big\}\cdot (\kappa_y\tfrac{\delta}{\epsilon}+1)}~\cdot\tfrac{L\cG(x_0^0,y_0^0)}{\epsilon^2}\Big)$. Specifically, if $\kappa_y \geq \sqrt{\frac{\delta}{\epsilon}}$, the total complexity is $\mathcal{O}\Big(\kappa_y^2 \delta \frac{L \cG(x_0^0, y_0^0)}{\epsilon^3}\Big)$; otherwise, the total complexity is $\mathcal{O}\Big(\kappa_y {\delta^{1.5}}\frac{L \cG(x_0^0, y_0^0)}{\epsilon^{3.5}}\Big)$.
\end{remark}}

\begin{remark}
\xzrev{The results in \cref{cor:complexity VR} continues to hold under a weaker form of Assumption~\ref{ASPT: lipshiz gradient VR} as in \cite{luo2020stochastic,xu2020enhanced}, i.e., we replace \cref{eq:LGX1-VR} with
{\small
\begin{align*}
    & \mathbb{E}\Big[\| \tilde\nabla_y \Phi(x,y;\omega) -\tilde \nabla_y \Phi(\bar{x},\bar{y};\omega)\|^2\Big]
        \leq 
        2L^2_{yx}\|x- \bar{x}\|^2
        + 2L^2_{yy}\|y- \bar{y}\|^2,
        \\
        &  \mathbb{E}\Big[\| \tilde\nabla_x \Phi(x,y;\omega) - \tilde\nabla_x \Phi(\bar{x},\bar{y};\omega)\|^2\Big]
        \leq 
        2L^2_{xx}\|x- \bar{x}\|^2+2L^2_{xy}\|y-\bar{y}\|^2.
\end{align*}}}%
\end{remark}
\vspace*{-8mm}

\section{Weakly convex-\sa{merely concave~(WCMC)} problems}
\label{sec:wcmc}
In this section, \sa{we state the convergence guarantees of \texttt{SAPD+} for solving WCMC problems. In particular, we will consider \eqref{eq:main problem} such that $f(\cdot)=0$ and $\mu_y=0$, i.e., $\Phi(x,\cdot)$ is \emph{merely} concave for all $x\in\cX$. Instead of directly solving \eqref{eq:main problem} in WCMC setting,}
we will solve an approximate model obtained by smoothing the primal problem in a similar spirit to 
the technique in \cite{nesterov2005smooth}.
More precisely, we approximate \eqref{eq:main problem} with the following 
WCSC problem: \sa{given an arbitrary $\hat y\in\dom g$, consider}
\vspace*{-2mm}
{
\begin{equation}\label{eq:weakly convex-merely concave problem approx}
    \min_{x\in \mathcal{X}} \max_{y\in \mathcal{Y}} \hat{\mathcal{L}}(x,y)\triangleq \hat{\Phi}(x,y) -g(y),\quad\mbox{where}\quad \sa{\hat{\Phi}(x,y) \triangleq \Phi(x,y)-\frac{\hat{\mu}_y}{2}\|y -\hat{y} \|^2.}
\end{equation}}%


\begin{theorem}
\label{Thm:WCMC}
Under Assumptions~\ref{ASPT: fg},~\ref{ASPT: lipshiz gradient},~\ref{ASPT: unbiased noise assumption}, consider the SP problem in~\eqref{eq:main problem} such that $f(\cdot)\equiv 0$, $\mu_y=0$, 
and $\cD_{\cY} \triangleq\sup_{y_1,y_2\in\dom g}\|y_1-y_2\|\sa{<\infty}$. 
When either \cref{assump:compact} or \cref{assump:bounded-subdifferential} holds, for any given $\epsilon>0$, \texttt{SAPD+} with $\texttt{VR-flag}=\textbf{false}$, applied to~\eqref{eq:weakly convex-merely concave problem approx} with $\hat{\mu}_y=\Theta(\epsilon^2/(L\cD_y^2))$,
is guaranteed to generate $x_\epsilon\in\cX$ such that $\mathbb{E}\left[\|\nabla\phi_{\lambda}(x_\epsilon)\|\right]\leq\epsilon$ for $\lambda=1/(2\gamma)$ within $\mathcal{O}(L^{3}\epsilon^{-6})$ stochastic first-order oracle calls.
\end{theorem}
\begin{proof}
\nsa{See~\cref{sec:Thm5_proof} for the proof.}
\end{proof}
\section{\mg{Numerical experiments}}
The experiments are conducted on \sa{a PC with 3.6 GHz Intel Core i7 CPU and NVIDIA
RTX2070 GPU.}\footnote{The code is made available at  \url{https://github.com/XuanZhangg/SAPD-PLUS}.} We consider distributionally robust optimization and fair classification. \sa{In the rest, $n$ and $d$ represent the number of samples in the dataset and the dimension of each data point, respectively.} \nsa{In this section, \texttt{SAPD+} means calling \texttt{SAPD+} with \texttt{VR-flag}=\textbf{false}, and  \texttt{SAPD+VR} means calling \texttt{SAPD+} with \texttt{VR-flag}=\textbf{true}.}


\paragraph{Distributionally Robust Optimization~\sa{(DRO)}.} 
First, we consider nonconvex-regularized variant of 
\sa{DRO} problem \cite{abadi2016deep,namkoong2016stochastic,kohler2017sub,luo2020stochastic,zhang2021robust,yan2019stochastic} \mg{which arises in distributionally robust learning}. Let $\{\mathbf{a}_i,b_i\}^{n}_{i=1}$ be the dataset where $\mathbf{a}_i\in\mathbb{R}^d$ \mg{are the features} and $b_i\in\{-1,1\}$ are labels. The DRO problem is
{\small
\begin{equation}\label{eq: experiment2 problem}
    \text{(DRO): } \min_{x\in \mathbb{R}^d} \max_{y \in Y}\frac{1}{n}\sum_{i=1}^n y_i \ell_i(x) + f(x) - g(y),
\end{equation}}%
where $\ell_i(x)=\log(1+\exp(-b_i\mathbf{a}^\top_i\mathbf{x}))$ \nsa{is the logistic loss}, $f(x) = \sa{\eta_1}\sum_{i=1}^d \frac{\alpha x_i^2}{1+\alpha x_i^2}$ is a nonconvex regularizer~\cite{antoniadis2011penalized}, 
$g(y)=\frac{1}{2}\sa{\eta_2}\|n {y}-\mathbf{1}\|^2$, and $Y \triangleq \{{y}\in \mathbb{R}^d_+:~\mathbf{1}^\top {y} = 1 \}$ -- here, $\mathbf{1}$ denotes the vector with all entries equal to one. \mg{This problem can be viewed as a robust formulation of empirical risk minimization where the weights $y_i$ are allowed to deviate from $1/n$; and the aim is to minimize the worst-case empirical risk.}
We perform experiments on three data sets: $i)$ \verb+a9a+ with $n = 32561$, $d = 123$; $ii)$ \verb+gisette+ with $n = 6000$, $d = 5000$; $iii)$ \verb+sido0+ with $n = 12678$, $d = 4932$. The dataset \verb+sido0+ is obtained from
Causality Workbench\footnote{http://www.causality.inf.ethz.ch/challenge.php?page=datasets} while the others can be downloaded from LIBSVM repository\footnote{https://www.csie.ntu.edu.tw/~cjlin/libsvmtools/datasets/binary.html}.

\emph{Parameter tuning.} We set the parameters \mg{according to} \cite{yan2019stochastic, luo2020stochastic,kohler2017sub}, i.e., , \sa{$\alpha=10$, $\eta_1=10^{-3}$, $\eta_2=1/n^2$}.
\mg{We compare \nsa{\texttt{SAPD+} and \texttt{SAPD+VR}} against \texttt{PASGDA}~\cite{boct2020alternating}, \texttt{SREDA}~\cite{luo2020stochastic}, \texttt{SMDA}, \texttt{SMDA-VR}~\cite{huang2021efficient} algorithms.} As suggested in \cite{luo2020stochastic}, we tune the primal stepsizes of all the algorithms \mg{based on a grid-search over the set} $\{10^{-3},10^{-2},10^{-1}\}$ and \mg{the ratio of the primal stepsize 
\nsa{to} dual stepsize, \nsa{i.e., $\tau/\sigma$,} is varied to take values \xzh{from} the set $\{10,10^{2},10^{3},10^{4}\}$}. For all variance reduction-based algorithms, i.e., for \texttt{SAPD+VR}, \texttt{SREDA}, \texttt{SMDA-VR}, we tune the large batch size $\sa{b}\triangleq|\cB|$ from the set $\{3000,6000\}$, and the small batch size $\sa{b'\triangleq}|I|$ from grid search over the set $\{10,100,200\}$. For the frequency parameter $q$, we let $q=b'=|I|$ \sa{for \texttt{SAPD+VR} and \texttt{SMDA-VR} (as suggested in \cite{huang2021efficient}); for \texttt{SREDA}, when we set $q$ and $m$ (\texttt{SREDA}'s inner loop iteration number) to $\mathcal{O}(n/|I|)$ as suggested in~\cite{luo2020stochastic}, we noticed that \texttt{SREDA} does not perform well against \texttt{SAPD+VR} and \texttt{SMDA-VR}. Therefore,} to optimize the performance of \texttt{SREDA} further, we 
tune $q,m$ from a grid search over $\{10,100,200\}$. For methods without variance reduction, i.e., for \texttt{SAPD+}, \texttt{SMDA} and \texttt{PASGDA}, we \nsa{also use mini-batch to estimate the gradients} and tune the batch size 
from $\{10,100,200\}$ as well. For \texttt{SAPD+} and \texttt{SAPD+VR}, we tune the momentum $\theta$ from $\{0.8,0.85,0.9\}$ and the inner iteration number from \xzh{$N = \{10,50,100\}$}.

\emph{Results.} \nsa{To fairly compare the performances of algorithms using different batch sizes, we plot \na{the primal function values} against epochs in x-axis\footnote{\nsa{an epoch is completed whenever an algorithm does one pass over the whole data set through sampling mini-bathes without replacement.}}.} In \cref{fig: DRO}, we plot the \mg{average} 
\na{primal function value} against the epoch number \mg{based on 30 simulations (runs)}. \mg{The standard deviations of the runs are also illustrated around 
\sa{the average} in lighter color \sa{as shaded regions}.}
We observe that \texttt{SAPD+} and \texttt{SAPD+VR} consistently outperforms over other algorithms. For \verb+a9a+, \verb+gisette+, \verb+sido0+ datasets, the average training accuracy of \texttt{SAPD+} are $84.06\%$, $95.41\%$, $96.43\%$, and of \texttt{SAPD+VR} are $84.33\%$, $97.69\%$, $97.46\%$, respectively. The best performance for \verb+a9a+, \verb+gisette+, \verb+sido0+ among all the other algorithms are $75.92\%$, $93.07\%$, $96.43\%$, respectively. More importantly, \mg{we observe that} \sa{as an accelerated method, \texttt{SAPD+VR} enjoys fast convergence properties while still being robust to gradient noise.}

\begin{figure}[t]
\centering
\caption{
Comparison of \nsa{\texttt{SAPD+} and \texttt{SAPD+VR}} against \texttt{PASGDA}~\cite{boct2020alternating}, \texttt{SREDA}~\cite{luo2020stochastic}, \texttt{SMDA}, \texttt{SMDA-VR}~\cite{huang2021efficient} on real-data for solving \cref{eq: experiment2 problem} with $30$ times simulation.}
\label{fig: DRO}
\includegraphics[width = 0.32\textwidth]{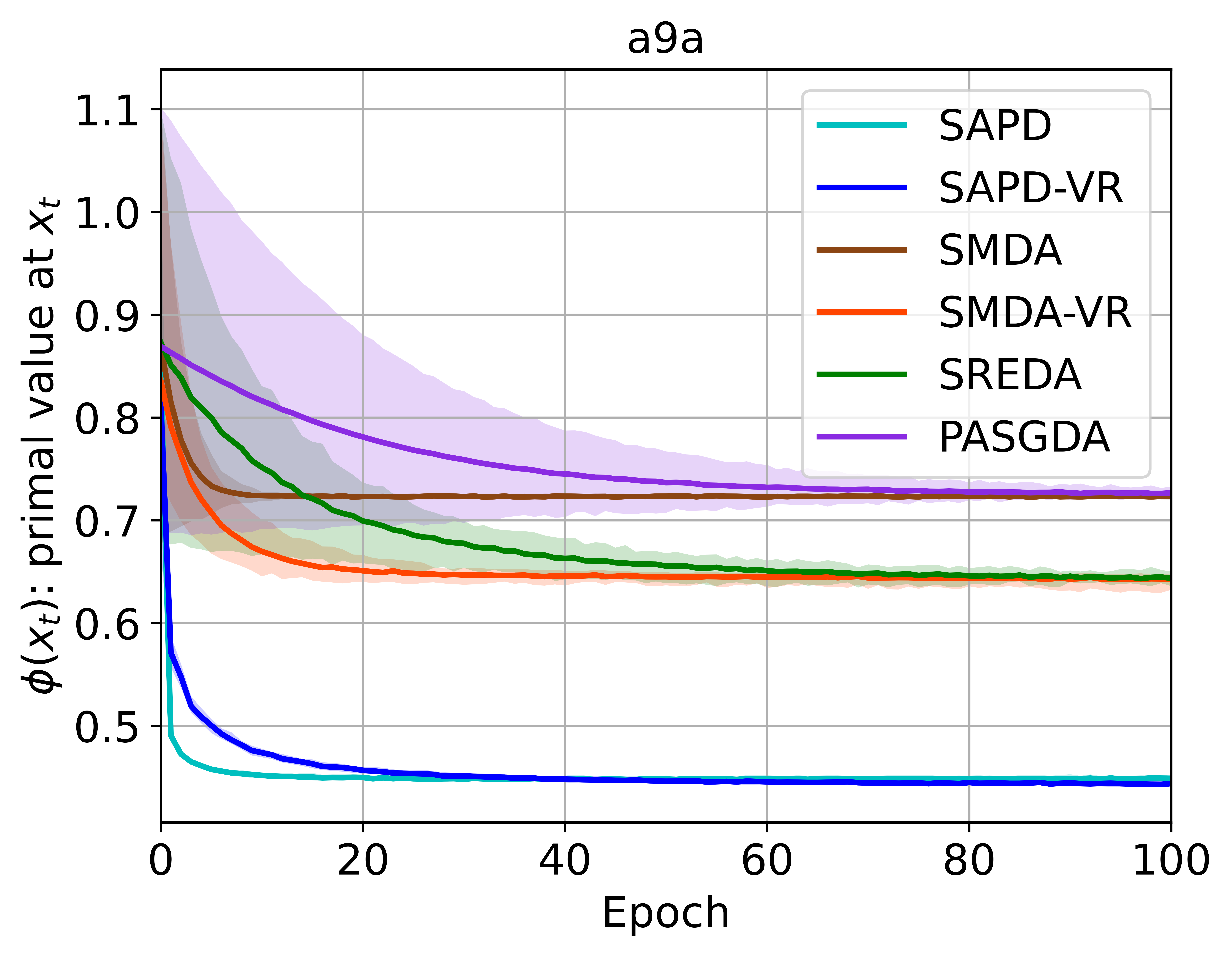}
\includegraphics[width = 0.32\textwidth]{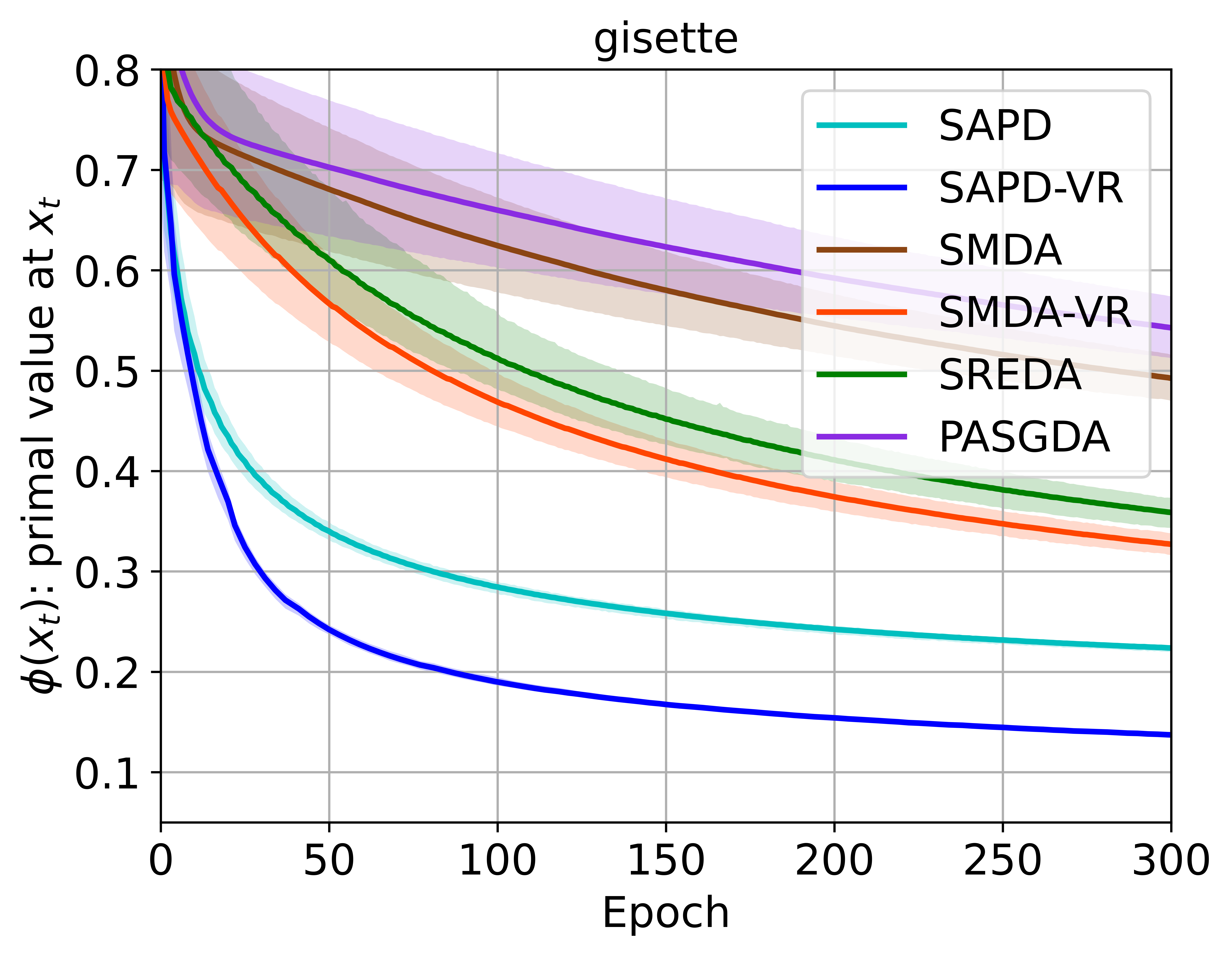}
\includegraphics[width = 0.32\textwidth]{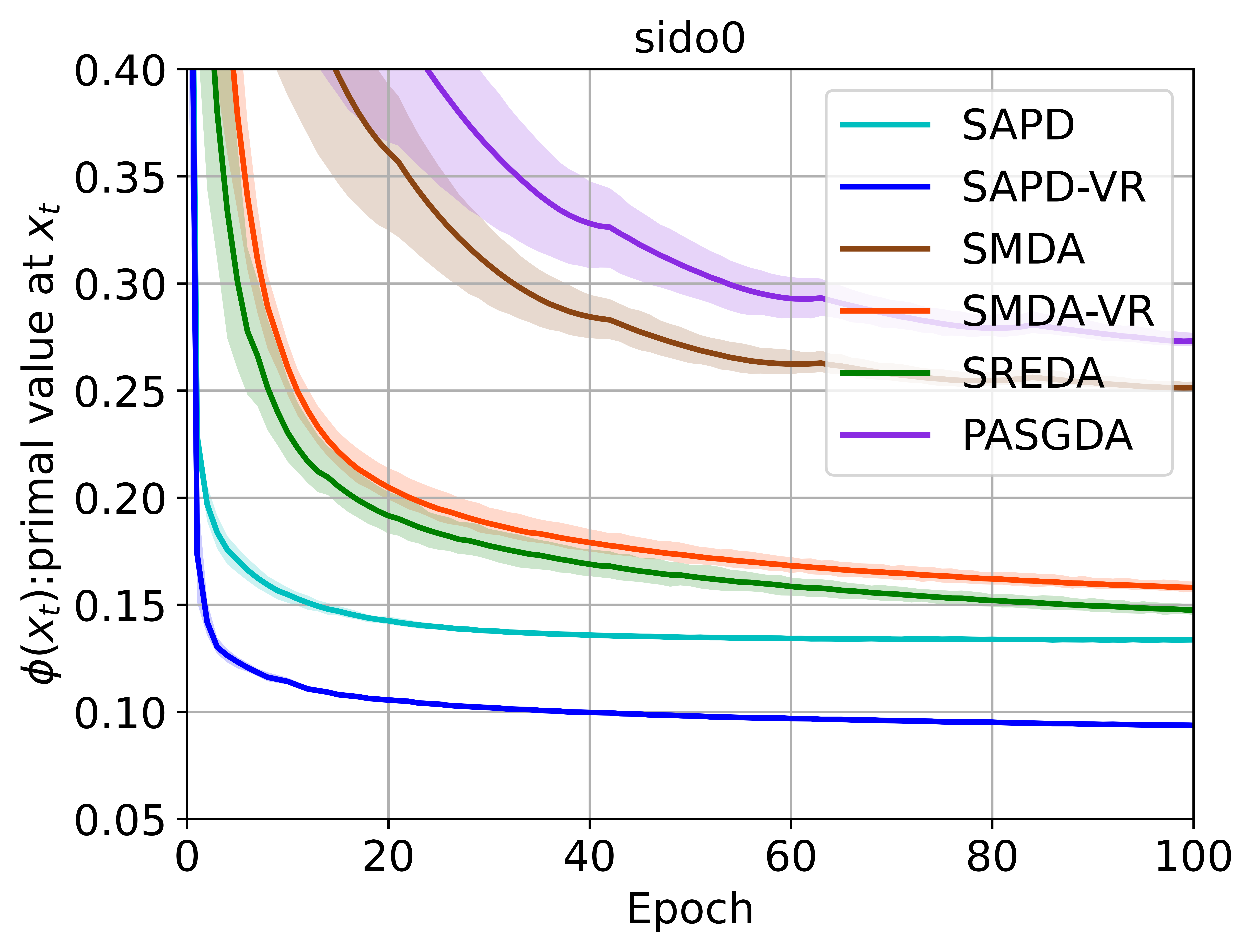}
\centering
\end{figure}

\paragraph{Fair Classification.} 
\mg{
For multi-class classification, Mohri \emph{et al.}~\cite{mohri2019agnostic} propose training a fair classifier 
\sa{thorough minimizing} the worst-case loss over the classification categories.}
In the spirit of \cite{nouiehed2019solving,huang2021efficient}, we adopt a nonconvex \mg{convolutional neural network} (CNN) model as a classifier and set the number of categories to $3$, \sa{resulting in a minimax problem of the form:}\vspace*{-2mm}
{\small
\begin{equation}
    \label{eq: experiment CNN}
    \min_{x\in\mathcal{X}} \max_{y\in\cY}\sum_{i=1}^{3}y_i \sa{\ell_i(x)} -g(y),\quad s.t.\quad \sum_{i=1}^{3}y_i=1,\;y_i\geq0,\; \forall\;i 
\end{equation}}%
where 
\sa{$x\in\reals^p$} represents the parameters of the CNN, and \mg{\sa{$\ell_1,\ell_2,\ell_3$} correspond to the loss of three categories \xzrev{whose details are given in \cref{sec:experiement-details}}, $g(y)=\frac{\sa{\eta}}{2}\|y\|_2^2$ is a regularizer with $\sa{\eta}>0$}. \sa{We 
train \eqref{eq: experiment CNN} on the datasets to classify: $i)$ gray-scale hand-written digits $\{0,2,3\}$ from \texttt{MNIST}; 
$ii)$ fashion images 
with target classes \{T-shirt/top, Sandal, Ankle boot\} from \texttt{F-MNIST}; 
$iii)$ RBG colored images with target classes \{Plane, Truck, Deer\} from \texttt{CIFAR10}. For both \texttt{MNIST} and \texttt{F-MNIST} $p=43831$, $n=18000$ and $d=28\times28\times 1$, and for \texttt{CIFAR10} $p=61411$, $n=15000$, and  $d=32\times 32\times 3$.} 

We let the regularization parameter $\sa{\eta}=0.1$ 
as suggested in \cite{huang2021efficient}. \mg{We compare \texttt{SAPD+VR} \sa{against the other} VR-based algorithms \texttt{SREDA} and \texttt{SMDA-VR} over 30 runs.}
We tune the primal stepsizes of \texttt{SAPD+VR} and \texttt{SREDA} \mg{by a grid search over the set} $\{10^{-2},5\times 10^{-3},10^{-3}\}$ and the ratio of primal \sa{to dual stepsizes, i.e., $\tau/\sigma$, is chosen from} 
$\{10,10^2,5\times 10^2,10^3\}$. For \texttt{SMDA-VR}, the primal and dual stepsizes are $10^{-3}$ and $10^{-5}$ as suggested in \cite{huang2021efficient} \sa{--we also tried stepsizes bigger than the suggested; but, it caused convergence issues in the experiments.} 
We set the large batchsize $|\cB|=3000$ and the small batchsize $|\cI|=200$ for all algorithms and data sets; the frequency $q=200$ is used for \texttt{SAPD+VR} and \texttt{SMDA-VR}, and \sa{we tune $q$ for \texttt{SREDA} taking values from $\{10,50,100,200\}$.} The momentum $\theta$ for \texttt{SAPD+VR} is tuned taking values from  $\{0.8,0.85,0.9\}$ and inner iteration number is tuned from \xzh{$N = \{10,50,100\}$}. 
For \texttt{SREDA}, we tune the inner loop iteration from $\{10,50,100\}$. \mg{Fig. \ref{fig: CNN WCSC} shows that \texttt{SAPD+VR} outperforms the other VR-based algorithms clearly in terms of both the average and the standard deviation of \na{the primal function values}.}

\begin{figure}[t]
\centering
\caption{
Comparison of \texttt{SAPD+VR} against other Variance Reduction algorithms, \texttt{SREDA}~\cite{luo2020stochastic}, \texttt{SMDA-VR}~\cite{huang2021efficient} on real-data for solving \cref{eq: experiment CNN} with $30$ times simulation.}
\label{fig: CNN WCSC}
\includegraphics[width = 0.32\textwidth]{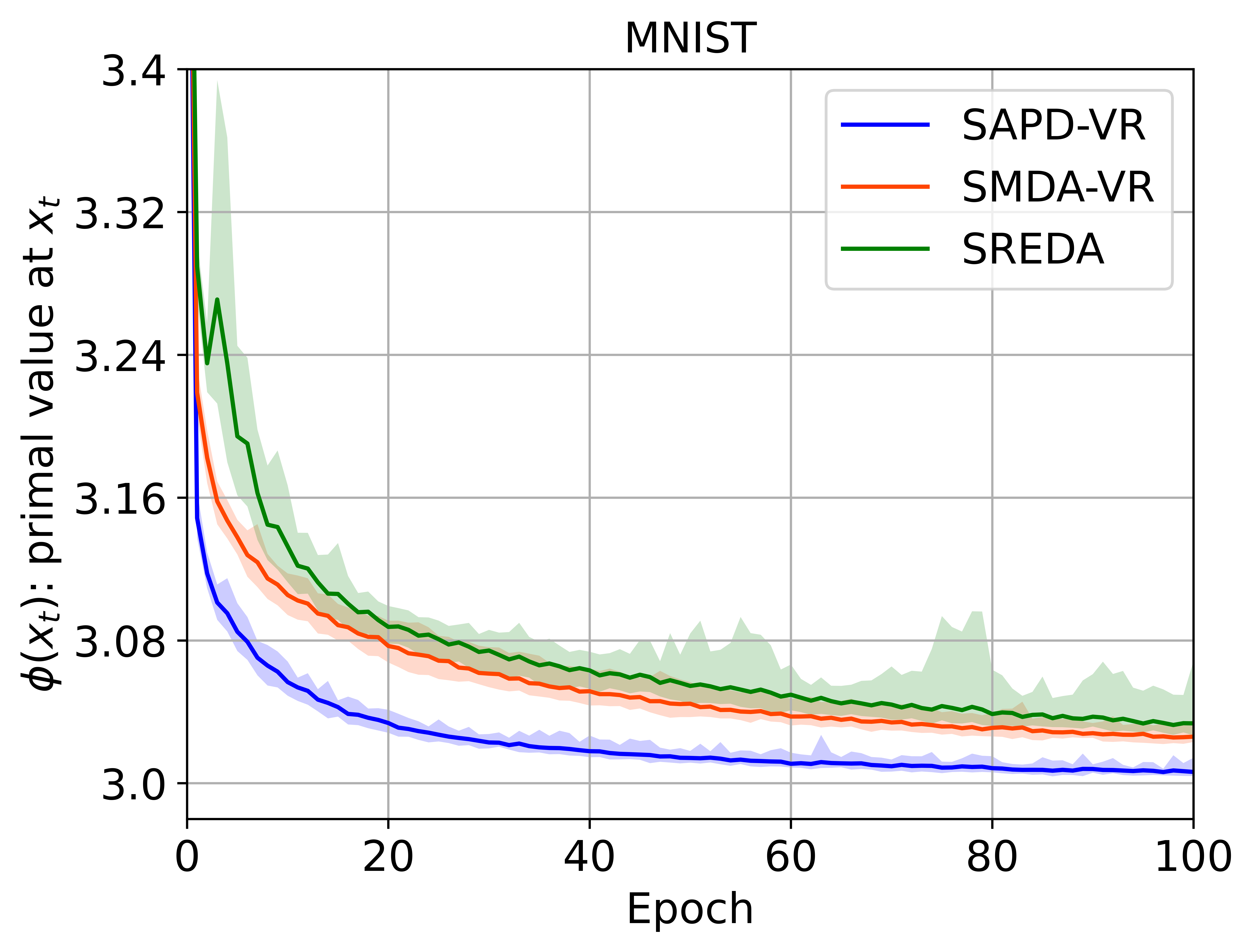}
\includegraphics[width = 0.32\textwidth]{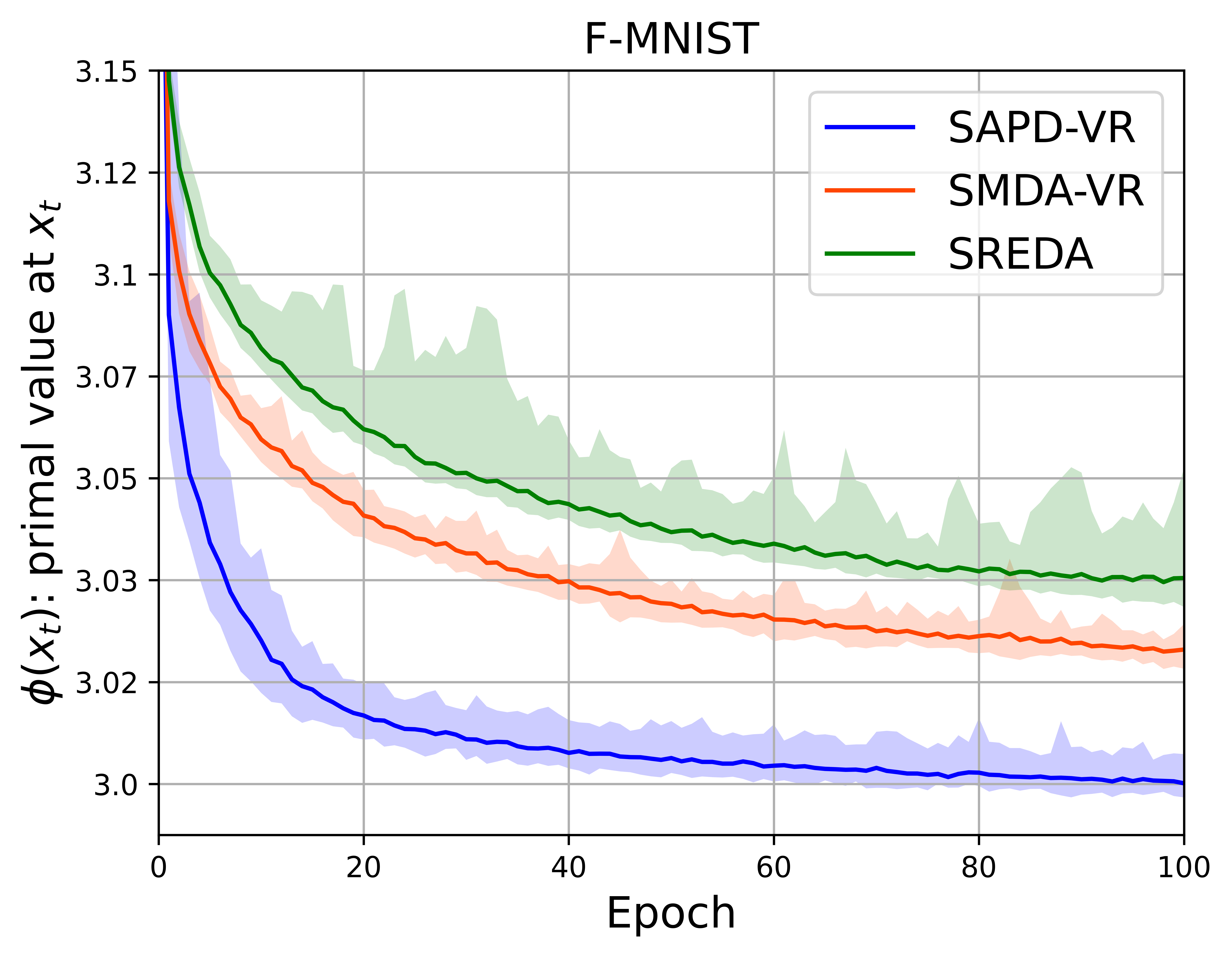}
\includegraphics[width = 0.32\textwidth]{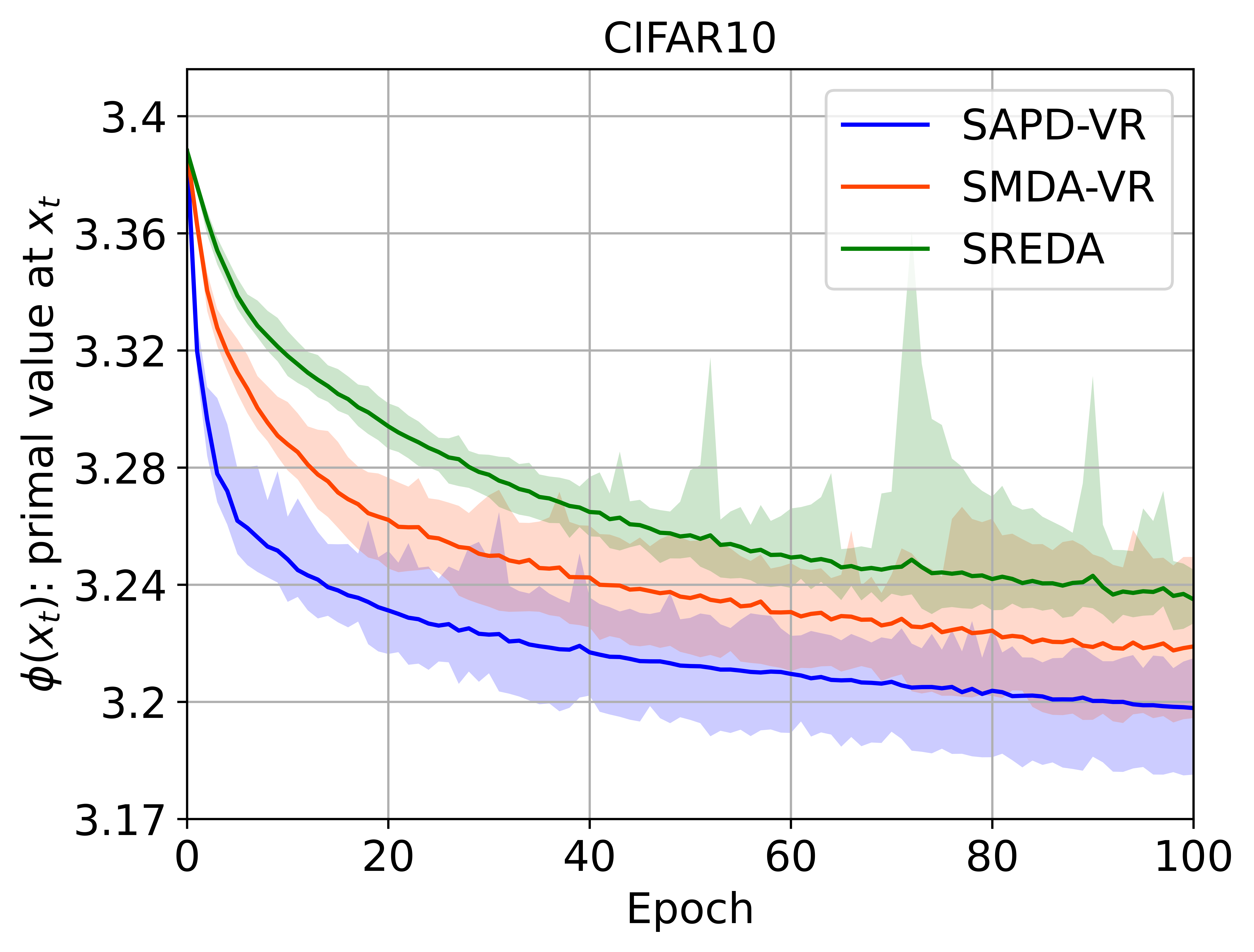}
\centering

\end{figure}

\section{Conclusion}

\mg{In this paper, we \sa{considered both WCSC and WCMC saddle-point problems assuming we only have an access to an unbiased stochastic first-oracle with a finite variance. This setting arises} in many applications \sa{ranging} from distributionally robust learning to GANs. We proposed a new method \texttt{SAPD+}, which achieves an improved complexity in terms of target accuracy $\epsilon$ for both WCSC and WCMC problems; \sa{moreover, our bound for \texttt{SAPD+} has a better dependency to the condition number $\kappa_y$ for the WCSC scenario.} \sa{We also showed that} our algorithm \texttt{SAPD+} can support the 
\sa{SPIDER} variance-reduction technique. \sa{Finally,} we provided numerical experiments 
\sa{demonstrating that \texttt{SAPD+} can achieve a} state-of-the-art performance on distributionally robust learning and on multi-class classification \sa{problems arising in ML}.\looseness=-1}

\bibliographystyle{plain} 
\bibliography{reference} 


\newpage

\section{The general construction used in the proof of Theorem~\ref{cor:complexity-fg}}
\label{sec:construction}
{In general, the proof of Theorem~\ref{cor:complexity-fg} can be divided into two parts: (1)~inner loop and outer loop convergence analysis, (2)~\sa{combining these results to derive the overall complexity}.}
{\begin{itemize}
    \item  \sa{We first study the convergence properties of Algorithm~\ref{Alg: SAPD} for solving the SCSC subproblems in \cref{eq: SCSC problem chpt2}. In \cref{lemma:bound-gap}, we provide guarantees for the inner loop iterates using the expected gap function as our metric.} 
    \item Since the convergence guarantee for the inner loop is provided in terms of $\mathcal{G}^t$, \sa{we also consider the relationship between $\mathcal{G}^t(x_0^t,y_0^t)$ and GNME, i.e., $||\nabla_x\phi_{\lambda}(x_0^t)||$.} Indeed, Lemmas~\ref{Lemma: gap convergence},\ref{lemma: different bound of gap}, and \ref{Lemma: telescope sum chpt2} allow us to translate the expected gap result of inner loops to the convergence in terms of GNME for the outer loops. In Theorem~\ref{thm:WC_SP}, we provide the convergence result in the GNME metric and \sa{state} the requirements on the parameters \sa{to be able to derive} the complexity bound in Theorem~\ref{cor:complexity-fg}.
    \item \sa{In \cref{LEMMA: Noise LMI after young's ineq-gap}, we provide a particular step size rule 
    for solving the SCSC subproblems in \cref{eq: SCSC problem chpt2}, and we use this specific choice to compute the overall complexity for solving the WCSC problem \cref{eq:main problem} by using \texttt{SAPD+}.} 
\end{itemize}}

\subsection{The construction for {the convergence analysis}}
\label{sec:pre_analysis_thm1}
\sa{Based on \cref{Lemma: graident of ME}, the key step for establishing \xz{\texttt{SAPD+}} convergence is to bound} $\|x^{t}_{0}-\prox{\lambda\phi}(x^{t}_{0})\|$, \sa{where $\phi(x) \triangleq \max_{y\in\cY}\mathcal{L}(x,y)$ for {every $x\in\cX$}
and $\lambda = (\gamma+\mu_x)^{-1}$. To achieve this, we first give a bound on the gap function $\cG^{t}$ at the $t$-th outer iteration.}
\begin{lemma}\label{lemma:bound-gap}
Suppose Assumptions~\ref{ASPT: fg},~\ref{ASPT: lipshiz gradient},~\ref{ASPT: unbiased noise assumption} hold. \sa{
Given $\{N_t\}_{t\geq 0}\subset\integers_+$, let $\{x^{t}_{0},y^{t}_{0}\}_{t\geq 1}$ be generated by \xz{\texttt{SAPD+}}, stated in~Algorithm~\ref{Alg: SAPD-plus}, when \texttt{VR-flag}=\textbf{false}, initialized from {$(x_0^0,y_0^0)\in\dom f\times\dom g$} and using} $\tau,\sigma,\theta,\mu_x>0$ that satisfy {\small
\begin{equation}
    \label{eq: SCSC SAPD LMI}
  \begin{pmatrix}
  \mu_y & (\theta  - 1)L_{yx} & (\theta  - 1)L_{yy}  & 0\\ 
   (\theta  - 1)L_{yx} & \tfrac{1}{\tau} - \sa{L'_{xx}} & 0 & -  \theta L_{yx}\\ 
  (\theta  - 1)L_{yy} & 0 & \frac{1}{\sigma} - \alpha & -  \theta L_{yy}\\
   0 & - \theta L_{yx} & -  \theta L_{yy}  &  \alpha 
\end{pmatrix}\succeq 0
\end{equation}}%
{for some} $\alpha \in [0,  {\tfrac{1}{\sigma}\sa{)}}$, \sa{where $L'_{xx}\triangleq L_{xx}+\mu_x+\gamma$}.  Then for all $t\geq {0}$, \sa{it holds that}
\begin{equation}\label{eq:Gap bound with fg}
  \mathbb{E}\left[\cG^{t}(x^{t+1}_{0}, y^{t+1}_{0})\right] \leq \frac{M_{\tau,\sigma,\theta}}{ N_t }\left(
\frac{\mu_x}{4}\ED{x^{t}_{*}(y^{t+1}_{0}) - x^{t}_{0}} + \frac{\mu_y}{4}\ED{y_*(x^{t+1}_{0}) - y^{t}_{0}}
\right) + \Xi_{\tau,\sigma,\theta},  
\end{equation}
where $ N_t \in \mathbb{N}^+$ and $\sM \triangleq \max \{ \frac{4}{\mu_x\tau}, \frac{4+4\theta}{\mu_y\sigma}\}$,
\begin{subequations}
\begin{align}
    & \Xi_{\tau,\sigma,\theta} \triangleq \tau\left( \Xi^x_{\tau,\sigma,\theta} + \frac{1}{2}\right)\delta_x^2  
     + \sigma\left( \Xi^y_{\tau,\sigma,\theta} +  \frac{1+2\theta}{2}\right)\delta_y^2,\nonumber\\
     &\Xi^x_{\tau,\sigma,\theta} \triangleq \left( 1+ \frac{\sigma\theta(1+\theta)L_{yx}}{2}\right),\label{eq:xi_x}\\
     &\Xi^y_{\tau,\sigma,\theta} \triangleq \left(1+3\theta + \sigma\theta(1+\theta)L_{yy}  +  \tau\sigma\theta(1+\theta)L_{yx}L_{xy} \right)(1+2\theta) + \frac{\tau\theta(1+\theta)L_{yx}}{2}.\label{eq:xi_y}
\end{align}
\end{subequations}
\end{lemma}
\begin{proof}
For easier readability, we provide the proof in a separate subsection, see~\cref{sec:pf-lema2-lemma7}.
\end{proof}
The following lemma provides \sa{a relation between $\cG^{t}(x^{t}_{0},y^{t}_{0})$ and $\cG^{t}(x^{t+1}_{0},y^{t+1}_{0})$.} 
\begin{lemma} 
\sa{Under the premise of Lemma~\ref{lemma:bound-gap} and \nsa{\cref{assump:compact}}, for all $t\geq 0$,}
\label{Lemma: gap convergence}
$$\left(1-\frac{\sM}{ N_t }\right)\nsa{\mathbb{E}[\cG^{t}(x^{t+1}_{0}, y^{t+1}_{0})]} \leq \frac{\sM}{ N_t }\nsa{\mathbb{E}[\cG^{t}(x^{t}_{0},y^{t}_{0})]}  + \Xi_{\tau,\sigma,\theta}.
$$
\end{lemma}
\begin{proof}
\sa{It is shown in~\cite[Lemma 1]{yan2020optimal} that 
$$\frac{\mu_x}{4}\|x^{t}_{*}(y) - \sa{x'} \|^{\sa{2}} + \frac{\mu_y}{4} \|y_*(x) - \sa{y'} \|^{\sa{2}}\leq \cG^{t}(x,y) + \cG^{t}(x',y')$$
holds {for all $(x,y),(x',y')\in \dom f \times \dom g$.}
\nsa{It is important to note that since $\dom f$ and $\dom g$ are compact sets, \eqref{eq:Gap bound with fg} implies that $\mathbb{E}[\cG^{t}(x_0^{t+1},y_0^{t+1})]<\infty$. Furthermore, since $\cG^t(\cdot,\cdot)\geq 0$, we also have $\mathbb{E}[\cG^{t}(x_0^{t+1},y_0^{t+1})]>-\infty$; hence, $-\infty<\mathbb{E}[\cG^{t}(x_0^{t+1},y_0^{t+1})]<\infty$ for all $t\geq 0$.}
\mg{Then \eqref{eq:Gap bound with fg} and above inequality with the choice of $x = x_0^{t+1}$, $y=y_0^{t+1}$, $x' = x_0^t$, $y' = y_0^t$ together yield the desired result} --\nsa{one can subtract $\frac{\sM}{ N_t }\mathbb{E}[\cG^{t}(x^{t+1}_{0}, y^{t+1}_{0})]$ from both sides $\mathbb{E}[\cG^{t}(x^{t+1}_{0}, y^{t+1}_{0})]$ is finite.}
}
\end{proof}

\sa{For the sake of completeness, we state \cite[Lemma 8]{yan2020optimal} below, which will be used in our analysis.}

\begin{lemma}
\label{lemma: different bound of gap}
\sa{\textbf{\cite[Lemma 8]{yan2020optimal}.}} Under the premise of \cref{lemma:bound-gap}, for any $\beta_{1},\beta_{2}\in(0,1)$ and $t\geq 0$,  
\begin{equation}
    \begin{aligned}
        &\cG^{t}(x^{t+1}_{0},y^{t+1}_{0})\geq \left(1-\frac{\gamma+\mu_x}{\gamma}\Big(\frac{1}{\beta_1}-1\Big)\cG_{t+1}(x^{t+1}_{0},y^{t+1}_{0})\right) - \frac{\gamma+\mu_x}{2}\frac{\beta_1}{1-\beta_1}\|x^{t+1}_{0}-x^{t}_{0}\|^2,\\
        &\cG^{t}(x^{t+1}_{0},y^{t+1}_{0})\geq \phi(x^{t+1}_{0})-\phi(x^{t}_{0}) + \frac{\gamma+\mu_x}{2}\|x^{t+1}_{0}-x^{t}_{0}\|^2,\\
        &\cG^{t}(x^{t+1}_{0},y^{t+1}_{0})\geq \frac{\gamma \beta_2}{2}\|x^{t}_{0}-x^{t}_{*}\|^2 - \frac{\gamma\beta_2}{2(1-\beta_2)}\|x^{t+1}_{0}-x^{t}_{0}\|^2,
        \end{aligned}
\end{equation}
\nsa{hold w.p. 1,} \mg{where $x^{t}_{*}=\prox{\lambda\phi}(x^{t}_{0})$.}
\end{lemma}
\sa{Recall that we aim to control the quantity $x^{t}_{0} - \prox{\lambda\phi}(x^{t}_{0})$ as it directly determines $\grad\phi_\lambda(x^{t}_{0})$, and we also have that $\|x^{t}_{0} - \prox{\lambda\phi}(x^{t}_{0})\|=\|x^{t}_{0} - x^{t}_{*}\|$. Thus, in the following result, we bound $\mathbb{E}[\|x^{t}_{0} - x^{t}_{*}\|^2]$. Moreover, this result will also help us construct a telescoping sum for analyzing the convergence of $\{x_0^t\}_{t\geq 0}$ to a stationary point.}
\begin{lemma} 
\label{Lemma: telescope sum chpt2}
\sa{Under the premise of Lemma~\ref{lemma:bound-gap} and \nsa{\cref{assump:compact}}, for any $\beta_{1},\beta_{2}\in(0,1)$, and $p_1,p_2,p_3>0$ such that $p_1+p_2+p_3=1$, it holds for all $t\geq 0$ that}
{\small
\begin{equation}
\label{eq: WCSC step inequality chpt2}
    \begin{aligned}
        &\left(1- \fixspp\right)\frac{\gamma p_3\beta_2}{2}\ED{x^{t}_{0} - x^{t}_{*}} 
        \\
        \leq &  \fixspp\mathbb{E}\left[\cG^{t}(x^{t}_{0},y^{t}_{0})\right] - 
        \left(1- \fixspp\right) p_1\Big(1-\frac{\gamma+\mu_x}{\gamma}\Big(\frac{1}{\beta_{1}}-1\Big)\Big)
        \mathbb{E}\left[\mathcal{G}^{t+1}(x^{t+1}_{0},\sa{y^{t+1}_{0}}) \right]
        \\
        & + \left(1- \fixspp\right)p_2\mathbb{E}\left[\phi(x^{t}_{0}) - \phi(x^{t+1}_{0})\right]
        \\
        & \sa{+\frac{1}{2}\left(1- \fixspp\right)\left(p_1(\gamma+\mu_x) \frac{\beta_1}{1-\beta_{1}} - p_2(\gamma+\mu_x) + p_3\gamma\frac{\beta_2}{1-\beta_{2}}\right)}\ED{x^{t+1}_{0} - x^{t}_{0}} +\sXi.
    \end{aligned}
\end{equation}}%
\end{lemma}
\begin{proof}
Using \cref{lemma: different bound of gap}
and 
$\cG^{t}(x^{t+1}_{0},y^{t+1}_{0}) = (p_1+p_2+p_3)\cG^{t}(x^{t+1}_{0},y^{t+1}_{0})$
leads to 
\begin{equation*}
    \begin{aligned}
    \mathbb{E}\left[\cG^{t}(x^{t+1}_{0}, y^{t+1}_{0})\right]\geq & - \left(p_1\frac{\gamma+\mu_x}{2} \frac{\beta_1}{1-\beta_{1}} - p_2\frac{\gamma+\mu_x}{2} + p_3\frac{\gamma\beta_2}{2(1-\beta_2)}\right)\ED{x^{t+1}_{0} - x^{t}_{0}} \\
    & + p_1\left(1-\frac{\gamma+\mu_x}{\gamma}\Big(\frac{1}{\beta_{1}}-1\Big)\right)\mathbb{E}\left[\nsa{\mathcal{G}^{t+1}}(x^{t+1}_{0},y^{t+1}_{0})\right] \\
        & + p_2\mathbb{E}\left[\phi(x^{t+1}_{0}) - \phi(x^{t}_{0})\right]  + p_3\frac{\gamma\beta_{2}}{2}\ED{x^{t}_{0} - x^{t}_{*}}.
    \end{aligned}
\end{equation*}
Then, combining \sa{this inequality with} \cref{Lemma: gap convergence} yields the desired result.
\end{proof}

\sa{Finally, in the following result, we establish a preliminary convergence result for \xz{\texttt{SAPD+}} under compactness assumption stated in \cref{assump:compact}.}
\begin{theorem}
\label{thm:WC_SP}
Under the premise of Lemma~\ref{lemma:bound-gap}, given $T\in\integers_+$, suppose $N_t=N$ for all $t=\sa{0},\ldots T$ for some \sa{$N\in\integers_+$} such that $ N  \geq \sa{(1+\zeta)} M_{\tau,\sigma,\theta}$ for some $\zeta>0$, and the inequality system,
\begin{subequations}
\label{eq: additional stepsize condition WCSC}
\begin{align}
        & \tfrac{M_{\tau,\sigma,\theta}}{ N } - 
        \left(1-\tfrac{M_{\tau,\sigma,\theta}}{ N }\right) p_1\Big(1-\frac{\gamma+\mu_x}{\gamma}\Big(\frac{1}{\beta_{1}}-1\Big)\Big) \leq 0, \label{eq: additional stepsize condition WCSC-1}
        \\& 
        \nsa{(\gamma+\mu_x)\Big(p_1\frac{\beta_1}{1-\beta_{1}} - p_2\Big)} + p_3\gamma\frac{\beta_2}{1-\beta_2} \leq 0,\label{eq: additional stepsize condition WCSC-2}
\end{align}
\end{subequations}
has a solution for some $\beta_{1},\beta_{2}\in(0,1)$ and $p_1,p_2,p_3>0$ such that $p_1+p_2+p_3=1$. Then, for $\lambda=(\gamma+\mu_x)^{-1}$, \sa{under \cref{assump:compact},} the following bound holds for all $T\geq 1$:
\begin{equation}\label{eq: sketch convergence result}
    \begin{aligned}
        \frac{1}{T+1}\sum_{t=0}^{T}\mathbb{E}\left[\|\nabla\phi_{\lambda}(x^{t}_{0}) \|^2 \right]
        \leq
        &       \frac{2(1+\zeta)(\gamma+\mu_x)^2}{\zeta\gamma p_3 \beta_2}\left(\frac{1}{T+1}
        \cG(x_0^0,y_0^0)
        +\sXi \right).
    \end{aligned}
\end{equation}
\end{theorem}
\begin{proof}
\sa{Since $\dom f$ and $\dom g$ are compact sets, $\mathbb{E}[\cG^{t}(x_0^t,y_0^t)]\in\reals$ exist for $t=0,\ldots,T$, i.e., $-\infty<\mathbb{E}[\cG^{t}(x_0^t,y_0^t)]<\infty$ for all $t$. Therefore,} if we sum up equation \eqref{eq: WCSC step inequality chpt2} from $0$ to T, 
we get
\begin{equation}\label{eq: proof step A}
    \begin{aligned}
        &\sum_{t=\sa{0}}^{T}\left(1- \fixspp\right)
        \frac{\gamma p_3 \beta_2}{2}
       \mathbb{E}\left[\|x^{t}_{0} - x^{t}_{*} \|^2 \right]
        \\
        & \leq  \tfrac{M_{\tau,\sigma,\theta}}{N_{0}}
        \nsa{\cG^0(x_0^0,y_0^0)}
        - \left(1-\tfrac{M_{\tau,\sigma,\theta}}{ N_t }\right) p_1\left(1-\frac{\gamma+\mu_x}{\gamma}\Big(\frac{1}{\beta_{1}}-1\Big)\right)
        \mathbb{E}\left[\mathcal{G}^{T+1}(\sa{x^{T+1}_{0},y^{T+1}_{0}}) \right]
        \\
        &
        + \sum_{t=0}^{T-1}\left(\tfrac{M_{\tau,\sigma,\theta}}{N_{t+1}} - 
        \left(1- \fixspp\right) p_1\left(1-\frac{\gamma+\mu_x}{\gamma}\Big(\frac{1}{\beta_{1}}-1\Big)\right)\right)
        \mathbb{E}\left[\mathcal{G}^{t+1}(x^{\sa{t+1}}_{0},y^{\sa{t+1}}_{0})\right] 
        \\
        & +\left(1-\tfrac{M_{\tau,\sigma,\theta}}{ N_0}\right)p_2
        \nsa{\phi(x^{\sa{0}}_{0})}
        - \left(1-\tfrac{M_{\tau,\sigma,\theta}}{N_T}\right)p_2\mathbb{E}\left[\phi(x^{\sa{T+1}}_{0})\right] + p_2\sum_{t=0}^{T-1}\underbrace{\left( \fixspp -\tfrac{M_{\tau,\sigma,\theta}}{N_{t+1}} \right)}_{\textbf{part 1}}\mathbb{E}\left[\phi(x^{\sa{t+1}}_{0}) \right]
        \\
        &
        +  \sum_{t=0}^T \left(1- \fixspp\right)\left(p_1\frac{\gamma+\mu_x}{2} \frac{\beta_{1}}{1-\beta_{1}} - p_2\frac{\gamma+\mu_x}{2} 
        +
        p_3\gamma\frac{\beta_2}{2(1-\beta_2)}\right)\mathbb{E}\left[\|x^{t+1}_{0} - x^{t}_{0}\|^2\right] \\
        &+(T+1)\sXi
    \end{aligned}
\end{equation}
\sa{Thus, using $N_t=N$ for $t=0,\ldots,N$,} it follows from \sa{the conditions in} \eqref{eq: additional stepsize condition WCSC} that
\begin{equation}
    \begin{aligned}
        &\frac{1}{T+1}\sum_{t=0}^{T}\left(1-\sa{\tfrac{M_{\tau,\sigma,\theta}}{ N }}\right) \frac{\gamma p_3 \beta_2}{2}\mathbb{E}\left[\|x^{t}_{0} - x^{t}_{*} \|^2 \right]
        \\
        \leq & \frac{1}{T+1}\tfrac{M_{\tau,\sigma,\theta}}{ N }\mathbb{E}\left[\cG^0(x_0^0,y_0^0) \right] \\
        &-
       \frac{1}{T+1} \left(1-\tfrac{M_{\tau,\sigma,\theta}}{ N }\right) p_1\left(1-\frac{\gamma+\mu_x}{\gamma}\Big(\frac{1}{\beta_{1}}-1\Big)\right)
       \mathbb{E}\left[ \mathcal{G}^{T+1}(x^{T+1}_{0},\sa{y^{T+1}_{0}}) \right]
        \\
        &
        + \frac{p_2\left(1-\tfrac{M_{\tau,\sigma,\theta}}{ N }\right)}{T+1}\mathbb{E}\left[\phi(x^0_0)  - \phi(x^{\sa{T}+1}_{0})\right]
        +\sXi
        \\
        \leq
        & \frac{1}{T+1} \tfrac{M_{\tau,\sigma,\theta}}{ N }
        \nsa{\cG^0(x_0^0,y_0^0)} 
         + \frac{p_2\left(1-\tfrac{M_{\tau,\sigma,\theta}}{ N }\right)}{T+1} 
         \nsa{\cG(x_0^0,y_0^0)} 
        +\sXi,
    \end{aligned}
\end{equation}
\sa{which follows from $(i)$} $\cG_{T+1}(x^{T+1}_{0},y^{T+1}_{0})\geq 0$, $(ii)$ $\phi(x^0_0)  - \phi(x^{T+1}_{0}) = \cL(x^0_0, y_*(x^0_0)) -\cL(x^{T+1}_{0}, y_*(x^{\sa{T}+1}_{0}))\leq \cL(x^0_0, y_*(x^0_0)) -\cL(x^{T+1}_{0}, y^0_0)\leq  \sa{\sup_{y'\in\cY}\cL(x^0_0, y')} -\inf_{x'\in\cX}\cL(x', y^0_0)=\cG(x^0_0,y^0_0)$, and also from the fact that \eqref{eq: additional stepsize condition WCSC-1} implies  $\left(1-\tfrac{M_{\tau,\sigma,\theta}}{ N }\right) p_1\Big(1-\frac{\gamma+\mu_x}{\gamma}\Big(\frac{1}{\beta_{1}}-1\Big)\Big)\geq 0$. Then dividing both sides by $\left(1-\tfrac{M_{\tau,\sigma,\theta}}{ N }\right) \frac{\gamma p_3 \beta_2}{2}$ gives us
{\small
\begin{eqnarray}
        \lefteqn{\frac{1}{T+1}\sum_{t=1}^{T}\|x^{t}_{0} - x^{t}_{*} \|^2}\\
        && \leq \frac{2}{(1-\tfrac{M_{\tau,\sigma,\theta}}{ N })\gamma p_3\beta_{2}}\Big(
        \frac{1}{T+1} \tfrac{M_{\tau,\sigma,\theta}}{ N }
        \nsa{\cG^0(x_0^0,y_0^0)} 
        + \frac{p_2\left(1-\tfrac{M_{\tau,\sigma,\theta}}{ N }\right)}{T\sa{+1}}
        \nsa{\cG(x_0^0,y_0^0)}
        +\sXi\Big),\nonumber
        \\
        && \leq \frac{2(1+\zeta)}{\zeta\gamma p_3 \beta_2}\left(\frac{1}{T+1}
        \nsa{\cG(x_0^0,y_0^0)}
        +\sXi \right),\nonumber
\end{eqnarray}}%
where the second inequality 
\sa{follows from} $\cG(x_0^0,y_0^0)\geq \cG^0(x_0^0,y_0^0)$, \sa{and for $p_2\in(0,1)$, we have} $ N \geq (1+\zeta)M_{\tau,\sigma,\theta}$. Finally, we get the desired result using Lemma \ref{Lemma: graident of ME}.
\end{proof}

{\subsection{A particular parameter choice}}
We employ the 
matrix inequality~(MI) in~\cref{eq: SCSC SAPD LMI} to describe the admissible set of algorithm parameters that guarantee convergence of \nsa{Algorithm~\ref{Alg: SAPD}, i.e., inner loop \xz{of} \texttt{SAPD+} when \texttt{VR-flag} is \textbf{false}.} In this subsection, we  compute a particular solution 
 \sa{by exploiting 
 the structure of MI in~\cref{eq: SCSC SAPD LMI}.}  
 \nsa{This particular solution is for} \sa{solving the SCSC subproblems in \cref{eq: SCSC problem chpt2}.}
\begin{lemma}\label{LEMMA: Noise LMI after young's ineq-gap}
For any $\mu_x\geq0$, let $L'_{xx}= L_{xx}+\gamma+\mu_x$. Suppose $\theta = 1$, and $\tau, \sigma>0$, \sa{satisfy}
{\small
\begin{equation}
\label{eq:sufficient_cond_noisy_LMI-gap}
\tau \leq \frac{1}{L'_{xx}+L_{yx}}, \quad \sigma \leq \frac{1}{2L_{yy}+L_{yx}}.
\end{equation}}
Then $\{\tau, \sigma,\theta,\alpha\}$ is a solution to \eqref{eq: SCSC SAPD LMI} for $\alpha =  L_{yx} +  L_{yy}$.
\end{lemma}
\begin{proof}
\sa{It follows from the choice of $\tau$ and $\sigma$ in \eqref{eq:sufficient_cond_noisy_LMI-gap} and $\theta=1$ that a sufficient condition for \eqref{eq: SCSC SAPD LMI} is given by the following smaller matrix inequality for $\alpha =  L_{yx} +  L_{yy}$,}
{\footnotesize
\begin{equation*}
\begin{aligned}
\sa{\mathbf{0}\preceq}
\begin{pmatrix}
  \tfrac{1}{\tau} - \xz{L'_{xx}} & 0 & - L_{yx}\\ 
 0 &\tfrac{1}{\sigma}- \alpha & - L_{yy}\\
 - L_{yx} & - L_{yy} & \alpha
\end{pmatrix}
= & 
    \begin{pmatrix}
  \tfrac{1}{\tau} - \xz{L'_{xx}} & 0 & - L_{yx}\\ 
 0 & \tfrac{1}{\sigma}-   L_{yx} -  L_{yy} & - L_{yy}\\
 - L_{yx} & - L_{yy} & L_{yx} + L_{yy}
\end{pmatrix}
\triangleq & 
M_1 + M_2,
\end{aligned}
\end{equation*}}
where
${\small M_1 
 \triangleq 
 \begin{pmatrix}
\tfrac{1}{\tau} - \xz{L'_{xx}} & 0 & - L_{yx}\\ 
 0 & 0 & 0 \\
 - L_{yx} & 0 & L_{yx}
\end{pmatrix}}$ and 
${\small
M_2 \triangleq 
    \begin{pmatrix}
  0 & 0 & 0\\ 
 0 & \tfrac{1}{\sigma}-   L_{yx} -  L_{yy} & - L_{yy}\\
 0 & - L_{yy} &  L_{yy}
\end{pmatrix}}$. 
Therefore, the Schur complement conditions together with eq.~\eqref{eq:sufficient_cond_noisy_LMI-gap}
 imply $M_1 
 \succeq 0$ and $M_2 \succeq 0$, respectively.
Thus, $M_1+M_2\succeq 0$. 
\end{proof}
{\subsection{Proof of Theorem~\ref{cor:complexity-fg}}}
\label{pf:complexity-fg}
\begin{proof}
\sa{Using the results we derived in the previous two subsections, we are now ready to provide the proof of Theorem~\ref{cor:complexity-fg}.}

\sa{For the inner loop iterations, \cref{LEMMA: Noise LMI after young's ineq-gap} ensures that \cref{eq: SCSC SAPD LMI} holds for our $\{\tau,\sigma,\theta\}$  choice in \cref{Condition: SP solution to noisy LMI label-gap}.} 
For the outer loop, if we \sa{set $N$ as in \cref{Condition: SP solution to noisy LMI label-gap} and}
\begin{equation}\label{eq: speical paramater-gap}
    \begin{aligned}
          p_1 = \frac{1}{16},~p_2 = \sa{\frac{19}{32}},~p_3=\sa{\frac{11}{32}},~\beta_1 = \frac{4}{5},~\beta_2 = \frac{1}{2},~\zeta=32,
    \end{aligned}
\end{equation}
\sa{all assumptions of \cref{thm:WC_SP} are satisfied, i.e., both the inequality system \cref{eq: additional stepsize condition WCSC} and $N \geq \sa{(1+\zeta)} M_{\tau,\sigma,\theta}$ hold.} 

Specifically, because $\mu_x=\gamma$ and $\theta=1$, we have $\sM = \max \{ \frac{4}{\gamma\tau}, \frac{8}{\mu_y\sigma}\}$. Therefore, we know that $N \geq \sa{(1+\zeta)} M_{\tau,\sigma,\theta}$ is trivially true.
Moreover, using $\sM/N\leq (1+\zeta)^{-1}$, it follows that \cref{eq: additional stepsize condition WCSC-1} holds for $\mu_x=\gamma$, $p_1 = \frac{1}{16}$ and $\beta_1 = \frac{4}{5}$, i.e.,
$$
 \sa{\frac{M_{\tau,\sigma,\theta}}{N} - 
        \left(1-\frac{M_{\tau,\sigma,\theta}}{N}\right) p_1\Big(1-\frac{\gamma+\mu_x}{\gamma}\Big(\frac{1}{\beta_{1}}-1\Big)\Big) =\frac{33}{32} \tfrac{M_{\tau,\sigma,\theta}}{ N } - \frac{1}{32}
        \leq \frac{33}{32}\frac{1}{1+\zeta}-\frac{1}{32}=0.}
$$
\sa{Moreover, it is trivial to check that \cref{eq: additional stepsize condition WCSC-2} holds for the parameter values given in \cref{eq: speical paramater-gap}.} 

Since all assumptions of \cref{thm:WC_SP} are satisfied for parameters chosen as in \cref{Condition: SP solution to noisy LMI label-gap} and \cref{eq: speical paramater-gap},
if we substitute \cref{eq: speical paramater-gap} into \cref{eq: sketch convergence result}, if follows that 
\begin{equation*}
    \begin{aligned}
        \frac{1}{T+1}\sum_{t=0}^{T} \mathbb{E}\left[\|\nabla\phi_{\lambda}(x^t_0) \|^2 \right]
        \leq
        &       
        \sa{48\gamma}
        \left(\frac{1}{T+1} 
        \sa{\cG(x_0^0,y_0^0)}
        +\sXi \right).
    \end{aligned}
\end{equation*}
Thus, for any $\epsilon>0$,
the right side of the above inequality can be bounded by $\epsilon^2$ when
\begin{equation}\label{eq: split bound-gap}
    \frac{\sa{48\gamma}}{T+1}
    \cG(x_0^0,y_0^0)
    \leq \frac{\epsilon^2}{2},\qquad 
    \sa{48\gamma}\sXi\leq\frac{\epsilon^2}{2}.
\end{equation}

Note that \xz{because
    $\Xi_{\tau,\sigma,\theta} = \tau\left( \Xi^x_{\tau,\sigma,\theta} + \tfrac{1}{2}\right)\delta_x^2  
     + \sigma\left( \Xi^y_{\tau,\sigma,\theta} +  \tfrac{3}{2}\right)\delta_y^2$}, a sufficient condition for the second inequality in \cref{eq: split bound-gap} is that 
\begin{equation}\label{eq: sufficient condition for variance complexity-gap}
    \sa{24\gamma}\tau(1+2\Xi^x_{\tau,\sigma,\theta})\delta^2_x\leq\frac{\epsilon^2}{4},
    \qquad
    \sa{24\gamma}\sigma(3+2\Xi^{\sa{y}}_{\tau,\sigma,\theta})\delta_y^2\leq\frac{\epsilon^2}{4}.
\end{equation}
{Moreover, \sa{recall that $\Xi^x_{\tau,\sigma,\theta}$ and $\Xi^y_{\tau,\sigma,\theta}$ are defined in \cref{lemma:bound-gap}; for $\theta=1$, they can be simplified as follows:}
 \begin{align*}
     \Xi^x_{\tau,\sigma,\theta} =  1+ \sigma L_{yx}, \quad
     \Xi^y_{\tau,\sigma,\theta} = 3\left(4 + 2\sigma L_{yy}  +  2\tau\sigma L_{yx}L_{xy} \right) + \tau L_{yx}.
\end{align*}}%
Because the choice of $\{\tau,~\sigma\}$ in \cref{Condition: SP solution to noisy LMI label-gap} implies that 
$$
\tau L_{yx}\leq 1,\quad \tau L_{xy}\leq 1,\quad\sigma L_{yy}\leq \frac{1}{2}, \quad \sigma L_{yx}\leq 1,
$$
we can upper bound $\Xi^x_{\tau,\sigma,\theta}$ and $\Xi^y_{\tau,\sigma,\theta}$ as follows:
$$
 \Xi^x_{\tau,\sigma,\theta}\leq 2,\quad \Xi^y_{\tau,\sigma,\theta} \leq 22.
$$
Therefore, with the choice of $\{\tau,\sigma\}$ in \cref{Condition: SP solution to noisy LMI label-gap}, we have a sufficient condition for \cref{eq: sufficient condition for variance complexity-gap} as follows:
\begin{equation*}
    120\gamma\tau\delta^2_x\leq\frac{\epsilon^2}{4},
    \qquad
    1128\gamma\sigma\delta_y^2\leq\frac{\epsilon^2}{4}.
\end{equation*}
Indeed, the above condition is trivially \sa{satisfied by our choice of $\{\tau,\sigma\}$ given in \cref{Condition: SP solution to noisy LMI label-gap}}. Therefore, \sa{the second condition in \eqref{eq: split bound-gap}, i.e., $
\sa{48\gamma}\sXi\leq\frac{\epsilon^2}{2}$, holds for} the choice of $\{\tau,\sigma\}$ in \cref{Condition: SP solution to noisy LMI label-gap}. Thus, from the first inequality in \cref{eq: split bound-gap}, we get 
{
\begin{align}
    \na{\frac{1}{T+1}\sum_{t=0}^{T} \mathbb{E}\left[\|\nabla\phi_{\lambda}(x^t_0) \|^2 \right]
        \leq\epsilon^2,}\quad \forall~T \geq \sa{96}\cG(x_0^0,y_0^0)  \cdot\frac{\gamma}{\epsilon^2}+1.
    \label{eq:T-bound-gap}
\end{align}}%
\nsa{
\na{Suppose that $t_*$ is chosen uniformly at random among $\{0,1,\ldots,T\}$, then $x_\epsilon=x_0^{t_*}$ satisfies $\mathbb{E}\left[\|\nabla\phi_{\lambda}(x_\epsilon) \|^2 \right]
        \leq\epsilon^2$. Combining with this observation,
to finally show} the complexity result, recall that}
$N = 33\max\{\frac{4}{\gamma\tau},\frac{8}{\mu_y\sigma}\}$. 
Using the the choice of $\{\tau,\sigma\}$ in \cref{Condition: SP solution to noisy LMI label-gap}  we derive that
{\footnotesize
\begin{equation}
\label{eq:N-bound-gap}
N = \mathcal{O}\Big(  
\frac{\max\{L_{xx},L_{yx},L_{xy}\}}{\gamma}  + \frac{\max\{L_{yy},L_{yx}\}}{\mu_y}
+ \Big(
\frac{\delta_x^2}{\gamma} +
\frac{\delta_y^2}{\mu_y} \Big)
\frac{\gamma}{\epsilon^2}\Big).
\end{equation}}%
\nsa{Moreover,} \sa{since \xz{\texttt{SAPD+}} requires $NT$ oracle calls in total, combining \eqref{eq:T-bound-gap} with \eqref{eq:N-bound-gap} leads to $\cO(\epsilon^{-4})$ bound on $C_\epsilon$ as stated in~\cref{cor:complexity-fg}, which completes the proof.}
\end{proof}

\section{Proof of Theorem~\ref{thm:metric-equivalance} and preliminary technical results}
\label{sec:Thm2_proof}
\sa{\xz{Suppose Assumptions~\ref{ASPT: fg},~\ref{ASPT: lipshiz gradient},~\ref{ASPT: unbiased noise assumption} hold}. Given \sa{$x_\epsilon$}, an $\epsilon$-stationary point for the $\gamma$-weakly convex function $\phi(\cdot)=\max_{y\in\mathcal{Y}} \cL(\cdot,y)$, i.e., $\mathbb{E}\big[\|\nabla\phi_{\lambda}(\sa{x_\epsilon})\|\big]\leq\frac{\epsilon}{2}$ for 
\sa{some fixed} $\lambda\in(0,\gamma^{-1})$. \sa{Let $\phi^s(\cdot)\triangleq\max_{y\in\cY}\Phi(\cdot,y)-g(y)$ so that $\phi=f+\phi^s$.} In this section we show that initialized from $x_\epsilon$ and using appropriately selected step size parameters, within 
$\tilde{\mathcal{O}}(\frac{1}{\xz{\epsilon^2}})$ stochastic first-order oracle calls,
\texttt{SAPD}, stated in~Algorithm~\ref{Alg: SAPD}, can generate \sa{$\tilde x$} such that
\mg{$\mathbb{E}\big[ \norm{G_\lambda(\tilde{x})} \leq \epsilon$,} where generalized gradient mapping $G_\lambda$ is defined in \eqref{eq:generalized_GM}.}
\begin{lemma}\label{lemma: dist bound with fg}
Suppose Assumptions~\ref{ASPT: fg},~\ref{ASPT: lipshiz gradient},~\ref{ASPT: unbiased noise assumption} hold. \sa{\nsa{Given some $(x_0,y_0)\in\dom f\times\dom g$,} consider the SCSC problem in~\eqref{eq:WCSC-subproblem-generic} for some  $\mu_x>0$. Let $\{x_k,y_k\}_{k\geq 0}$ be generated by \xz{\texttt{SAPD}}, stated in~Algorithm~\ref{Alg: SAPD}, initialized from {$(x_0,y_0)$} and using} $\tau,\sigma,\theta>0$ that satisfy {\small
\begin{equation}
    \label{eq: SCSC SAPD LMI-dist}
  \sa{G\triangleq}\begin{pmatrix}
    \frac{1}{\tau}(1-\frac{1}{\rho})+\frac{\mu_x}{\rho} & 0 & 0 & 0 & 0\\ 
  0 & \frac{1}{\sigma}(1-\frac{1}{\rho})+\mu_y & (\frac{\theta}{\rho} - 1)L_{yx} & (\frac{\theta}{\rho} - 1)L_{yy}  & 0\\ 
  0 & (\frac{\theta}{\rho} - 1)L_{yx} & \tfrac{1}{\tau} - L'_{xx} & 0 & -  \frac{\theta}{\rho}L_{yx}\\ 
  0& (\frac{\theta}{\rho} - 1)L_{yy} & 0 & \frac{1}{\sigma} - \alpha & -  \frac{\theta}{\rho}L_{yy}\\
  0 & 0 & - \frac{\theta}{\rho}L_{yx} & -  \frac{\theta}{\rho}L_{yy}  & \frac{\alpha}{\rho}
\end{pmatrix}\succeq 0
\end{equation}}%
{for some} $\alpha \in [0,  {\tfrac{1}{\sigma}\sa{)}}$ and \xzh{$\rho\in (0,1\sa{)}$}, \sa{where $L'_{xx}\triangleq L_{xx}+\mu_x+\gamma$}. Define $\phi(x)=\max_{y\in\cY}\cL(x,y)$; and let \sa{$\hat x=\prox{\lambda\phi}(\xzh{x_0})$ for $\lambda=(\mu_x+\gamma)^{-1}$ and  $\nsa{y_*(\hat{x})}=\argmax_{y\in\cY}\cL(\hat{x},y)$.} Then for all $N\in\integers_+$, \sa{it holds that}
\begin{equation}\label{eq:dist bound}
      \begin{aligned}
        \mathbb{E}\Big[& \Big(\frac{1}{\tau}-\mu_x\Big) \|x_N-\hat{x}\|^2 +  {\Big(\frac{1}{\sigma} - \alpha \Big)}\|y_N-\nsa{y_{*}(\hat x)}\|^2 \Big]  \\
        & \leq
\rho^{N}\left( \frac{1}{\tau}\|x_{0}-\hat x\|^2
      +\frac{1}{\sigma}\|y_{0}-\nsa{y_{*}(\hat x)}\|^2\right)
      +\frac{\rho}{1-\rho}\Big(\tau\Xi_{\tau,\sigma,\theta}^x \delta_x^2  +  \sigma\Xi_{\tau,\sigma,\theta}^y\delta_y^2\Big),
    \end{aligned}
\end{equation}
where \sa{$\Xi^x_{\tau,\sigma,\theta}$ and $\Xi^y_{\tau,\sigma,\theta}$ are defined in~\eqref{eq:xi_x} and \eqref{eq:xi_y}, respectively.}
\end{lemma}
\begin{proof}
For easier readability, we provide the proof in a separate subsection, see~\cref{sec:pf-lema2-lemma7}.
\end{proof}
In the following part, we will compute a particular solution 
 \sa{by exploiting 
 the structure of MI in~\cref{eq: SCSC SAPD LMI-dist} and use this particular solution 
 \nsa{for the rest of the proof}.} First, in \cref{LEMMA: Noise LMI after young's ineq}, we give an intermediate condition to help us construct the particular solution subsequently provided in \cref{lemma: explicit solution to noisy LMI} for \sa{solving the generic SCSC subproblems in \cref{eq:WCSC-subproblem-generic}.}
\begin{lemma}\label{LEMMA: Noise LMI after young's ineq}
For any $\mu_x>0$, let $L'_{xx}= L_{xx}+\gamma+\mu_x$. Suppose $\rho = \theta$, and $\tau, \sigma>0$, $\theta\in(0,1\sa{)}$ \sa{satisfy}
{\small
\begin{equation}
\label{eq:sufficient_cond_noisy_LMI}
\tau \geq \frac{1-\theta}{\mu_x}, \quad \sigma\geq \frac{1-\theta}{\mu_y\theta},\quad \frac{1}{\tau} \geq \xz{L'_{xx}} + \pi_1  L_{yx},
\quad \frac{1}{\sigma} \geq \frac{ \theta L_{yx}}{\pi_1} + \Big(\frac{\theta}{\pi_2} +\pi_2\Big)  L_{yy},
\end{equation}}
for some $\pi_1,\pi_2>0$.
Then $\{\tau, \sigma,\theta,\alpha\}$ is a solution to \eqref{eq: SCSC SAPD LMI-dist} for $\alpha = \frac{\theta L_{yx}}{\pi_1} + \frac{\theta L_{yy}}{\pi_2}$.
\end{lemma}
\begin{proof}
\sa{It follows from the choice of $\tau$ and $\sigma$ in \eqref{eq:sufficient_cond_noisy_LMI} and $\rho=\theta$ that a sufficient condition for 
\cref{eq: SCSC SAPD LMI-dist}, i.e., for $G\succeq 0$, is given by the following smaller matrix inequality for $\alpha = \frac{\theta L_{yx}}{\pi_1} + \frac{\theta L_{yy}}{\pi_2}$,}
{\footnotesize
\begin{equation*}
\begin{aligned}
\sa{\mathbf{0}\preceq}
\begin{pmatrix}
  \tfrac{1}{\tau} - \xz{L'_{xx}} & 0 & - L_{yx}\\ 
 0 &\tfrac{1}{\sigma}- \alpha & - L_{yy}\\
 - L_{yx} & - L_{yy} & \tfrac{\alpha}{ \theta} 
\end{pmatrix}
= & 
    \begin{pmatrix}
  \tfrac{1}{\tau} - \xz{L'_{xx}} & 0 & - L_{yx}\\ 
 0 & \tfrac{1}{\sigma}-  \tfrac{\theta L_{yx}}{\pi_1} - \tfrac{\theta L_{yy}}{\pi_2} & - L_{yy}\\
 - L_{yx} & - L_{yy} & \tfrac{L_{yx}}{\pi_1} + \tfrac{ L_{yy}}{\pi_2}
\end{pmatrix}
\triangleq & 
M_1 + M_2,
\end{aligned}
\end{equation*}}
where
${\small M_1 
 \triangleq 
 \begin{pmatrix}
\tfrac{1}{\tau} - \xz{L'_{xx}} & 0 & - L_{yx}\\ 
 0 & 0 & 0 \\
 - L_{yx} & 0 & \tfrac{L_{yx}}{\pi_1}
\end{pmatrix}}$ and 
${\small
M_2 \triangleq 
    \begin{pmatrix}
  0 & 0 & 0\\ 
 0 & \tfrac{1}{\sigma}-  \tfrac{\theta L_{yx}}{\pi_1} - \tfrac{\theta L_{yy}}{\pi_2} & - L_{yy}\\
 0 & - L_{yy} & \tfrac{ L_{yy}}{\pi_2}
\end{pmatrix}}$. 
Therefore, since $\pi_1,\pi_2>0$, the Schur complement conditions in~\eqref{eq:sufficient_cond_noisy_LMI}, i.e., \sa{the third and the fourth} inequalities,
 imply $M_1 
 \succeq 0$ and $M_2 \succeq 0$, respectively.
Thus, $M_1+M_2\succeq 0$. 
\end{proof}
Lemma~\ref{LEMMA: Noise LMI after young's ineq} shows
 that every solution to 
\eqref{eq:sufficient_cond_noisy_LMI} can be converted to a solution to 
\eqref{eq: SCSC SAPD LMI-dist}. Next, based on
\cref{LEMMA: Noise LMI after young's ineq}, we will give another explicit parameter choice for Algorithm~\ref{Alg: SAPD} \sa{in addition to the solution we provided earlier in} \cref{LEMMA: Noise LMI after young's ineq-gap}.

\begin{lemma}\label{lemma: explicit solution to noisy LMI}
\sa{For any $\mu_x>0$, let} $\xz{L'_{xx}=} L_{xx}+\gamma+\mu_x$. For any given $\beta\in(0,1]$, let $\tau, \sigma>0$ and $\theta\in (0,1)$ be chosen satisfying
{\small
\begin{equation}
\label{Condition: SP solution to noisy LMI}
\tau = \frac{1-\theta}{\mu_x },\quad \sigma = \frac{1-\theta}{\mu_y\theta},
\quad 
\theta  \geq \bar{\theta}(\beta),
\end{equation}}%
 where $\bar{\theta}(\beta)\triangleq\max\{\bar{\theta}_1(\beta),~\bar{\theta}_2(\beta)\}\in(0,1)$ such that
\sa{
\begin{equation*}
        \bar{\theta}_1(\beta)\triangleq 1 
         - \tfrac{\beta\mu_y\xz{L'_{xx}}}{2L^2_{yx}}\big(\sqrt{1 + \tfrac{4L^2_{yx}\mu_x}{\beta \xz{{L'_{xx}}^2}\mu_y}} - 1 \big),
\end{equation*}
\begin{equation*}
        \bar{\theta}_2(\beta)\triangleq 
        \begin{cases}
        1 - \tfrac{(1-\beta)^2}{8}\tfrac{\mu_y^2}{L_{yy}^2} \Big( \sqrt{1+\tfrac{16L_{yy}^2}{(1-\beta)^2\mu_y^2}}-1\Big) & L_{yy}>0\\
        0 & L_{yy}=0.
        \end{cases}
\end{equation*}}
Then, $\{\tau, \sigma,\theta,\alpha,\xz{\rho}\}$ \sa{with $\alpha = \frac{1}{\sigma}-\sqrt{\theta}L_{yy}>0$} and $\xz{\rho=\theta}$ is a solution to \eqref{eq: SCSC SAPD LMI-dist}. 
\end{lemma}
\begin{proof}
Consider arbitrary $\tau,\sigma,\pi_1,\pi_2>0$ and $\theta\in(0,1)$. By a straightforward
calculation,  
$\{\tau, \sigma,\theta,\pi_1,\pi_2\}$ is a solution to \eqref{eq:sufficient_cond_noisy_LMI} 
if and only if
{\small
\begin{subequations}
\label{Condition: SP solution sufficient}
\begin{gather}
        \tau\geq \frac{1-\theta}{\mu_x},\quad \sigma\geq \frac{1-\theta}{\theta\mu_y},\quad 
    \pi_1 \geq  \frac{\sigma\theta L_{yx}}{1- \sigma(\pi_2+\frac{\theta}{\pi_2})L_{yy}},\label{Condition: SP solution sufficient 123}\\
    \sigma(\pi_2+\frac{\theta}{\pi_2})L_{yy} < 1,\quad
    \frac{1}{\tau}-\xz{L'_{xx}} \geq\pi_1 L_{yx}. \label{Condition: SP solution sufficient 4}
\end{gather}
\end{subequations}}
In the remainder of the proof, we fix $(\pi_1,\pi_2)$ as follows:
{\small
\begin{align}
        \pi_1 =   \frac{\sigma\theta L_{yx}}{1-\sigma\left(\pi_2+\frac{\theta}{\pi_2}\right)L_{yy}},\quad \pi_2 = \sqrt{\theta}.
        \label{Condition: SP solution to noisy LMI 3}
\end{align}}
Note the definition of $\xz{\bar{\theta}(\beta)}$ implies that $\xz{\bar{\theta}(\beta)}\in (0,1)$. Next, we show that $\theta\in [\xz{\bar{\theta}(\beta)},1)$ implies $\pi_1,\pi_2>0$; furthermore, we also show that $\tau,\sigma>0$ defined as in \eqref{Condition: SP solution to noisy LMI} for $\theta\in[\xz{\bar{\theta}(\beta)}, 1)$ together with $(\pi_1, \pi_2)$ as in \eqref{Condition: SP solution to noisy LMI 3} is a solution to \eqref{Condition: SP solution sufficient}.

First, setting $\tau,\sigma$ as in~\eqref{Condition: SP solution to noisy LMI} and $\pi_1,\pi_2$ as in~\eqref{Condition: SP solution to noisy LMI 3} imply that \eqref{Condition: SP solution sufficient 123}  
is trivially satisfied. Next, by substituting $\{\tau,\sigma,\pi_1,\pi_2\}$, chosen as in \eqref{Condition: SP solution to noisy LMI} and \eqref{Condition: SP solution to noisy LMI 3}, into \eqref{Condition: SP solution sufficient 4}, we conclude that $\{\tau,\sigma,\theta,\pi_1,\pi_2\}$ satisfies \eqref{Condition: SP solution sufficient} for any $\theta\in (0,1)$ such that
{\small
\begin{gather}
         \frac{2L_{yy}}{\mu_y}\cdot\frac{1-\theta}{\sqrt{\theta}}  \leq 1-\beta,
         \label{Condition: SP solution from sigma 2}\\
         \frac{\mu_x}{1-\theta} - \xz{L'_{xx}} \geq (1-\theta)\frac{L_{yx}^2}{\mu_y}\cdot\Big(1- \frac{2L_{yy}}{\mu_y}\cdot\frac{1-\theta}{\sqrt{\theta}}\Big)^{-1}, \label{Condition: SP solution from tau}
\end{gather}}%
for some $\beta\in(0,1]$. 
Clearly, a sufficient condition for \eqref{Condition: SP solution from tau} is
{\small
\begin{equation}\label{Condition: SP solution from tau 2}
            \frac{\mu_x}{1-\theta} - \xz{L'_{xx}} \geq (1-\theta)\frac{L_{yx}^2}{\mu_y}\cdot\frac{1}{\beta}. 
\end{equation}}%
Note that \eqref{Condition: SP solution from sigma 2} implies that $\pi_1>0$. We also have $\pi_2=\sqrt{\theta}>0$ trivially.

When $L_{yy}>0$, given any $\beta\in(0,1)$, solving \cref{Condition: SP solution from sigma 2,Condition: SP solution from tau 2} for $\theta\in(0,1)$, we get the third condition in \eqref{Condition: SP solution to noisy LMI}. Indeed, it can be checked that $\theta\in[\bar{\theta}_2\xz{(\beta)},1)$ satisfies \eqref{Condition: SP solution from sigma 2} and $\theta\in[\bar{\theta}_1\xz{(\beta)},1)$ satisfies \eqref{Condition: SP solution from tau 2}; thus, $\theta\in[\bar{\theta}\xz{(\beta)},1)$ satisfies \eqref{Condition: SP solution from sigma 2} and  \eqref{Condition: SP solution from tau 2} simultaneously. Moreover, when $L_{yy}=0$, one does not need to solve \cref{Condition: SP solution from sigma 2} as the first inequality in \eqref{Condition: SP solution sufficient 4} holds trivially; thus, the only condition on $\theta$ comes from \eqref{Condition: SP solution from tau} which is equivalent to \eqref{Condition: SP solution from tau 2} with $\beta=1$.  
The rest follows from  Lemma~\ref{LEMMA: Noise LMI after young's ineq} by setting $\alpha = \frac{\theta L_{yx}}{\pi_1} + \frac{\theta L_{yy}}{\pi_2}$. Indeed, the particular choice of $(\pi_1,\pi_2)$ in \eqref{Condition: SP solution to noisy LMI 3} gives us $\alpha = \frac{1}{\sigma}-\sqrt{\theta}L_{yy}$.
\end{proof}
\sa{Now that we have provided a particular solution to \cref{eq: SCSC SAPD LMI-dist}, we will next use this particular solution within \cref{lemma: dist bound with fg} to derive an error bound customized for this choice of parameters.} The following two technical results, i.e., Lemmas~\ref{lemma: SAPD single call complexity} and~\ref{lemma: proximal point lemma v2},  will be used later within the proof of \cref{thm:metric-equivalance}.
\begin{lemma}\label{lemma: SAPD single call complexity}
\sa{Consider $\cL$ defined in \eqref{eq:main problem}. \xz{Suppose Assumptions~\ref{ASPT: fg},~\ref{ASPT: lipshiz gradient},~\ref{ASPT: unbiased noise assumption} hold}. Given arbitrary $\xzh{x_0}$, let $\hat x=\prox{\lambda\phi}(\xzh{x_0})$, where $\phi(\cdot)=\max_{y\in\cY}\cL(\cdot,y)$ and $\lambda=(2\gamma)^{-1}$. 
For any given $\hat{\epsilon}>0$, \texttt{SAPD}, displayed in Algorithm~\ref{Alg: SAPD}, can generate $\tilde x_*\in\cX$ such that $\mathbb{E}\left[\|\tilde x_* - \hat{x}\|\right]\leq\hat{\epsilon}$ 
within $\mathcal{\tilde{O}}(\frac{1}{\hat{\epsilon}^2})$ stochastic first-order oracle calls.}
\end{lemma}
\begin{proof}
Recall that $y_*(x)=\argmax_{y\in\cY}\cL(x,y)$ for {$x\in\dom f$}.
Hence, 
$(\hat{x},y_*(\hat{x}))$ is the unique saddle point to the SCSC problem: 
\begin{align}
\label{eq:additional_SAPD label 2}
\min_{x\in\mathcal{X}} \max_{y\in\mathcal{Y}} \bar{\cL}(x,y) \triangleq f(x) + \Phi(x,y) + 
\sa{\gamma}\|x- \xzh{x_0}\|^2-g(y),
\end{align}
\sa{which is equivalent to the SCSC problem in \cref{eq:WCSC-subproblem-generic} with $\mu_x=\gamma$.}
Let $\{x_k,y_k\}$ be the iterate sequence generated by \texttt{SAPD} running on \eqref{eq:additional_SAPD label 2}, initialized from an arbitrary point $(x_0,y_0)$,
with parameters $\{\tau,\sigma,\theta\}$ chosen as follows:
{\small
\begin{equation}
\label{Condition: SP solution to noisy LMI label-dist}
\tau = \frac{1-\theta}{
\gamma},\quad \sigma = \frac{1-\theta}{\mu_y\theta},
\quad 
\theta  \sa{=} \max\{\bar{\theta}(\beta),~\hat{\theta}_1,~\hat{\theta}_2\},
\end{equation}}
for $\beta=\min\{\frac{1}{2},~\frac{\mu_y}{\gamma},~\frac{\gamma}{\mu_y},~\frac{L_{yx}}{L_{xy}}\}$, 
 where $\bar{\theta}(\beta)\triangleq\max\{\bar{\theta}_1(\beta),~\bar{\theta}_2(\beta)\}\in(0,1)$ such that
\sa{
\begin{equation*}
        \bar{\theta}_1(\beta)\triangleq 1 
         - \tfrac{\beta\mu_y\xz{L'_{xx}}}{2L^2_{yx}}\big(\sqrt{1 + \tfrac{4L^2_{yx}\gamma}{\beta \xz{{L'_{xx}}^2}\mu_y}} - 1 \big),
\end{equation*}
\begin{equation*}
        \bar{\theta}_2(\beta)\triangleq 
        \begin{cases}
        1 - \tfrac{(1-\beta)^2}{8}\tfrac{\mu_y^2}{L_{yy}^2} \Big( \sqrt{1+\tfrac{16L_{yy}^2}{(1-\beta)^2\mu_y^2}}-1\Big) & L_{yy}>0\\
        0 & L_{yy}=0,
        \end{cases}
\end{equation*}
}
with $L'_{xx}=L_{xx}+2\gamma$ and
{
\begin{equation}
\label{eq:theta_bound single call SAPD}
    \hat{\theta}_1  \triangleq \max\Big\{0,1 - \frac{1}{8}\cdot\gamma^2\cdot\frac{\hat{\epsilon}^2}{\delta_x^2}\Big\},
\quad 
\hat{\theta}_2 \triangleq \left(1+ \frac{1}{8}\cdot\frac{\mu_y\gamma}{11}\cdot\frac{\hat{\epsilon}^2}{\delta_y^2}\right)^{-1}.
\end{equation}}%
In fact, \sa{in \cref{lemma: dist bound with fg} we provide a convergence guarantee for solving the above problem in~\eqref{eq:additional_SAPD label 2} using Algorithm~\ref{Alg: SAPD}.} \sa{Since the parameter choice above satisfies the condition \eqref{eq: SCSC SAPD LMI-dist} in \cref{lemma: dist bound with fg}, we can invoke} \cref{eq:dist bound} to complete the rest of the analysis. To be more precise, the problem in~\cref{eq:additional_SAPD label 2} is a generic form of the SCSC subproblems given in~\cref{eq: SCSC problem chpt2} with $\mu_x = \gamma$; furthermore, by \cref{lemma: explicit solution to noisy LMI}, $(\tau,\sigma,\theta)$ chosen as in~\eqref{Condition: SP solution to noisy LMI label-dist} satisfies \eqref{eq: SCSC SAPD LMI-dist} with $\rho=\theta$, $\mu_x=\gamma$, $\alpha = \frac{1}{\sigma}-\sqrt{\theta}L_{yy}>0$, and $L'_{xx}=L_{xx}+2\gamma$.

Since $(\hat x, y_*(\hat x))$ is the saddle point of $\bar\cL$, then by \cref{lemma: dist bound with fg}, we get
\begin{small}
\begin{equation*}
          \begin{aligned}
        \mathbb{E}\Big[& \Big(\frac{1}{\tau}-\gamma\Big) \|x_{N}-\hat{x}\|^2  \Big]  \leq
\theta^{N}\left( \frac{1}{\tau}\|x_0-\hat{x}\|^2
      +\frac{1}{\sigma}\|y_{0}-y_*(\hat{x})\|^2\right)
      +\frac{\theta}{1-\theta}\Big(\tau\Xi_{\tau,\sigma,\theta}^x \delta_x^2  +  \sigma\Xi_{\tau,\sigma,\theta}^y\delta_y^2\Big).
    \end{aligned}
\end{equation*}
\end{small}
If we 
substitute the choice of $\{\tau,\sigma\}$ in \cref{Condition: SP solution to noisy LMI label-dist} into the above inequality, it follows that
\begin{small}
\begin{equation*}
          \begin{aligned}
        \mathbb{E}\Big[\|x_{N}-\hat{x}\|^2  \Big]  \leq
\theta^{N-1}\max\Big\{1,\frac{\mu_y}{\gamma}\Big\}\left( \|x_0-\hat{x}\|^2
      +\|y_{0}-y_*(\hat{x})\|^2\right)
      +\frac{1}{\gamma}\Big(\tau\Xi_{\tau,\sigma,\theta}^x \delta_x^2  +  \sigma\Xi_{\tau,\sigma,\theta}^y\delta_y^2\Big).
    \end{aligned}
\end{equation*}
\end{small}
Then, by Jensen's inequality, it follows that
\begin{equation*}
    \Big(\mathbb{E}\left[ \|x_{N}-\hat x\| \right]\Big)^2 \leq \mathbb{E}\left[ \|x_{N}-\hat x\|^2 \right] \leq \theta^{N-1}\max\Big\{1,\frac{\mu_y}{\gamma}\Big\}\cD_0^2 + \frac{1}{\gamma}\Big(\tau\Xi_{\tau,\sigma,\theta}^x \delta_x^2  +  \sigma\Xi_{\tau,\sigma,\theta}^y\delta_y^2\Big),
\end{equation*}
where $\cD_0 \triangleq \big(\|\hat{x}-x_{0}\|^2 + \|y_*(\hat{x})-y_{0}\|^2\big)^{1/2}$.
\nsa{Thus,} for any given $\hat{\epsilon}>0$, $\mathbb{E}\left[ \|x_{N}-\hat{x}\| \right]$ can be bounded by $\hat{\epsilon}$ when
\begin{subequations}
\label{eq: split bound of distance metric}
\begin{align}
    &\frac{1}{\gamma}\Big(\tau\Xi_{\tau,\sigma,\theta}^x \delta_x^2  +  \sigma\Xi_{\tau,\sigma,\theta}^y\delta_y^2\Big)\leq \frac{\hat{\epsilon}^2}{2},
    \label{eq: split bound of distance metric-xi}\\
    &\theta^{N-1}\max\Big\{1,\frac{\mu_y}{\gamma}\Big\}\cD_0^2 \leq\frac{\hat{\epsilon}^2}{2}.
    \label{eq: split bound of distance metric-N}
\end{align}
\end{subequations}
Recall that $\Xi^x_{\tau,\sigma,\theta}$, $\Xi^y_{\tau,\sigma,\theta}$ are defined in \cref{lemma: dist bound with fg}. 
Thus, the choice of $\tau$ and $\sigma$ in~\eqref{Condition: SP solution to noisy LMI label-dist} further implies that 
{
\begin{align*}
 &   \Xi^x_{\tau,\sigma,\theta} = \sa{1 + 
(1-\theta^2)\frac{L_{yx}}{2\mu_y}},\\
&     \Xi^y_{\tau,\sigma,\theta} = \Big( 1 + 3\theta + (1-\theta^2)\frac{L_{yy}}{\mu_y} + (1+\theta)(1-\theta)^2\frac{ L_{yx}L_{xy}}{\sa{\gamma}\mu_y}\Big)(1+2\theta) + \sa{\theta}(1-\theta^2)\frac{L_{yx}}{2\sa{\gamma}}.
\end{align*}
}%
Since $0<\theta < 1$ and $1-\theta^2\leq 2(1-\theta)$, we have
{
\begin{subequations}
\begin{align}
\label{eq:Xi_x_1}
& \Xi^x_{\tau,\sigma,\theta} \leq 1+ (1-\theta)\frac{L_{yx}}{\mu_y},
\\
    \label{eq:Xi_y_1}
  &  \Xi^y_{\tau,\sigma,\theta} \leq  3\Big(4 + 2(1-\theta)\frac{L_{yy}}{\mu_y} + 2(1-\theta)^2\frac{ L_{yx}L_{xy}}{\gamma\mu_y}\Big) + (1-\theta)\frac{L_{yx}}{\gamma}.
\end{align}
\end{subequations}}%
On the other hand, since $\theta\geq \bar{\theta}(\beta)=\max\{\bar{\theta}_1(\beta),\bar{\theta}_2(\beta)\}$, the inequality $\sqrt{a+b} \leq \sqrt{a} + \sqrt{b}$ for all $a,b\geq 0$, and the definition of $\bar{\theta}(\beta)$ 
  together imply that
  \begin{equation}
      \label{eq:1-theta-bound}
      1-\theta\leq \min\Big\{\frac{\sqrt{\beta\gamma\mu_y}}{L_{yx}}, \frac{1-\beta}{2}\frac{\mu_y}{L_{yy}}\Big\}.
  \end{equation}
Therefore, by \cref{eq:1-theta-bound}, we can derive that
$$
(1-\theta)\frac{L_{yx}}{\mu_y}\leq \sqrt{\frac{\beta \gamma}{\mu_y}},\;
(1-\theta)\frac{L_{yy}}{\mu_y}\leq \frac{1-\beta}{2},\;
(1-\theta)^2\frac{ L_{yx}L_{xy}}{\gamma\mu_y} \leq \frac{\beta L_{xy}}{L_{yx}},\;
(1-\theta)\frac{L_{yx}}{\gamma}\leq \sqrt{\frac{\beta\mu_y}{\gamma}};
$$
thus, using those inequalities
within \cref{eq:Xi_x_1} and \cref{eq:Xi_y_1}, we get
{
\begin{subequations}
\begin{align}
\label{eq: Xi_x bound}
    & \Xi^x_{\tau,\sigma,\theta} \leq 1+ \sqrt{\frac{\beta\gamma}{\mu_y}},
\\
\label{eq: Xi_y bound}
    &\Xi^y_{\tau,\sigma,\theta} \leq  
    15-3\beta +6\beta\frac{L_{xy}}{L_{yx}} +\sqrt{\frac{\beta\mu_y}{\gamma}}.
\end{align}
\end{subequations}
}
Note that $\beta=\min\{\frac{1}{2},~\frac{\mu_y}{\gamma},~\frac{\gamma}{\mu_y},~\frac{L_{yx}}{L_{xy}}\}\in(0,1)$ implies that
$$
\Xi^x_{\tau,\sigma,\theta}\leq 2,\qquad \Xi^y_{\tau,\sigma,\theta}\leq 22.
$$
Therefore, using the choice of $\{\tau,\sigma\}$ in \cref{Condition: SP solution to noisy LMI label-dist}, we obtain a sufficient condition \sa{for}
    \cref{eq: split bound of distance metric-xi} as given below:
    \begin{equation}
    \label{eq:sufficient_cond_theta}
    \frac{1-\theta}{\gamma}\frac{2}{\gamma}\delta_x^2  +  \frac{1-\theta}{\mu_y\theta} \frac{22}{\gamma}\delta_y^2\leq \frac{\hat{\epsilon}^2}{2}.
    \end{equation}
\sa{Our choice of $\theta\in(0,1)$ in~\eqref{Condition: SP solution to noisy LMI label-dist} implies that $\theta\geq \max\{\hat{\theta}_1,\hat{\theta}_2\}$, where $\hat{\theta}_1$ and $\hat{\theta}_2$ are defined in \cref{eq:theta_bound single call SAPD}. Note $\theta\geq \max\{\hat{\theta}_1,\hat{\theta}_2\}$ immediately implies that the above sufficient condition in~\eqref{eq:sufficient_cond_theta} holds.} Therefore, with the choice of $\{\tau,\sigma,\theta\}$ in \cref{Condition: SP solution to noisy LMI label-dist} we obtain that \cref{eq: split bound of distance metric-xi} holds, i.e.,
$$\frac{1}{\gamma}\Big(\tau\Xi_{\tau,\sigma,\theta}^x \delta_x^2  +  \sigma\Xi_{\tau,\sigma,\theta}^y\delta_y^2\Big)\leq \frac{\hat{\epsilon}^2}{2}.$$

Furthermore, \eqref{eq: split bound of distance metric-N} holds when $N\geq \ln\Big( \frac{2\max\{1,~\mu_y/\gamma\}\cD_0^2}{\hat{\epsilon}^2}\Big)/\ln\Big(\frac{1}{\theta}\Big) + 1$. Thus, we conclude that for any $\hat{\epsilon}>0$, \texttt{SAPD}, stated in Algorithm~\ref{Alg: SAPD}, can generate $x_N$ such that $\mathbb{E}\left[ \|x_{N}-\hat{x}\| \right]\leq\hat{\epsilon}$ within $N_{\hat\epsilon}$ iterations for $\theta  \sa{=} \max\{\bar{\theta}(\beta),~\hat{\theta}_1,~\hat{\theta}_2\}$, where 
\begin{equation}\label{eq: N bound with theta single SAPD}
    N_{\hat\epsilon} = \mathcal{O}\Big(\ln\Big( \frac{\max\{1,\mu_y/\gamma\}}{\hat{\epsilon}}\Big)/\ln\Big(\frac{1}{\theta}\Big) + 1 \Big).
\end{equation}
Note $\frac{1}{\ln(\frac{1}{\theta})}\leq(1-\theta)^{-1}$ for $\theta\in(0,1)$ implies that
$$
\frac{1}{\ln(\frac{1}{\theta})} \leq \mathcal{O}\Big(\max\{(1-\overline{\theta}_1(\beta))^{-1},(1-\overline{\theta}_2(\beta))^{-1},(1-\hat{\theta}_1)^{-1},(1-\hat{\theta}_2)^{-1}\}\Big).
$$
First, we equivalently rewrite $(1-\overline{\theta}_1(\beta))^{-1}$
and $(1-\overline{\theta}_2(\beta))^{-1}$ 
as follows:
{
\begin{equation*}
         (1-\overline{\theta}_1(\beta))^{-1}
         =  \frac{1}{2}\frac{\xz{L'_{xx}}}{\gamma} + \sqrt{\frac{1}{4}\frac{\xz{{L'_{xx}}^2}}{\gamma^2}+ \frac{L_{yx}^2}{\beta \gamma\mu_y}},\quad
         (1-\overline{\theta}_2(\beta))^{-1}
     = \frac{1}{2} + \sqrt{\frac{1}{4} +\frac{4L_{yy}^2}{\left(1-\beta\right)^2\mu_y^2}};
\end{equation*}}%
thus, 
$$(1-\overline{\theta}_1)^{-1}\leq \frac{\xz{L'_{xx}}}{\gamma}+\frac{L_{yx}}{\sqrt{\beta\gamma\mu_y}},\qquad (1-\overline{\theta}_2)^{-1}\leq 1+\frac{2}{1-\beta}\cdot\frac{L_{yy}}{\mu_y}.$$
Finally, 
$$(1-\hat{\theta}_1)^{-1}=\cO\Big(\frac{1}{\gamma^2}\cdot\frac{\delta_x^2}{\hat{\epsilon}^2}\Big),\qquad  (1-\hat{\theta}_2)^{-1}=\cO\Big(\frac{1}{\gamma\mu_y}\cdot\frac{\delta_y^2}{\hat{\epsilon}^2}+1\Big).$$ Recall that $\xz{L'_{xx}} = 2\gamma+L_{xx}$, using the above four identities that and our choice of $\beta=\min\{\frac{1}{2},~\frac{\mu_y}{\gamma},~\frac{\gamma}{\mu_y},~\frac{L_{yx}}{L_{xy}}\}$  we derive that
{
\begin{equation*}
\frac{1}{\ln(\frac{1}{\theta})}= \mathcal{O}\Big(  
\frac{\max\{L_{xx},L_{yx}\}}{\gamma} + \frac{\max\{L_{yx},L_{xy}\}}{\sqrt{\gamma\mu_y}} + \frac{\max\{L_{yy},L_{yx}\}}{\mu_y}
+ \Big(
\frac{\delta_x^2}{\gamma} +
\frac{\delta_y^2}{\mu_y} \Big)
\frac{1}{\gamma\hat{\epsilon}^2}\Big),
\end{equation*}}%
From \eqref{eq: N bound with theta single SAPD}, we conclude that
{\footnotesize
\begin{equation*}
N_{\hat\epsilon} = \mathcal{O}\left(  
\frac{\max\{L_{xx},L_{yx}\}}{\gamma} + \frac{\max\{L_{yx},L_{xy}\}}{\sqrt{\gamma\mu_y}} + \frac{\max\{L_{yy},L_{yx}\}}{\mu_y}
+ \Big(
\frac{\delta_x^2}{\gamma} +
\frac{\delta_y^2}{\mu_y} \Big)
\frac{1}{\gamma\hat{\epsilon}^2}\right)\cdot\ln\Big( \frac{\max\{1,\mu_y/\gamma\}}{\hat{\epsilon}}\Big),
\end{equation*}}%
which completes the proof.
\end{proof}

\begin{lemma}\label{lemma: proximal point lemma v2}
\sa{Suppose $f:\cX\to\reals\cup\{+\infty\}$ is closed convex, and $V$ is a strictly convex function on $\dom f$ and differentiable on an open set containing $\dom f$.} Let $x_* = \argmin_{x\in\cX} f(x) + V(x)$. Then, for any $\alpha>0$, it holds that $x_* =  \prox{\alpha f}(x_* - \alpha \nabla V(x_*))$.
\end{lemma}
\begin{proof}
\sa{From the first-order optimality condition}, we have
\begin{equation}\label{eq: optimality condition lemma v2}
   0 \in \partial f(x_*) + \nabla V(x_*). 
\end{equation}
\sa{Moreover, from the definition of $\prox{\alpha f}(\cdot)$ operator, it follows that}
\begin{align}
\label{eq:prox-condition}
\prox{\alpha f}(x_* - \alpha \nabla V(x_*)) = \argmin_{x\in\cX} f(x) + \nabla V(x_*)^\top(x-x_*) + \frac{1}{2\alpha}\|x-x_*\|^2.
\end{align}
\sa{Finally, \eqref{eq: optimality condition lemma v2} implies that $x_*$ is the unique minimizer of the problem on the rhs of \eqref{eq:prox-condition}. Therefore, we get that $x_* =\prox{\alpha f}(x_* - \alpha \nabla V(x_*))$, which completes the proof.}
\end{proof}
\subsection{Proof of Theorem~\ref{thm:metric-equivalance}}
\sa{We are now ready to prove \cref{thm:metric-equivalance}.}
\begin{proof}
Let $\hat{x}= \prox{\lambda\phi}(\sa{x_\epsilon})$, and $\phi^s$ be the smooth part of $\phi$, i.e., $\phi=f+\phi^s$.
Moreover, since $\Phi(x,\cdot)-g(\cdot)$ is strongly concave and $\Phi(\cdot,y)$ is differentiable, we have that $\phi^s$ is differentiable; hence, for any \sa{$x\in\dom f$},
$$
\nabla \phi^s(x) = \nabla_x \Phi(x,y_*(x)), \quad \text{where} \quad y_*(x) = \argmax_{y\in\mathcal{Y}} \Phi(x,y)-g(y).
$$
Then we can explicitly write $\hat{x}$ as
\begin{equation*}
    \hat{x} = \argmin_{x\in\cX} f(x) + \phi^s(x) + \frac{1}{2\lambda}\|x- x_\epsilon\|^2.
\end{equation*}
\sa{Since $\phi^s(\cdot) + \frac{1}{2\lambda}\|\cdot- x_\epsilon\|^2$ is smooth and strongly convex,} for any $\alpha>0$, \cref{lemma: proximal point lemma v2} implies that
$$
\hat{x} = \prox{\alpha f}\left(\hat{x} - \alpha\big(\nabla\phi^s(\hat{x})\xz{+}\frac{1}{\lambda}(\hat{x}-x_\epsilon) \big)\right). 
$$
If we let $\alpha=\lambda$, it follows that 
$$
\hat{x} = \prox{\lambda f}\big(x_\epsilon - \lambda\nabla_x\phi^s(\hat{x}) \big). 
$$
\sa{Moreover, since $f$ is closed convex, $\prox{f}(\cdot)$ is nonexpansive; hence,}
\begin{equation}
    \begin{aligned}
 \mathbb{E}\left[\|\hat{x} - \prox{\lambda f}\big(
\hat{x} - \lambda\nabla_x\phi^s(\hat{x})\big) \|\right] 
\leq  \mathbb{E}\left[\|x_\epsilon-\hat{x}\| \right]\leq \frac{\lambda \epsilon}{2},
    \end{aligned}
\end{equation}
where we  used  \cref{Lemma: graident of ME} for the last inequality, i.e.,
$
\|x_\epsilon - \hat{x}\|=\lambda\|\nabla \phi_{\lambda}(x_\epsilon)\|.
$
On the other hand, for any $\tilde{x}\in\sa{\dom f}$,
\begin{equation}\label{eq: gradient bound v2}
    \begin{aligned}
   &\mathbb{E}\left[ \|
\tilde{x} - 
\prox{\lambda f} 
\big(
\tilde{x} - \lambda\nabla_x\phi^s(\tilde{x})  
\big)  \|\right]
 \\
\leq &         \mathbb{E}\left[   \|
\tilde{x} - 
\prox{\lambda f} 
\big(
\tilde{x} - \lambda\nabla_x\phi^s(\tilde{x})  
\big)  -  \hat{x} + 
\prox{\lambda f} 
\big(
\hat{x} - \lambda\nabla_x\phi^s(\hat{x}) 
\big)\|\right] + \frac{\lambda\epsilon }{2}
\\
\leq & 2\mathbb{E}\left[\|\tilde{x}-\hat{x}\|\right] + \lambda
\mathbb{E}\left[\norm{\nabla_x \Phi(\tilde x,y_*(\tilde x))-\nabla_x \Phi(\hat x,y_*(\hat x))} \right]+ \frac{\lambda\epsilon}{2}.
\end{aligned}
\end{equation}
According to 
\cite[Proposition 1]{chen2021proximal}, $y_*(\cdot)$ is Lipschitz with constant \xz{$\kappa_{yx}=\frac{L_{yx}}{\mu_y}$}.
Therefore, we get
$$
\norm{\nabla_x \Phi(\tilde x,y_*(\tilde x))-\nabla_x \Phi(\hat x,y_*(\hat x))}
\leq L_{xx}\|\tilde x-\hat x\| + L_{xy}\|y_*(\tilde x)-y_*(\hat x)\|\leq \big(L_{xx}+L_{xy}\xz{\kappa_{yx}}\big)\|\tilde x-\hat x\|,
$$
which together with \cref{eq: gradient bound v2} implies that
\begin{equation}\label{eq:gradient bound 2 v2}
    \frac{1}{\lambda}\mathbb{E}\left[\norm{\tilde{x} - 
\prox{\lambda f} 
\big(
\tilde{x} - \lambda\nabla_x\phi^s(\tilde{x})  
\big) } \right]
    \leq  \big(\frac{2}{\lambda}+L_{xx}+L_{xy}\xz{\kappa_{yx}}\big)\mathbb{E}\left[\|\tilde{x}-\hat{x}\|\right] + \frac{\epsilon}{2}. 
\end{equation}
Let $\lambda^{-1}=2\gamma$, and $C\triangleq (4\gamma+ L_{xx}+L_{xy}\xz{\kappa_{yx}})^{-1}/2$. Thus, for any $\tilde x\in\sa{\dom f}$ such that {$\mathbb{E}\left[\norm{\tilde x-\hat x}\right]\leq C\epsilon$}, we have 
$$
 \mathbb{E}\left[\frac{1}{\lambda}\norm{\tilde{x} - 
\prox{\lambda f} 
\big(
\tilde{x} - \lambda\nabla_x\phi^s(\tilde{x})  
\big) }\right]\leq\epsilon.
$$
Indeed, when $f(x)=0$ for all $x\in\cX$, we get $\phi(x)=\phi^s(x)$ and the above inequality implies that
$$
\mathbb{E}\left[\|\nabla\phi(\tilde{x})\|\right]\leq\epsilon.
$$
The rest directly follows from invoking \cref{lemma: SAPD single call complexity} with $\hat\epsilon=C\epsilon$, and $\xzh{x_0}=x_\epsilon$. 
\end{proof}
\section{Proofs of Lemma~\ref{lemma:bound-gap} and Lemma~\ref{lemma: dist bound with fg}}
\label{sec:pf-lema2-lemma7}
\sa{We first discuss the proof of Lemma~\ref{lemma: dist bound with fg} and later establish Lemma~\ref{lemma:bound-gap} through specializing some parts of this proof. Indeed recall that Lemma~\ref{lemma: dist bound with fg} is stated for a generic \texttt{SAPD+} subproblem of the form \eqref{eq:WCSC-subproblem-generic}. In \cref{lemma: dist bound with fg-general} below, we restate Lemma~\ref{lemma: dist bound with fg} and rather than using a generic subproblem, we state it for the specific subproblems as in \eqref{eq: SCSC problem chpt2}, which arise while implementing \texttt{SAPD+}.} It is crucial to remind that the matrix inequality~(MI) we establish in~ \cref{lemma: dist bound with fg}, i.e., \cref{eq: SCSC SAPD LMI-dist}, helps us describe the admissible set of algorithm parameters that guarantee the \sa{\emph{linear} convergence of inner loop iterates \nsa{generated by} \texttt{SAPD}, i.e.,  
$\Big\{\mathbb{E}\big[\|x^t_k-x^t_*\|^2 + \|y_k^t-y^t_*\|^2\big]\Big\}_{k\geq 0}$, for any $t\geq 0$.}

\begin{lemma}\label{lemma: dist bound with fg-general}
Suppose Assumptions~\ref{ASPT: fg},~\ref{ASPT: lipshiz gradient},~\ref{ASPT: unbiased noise assumption} hold. \sa{For any given $\mu_x>0$ and $t\in\integers_+$, consider solving the SCSC subproblem in~\eqref{eq: SCSC problem chpt2} using \texttt{SAPD}, displayed in~Algorithm~\ref{Alg: SAPD}. Let $(x_*^t,y_*^t)$ denote the unique saddle point of \eqref{eq: SCSC problem chpt2}, and let $\{x^{t}_{k},y^{t}_{k}\}_{k\geq 0}$ be the iterate sequence when initialized from {$(x_0^t,y_0^t)\in\dom f\times\dom g$} and using} $\tau,\sigma,\theta$ that satisfy \eqref{eq: SCSC SAPD LMI-dist}
{for some} $\alpha \in [0,  {\tfrac{1}{\sigma}\sa{)}}$ and {$\rho\in (0,1\sa{)}$}, \sa{where $L'_{xx}\triangleq L_{xx}+\mu_x+\gamma$}.  Then for all $N\geq \integers_+$, \sa{it holds that}
\begin{equation}\label{eq:dist bound-general}
      \begin{aligned}
        \mathbb{E}\Big[& \Big(\frac{1}{\tau}-\mu_x\Big) \|x^{t}_{N}-x^t_*\|^2 +  {\Big(\frac{1}{\sigma} - \alpha \Big)}\|y^{t}_{N}-y^t_*\|^2 \Big]  \\
        & \leq
\rho^{N}\mathbb{E}\left[ \frac{1}{\tau}\|x^{t}_{0}-x^{t}_*\|^2
      +\frac{1}{\sigma}\|y^{t}_{0}-y^{t}_{*}\|^2\right]
      +\frac{\rho}{1-\rho}\Big(\tau\Xi_{\tau,\sigma,\theta}^x \delta_x^2  +  \sigma\Xi_{\tau,\sigma,\theta}^y\delta_y^2\Big),
    \end{aligned}
\end{equation}
where \sa{$\Xi^x_{\tau,\sigma,\theta}$ and $\Xi^y_{\tau,\sigma,\theta}$ are defined in~\eqref{eq:xi_x} and \eqref{eq:xi_y}, respectively.}
\end{lemma}
\begin{proof}
\sa{The proof is provided in \cref{pf:dist-lemma}.}
\end{proof}
\sa{Clearly, it is sufficient to prove \cref{lemma: dist bound with fg-general} in order to establish \cref{lemma: dist bound with fg}. Moreover, as stated earlier, we will exploit the techniques used in the proof of \cref{lemma: dist bound with fg-general} when showing \cref{lemma:bound-gap} --this is why we restated \cref{lemma: dist bound with fg}.} 
\subsection{Preliminary technical results}
\sa{Recall that given some {$x^t_0\in\dom f$} and $\mu_x>0$, we define
\begin{subequations}
\begin{align}
    \cL^{t}(x,y)&\triangleq f(x)+\Phi^t(x,y)-g(y),\\
    \Phi^t(x,y) &\triangleq\Phi(x,y) + \frac{\mu_x+\gamma}{2}\|x-\xz{x^t_0}\|^2,\label{eq:Phi_t}
\end{align}
\end{subequations}
where $\gamma>0$ is the weak-convexity constant of $\Phi(\cdot,y)$ for any $y\in\dom \xz{g}$. It follows from \cref{ASPT: lipshiz gradient} that $\grad_y \Phi^t$ and $\grad_x \Phi^t$ are Lipschitz such that
\begin{align}
    \norm{\grad_y \Phi^t(x,y)-\grad_y \Phi^t(x',y')}\leq L_{yx}\norm{x-x'}+L_{yy}\norm{y-y'},\\
    \norm{\grad_x \Phi^t(x,y)-\grad_x \Phi^t(x',y')}\leq L'_{xx}\norm{x-x'}+L_{xy}\norm{y-y'}
\end{align}
such that $L'_{xx}\triangleq L_{xx}+\mu_x+\gamma$. Furthermore, \eqref{eq:Phi_t} implies that for any $y\in\dom \xz{g}$, $\Phi^t(\cdot,y)$ is strongly convex with modulus $\mu_x>0$.}%

We will derive some key inequalities below for \xz{\texttt{SAPD}} iterates \sa{$\{x_k^t,y_k^t\}_{k\geq 0}$} generated
by Algorithm~\ref{Alg: SAPD} \sa{to solve $\min_x\max_y\cL^{t}(x,y)$}. Let $x_{-1}^t = x_0^t$, $y_{-1}^t = y_0^t$, and for $k\geq 0$, define
{\small
\begin{align}
\label{eq:qksk}
{q}_k^t \triangleq \nabla_y \sa{\Phi^t}(x_k^t,y_k^t) - \nabla_y\sa{\Phi^t}(x_{k-1}^t,y_{k-1}^t),\qquad s^{t}_{k} \triangleq \nabla_y \sa{\Phi^t}(x_k^t, y_k^t) + \theta q_k^t.
\end{align}}%
\sa{Thus $q_0^t=\mathbf{0}$; and for $k\geq 0$, \cref{ASPT: lipshiz gradient} implies that
{\small
\begin{equation}
           \norm{q_{k+1}^t} \leq  L_{yx} \|x_{k+1}^t - x_{k}^t \|+  L_{yy} \|y_{k+1}^t - y_{k}^t \|. \label{INEQ: Cauchy Ineqaulity 1}
\end{equation}}}%

\begin{lemma}\label{lem: basic lemma for SCSC}
\xz{Suppose Assumptions~\ref{ASPT: fg},~\ref{ASPT: lipshiz gradient},~\ref{ASPT: unbiased noise assumption} hold.} Let $\{x_k^t,y_k^t\}_{k\geq 0}$ be \xz{\texttt{SAPD}} iterates generated according to Algorithm~\ref{Alg: SAPD} \sa{for solving $\min_x\max_y\cL^{t}(x,y)$}. 
Then 
for all $x\in\dom f\subset \cX$, $y\in\dom g\subset \cY$, and $k\geq 0$,
{\small
\begin{equation}\label{D1}
    \begin{aligned}
       \cL^{t}( & x^{t}_{k+1}, y)  - \cL^{t}(x, y^{t}_{k+1})
        \\
    \leq
    &-\langle q^{t}_{k+1}, y^{t}_{k+1} - y \rangle + \theta \langle q^{t}_{k}, y^{t}_{k} - y \rangle  
    + \Lambda^{t}_k(x,y) - \Sigma^{t}_{k+1}(x,y)+ \Gamma^{t}_{k+1}+\sa{\varepsilon^{t,x}_{k}(x)+\varepsilon^{t,y}_{k}(y)},
\end{aligned}
\end{equation}}%
where 
\begin{align*}
    &\sa{\varepsilon^{t,x}_k(x)}\triangleq 
       \langle  \tilde{\nabla}_x \Phi^{t}(x^{t}_{k}, y^{t}_{k+1};\omega_k^x) - \nabla_x \Phi^{t}(x^{t}_{k}, y^{t}_{k+1}),~x-x^{t}_{k+1} \rangle
       ,\quad
       \sa{\varepsilon^{t,y}_{k}(y)}\triangleq\langle \tilde{s}^{t}_{k} -s^{t}_{k}, y^{t}_{k+1} - y \rangle,
\end{align*}
$q^{t}_{k}$ and $s^{t}_{k}$ are defined as in \eqref{eq:qksk}, and
{\small
\begin{align*}
\Lambda^{t}_{k}(x,y) &\triangleq  (\frac{1}{2\tau}-\frac{\mu_x}{2}) \|x - x^{t}_{k}\|^2 + \frac{1}{2\sigma}\|y-y^{t}_{k}\|^2,\\
\Sigma^{t}_{k+1}(x,y) &\triangleq  \frac{1}{2\tau} \|x - x^{t}_{k+1}\|^2 + (\frac{1}{2\sigma} + \frac{\mu_y}{2})\|y-y^{t}_{k+1}\|^2,\\ \Gamma^{t}_{k+1} &\triangleq
     (\frac{ L'_{xx} }{2} - \frac{1}{2\tau}) \| x^{t}_{k+1} - x^{t}_{k} \|^2 - \frac{1}{2\sigma}\|y^{t}_{k+1}- y^{t}_{k}\|^2 \\ &\quad + \theta L_{yx} \|x^{t}_{k} - x^{t}_{k-1} \|\| y^{t}_{k+1} - y^{t}_{k} \| +   \theta   L_{yy} \|y^{t}_{k} - y^{t}_{k-1} \| \| y^{t}_{k+1} - y^{t}_{k} \|.
\end{align*}}%
\end{lemma}
\begin{proof}
Fix $x\in\dom f$ and $y\in\dom g$.
Using Lemma 7.1 from \cite{hamedani2018primal} for the $y-$ and $x-$subproblems in Algorithm~\ref{Alg: SAPD}, we get 
\begin{small}
\begin{align*}
    f(x^{t}_{k+1}) + \langle \tilde{\nabla}_x \Phi^{t}(x^{t}_{k}, y^{t}_{k+1};\omega^x_k) , x^{t}_{k+1} - x \rangle 
       & \leq  
       f(x) + \frac{1}{2\tau} \left[  \|x- x^{t}_{k}\|^2 -\|x- x^{t}_{k+1}\|^2 - \|x^{t}_{k+1}- x^{t}_{k}\|^2
       \right],\\
       g(y^{t}_{k+1}) - \langle \tilde{s}^{t}_k,
       y^{t}_{k+1} - y \rangle 
       & \leq 
       g(y) + \frac{1}{2\sigma} \left[  \|y-y^{t}_{k}\|^2- \|y -  y^{t}_{k+1}\|^2 - 
      \|y^{t}_{k+1}- y^{t}_{k}\|^2 
       \right].
\end{align*}
\end{small}
Thus, by adding and subtracting we further get
\begin{subequations}
\begin{equation} \label{IX}
    \begin{aligned}
 f(x^{t}_{k+1}) &+ \langle {\nabla}_x \Phi^{t}(x^{t}_{k}, y^{t}_{k+1}
 ),~x^{t}_{k+1} - x \rangle\\
       \leq &  f(x) + \tfrac{1}{2\tau} (  \|x- x^{t}_{k}\|^2 - \|x- x^{t}_{k+1}\|^2 - \|x^{t}_{k+1}- x^{t}_{k}\|^2
       )+\varepsilon^{t,x}_k(x),
    \end{aligned}
\end{equation}
\begin{equation}\label{IY1}
    \begin{aligned}
    g(y^{t}_{k+1})  &- \langle s^{t}_{k}, 
       y^{t}_{k+1} - y \rangle 
       \\
       \leq & g(y) + \tfrac{1}{2\sigma} (  \|y- y^{t}_{k}\|^2 - \|y- y^{t}_{k+1}\|^2 - 
       \|y^{t}_{k+1}- y^{t}_{k}\|^2
       )+\varepsilon^{t,y}_k(y). 
    \end{aligned}
\end{equation}
\end{subequations} 
Rearranging the terms in \eqref{IY1}, we get
\begin{equation}\label{IY2}
\begin{aligned}
    -g(y)&+ g(y^{t}_{k+1}) 
       \\
       \leq &\langle s^{t}_{k},  y^{t}_{k+1} - y \rangle + \tfrac{1}{2\sigma} \left[  \|y-y^{t}_{k}\|^2- \|y -  y^{t}_{k+1}\|^2 - 
      \|y^{t}_{k+1}- y^{t}_{k}\|^2
       \right] +\varepsilon^{t,y}_k(y).
\end{aligned}
\end{equation}
Since $y^{t}_{k+1}\in \dom g$, the inner product in (\ref{IX}) can be lower bounded using convexity of $\Phi^{t}(\cdot, y^{t}_{k+1})$ as follows (see \cref{ASPT: lipshiz gradient}):
\begin{eqnarray*}
    \lefteqn{\langle \nabla_x  \Phi^{t}(x^{t}_{k}, y^{t}_{k+1}), x^{t}_{k+1} - x \rangle
     =  
    \langle \nabla_x \Phi^{t}(x^{t}_{k}, y^{t}_{k+1}), x^{t}_{k} - x \rangle
    +
    \langle \nabla_x \Phi^{t}(x^{t}_{k}, y^{t}_{k+1}), x^{t}_{k+1} - x^{t}_{k} \rangle}
    \\
    &&\geq 
    \Phi^{t}(x^{t}_{k}, y^{t}_{k+1}) - \Phi^{t}(x, y^{t}_{k+1})  +\frac{\mu_x}{2}\|x-x^{t}_{k}\|^2
    + \langle \nabla_x \Phi^{t}(x^{t}_{k}, y^{t}_{k+1}),~x^{t}_{k+1} - x^{t}_{k} \rangle.
\end{eqnarray*}
Using this inequality after adding $\Phi^{t}(x^{t}_{k+1}, y^{t}_{k+1})$ to both sides of \eqref{IX}, we get
{
\begin{equation}
\begin{aligned}
   \Phi^{t}( & x^{t}_{k+1}, y^{t}_{k+1}) + f(x^{t}_{k+1}) \\ 
    \leq &  \Phi^{t}(x, y^{t}_{k+1}) + f(x) 
     + \Phi^{t}(x^{t}_{k+1}, y^{t}_{k+1}) - \Phi^{t}(x^{t}_{k}, y^{t}_{k+1}) - \langle \nabla_x \Phi^{t}(x^{t}_{k}, y^{t}_{k+1}), x^{t}_{k+1} - x^{t}_{k} \rangle 
   \\
   & 
   +  \tfrac{1}{2\tau} \left[ \| x - x^{t}_{k} \| ^ 2 - \| x - x^{t}_{k+1}  \| ^ 2 - 
       \|x^{t}_{k+1}- x^{t}_{k}\|^2\right] - \tfrac{\mu_x}{2}\|x - x^{t}_{k}\|^2+\varepsilon^{t,x}_k(x)
   \\
   \leq & \Phi^{t}(x, y^{t}_{k+1}) +  f(x) + \frac{L'_{xx}}{2} \| x^{t}_{k+1} - x^{t}_{k} \|^2
   \\
   &    +  \tfrac{1}{2\tau} \left[ \| x - x^{t}_{k} \| ^ 2 - \|x-x^{t}_{k+1}\| ^ 2 - 
       \|x^{t}_{k+1}- x^{t}_{k}\|^2\right]- \tfrac{\mu_x}{2}\|x - x^{t}_{k}\|^2+\sa{\varepsilon^{t,x}_k(x)},
\end{aligned}
\end{equation}}%
where the last step uses 
\cref{ASPT: lipshiz gradient}. Rearranging the terms gives us
{
\begin{equation}\label{IX2}
\begin{aligned}
     f(&x^{t}_{k+1}) - f(x) - \Phi^{t}(x,y^{t}_{k+1}) 
    \leq -\Phi^{t}(x^{t}_{k+1},y^{t}_{k+1}) + \frac{L'_{xx}}{2} \| x^{t}_{k+1} - x^{t}_{k} \|^2
    \\
    &    +  \frac{1}{\sa{2}\tau} \left[ \| x - x^{t}_{k} \| ^ 2 - \| x - x^{t}_{k+1}  \| ^ 2 - 
       \|x^{t}_{k+1}- x^{t}_{k}\|^2\right]- \frac{\mu_x}{2}\|x - x^{t}_{k}\|^2+\sa{\varepsilon^{t,x}_k(x)}.
\end{aligned}
\end{equation}}
Then, for $k \geq 0$, by summing \eqref{IY2} and \eqref{IX2}, we obtain
{
\begin{equation}
\label{eq:one-step-aux1}
\begin{aligned}
   \mathcal{L}( & x^{t}_{k+1}, y)  - \mathcal{L}(x, y^{t}_{k+1}) 
   = 
   f(x^{t}_{k+1}) + \Phi^{t}(x^{t}_{k+1},y) - g(y) - f(x) - \Phi^{t}(x,y^{t}_{k+1}) + g(y^{t}_{k+1})
   \\
   \leq & \Phi^{t}(x^{t}_{k+1}, y) - \Phi^{t}(x^{t}_{k+1}, y_{k+1}) + \langle s^{t}_{k}, y^{t}_{k+1} - y \rangle + \frac{L'_{xx}}{2} \| x^{t}_{k+1} - x^{t}_{k} \|^2
   \\
   &  +  \frac{1}{2\sigma} \left[  \|y-y^{t}_{k}\|^2- \|y -  y^{t}_{k+1}\|^2 - 
      \|y^{t}_{k+1}- y^{t}_{k}\|^2 \right]+ \sa{\varepsilon_k^{t,y}(y)}
   \\
   & +  \frac{1}{2\tau} \left[ \| x - x^{t}_{k} \| ^ 2 - \| x - x^{t}_{k+1}\| ^ 2 - 
       \|x^{t}_{k+1}- x^{t}_{k}\|^2\right]
       - \frac{\mu_x}{2}\|x - x^{t}_{k}\|^2 + \sa{\varepsilon_k^{t,x}(x)}. 
\end{aligned}
\end{equation}}
From \cref{ASPT: lipshiz gradient},
the $\mu_y$-strongly concavity of $\Phi^t(x, \cdot)$ for fixed $x \in \dom f\subset \mathcal{X}$ implies
{
\begin{equation*}
    \begin{aligned}
         \Phi^{t}( & x^{t}_{k+1}, y) -  \Phi^{t}(x^{t}_{k+1}, y^{t}_{k+1}) + \langle s^{t}_{k}, y^{t}_{k+1} - y \rangle
         \\
         \leq & \langle \nabla_y\Phi^{t}(x^{t}_{k+1},y^{t}_{k+1}), y - y^{t}_{k+1} \rangle - \frac{\mu_y}{2}\|y - y^{t}_{k+1}\|^2 + \langle \nabla_y \Phi^{t}(x^{t}_{k}, y^{t}_{k}) + \theta q^{t}_{k}, y^{t}_{k+1} - y \rangle
         \\
         = & -\langle q^{t}_{k+1}, y^{t}_{k+1} - y \rangle - \frac{\mu_y}{2}\|y - y^{t}_{k+1}\|^2 + \theta \langle q^{t}_{k}, y^{t}_{k} - y \rangle +\theta \langle q^{t}_{k}, y^{t}_{k+1}-y^{t}_{k}\rangle.
    \end{aligned}
\end{equation*}}%
Thus, using the above inequality within~\eqref{eq:one-step-aux1}, we get
{
\begin{equation*}
\begin{aligned}
    \cL^{t}( & x^{t}_{k+1}, y)  - \cL^{t}(x, y^{t}_{k+1}) \leq 
    -\langle q^{t}_{k+1}, y^{t}_{k+1} - y \rangle + \theta \langle q^{t}_{k}, y^{t}_{k} - y\rangle+\theta\langle q^t_{k},y^{t}_{k+1}-y^{t}_{k} \rangle \\
    &
    + \frac{L'_{xx}}{2} \| x^{t}_{k+1} - x^{t}_{k} \|^2
    +  \frac{1}{2\sigma} \left[  \|y-y^{t}_{k}\|^2- \|y -  y^{t}_{k+1}\|^2 -
      \|y^{t}_{k+1}- y^{t}_{k}\|^2
       \right]- \frac{\mu_y}{2}\|y - y^{t}_{k+1}\|^2 
   \\
   & +  \frac{1}{2\tau} \left[ \| x - x^{t}_{k} \| ^ 2 - \| x - x^{t}_{k+1}  \| ^ 2 - 
       \|x^{t}_{k+1}- x^{t}_{k}\|^2\right]
       - \frac{\mu_x}{2}\|x - x^{t}_{k}\|^2+\varepsilon^{t,x}_k(x)+\sa{\varepsilon^{t,y}_k(y).}
\end{aligned}
\end{equation*}}%
Finally, \eqref{D1} follows from using Cauchy-Schwarz 
for \sa{$\fprod{q^t_k,y^t_{k+1}-y^t_k}$} and \eqref{INEQ: Cauchy Ineqaulity 1}.
\end{proof}
\begin{lemma}\label{lemma: proximal ineq}
\cite[Theorem 6.42]{beck2017first} Let $f$ be proper, closed and convex function. Then for any {$x,x'\in \cX$},
we get
{\small
$
\|\prox{f}(x) - \prox{f}(x') \| \leq \|  x- x' \|.
$}
\end{lemma}
Next, based on the above inequality, we prove an intermediate result, \sa{which we use later} to bound the variance of the \xz{\texttt{SAPD}} iterate sequence.

\begin{lemma}\label{lemma: intermediate noisy bound}
\xz{Suppose Assumptions~\ref{ASPT: fg},~\ref{ASPT: lipshiz gradient},~\ref{ASPT: unbiased noise assumption} hold.} Let $\{x_k^t,y_k^t\}_{k\geq 0}$ be \xz{\texttt{SAPD}}  iterates generated 
as in Algorithm~\ref{Alg: SAPD} \sa{for solving $\min_x\max_y\cL^{t}(x,y)$}.  
\sa{For $k\geq 0$, let $q^{t}_{k}$ and $s^{t}_{k}$ be defined as in \eqref{eq:qksk}, and 
let} 
{\small
\begin{equation*}
    \begin{aligned}
       &{\hat{x}^{t}_{k+1}} \triangleq   \prox{\tau f}\left({x^{t}_{k}- \tau{\nabla}_x \Phi^{t}(x^{t}_{k}, y^{t}_{k+1}})\right), \quad
       {\hat{\hat{x}}^{t}_{k+1}  \triangleq   \prox{\tau f}\left({x^{t}_{k}- \tau{\nabla}_x \Phi^{t}(x^{t}_{k}, \hat{y}^{t}_{k+1}})\right)},\\
       &\hat{y}^{t}_{k+1} \triangleq  \prox{\sigma g}\left(y^{t}_{k} + \sigma s^{t}_{k}\right), \quad \hat{\hat{y}}^{t}_{k+1}\triangleq \prox{\sigma \sa{g}}\left( \hat{y}^{t}_{k} + \sigma(1+\theta){\nabla}_y \Phi^{t}(\hat{\hat{x}}^{t}_{k}, \hat{y}^{t}_{k}) - \sigma\theta{\nabla}_y \Phi^{t}(x^{t}_{k-1}, y^{t}_{k-1})  \right),
    \end{aligned}
\end{equation*}}%
then the following inequalities hold for $k\geq 0$:
{
\begin{subequations}
\begin{align}
    \|x^{t}_{k+1} - \hat{x}^{t}_{k+1} \| &\leq \tau\| \Delta^{t,x}_{k}\|,
    \qquad \|y^{t}_{k+1} - \hat{y}^{t}_{k+1}\| \leq \sigma\left((1+\theta)\|\Delta^{t,y}_{k}\| + \theta\|\Delta^{t,y}_{k-1}\|\right),\label{eq:y_k - y_hat_k}\\
    \|{y}^{t}_{k+1} - \hat{\hat{y}}^{t}_{k+1}\| &\leq  \sigma\left( 
    (1+\theta)\|\Delta^{t,y}_{k}\| + \theta\|\Delta^{t,y}_{k-1}\|
    +
   \tau(1+\theta)L_{yx}\|\Delta^{t,x}_{k-1}\|\right)
    \label{eq:y_k - y_hat_hat_k}\\
    & 
    \quad + \sigma\left(1 + \sigma(1+\theta)L_{yy} +  \tau\sigma(1+\theta)L_{yx}L_{xy} \right)\left((1+\theta)\|\Delta^{t,y}_{k-1}\| + \theta\|\Delta^{t,y}_{k-2}\|\right), \nonumber
    \end{align}
\end{subequations}}%
where
{\small
    $\Delta^{t,x}_{k} \sa{\triangleq} \tilde{\nabla}_x \Phi^{t}(x^{t}_{k}, y^{t}_{k+1};\omega_k^x) - \nabla_x \Phi^{t}(x^{t}_{k}, y^{t}_{k+1})$, and  $\Delta^{t,y}_{k} \sa{\triangleq} \sa{\tilde\nabla_y}\Phi^{t}(x^{t}_{k}, y^{t}_{k};\omega_k^y) - \nabla_y \Phi^{t}(x^{t}_{k}, y^{t}_{k})$.
}
\end{lemma}
\begin{proof}
The first inequality in~\cref{eq:y_k - y_hat_k} is 
from \cref{lemma: proximal ineq}; for the second, we have
{
\begin{equation*}
\begin{aligned}
    \|y^{t}_{k+1} - \hat{y}^{t}_{k+1}\| \leq \sigma\|\tilde{s}^{t}_{k} - s^{t}_{k} \| \leq \sigma\left((1+\theta)\|\Delta^{t,y}_{k}\| + \theta\|\Delta^{t,y}_{k-1}\|\right),
\end{aligned}
\end{equation*}}%
\sa{which follows from \cref{lemma: proximal ineq} and the 
triangle inequality. 
To show \cref{eq:y_k - y_hat_hat_k}, 
we bound $\| y^{t}_{k+1} - \hat{y}^{t}_{k+1}\|$ and $\| \hat{y}^{t}_{k+1}- \hat{\hat{y}}^{t}_{k+1}\|$ separately.} It follows from \cref{lemma: proximal ineq} that
    $\|x^{t}_{k+1} - \hat{\hat{x}}^{t}_{k+1} \| \leq \tau\|\tilde{\nabla}_x\Phi^{t}(x^{t}_{k},y^{t}_{k+1};\omega_k^x) - \nabla_x\Phi^{t}(x^{t}_{k},\hat{y}^t_{k+1})\|$.
\sa{After adding and subtracting $\nabla_x\Phi^{t}(x^{t}_{k},y^{t}_{k+1})$, \cref{ASPT: lipshiz gradient} implies that} 
{
\begin{equation}\label{ineq: x_k - x_hat_hat_k}
    \begin{aligned}
    \|x^{t}_{k+1} - \hat{\hat{x}}^{t}_{k+1} \| \leq \tau\left(\|\Delta^{t,x}_{k}\| + L_{xy}\|y^{t}_{k+1} - \hat{y}^{t}_{k+1}\|\right).
    \end{aligned}
\end{equation}}
\sa{We will use this relation to bound $\|\hat{y}^{t}_{k+1} - \hat{\hat{y}}^{t}_{k+1}\|$. 
Indeed,} using \cref{lemma: proximal ineq}, we have 
{
\begin{align*}
    \|\hat{y}^{t}_{k+1} - \hat{\hat{y}}^{t}_{k+1}\| \leq & \|y^{t}_{k} - \hat{y}^{t}_{k} + \sigma(1+\theta) \left( \nabla_y\Phi^{t}(x^{t}_{k},y^{t}_{k}) - \nabla_y\Phi^{t}(\hat{\hat{x}}^{t}_{k},\hat{y}^{t}_{k}) \right) \|\\
    \leq &(1 + \sigma(1+\theta)L_{yy}) \|y^{t}_{k} - \hat{y}^{t}_{k}\| + \sigma(1+\theta)L_{yx}\|x^{t}_{k} - \hat{\hat{x}}^{t}_{k} \|\\
    \leq & \Big(1 + \sigma(1+\theta)L_{yy} + \tau\sigma(1+\theta)L_{yx}L_{xy}\Big) \|y^{t}_{k} - \hat{y}^{t}_{k}\| 
    + \tau\sigma(1+\theta)L_{yx}\| \Delta^{t,x}_{k-1}\|\\
    \leq & \sigma\left(1 + \sigma(1+\theta)L_{yy} +  \tau\sigma(1+\theta)L_{yx}L_{xy} \right) \cdot \left((1+\theta)\|\Delta^{t,y}_{k-1}\| + \theta\|\Delta^{t,y}_{k-2}\|\right)\\
    & 
    + \sigma\tau(1+\theta)L_{yx}\|\Delta^{t,x}_{k-1}\|, 
\end{align*}}%
\sa{where the second, third and fourth inequalities follow from \cref{ASPT: lipshiz gradient}, \cref{ineq: x_k - x_hat_hat_k} and the second inequality in~\cref{eq:y_k - y_hat_k}, respectively.  Combining this 
with
    $\| y^{t}_{k+1} - \hat{\hat{y}}^{t}_{k+1}\| 
    \leq \| y^{t}_{k+1} - \hat{y}^{t}_{k+1}\| + \| \hat{y}^{t}_{k+1}- \hat{\hat{y}}^{t}_{k+1}\|$, 
and the second inequality in~\cref{eq:y_k - y_hat_k} 
give us the desired bound.} 
\end{proof}

\sa{Next, we provide some inequalities to bound the \xz{\texttt{SAPD}} variance term later} in our analysis.

\begin{lemma}\label{Lemma:noisy bound with fg}
\xz{Suppose Assumptions~\ref{ASPT: fg},~\ref{ASPT: lipshiz gradient},~\ref{ASPT: unbiased noise assumption} hold.} Let $\{x_k^t,y_k^t\}_{k\geq 0}$ be \xz{\texttt{SAPD}} iterates generated according to Algorithm~\ref{Alg: SAPD} for solving $\min_x\max_y\cL^{t}(x,y)$.  
The following inequality holds for all $k\geq 0$:
{\begin{small}
\begin{eqnarray*}
\lefteqn{\mathbb{E}\left[\langle \Delta^{t,x}_{k},  \hat{x}^{t}_{k+1}-x^{t}_{k+1} \rangle\right] \leq \tau\delta_x^2,\qquad
     \mathbb{E}\left[\langle \Delta^{t,y}_{k},  y^{t}_{k+1} - \hat{y}^{t}_{k+1}\rangle\right] \leq \sigma(1+2\theta)\delta_y^2,}\\
    \lefteqn{\mathbb{E}\left[ \langle \Delta^{t,y}_{k-1},  \hat{\hat{y}}^{t}_{k+1}-y^{t}_{k+1} \rangle \right]}\\
    &&\leq\sigma \left[  \left( \left(2 + \sigma(1+\theta)L_{yy}  +  \tau\sigma(1+\theta)L_{yx}L_{xy} \right)\cdot 
    (1+2\theta)
    + \frac{\tau(1+\theta)L_{yx}}{
     \sa{2}}\right)\delta_y^2 + \frac{\tau(1+\theta)L_{yx}}{
     \sa{2}}\delta_x^2\right],
\end{eqnarray*}
\end{small}
}%
\sa{where $\Delta_k^{t,x}$ and $\Delta_k^{t,y}$ are defined in~\cref{lemma: intermediate noisy bound}.}
\end{lemma}
\begin{proof}
With the convention that $y^{t}_{-2} = y^{t}_{-1} = y^{t}_0$, and $x^{t}_{-2} = x^{t}_{-1} = x^{t}_0$, \cref{lemma: intermediate noisy bound} and Cauchy-Schwarz inequality
imply for all $k\geq 0$ that 
{
\begin{equation*}
    \begin{aligned}
    & \langle \Delta^{t,x}_{k}, x^{t}_{k+1} - \hat{x}^{t}_{k+1} \rangle \leq \tau\|\Delta^{t,x}_{k} \|^2, \\
    & \langle \Delta^{t,y}_{k},  y^{t}_{k+1} - \hat{y}^{t}_{k+1}\rangle \leq \sigma\left((1+\theta)\|\Delta^{t,y}_{k}\|^2 + \theta\|\Delta^{t,y}_{k-1}\|\|\Delta^{t,y}_{k}\|\right),\\
    &  \langle \Delta^{t,y}_{k-1},  y^{t}_{k+1} - \hat{\hat{y}}^{t}_{k+1}\rangle
    \leq
    \sigma\Bigg( 
    (1+\theta)\|\Delta^{t,y}_{k}\|\|\Delta^{t,y}_{k-1}\|
    +
    \theta\|\Delta^{t,y}_{k-1}\|^2+
    \tau(1+\theta)L_{yx}\|\Delta^{t,x}_{k-1}\|
    \|\Delta^{t,y}_{k-1}\|\\
    & 
    + 
    \Big(1 + \sigma(1+\theta)L_{yy} +  \tau\sigma(1+\theta)L_{yx}L_{xy} \Big)\cdot
    \Big((1+\theta)\|\Delta^{t,y}_{k-1}\|^2
    +
    \theta\|\Delta^{t,y}_{k-2}\|\|\Delta^{t,y}_{k-1}\|\Big)
    \Bigg).
    \end{aligned}
\end{equation*}}%
Next, using \cref{ASPT: unbiased noise assumption} and 
$\| a\|\|b\|\leq  \frac{1}{2}\| \xzh{a}\|^2
 + \frac{1}{2}\| b\|^2$,
which holds for $a,b \in \mathbb{R}^n$, and taking the expectation leads to the desired result.
\end{proof}
Before we move on to prove our intermediate result \sa{in~\cref{lemma: dist bound with fg-general},} 
we give two technical lemmas that help us simplify the \xz{\texttt{SAPD}} parameter selection rule \sa{and lead to} the \sa{matrix inequality in \cref{eq: SCSC SAPD LMI-dist}.} 
\begin{lemma}
\label{lem:equivalent_systems}
Given $\tau,\sigma>0$, $\theta,\alpha\geq 0$, and $\rho\in(0,1)$, let
{\footnotesize
\begin{equation}
\label{eq:Gp-def}
 \sa{G'} \triangleq  
 \begin{pmatrix}
    \frac{1}{\tau}(1-\frac{1}{\rho})+\frac{\mu_x}{\rho} & 0 & 0 & 0 & 0\\ 
  0 & \frac{1}{\sigma}(1-\frac{1}{\rho})+\mu_y & \sa{-|1-\frac{\theta}{\rho}|}~L_{yx} & \sa{-|1-\frac{\theta}{\rho}|}~L_{yy} & 0\\ 
  0 & \sa{-|1-\frac{\theta}{\rho}|}~L_{yx} & \tfrac{1}{\tau} - \sa{L'_{xx}} & 0 & -  \frac{\theta}{\rho}L_{yx}\\ 
  0& \sa{-|1-\frac{\theta}{\rho}|}~L_{yy} & 0 & \frac{1}{\sigma} - \alpha & -  \frac{\theta}{\rho}L_{yy}\\
  0 & 0 & - \frac{\theta}{\rho}L_{yx} & -  \frac{\theta}{\rho}L_{yy} & \frac{\alpha}{\rho}
\end{pmatrix},
\end{equation}}%
then \sa{$G\succeq 0$ if and only if $G'\succeq 0$, where $G$ is defined in \cref{eq: SCSC SAPD LMI-dist}}.
\end{lemma}
\begin{proof}
$\forall ~ \mathbf{y} = (y_1,y_2,y_3,y_4,y_5)^\top\in \mathbb{R}^5$, letting $\tilde{\mathbf{y}}= (y_1,-y_2,y_3,y_4,y_5)^\top$, we have
{\footnotesize
\begin{equation*}
    \mathbf{y}^\top G' \mathbf{y} = 
    \begin{cases}
    \mathbf{y}^\top \sa{G} \mathbf{y} & \text{if} ~\sa{\theta\leq \rho,}\\
    \tilde{\mathbf{y}}^\top G \tilde{\mathbf{y}} & \text{else;}
    \end{cases}\quad
    \mathbf{y}^\top G \mathbf{y} = 
    \begin{cases}
    \mathbf{y}^\top G' \mathbf{y} & \text{if} ~\sa{\theta\leq \rho},\\
    \tilde{\mathbf{y}}^\top G' \tilde{\mathbf{y}} & \text{else.}
    \end{cases}
\end{equation*}}%
\sa{Thus, $G\succeq 0$ is equivalent to $G'\succeq 0$.}
\end{proof}
\begin{lemma}\label{lemma: sub positive matrix}
\sa{Given $\tau,\sigma>0$, $\theta,\alpha\geq 0$, and $\rho\in(0,1)$, consider $G$ defined in \cref{eq: SCSC SAPD LMI-dist}. If $G\succeq 0$, then $G''\succeq 0$, where}
{\footnotesize
\begin{equation}
    \label{eq: sub matrix of general SAPD LMI}
  \sa{G''} \triangleq \begin{pmatrix}
   \frac{1}{\sigma}(1-\frac{1}{\rho})+\mu_y +\frac{\alpha}{\rho} & (-| 1 - \tfrac{\theta}{\rho} | - \tfrac{\theta}{\rho})L_{yx} & (-| 1 - \tfrac{\theta}{\rho} | - \tfrac{\theta}{\rho})L_{yy} \\ 
 (-| 1 - \tfrac{\theta}{\rho} | - \tfrac{\theta}{\rho}) L_{yx} & \tfrac{1}{\tau} - \sa{L'_{xx}} & 0 \\ 
 (-| 1 - \tfrac{\theta}{\rho} | - \tfrac{\theta}{\rho})L_{yy} & 0 & \frac{1}{\sigma} - \alpha \\
\end{pmatrix}\sa{\succeq 0}.
\end{equation}}%
\end{lemma}
\begin{proof}
\sa{Note that}
$\bx^\top
G''
\bx
= 
{\bx'}^\top
G'
\bx'\geq 0$ for all $\bx=[x_1~x_2~x_3]^\top\in \mathbb{R}^3$, where  $\bx'=[0~x_1~x_2~x_3~x_1]^\top$ and $G'$ is defined in \eqref{eq:Gp-def}. Then the desired result follows from \cref{lem:equivalent_systems}.
\end{proof}


\sa{Finally, with the following observation, we will be ready to proceed to the proof of \cref{lemma: dist bound with fg-general}. Let $\{\cF_k^{t,x}\}$ and $\{\cF_k^{t,y}\}$ be the filtrations such that $\mathcal{F}^{t,x}_k \triangleq \mathcal{F}(\{x^t_i\}_{i=0}^{k},\{y^t_i\}_{i=0}^{k+1})$ and $\mathcal{F}^{t,y}_k \triangleq \mathcal{F}(\{x^t_i\}_{i=0}^{k},\{y^t_i\}_{i=0}^{k})$ denote the $\sigma$-algebras generated by the random variables in their arguments.} 
A consequence of Assumption~\ref{ASPT: unbiased noise assumption} is that
\sa{for $\mathcal{F}^{t,x}_k$-measurable random variable $v$, i.e.,} $v\in\mathcal{F}^{t,x}_k$, we have that 
$\mathbb{E}\left[\langle \tilde{\nabla}\Phi_x(x^{t}_{k},y^{t}_{k+1};\omega_k^x) - \nabla\Phi_x(x^{t}_{k},y^{t}_{k+1}),v \rangle \right]= 0$; similarly,
for \sa{$v\in\mathcal{F}^{t,y}_k$}, it holds that
$\mathbb{E}\left[\langle \tilde{\nabla}\Phi_y(x^{t}_{k},y^{t}_{k};\omega_k^y) - \nabla\Phi_y(x^{t}_{k},y^{t}_{k}),v \rangle \right]= 0$.
\subsection{Proof of \cref{lemma: dist bound with fg-general}}\label{pf:dist-lemma}
\begin{proof}
Fix arbitrary $(x,y)\in\dom f\times \dom g$. Since $(x^{t}_{k+1},y^{t}_{k+1})\in\dom f\times \dom g$, using the concavity 
of $\cL^{t}(x^{t}_{k+1}, \cdot)$
and the convexity
of $\cL^{t}(\cdot, y^{t}_{k+1})$, Jensen's lemma immediately implies that
{\small
\begin{align}
\label{eq:jensen bounded version dist}
 K_{N}(\rho)  \left(\cL^{t}(\bar{x}^{t}_{N}, y)  - \cL^{t}(x, \bar{y}^{t}_{N}) \right) 
            \leq 
            \sum_{k=0}^{N-1}\rho^{-k} \left( \cL^{t}(x^{t}_{k+1}, y) - \cL^{t}(x, y^{t}_{k+1}) \right),
            \;\sa{\forall}\rho\in \sa{(0,1]},
\end{align}}%
where $\bar{x}^{t}_N=\frac{1}{K_N(\rho)}\sum_{k=0}^{N-1}\sa{\rho^{-k}}x^{t}_{k+1},~\bar{y}^{t}_N=\frac{1}{K_N(\rho)}\sum_{k=0}^{N-1}\sa{\rho^{-k}}y^{t}_{k+1}$, $K_N(\rho)=\sum_{k=0}^{N-1}\rho^{-k+1}$. Thus, if we multiply both sides of 
\eqref{D1} \sa{by $\rho^{-k}$} and \sa{sum the resulting inequality from $k=0$ to $N-1$, then using \eqref{eq:jensen bounded version dist} we get}
{\footnotesize
\begin{equation}\label{INEQ: difference of gap  _bounded version}
    \begin{aligned}
            K_{N} &(\rho)  \left(\cL^{t}(\bar{x}^{t}_{N}, y)  - \mathcal{L}(x, \bar{y}^{t}_{N}) \right)  \\
            \leq & 
             \sum_{k=0}^{N-1}\rho^{-k} \Big( 
                \underbrace{ -\langle q^{t}_{k+1}, y^{t}_{k+1} - y \rangle + \theta \langle q^{t}_{k}, y^{t}_{k} - y \rangle}_{\text{\bf part 1}} + \Lambda^{t}_{k} \sa{(x,y)} - \Sigma^{t}_{k+1} \sa{(x,y)}+ \Gamma^{t}_{k+1} 
            \\
             & \quad  \underbrace{-
           \langle  \tilde{\nabla}_x \Phi^{t}(x^{t}_{k}, y^{t}_{k+1};\omega_k^x) - \nabla_x \Phi^{t}(x^{t}_{k}, y^{t}_{k+1}) , x^{t}_{k+1} - x \rangle}_{\text{\bf part 2} }
            +
           \underbrace{\langle \tilde{s}^{t}_k -s^{t}_k, 
          y^{t}_{k+1} - y \rangle}_{ \text{\bf part 3} }
          \Big).
    \end{aligned}
\end{equation}}%
\sa{Using Cauchy–Schwarz inequality and \eqref{INEQ: Cauchy Ineqaulity 1}} leads to
{\begin{equation}\label{eq:cauchy-S-dist}
    |\langle q^{t}_{k+1}, y^{t}_{k+1} - y \rangle|\leq  \sa{S^{t}_{k+1}}\xzh{(x,y)} \triangleq L_{yx}\|x^{t}_{k+1}-x^{t}_{k}\| \|y^{t}_{k+1}-y\|+L_{yy}\|y^{t}_{k+1}-y^{t}_{k}\| \|y^{t}_{k+1}-y\|
\end{equation}}%
for $k \geq -1$.
Recall $x^{t}_{-1} = x^{t}_0,\;y^{t}_{-1} = y^{t}_0$, thus \sa{$q_0=\mathbf{0}$; therefore, for \sa{\bf part\;1},}
{\small
\begin{align}
\label{eq:noisy-rate-part1-dsit}
         \sum_{k=0}^{N-1}&\rho^{-k} 
                 ( \theta \langle q^{t}_k, y^{t}_{k} - y \rangle-\langle q^{t}_{k+1}, y^{t}_{k+1} - y \rangle )
                 = \sum_{k=0}^{N-2}\rho^{-k}\Big(\frac{\theta}{\rho}-1\Big)\langle q^{t}_{k+1}, y^{t}_{k+1} - y \rangle - \rho^{-N+1}\langle q^{t}_{N}, y^{t}_{N} - y \rangle
                 \\
                 \leq & \sum_{k=0}^{N-2}\rho^{-k}\sa{|1-\frac{\theta}{\rho}|~S^{t}_{k+1}\xzh{(x,y)}} + \rho^{-N+1}\sa{S^t_{N}\xzh{(x,y)}}
                 =  \sum_{k=0}^{N-1}\rho^{-k}\sa{|1-\frac{\theta}{\rho}|~S^{t}_{k+1}\xzh{(x,y)}} 
                 + \rho^{-N+1}\frac{\theta}{\rho}\sa{S^{t}_{N}\xzh{(x,y)}},\nonumber
\end{align}}%
where \sa{the first inequality follows from 
\cref{eq:cauchy-S-dist}.}

Next, letting $\Delta^{t,x}_{k}$ and $\hat{x}_{k+1}$ be defined as in
\cref{lemma: intermediate noisy bound},  we equivalently write \textbf{part 2} as
{\footnotesize
\begin{equation}\label{eq:noisy-rate-part2-dist}
    \begin{aligned}
           \sum_{k=0}^{N-1} -\rho^{-k} \langle \Delta^{t,x}_{k} , x^{t}_{k+1} -  x \rangle
            =
             \sum_{k=0}^{N-1} \rho^{-k}
            \Big(\langle \Delta^{t,x}_{k} , {\hat{x}^{t}_{k+1}}-x^{t}_{k+1} \rangle  - \langle \Delta^{t,x}_{k} , {\hat{x}^{t}_{k+1}} -  x \rangle\Big). 
    \end{aligned}
\end{equation}}%
Moreover, for $\Delta^{t,y}_{k}$, $\hat{y}^{t}_{k+1}$ and $\hat{\hat{y}}^{t}_{k+1}$ defined as in 
\cref{lemma: intermediate noisy bound}, we also equivalently write \textbf{part 3} as 
{\small
\begin{equation}\label{eq:noisy-rate-part3-dist}
\begin{aligned}
 \sum_{k=0}^{N-1} &\rho^{-k} \langle \tilde{s}^{t}_k -s_k, 
       y^{t}_{k+1} - y \rangle \\
      &=  
         \sum_{k=0}^{N-1} \rho^{-k}\Big[(1 +\theta) \langle \Delta^{t,y}_{k}, 
       y^{t}_{k+1} - {\hat{y}^{t}_{k+1}}  + {\hat{y}^{t}_{k+1}}  - y \rangle
       - \theta \langle \Delta^{t,y}_{k-1}, 
       y^{t}_{k+1} - {\hat{\hat{y}}^{t}_{k+1}} + {\hat{\hat{y}}^{t}_{k+1}}  - y \rangle\Big].
    \end{aligned}
\end{equation}}%
Adding $\rho^{-N+1} D^{t}_{N}(x,y)$ to both sides of \eqref{INEQ: difference of gap  _bounded version}, then using \eqref{eq:noisy-rate-part1-dsit}, \eqref{eq:noisy-rate-part2-dist} and \eqref{eq:noisy-rate-part3-dist}, \sa{for any fixed $(x,y)\in \cX\times \cY$, we get} 
{
\begin{equation}
\label{INEQ: ergordic gap-dist}
\begin{aligned} 
        K_{N} ({\rho})  \left(\mathcal{L}^{t}(\bar{x}^{t}_{N}, y)  - \mathcal{L}^{t}(x, \bar{y}^{t}_{N}) \right) 
        + {\rho}^{-N+1} D^{t}_{N}(x,y)
        \leq
      U^{t}_N(x,y) + \sum_{k=0}^{N-1}{\rho}^{-k}(P^{t}_k(x,y) + Q^{t}_k), 
\end{aligned}
\end{equation}}%
where $U^{t}_N(x,y)$, $ D^{t}_{N}(x,y)$ are defined as
\begin{subequations}
\label{eq:U_N-DeltaN-dist}
\begin{equation}
    \begin{aligned}
        \sa{U^{t}_N(x,y)}   \triangleq 
      &  \sum_{k=0}^{N-1}{\rho}^{-k}
      \left(\Gamma^{t}_{k+1} + \Lambda^{t}_k{(x,y)} - \Sigma_{k+1} \sa{(x,y)} + \sa{|1-\frac{\theta}{\rho}|~S^{t}_{k+1}}\xzh{(x,y)} \right) 
      \\
      & 
      - {\rho}^{-N+1}\Big( -  D^{t}_{N}(x,y) - \frac{\theta}{\rho}\sa{S^{t}_{N}}\xzh{(x,y)}\Big), 
      \label{eq:U_N-dist} 
    \end{aligned}
\end{equation}
 \begin{equation}
       D^{t}_{N}(x,y)\triangleq \frac{1}{2\rho}\Big(\frac{1}{\tau}-\mu_x\Big) \|x^{t}_{N}-x\|^2 +  \frac{1}{2\rho}{\left(\frac{1}{\sigma} - \alpha \right)}\|y^{t}_{N}-y\|^2, \label{eq:DeltaN-dist}
\end{equation}
\end{subequations}
and $P^{t}_k(x,y)$, $Q^{t}_k$ for $k=0,\cdots,N-1$ are defined as
{
\begin{subequations}
\begin{align}
           P^{t}_k(x,y)  \triangleq 
      & -\langle \Delta^{t,x}_{k} , {\hat{x}^{t}_{k+1}} -  x \rangle 
     +    (1+\theta)\langle \Delta^{t,y}_{k}, {\hat{y}^{t}_{k+1}} - y \rangle 
     -\theta \langle \Delta^{t,y}_{k-1}, {\hat{\hat{y}}^{t}_{k+1}} -y  \rangle,
     \label{eq:Pk-dist}
     \\
     Q^{t}_k  \triangleq &
     \langle \Delta^{t,x}_{k} , {\hat{x}^{t}_{k+1} -  x^{t}_{k+1}} \rangle 
     +   
     (1+\theta)\langle \Delta^{t,y}_{k}, {y^{t}_{k+1} - \hat{y}^{t}_{k+1}}  \rangle 
     -\theta \langle \Delta^{t,y}_{k-1}, {y^{t}_{k+1} - \hat{\hat{y}}^{t}_{k+1}}  \rangle.
     \label{eq:Qk-dist}
    \end{align}
\end{subequations}
}%
{For any fixed $(x,y)\in \cX \times \cY$}, we first analyze $U^{t}_N(x,y)$. 
After adding and subtracting $\tfrac{\alpha}{2}\|y^{t}_{k+1}-y^{t}_{k}\|^2$,  and rearranging the terms, we get
{\footnotesize
\begin{equation}
\label{eq:U_def-dist}
    \begin{aligned}
           U^{t}_N(x,y) =  & 
           \frac{1}{2}\sum_{k=0}^{N-1}{\rho}^{-k}\Big(\xi_k^\top A\xi_k-\xi_{k+1}^\top B \xi_{k+1}\Big)
           -{\rho}^{-N+1}( -  D^{t}_{N}(x,y) - \frac{\theta}{\rho}\sa{S^{t}_{N}}\xzh{(x,y)})\\
           =  & 
           \frac{1}{2}\xi_0^\top A\xi_0- \frac{1}{2}\sum_{k=1}^{N-1}{\rho}^{-k+1}[\xi_k^\top{( B 
           - \tfrac{1}{\rho}A)}\xi_k]
           -\rho^{-N + 1 }\Big(\frac{1}{2}\xi_N^\top   B \xi_N  -  D^{t}_{N}(x,y) -\frac{\theta}{\rho}S^{t}_{N}\xzh{(x,y)}\Big),    
    \end{aligned}
\end{equation}}%
\sa{
where $A, B \in\reals^{5\times 5}$ 
and $\xi_k\in\reals^5$ are defined for $k\geq 0$ as follows:}
{\small
$\xi_k \triangleq \left( 
    \begin{array}{*{20}{c}}
  \|x^{t}_{k} - x\| \\ 
  \|y^{t}_{k} - y\| \\ 
  \|x^{t}_{k} - x^{t}_{k-1}\| \\ 
  \|y^{t}_{k} - y^{t}_{k-1}\| \\
  \|y^{t}_{k+1} - y^{t}_{k}\| 
   \end{array}
   \right)$ {\normalsize such that $x^{t}_{-1} = x^{t}_{0}$, $y^{t}_{-1} = y^{t}_{0}$, and} 
\begin{small}   
$$
A \triangleq   
  \begin{pmatrix}
    \frac{1}{\tau}-\mu_x & 0 & 0 & 0 & 0\\ 
  0 & \frac{1}{\sigma} & 0 & 0 & 0\\ 
  0 & 0 & 0 & 0 & {\theta L_{yx}}\\ 
  0& 0 & 0 & 0 & {\theta L_{yy}}\\
  0 & 0 & \theta L_{yx} & \theta L_{yy} & -\alpha
\end{pmatrix},
\quad
   B \triangleq   
  \begin{pmatrix}
    \frac{1}{\tau} & 0 & 0 & 0 & 0\\ 
  0 & \frac{1}{\sigma}+\mu_y & \sa{-|1-\frac{\theta}{\rho}|}~L_{yx} & \sa{-|1-\frac{\theta}{\rho}|}~L_{yy} & 0\\ 
  0 & \sa{-|1-\frac{\theta}{\rho}|}~L_{yx} & \tfrac{1}{\tau} - L'_{xx} & 0 & 0\\ 
  0& \sa{-|1-\frac{\theta}{\rho}|}~L_{yy} & 0 & \frac{1}{\sigma} - \alpha & 0\\
  0 & 0 & 0 & 0 & 0
\end{pmatrix}.$$
\end{small}}%
\sa{In Lemma~\ref{lem:equivalent_systems} we show that \cref{eq: SCSC SAPD LMI-dist} 
is equivalent to $B -\tfrac{1}{\rho}A\succeq 0$;} therefore, it follows from \eqref{eq:U_def-dist} that for any given $(x,y)\in \cX\times \cY$,
\begin{equation*}
    U^{t}_N(x,y) \leq  \sa{\frac{1}{2}\xi_0^\top A\xi_0}- {\rho}^{-N + 1 } ( \tfrac{1}{2}\xi_N^\top  B \xi_N  -   D^{t}_{N}(x,y) -\frac{\theta}{\rho}\sa{S^{t}_{N}\xzh{(x,y)}}), \text{ holds w.p.~1.}
\end{equation*}
 Furthermore, we have
{
$$\frac{1}{2}\xi_N^\top    B  \xi_N  - D^{t}_{N}(x,y)-\frac{\theta}{\rho}\sa{S^{t}_{N}\xzh{(x,y)}}=
\frac{1}{2}\xi_N^\top  
     \begin{pmatrix} 
     \sa{\frac{1}{\tau}(1-\frac{1}{\rho})}+\frac{\mu_x}{\rho}
       & \mathbf{0}_{1\times3} & 0\\ 
       \mathbf{0}_{3\times 1} & \sa{G''} & \mathbf{0}_{3\times 1} \\
       0 & \mathbf{0}_{1\times3} & 0\\ 
     \end{pmatrix}
\xi_N
  \geq 
0,$$}%
which 
follows from 
\cref{eq: SCSC SAPD LMI-dist,lemma: sub positive matrix}, where \sa{$G''$} is defined in \cref{eq: sub matrix of general SAPD LMI}.
Finally, 
$$\frac{1}{2}\xi_0^\top A\xi_0\sa{\leq} \frac{1}{2\tau}\norm{x-x^{t}_{0}}^2+\frac{1}{2\sigma}\norm{y-y^{t}_{0}}^2.$$ \xzd{Thus, for any $(x,y)\in\cX\times \cY$,}
\begin{align}
\label{eq:U-bound-dist}
    U^{t}_N(x,y)\leq \frac{1}{2\tau}\norm{x-x^{t}_{0}}^2+\frac{1}{2\sigma}\norm{y-y^{t}_{0}}^2,\quad \text{w.p.}~1.
\end{align}%

Now, we are ready to show 
\cref{eq:dist bound-general}. \xz{It follows from \cref{INEQ: ergordic gap-dist} and \cref{eq:U-bound-dist} that, for any $(x,y)\in\cX\times \cY$,}
{
\begin{equation}\label{eq:gap_bound_fg_2-dist}
\begin{aligned}
K_{N}& ({\rho})\Big( \cL^{t}(\bar{x}^{t}_{N}, y)  - \cL^{t}(x, \bar{y}^{t}_{N})\Big) +  \xz{{\rho}^{-N+1} D^{t}_{N}(x,y)} \\
&\leq 
      \frac{1}{2\tau}\|x-x^{t}_{0}\|^2
      +\frac{1}{2\sigma}\|y-y^{t}_{0}\|^2 
       +\sum_{k=0}^{N-1}{\rho}^{-k}(P^{t}_k(x,y)+ Q^{t}_k). 
\end{aligned}
\end{equation}}%
Let $(x^t_*,y^t_*)$ be the unique saddle point of $\cL^t$. If we substitute $(x,y)=(x^t_*,y^t_*)$ into \cref{eq:gap_bound_fg_2-dist} and use the fact $ \cL^{t}(\bar{x}^{t}_{N}, y^{t}_{*})  - \cL^{t}(x^{t}_{*}, \bar{y}^{t}_{N})\geq 0$, we obtain that
{
\begin{equation}\label{eq:gap_bound_fg_3}
\begin{aligned}
 \xz{{\rho}^{-N+1} D^{t}_{N}(x^{t}_*,y^{t}_{*})} \leq 
     \frac{1}{2\tau}\|x^{t}_*-x^{t}_{0}\|^2
      +\frac{1}{2\sigma}\|y^{t}_{*}-y^{t}_{0}\|^2 
       +\sum_{k=0}^{N-1}{\rho}^{-k}(P^{t}_k(x^{t}_{*},y^{t}_{*})+ Q^{t}_k). 
\end{aligned}
\end{equation}}%

From \cref{ASPT: unbiased noise assumption}, for 
$k\geq-1$, we have
$$\mathbb{E}\left[ \langle \Delta^{t,x}_{k} , {\hat{x}^{t}_{k+1}} -  x^t_* \rangle \right] = \mathbb{E}\left[ \langle \Delta^{t,y}_{k}, {\hat{y}^{t}_{k+1}} - y_*^t  \rangle \right] = \mathbb{E}\left[ \langle \Delta^{t,y}_{k-1},{\hat{\hat{y}}^{t}_{k+1}} - y_*^t  \rangle \right] = 0.
$$
Thus, 
$$\mathbb{E}[P^{t}_k(x^{t}_{*},y^{t}_{*})]=0.$$
Moreover, from \cref{ASPT: unbiased noise assumption}, for $k\geq -1$, we have
$$
\mathbb{E}\left[ \|\Delta^{t,x}_{k}\|^2 \right] \leq \delta_x^2,\quad \mathbb{E}\left[ \|\Delta^{t,y}_{k}\|^2 \right] \leq \delta_y^2.
$$
Therefore, we uniformly upper bound $\mathbb{E}\left[ Q^{t}_k \right]$ for $k\geq 0$ using \cref{Lemma:noisy bound with fg}, i.e.,
     $$\mathbb{E}[\sum_{k=0}^{N-1}{\rho}^{-k} Q^{t}_k]  \leq
    \Big(\tau\Xi_{\tau,\sigma,\theta}^x \delta_x^2  +  \sigma\Xi_{\tau,\sigma,\theta}^y\delta_y^2\Big) \sum_{k=0}^{N-1}{\rho}^{-k},$$
where $\Xi^x_{\tau,\sigma,\theta}$ and $\Xi^y_{\tau,\sigma,\theta}$ are defined in~\eqref{eq:xi_x} and \eqref{eq:xi_y}. Therefore, combining this result with $\mathbb{E}[P^{t}_k(x^{t}_{*},y^{t}_{*})]=0$ for any $k\in\{0,\ldots, N-1\}$, we get
{
\begin{equation}
\label{eq:variance-bound-dist}
   \mathbb{E}[ \sum_{k=0}^{N-1}{\rho}^{-k}(P^{t}_k(x^{t}_{*},y^{t}_{*})+ Q^{t}_k)]\leq \xzh{\sum_{k=0}^{N-1}{\rho}^{-k}}~\Big(\tau\Xi_{\tau,\sigma,\theta}^x \delta_x^2  +  \sigma\Xi_{\tau,\sigma,\theta}^y\delta_y^2\Big) .
\end{equation}}
Then, 
using the definition of $ D^{t}_{N}(x^t_*,y^t_*)$ in \cref{eq:DeltaN-dist} and the fact $$\xzh{\sum_{k=0}^{N-1}{\rho}^{-k}}=\rho^{-N+1}\frac{1-\rho^N}{1-\rho}\leq \rho^{-N+1}\frac{1}{1-\rho},$$
\sa{for any $\rho\in(0,1)$, the desired inequality in~\eqref{eq:dist bound-general} follows from} \eqref{eq:gap_bound_fg_3} and \eqref{eq:variance-bound-dist}. 
\end{proof}

\subsection{Proof of Lemma~\ref{lemma:bound-gap}}\label{pf:gap-lemma}
Throughout this proof, our analysis is based on the proof of \cref{lemma: dist bound with fg-general}. To analyze the expected gap in \cref{lemma:bound-gap}, we consider the setting with $\rho=1$,  which implies that $K_N(\rho)=N$ and $\bar{x}^{t}_N=\frac{1}{ N }\sum_{k=0}^{N-1}x^{t}_{k+1},~\bar{y}^{t}_N=\frac{1}{ N }\sum_{k=0}^{N-1}y^{t}_{k+1}$. The proof of \cref{lemma:bound-gap} is different than that of \cref{lemma: dist bound with fg-general} in the way we analyze the variance terms. To be precise, we construct the auxiliary sequences --see $\tilde{x}_{k}$, $\tilde{y}^{+}_k$, $\tilde{y}^{-}_k$ defined in~\cref{eq:tilde-x} and \cref{eq:tilde-y} --for the analysis of \textbf{part 2} and \textbf{part 3} in \cref{INEQ: difference of gap  _bounded version} to provide guarantees on the expected gap function. 

\begin{proof}
Fix arbitrary $(x,y)\in\dom f\times \dom g$. Since $(x^{t}_{k+1},y^{t}_{k+1})\in\dom f\times \dom g$, using the concavity 
of $\cL^{t}(x^{t}_{k+1}, \cdot)$
and the convexity
of $\cL^{t}(\cdot, y^{t}_{k+1})$, Jensen's lemma immediately implies that
{
\begin{align}
\label{eq:jensen bounded version gap}
  N   \left(\cL^{t}(\bar{x}^{t}_{N}, y)  - \cL^{t}(x, \bar{y}^{t}_{N}) \right) 
            \leq 
            \sum_{k=0}^{N-1} \left( \cL^{t}(x^{t}_{k+1}, y) - \cL^{t}(x, y^{t}_{k+1}) \right),
\end{align}}%
where $\bar{x}^{t}_N=\frac{1}{ N }\sum_{k=0}^{N-1}x^{t}_{k+1},~\bar{y}^{t}_N=\frac{1}{ N }\sum_{k=0}^{N-1}y^{t}_{k+1}$. \sa{ Summing \cref{D1} from $k=0$ to $N-1$ and using \eqref{eq:jensen bounded version gap}, we get}
{
\begin{equation}\label{INEQ: difference of gap  bounded version gap}
    \begin{aligned}
            N & \left(\cL^{t}(\bar{x}^{t}_{N}, y)  - \mathcal{L}(x, \bar{y}^{t}_{N}) \right)  \\
            \leq & 
             \sum_{k=0}^{N-1}
                \underbrace{ -\langle q^{t}_{k+1}, y^{t}_{k+1} - y \rangle + \theta \langle q^{t}_{k}, y^{t}_{k} - y \rangle}_{\text{\bf part 1}} + \Lambda^{t}_{k} \sa{(x,y)} - \Sigma^{t}_{k+1} \sa{(x,y)}+ \Gamma^{t}_{k+1} 
            \\
             & \quad  \underbrace{-
          \langle  \tilde{\nabla}_x \Phi^{t}(x^{t}_{k}, y^{t}_{k+1};\omega_k^x) - \nabla_x \Phi^{t}(x^{t}_{k}, y^{t}_{k+1}) , x^{t}_{k+1} - x \rangle}_{\text{\bf part 2} }
            +
          \underbrace{\langle \tilde{s}^{t}_k -s^{t}_k, 
          y^{t}_{k+1} - y \rangle}_{ \text{\bf part 3} }
         .
    \end{aligned}
\end{equation}}%
The bound on \textbf{Part 1} immediately follows from \cref{eq:noisy-rate-part1-dsit} with $\rho=1$, i.e.,
{
\begin{align}
\label{eq:noisy-rate-part1-gap}
         \sum_{k=0}^{N-1}&
                  \theta \langle q^{t}_k, y^{t}_{k} - y \rangle-\langle q^{t}_{k+1}, y^{t}_{k+1} - y \rangle 
                 \leq \sum_{k=0}^{N-1}\sa{|1-\theta |~S^{t}_{k+1}\xzh{(x,y)}} 
                 + \sa{\theta} \sa{S^{t}_{N}\xzh{(x,y)}}.
\end{align}}%
Recall that $x^{t}_{-1} = x^{t}_0$ and $y^{t}_{-1} = y^{t}_0$; thus, \sa{$q_0^t=\mathbf{0}$.}

Next we consider \textbf{part 2}, 
\sa{let $\Delta^{t,x}_{k}$ be defined as in 
\cref{lemma: intermediate noisy bound}.} \sa{For some arbitrary $\eta_x> 0$}, define \sa{$\{\tilde{x}_k\}$} sequence as follows:
{
\begin{equation}\label{eq:tilde-x}
    \tilde{x}_0 \triangleq x^{t}_0, \quad \tilde{x}_{k+1} \triangleq \argmin_{x'\in \sa{\cX}}  - \langle \Delta^{t,x}_{k} , x' \rangle + \frac{\sa{\eta_x}}{2}\| x' - \tilde{x}_k\|^2,\quad \forall~k\geq 0.
\end{equation}
}%
Then by \sa{\cite[Lemma 2.1]{nemirovski2009robust}}, \sa{for all $k\geq 0$ and {$x\in \cX$}}, we have that
      $$\langle \Delta^{t,x}_{k} , x -  \tilde{x}_{k} \rangle   \leq \frac{\eta_x}{2}\| x - \tilde{x}_{k}\|^2 -\frac{\eta_x}{2}\| x - \tilde{x}_{k+1}\|^2  +\frac{1}{2\eta_x}\| \Delta^{t,x}_{k}\|^2.$$
\sa{Thus, using $\tilde{x}_0 = x^{t}_{0}$ we get}
{\footnotesize
\begin{equation}\label{INEQ: bound of x - tilde x }
    \begin{aligned}
          \sum_{k=0}^{N-1} & \langle \Delta^{t,x}_{k} , x -  \tilde{x}_{k} \rangle   \leq
          \sum_{k=0}^{N-1} \Big(\frac{\eta_x}{2}\| x - \tilde{x}_{k}\|^2 -\frac{\eta_x}{2}\| x - \tilde{x}_{k+1}\|^2  +\frac{1}{2\eta_x}\| \Delta^{t,x}_{k}\|^2\Big) \\
            = & \frac{\eta_x}{2}(\| x - {{x}^{t}_{0}}\|^2 -\| x - \tilde{x}_{N}\|^2)+ \sum_{k=0}^{N-1}\frac{1}{2\eta_x} \| \Delta^{t,x}_{k}\|^2 
            \leq  {\frac{\eta_x}{2}} \sa{\norm{x-x^{t}_{0}}^2} 
            + {\frac{1}{2\eta_x}}\sum_{k=0}^{N-1}\| \Delta^{t,x}_{k}\|^2;
     \end{aligned}
\end{equation}}%
hence, \textbf{part 2} becomes
{\footnotesize
\begin{equation}\label{eq:noisy-rate-part2-gap}
    \begin{aligned}
            & \sum_{k=0}^{N-1}  \langle \Delta^{t,x}_{k} , \sa{x-x^{t}_{k+1}} \rangle\\  
            =&    \sum_{k=0}^{N-1} 
            \langle \Delta^{t,x}_{k} , {\hat{x}_{k+1}}^{t}-x^{t}_{k+1} \rangle  - \langle \Delta^{t,x}_{k} , {\hat{x}^{t}_{k+1}} -  \tilde{x}_k \rangle
            + \langle \Delta^{t,x}_{k} , x -\tilde{x}_k   \rangle
            \\
            \leq&    \frac{\sa{\eta_x}}{2}\sa{\norm{x-x^{t}_0}^2}+\sum_{k=0}^{N-1} 
            \langle \Delta^{t,x}_{k} , {\hat{x}^{t}_{k+1}}-x^{t}_{k+1} \rangle  - \langle \Delta^{t,x}_{k} , {\hat{x}^{t}_{k+1}} -  \tilde{x}_k \rangle
          +\frac{1}{2\sa{\eta_x}}\| \Delta^{t,x}_{k}\|^2,
    \end{aligned}
\end{equation}}%
\sa{which follows} from \cref{INEQ: bound of x - tilde x }, and $\hat{x}_{k+1}$ is defined in \cref{lemma: intermediate noisy bound}.\footnote{When $\delta_x=0$, clearly $\Delta^{t,x}_{k}=\mathbf{0}$; thus, \textbf{part 2} is equal to $0$ and we can set $\eta_x=0$ for which \eqref{eq:noisy-rate-part2-gap} 
becomes $0\leq 0$.} 

Next, we consider \textbf{part 3}, let $\Delta^{t,y}_{k}$ be defined as in 
\cref{lemma: intermediate noisy bound}.
{For some arbitrary $\eta_y> 0$, we construct two auxiliary sequences:
let $\tilde{y}_0^+  = \tilde{y}_0^- = y^{t}_{0}$, and  for $k\geq 0$, we define}
{
\begin{equation}\label{eq:tilde-y}
    \tilde{y}^+_{k+1} \triangleq \argmin_{y'\in \cY}   \langle \Delta^{t,y}_{k} , y' \rangle + \frac{\sa{\eta_y}}{2}\| y' - \tilde{y}^+_k\|^2, \quad 
          \tilde{y}^-_{k+1} \triangleq \argmin_{y'\in \cY}   -\langle \Delta^{t,y}_{k} , y' \rangle + \frac{\sa{\eta_y}}{2}\| y' - \tilde{y}^-_k\|^2.
\end{equation}
}%
{Thus, it follows from \cite[Lemma 2.1]{nemirovski2009robust} that for $y\in \cY$,}
\begin{align*}
      &\langle \Delta^{t,y}_{k} ,   \tilde{y}_{k}^+ - y  \rangle   \leq \frac{\eta_y}{2}\| y - \tilde{y}_{k}^+\|^2 -\frac{\eta_y}{2}\| y - \tilde{y}_{k+1}^+\|^2  +\frac{1}{2\eta_y}\| \Delta^{t,y}_{k}\|^2, 
      \\
      & \langle \Delta^{t,y}_{k} ,   y-\tilde{y}_{k}^-  \rangle   \leq \frac{\eta_y}{2}\| y - \tilde{y}_{k}^-\|^2 -\frac{\eta_y}{2}\| y - \tilde{y}_{k+1}^-\|^2  +\frac{1}{2\eta_y}\| \Delta^{t,y}_{k}\|^2.
\end{align*}
\sa{Therefore, 
as in~\eqref{INEQ: bound of x - tilde x }, 
we get\footnote{\sa{As in~\textbf{part 2}, when $\delta_y=0$, we can set $\eta_y=0$.}}
{\footnotesize
\begin{equation}\label{INEQ:bound of y-tilde y}
\begin{aligned}
            \sum_{k=0}^{N-1}  &(1+\theta)\langle \Delta^{t,y}_{k} , \tilde{y}_{k}^+ -  y \rangle+ \theta \langle \Delta^{t,y}_{k-1} ,y-\tilde{y}_{k-1}^-   \rangle 
            \\ 
              \leq&
            \frac{\eta_y}{2}(1+2\theta) \sa{\norm{y-y^{t}_{0}}^2} 
            + \frac{1}{2\eta_y}\sum_{k=0}^{N-1}\Big( (1+\theta)\| \Delta^{t,y}_{k}\|^2 + \theta\| \Delta^{t,y}_{k-1}\|^2\Big).
\end{aligned}
\end{equation}}}%
Next, using \cref{INEQ:bound of y-tilde y}, we can bound \textbf{part 3} as follows:
{\footnotesize
\begin{equation}\label{eq:noisy-rate-part3-gap}
    \begin{aligned}
        &\sum_{k=0}^{N-1} \langle \tilde{s}^{t}_{k} -s^{t}_{k},
      y^{t}_{k+1} - y \rangle 
      \\
      = &  
         \sum_{k=0}^{N-1}(1 +\theta) \langle \Delta^{t,y}_{k}, 
      y^{t}_{k+1} - {\hat{y}^{t}_{k+1}}  + {\hat{y}^{t}_{k+1}} - \tilde{y}_k^+ + \tilde{y}_k^+ - y \rangle 
      - \theta \langle \Delta^{t,y}_{k-1}, 
      y^{t}_{k+1} - {\hat{\hat{y}}^{t}_{k+1}} + {\hat{\hat{y}}^{t}_{k+1}}  - \tilde{y}_{k-1}^- + \tilde{y}_{k-1}^- - y \rangle
      \\
      \leq & 
        \sum_{k=0}^{N-1} (1 +\theta) \langle \Delta^{t,y}_{k},
      y^{t}_{k+1} -{\hat{y}^{t}_{k+1}} + {\hat{y}^{t}_{k+1}} - \tilde{y}_k^+  \rangle  
      - \theta \langle \Delta^{t,y}_{k-1}, 
      y^{t}_{k+1} - {\hat{\hat{y}}^{t}_{k+1}} + {\hat{\hat{y}}^{t}_{k+1}} - \tilde{y}_{k-1}^-  \rangle 
      \\
      & + \frac{1}{2{\eta_y}}\sum_{k=0}^{N-1} \Big((1+\theta)\| \Delta^{t,y}_{k}\|^2 +  \theta\| \Delta^{t,y}_{k-1}\|^2 \Big) +\frac{\sa{\eta_y}}{2}(1+2\theta) \sa{\norm{y-y^t_0}^2}, 
    \end{aligned}
\end{equation}}%
where $\hat{y}^{t}_{k+1}$ and $\hat{\hat{y}}^{t}_{k+1}$ are defined in \cref{lemma: intermediate noisy bound}. 

For any fixed $(x,y)\in \dom f\times \dom g$,
we use \eqref{eq:noisy-rate-part1-gap}, \eqref{eq:noisy-rate-part2-gap} and \eqref{eq:noisy-rate-part3-gap}
{to get}
{
\begin{equation}
\label{INEQ: ergordic gap-gap}
\begin{aligned} 
        N & \left(\cL^{t}(\bar{x}^{t}_{N}, y)  - \cL^{t}(x, \bar{y}^{t}_{N}) \right) \\
        &\leq 
      \sa{\tilde U^{t}_N(x,y)} + \frac{\sa{\eta_x}}{2} \norm{x-x^{t}_0}^2 
      +\frac{\sa{\eta_y}}{2}(1+2\theta)\sa{\norm{y-y^{t}_0}^2} 
      + \sum_{k=0}^{N-1}(\tilde{P}^{t}_{k}+ \tilde{Q}^{t}_{k}), 
\end{aligned}
\end{equation}}%
where \sa{$\tilde U_N^t(x,y)$} and $\tilde{P}^{t}_{k}$, $\tilde{Q}^{t}_{k}$ for $k=0,\ldots,N-1$ are defined as follows:
\begin{subequations}
\label{eq:U_N-DeltaN-gap}
\begin{align}
   \label{eq:U_N-gap}
     \sa{\tilde U^{t}_N(x,y)}  & \triangleq 
        {\sum_{k=0}^{N-1}
      \left(\Gamma^{t}_{k+1} + \Lambda^{t}_k{(x,y)} - \Sigma^{t}_{k+1} \sa{(x,y)} + {\sa{|1-\theta |~S^{t}_{k+1}}\xzh{(x,y)}} \right) 
      +\sa{\theta \sa{S^{t}_{N}}\xzh{(x,y)}},}
      \\
      \tilde{P}^{t}_{k}  &\triangleq 
       -\langle \Delta^{t,x}_{k} , {\hat{x}^{t}_{k+1}} -  \tilde{x}_k \rangle 
     +    (1+\theta)\langle \Delta^{t,y}_{k}, {\hat{y}^{t}_{k+1}} - \sa{\tilde{y}_k^+}  \rangle 
     -\theta \langle \Delta^{t,y}_{k-1}, {\hat{\hat{y}}^{t}_{k+1}} - \sa{\tilde{y}_{k-1}^-}  \rangle,
     \label{eq:tilde-Pk-def}
     \\
     \tilde{Q}^{t}_{k}  &\triangleq 
     \langle \Delta^{t,x}_{k} , {\hat{x}^{t}_{k+1} -  x^{t}_{k+1}} \rangle 
     +   
     (1+\theta)\langle \Delta^{t,y}_{k}, {y^{t}_{k+1} - \hat{y}^{t}_{k+1}}  \rangle 
     -\theta \langle \Delta^{t,y}_{k-1}, {y^{t}_{k+1} - \hat{\hat{y}}^{t}_{k+1}}  \rangle \nonumber
     \\
     & \quad +
     \frac{1}{2{\eta_x}}\|\Delta^{t,x}_{k} \|^2  
     +
     \frac{1+\theta}{2{\eta_y}}\| \Delta^{t,y}_{k}\|^2
     + 
     \frac{\theta}{2{\eta_y}}\|\Delta^{t,y}_{k-1}\|^2. 
     \label{eq:tilde-Qk-def}
    \end{align}
\end{subequations}%

\sa{The remaining part of the analysis directly follows from the arguments we used in the proof of \cref{lemma: dist bound with fg-general}.} {For any fixed $(x,y)\in \cX \times \cY$}, we first analyze \sa{$\tilde U^{t}_N(x,y)$}. \sa{For some given $\alpha>0$,} after adding and subtracting $\tfrac{\alpha}{2}\|y^{t}_{k+1}-y^{t}_{k}\|^2$, and rearranging the terms, we get
{
\begin{equation}
\label{eq:U_def-gap}
    \begin{aligned}
          \tilde U^{t}_N(x,y) =  & 
          \frac{1}{2}\sum_{k=0}^{N-1}\Big(\xi_k^\top \tilde{A}\xi_k-\xi_{k+1}^\top \tilde{B} \xi_{k+1}\Big)
          +\theta \sa{S^{t}_{N}}\xzh{(x,y)}\\
          =  & 
         \frac{1}{2}\xi_0^\top \tilde{A}\xi_0- \frac{1}{2}\sum_{k=1}^{N-1}[\xi_k^\top{( \tilde{B} 
          - \tilde{A})}\xi_k]
          \sa{-\Big( \frac{1}{2}\xi_N^\top \tilde{B} \xi_N  
          -\theta S^{t}_{N}\xzh{(x,y)}\Big),}  
          \end{aligned}
\end{equation}}%
\sa{
where $A, B \in\reals^{5\times 5}$ 
and $\xi_k\in\reals^5$ are defined for $k\geq 0$ as follows:}
{
$$
\xi_k \triangleq
\begin{pmatrix}
  \|x^{t}_{k} - x\| \\ 
  \|y^{t}_{k} - y\| \\ 
  \|x^{t}_{k} - x^{t}_{k-1}\| \\ 
  \|y^{t}_{k} - y^{t}_{k-1}\| \\
  \|y^{t}_{k+1} - y^{t}_{k}\| 
\end{pmatrix},
\qquad
\tilde{A} \triangleq   
  \begin{pmatrix}
    \frac{1}{\tau}-\mu_x & 0 & 0 & 0 & 0\\ 
  0 & \frac{1}{\sigma} & 0 & 0 & 0\\ 
  0 & 0 & 0 & 0 & {\theta L_{yx}}\\ 
  0& 0 & 0 & 0 & {\theta L_{yy}}\\
  0 & 0 & \theta L_{yx} & \theta L_{yy} & -\alpha
\end{pmatrix},
$$
and
$$
\quad
\tilde{B} \triangleq   
  \begin{pmatrix}
    \frac{1}{\tau} & 0 & 0 & 0 & 0\\ 
  0 & \frac{1}{\sigma}+\mu_y & \sa{-|1-\theta |}~L_{yx} & \sa{-|1-\theta |}~L_{yy} & 0\\ 
  0 & \sa{-|1-\theta |}~L_{yx} & \tfrac{1}{\tau} - L'_{xx} & 0 & 0\\ 
  0& \sa{-|1-\theta |}~L_{yy} & 0 & \frac{1}{\sigma} - \alpha & 0\\
  0 & 0 & 0 & 0 & 0
\end{pmatrix}
$$}%
such that  $x^{t}_{-1} = x^{t}_{0}$,  $y^{t}_{-1} = y^{t}_{0}$.
\sa{Lemma~\ref{lem:equivalent_systems} together with $\rho=1$ implies that \cref{eq: SCSC SAPD LMI} is equivalent to $\tilde{B} -\tilde{A}\succeq 0$;}  therefore, it follows from \eqref{eq:U_def-gap} that, for any given $(x,y)\in \cX\times \cY$, \sa{we have}
$$
U^{t}_N(x,y) \leq  \sa{\frac{1}{2}\xi_0^\top \tilde{A}\xi_0}- ( \tfrac{1}{2}\xi_N^\top  \tilde{B} \xi_N  
-\theta \sa{S^{t}_{N}}\xzh{(x,y)}),\quad \text{w.p.}~1.
$$
Furthermore, 
we also have
{
$$\frac{1}{2}\xi_N^\top    \tilde{B}  \xi_N  
-\theta \sa{S^{t}_{N}\xzh{(x,y)}} \sa{\geq}
\frac{1}{2}\xi_N^\top  
     \begin{pmatrix}
      \mu_x
      & \mathbf{0}_{1\times3} & 0\\ 
      \mathbf{0}_{3\times 1} & \sa{G''} & \mathbf{0}_{3\times 1} \\
      0 & \mathbf{0}_{1\times3} & 0\\ 
     \end{pmatrix}
\xi_N
  \geq 
0,$$}%
which 
follows from \cref{eq: SCSC SAPD LMI} \sa{and \cref{lemma: sub positive matrix} with $\rho=1$, where} $G''$ is defined in \cref{eq: sub matrix of general SAPD LMI}.
Finally, 
$$
\frac{1}{2}\xi_0^\top \tilde{A}\xi_0\sa{\leq} \frac{1}{2\tau}\norm{x-x^{t}_{0}}^2+\frac{1}{2\sigma}\norm{y-y^{t}_{0}}^2.
$$ \xzd{Thus, the above three inequalities imply that, for any $(x,y)\in\cX\times \cY$,}
\begin{align}
\label{eq:U-bound-gap}
    \sa{\tilde U^{t}_N(x,y)}\leq \frac{1}{2\tau}\norm{x-x^{t}_{0}}^2+\frac{1}{2\sigma}\norm{y-y^{t}_{0}}^2,\quad \text{w.p.}~1.
\end{align}%

Now, we are ready to show \cref{eq:Gap bound with fg}. \xz{It follows from \cref{INEQ: ergordic gap-gap} and \cref{eq:U-bound-gap} that}
{
\begin{equation}\label{eq:gap_bound_fg_2-gap}
\begin{aligned}
N &    \sup_{(x,y)\in\cX\times\cY}\left\{\cL^{t}(\bar{x}^{t}_{N}, y)  - \cL^{t}(x, \bar{y}^{t}_{N}) \right\} \\
&\leq 
      \left(\frac{1}{2\tau}+\frac{\eta_x}{2}\right)\|x^{t}_*(\bar{y}^{t}_N)-x^{t}_{0}\|^2
      +\left(\frac{1}{2\sigma}+\frac{\eta_y(1+2\theta)}{2}\right)\|y_{*}(\bar{x}^{t}_{N})-y^{t}_{0}\|^2 
      +\sum_{k=0}^{N-1}(\tilde{P}^{t}_{k}+ \tilde{Q}^{t}_{k}), 
\end{aligned}
\end{equation}}%
\sa{where $(x^{t}_*(\bar{y}^{t}_N),y_{*}(\bar{x}^{t}_{N}))$ is the point achieving the supremum on the left hand side.}
Indeed, to derive the above inequality, we substitute $(x,y) = (x^{t}_*(\bar{y}^{t}_N),y_{*}(\bar{x}^{t}_{N}))$ into the \cref{INEQ: ergordic gap-gap} and use the fact that 
$$
\sup_{(x,y)\in\cX\times\cY}\left\{\cL^{t}(\bar{x}^{t}_{N}, y)  - \cL^{t}(x, \bar{y}^{t}_{N}) \right\} = \cL^{t}(\bar{x}^{t}_{N}, y_{*}(\bar{x}^{t}_{N}))  - \cL^{t}(x^{t}_*(\bar{y}^{t}_N), \bar{y}^{t}_{N}).
$$
From \cref{ASPT: unbiased noise assumption}, for 
$k\geq-1$, we have
$$
\mathbb{E}\left[ \langle \Delta^{t,x}_{k} , {\hat{x}^{t}_{k+1}} -  \tilde{x}_k \rangle \right] = \mathbb{E}\left[ \langle \Delta^{t,y}_{k},
\;
{\hat{y}^{t}_{k+1}} - \tilde{y}_k^{\pm}  \rangle \right] =
\mathbb{E}\left[ \langle \Delta^{t,y}_{k-1},{\hat{\hat{y}}^{t}_{k+1}} - \tilde{y}_{k-1}^{-}  \rangle \right] = 0.$$
Thus, 
$\mathbb{E}[\tilde{P}^{t}_{k}]=0$.
Moreover, for $k\geq -1$, from \cref{ASPT: unbiased noise assumption} we also have
$$
\mathbb{E}\left[ \|\Delta^{t,x}_{k}\|^2 \right] \leq \delta_x^2,\qquad \mathbb{E}\left[ \|\Delta^{t,y}_{k}\|^2 \right] \leq \delta_y^2.
$$
Next, we uniformly upper bound $\mathbb{E}\left[ \tilde{Q}^{t}_{k} \right]$ for $k\geq 0$ using \cref{Lemma:noisy bound with fg}, i.e.,
$$
\mathbb{E}[\sum_{k=0}^{N-1}\tilde{Q}^{t}_{k}]  \leq
     N\Big[\Big(\tau\Xi_{\tau,\sigma,\theta}^x + \frac{1}{2\eta_x}\Big)\delta_x^2  + \Big(     \sigma\Xi_{\tau,\sigma,\theta}^y +  \sa{\frac{1+2\theta}{2\eta_y}}\Big)\delta_y^2\Big] .
     $$   
Therefore, combining this result with $\mathbb{E}[\tilde{P}^{t}_{k}]=0$ for any $k\in\{0,\ldots, N-1\}$, we get
{
\begin{equation}
\label{eq:variance-bound-gap}
  \mathbb{E}[ \sum_{k=0}^{N-1}(\tilde{P}^{t}_{k}+ \tilde{Q}^{t}_{k})]\leq N~\Xi_{\tau,\sigma,\theta}.
\end{equation}}%
Finally, 
setting $\eta_x=\frac{1}{\tau}$, $\eta_y=\frac{1}{\sigma}$, $x^{t+1}_{0}=\bar{x}^{t}_{N}$, and $y^{t+1}_{0}=\bar{y}^{t}_{N}$, the desired result in~\eqref{eq:Gap bound with fg} follows from \eqref{eq:gap_bound_fg_2-gap} and \eqref{eq:variance-bound-gap}.
\end{proof}
\section{Computation of $\epsilon$-stationary point in practice}
\label{sec:eps-stationary}
In this section, we discuss how to compute a point $x_\epsilon$ such that $\mathbb{E}[\|\grad\phi_{\lambda}(x_\epsilon)\|]\leq \epsilon$ --as the \na{$t_*=\argmin_{t=0,1,\ldots,T}\|\grad\phi_{\lambda}(x_0^t)\|$ in \cref{remark-exchange-min and expectation}} can not be computed \na{trivially}. This result is shown in Theorem~\ref{thm:find-min-GNME}, which directly follows from Theorem~\ref{cor:complexity-fg} and Lemma~\ref{lemma: SAPD single call complexity}. Below, for the sake of completeness, we state a known technical result that we need for the proof of Theorem~\ref{thm:find-min-GNME}.
\begin{lemma}\label{lemma:smooth-GNME}
Suppose Assumptions~\ref{ASPT: fg} and \ref{ASPT: lipshiz gradient} hold. Then $\phi_\lambda(\cdot)$ is $\frac{1}{\lambda}$-smooth for $\lambda\in (0,\gamma^{-1})$, where $\phi_\lambda(\cdot)$ is defined in \cref{Def: Moreau envelope} for $\phi(\cdot) = \max_{y\in\cY}\cL(\cdot,y)$.
\end{lemma}
\begin{proof}
Let $R(x)\triangleq x - \prox{\lambda\phi}(x)$ for $\lambda\in(0,\gamma^{-1})$ and $x\in\dom f$. Indeed, by \cref{Def: Moreau envelope}, we know $R(x)=\lambda\grad\phi_\lambda(x)$. Then by the optimality condition of $\prox{\lambda\phi}(x)$, we obtain that
$$
R(x)\in \partial f(\prox{\lambda\phi}(x))
$$
holds for $x\in \dom f$. Hence, for $x_1,x_2\in \dom f$, we have that
$$
\langle R(x_1)-R(x_2), \prox{\lambda\phi}(x_1) - \prox{\lambda\phi}(x_2)\rangle \geq 0,
$$
which further implies that
\begin{align*}
    \|x_1-x_2\|^2 & = \|R(x_1)-\prox{\lambda\phi}(x_1) -R(x_2)+\prox{\lambda\phi}(x_2)\|^2
    \\
    & \geq \|R(x_1) -R(x_2)\|^2+\|\prox{\lambda\phi}(x_1)-\prox{\lambda\phi}(x_2)\|^2
    \\ &
    \geq \|R(x_1) -R(x_2)\|^2.
\end{align*}
Then using the fact $R(x) = \lambda \grad \phi_\lambda(x)$ completes the proof.
\end{proof}
\begin{theorem}\label{thm:find-min-GNME}
Consider $\cL$ defined in \eqref{eq:main problem}. Suppose Assumptions~\ref{ASPT: fg},~\ref{ASPT: lipshiz gradient},~\ref{ASPT: unbiased noise assumption} hold. Under the premise of \cref{cor:complexity-fg}, for any $\epsilon>0$, \texttt{SAPD+} can generate an point $x_\epsilon$ such that $\mathbb{E}[\|\grad\phi_{\lambda}(x_\epsilon)\|]\leq \epsilon$
within $ \cO\left(\frac{L\kappa_y\cG(x_0^0,y_0^0) }{\epsilon^2}\ln(1/\epsilon) + \frac{ L\kappa_y\delta^2\cG(x_0^0,y_0^0) }{\epsilon^4}\ln(1/\epsilon) \right)$ stochastic first-order oracle calls.
\end{theorem}
\begin{proof}
\xzrev{Under the premise of \cref{cor:complexity-fg}, given $\epsilon>0$, \texttt{SAPD+} generates $\{x^t_0\}_{t=0}^T$ such that  $\min_{t=0,\ldots,T}\mathbb{E}\left[\|\nabla\phi_{\lambda}(x_0^t)\|\right]\leq \epsilon/4$, 
for $T \geq 96\cG(x_0^0,y_0^0)  \cdot\tfrac{16\gamma}{\epsilon^2}+1$. Therefore, for each $x^t_0$, if we let $\hat{x}^t_0= \prox{\lambda\phi}(x^t_0)$, then \cref{lemma: SAPD single call complexity} ensures that we can generate a point $\tilde{x}^t_{*}$ such that 
\begin{equation*}
    \mathbb{E}[\|\tilde{x}^t_{*}-\hat{x}^t_0\|]\leq \hat{\epsilon}
\end{equation*}
within $N_{\hat{\epsilon}}$ many iterations, 
where
{\footnotesize
\begin{equation*}
N_{\hat\epsilon} = \mathcal{O}\left(  
\frac{\max\{L_{xx},L_{yx}\}}{\gamma} + \frac{\max\{L_{yx},L_{xy}\}}{\sqrt{\gamma\mu_y}} + \frac{\max\{L_{yy},L_{yx}\}}{\mu_y}
+ \Big(
\frac{\delta_x^2}{\gamma} +
\frac{\delta_y^2}{\mu_y} \Big)
\frac{1}{\gamma\hat{\epsilon}^2}\right)\cdot\ln\Big( \frac{\max\{1,\mu_y/\gamma\}}{\hat{\epsilon}}\Big)
\end{equation*}}%
Moreover, if we compute the GNME of $\phi(\cdot)$ at $\tilde{x}^t_*$, it follows that
\begin{equation*}
\begin{aligned}
         \|\grad \phi_{\lambda}(\tilde{x}^t_*)\| \leq &  \|\grad \phi_{\lambda}(\tilde{x}^t_*)-\grad \phi_{\lambda}(\hat{x}^t_0)\| + \|\grad \phi_{\lambda}(\hat{x}^t_*)-\grad \phi_{\lambda}(x^t_0)\| + \|\grad \phi_{\lambda}(x^t_0)\| 
         \\
        \leq& \frac{2}{\lambda}\|\tilde{x}^t_*-\hat{x}^t_0\| + \frac{1}{\lambda}\|\hat{x}^t_*-x^t_0\| + \|\grad \phi_{\lambda}(x^t_0)\| \\
         =& \frac{2}{\lambda}\|\tilde{x}^t_*-\hat{x}^t_0\| + 2\|\grad \phi_{\lambda}(x^t_0)\|
                  \\
         \leq & \frac{2}{\lambda}\hat{\epsilon} +  2 \|\grad \phi_{\lambda}(x^t_0)\|.
\end{aligned}
\end{equation*}
where the second inequality is by \cref{lemma:smooth-GNME}; the first equality is by \cref{Def: Moreau envelope} and the fact $\hat{x}^t_0=\prox{\lambda\phi}(x^t_0)$. Furthermore, because $\min_{t=0,\ldots,T}\mathbb{E}\left[\|\nabla\phi_{\lambda}(x_0^t)\|\right]\leq \epsilon/4$, then we have 
$$
\min_{t=0,\ldots,T}\mathbb{E} [\|\grad \phi_{\lambda}(\tilde{x}^t_*)\| ]\leq \frac{2}{\lambda}\hat{\epsilon} + \frac{\epsilon}{2}
$$
and we let $x_\epsilon=\tilde{x}^{\tilde{t}}_*$, where $ \tilde{t}_*\triangleq\argmin_{\{ t=0,..,T\}}\mathbb{E}[\|\grad \phi_{\lambda}(\tilde{x}^t_*)\| ]$.
Therefore, 
setting $\lambda=\frac{1}{2\gamma}$ and $\hat{\epsilon}=\frac{1}{8\gamma}$, \cref{lemma: SAPD single call complexity} implies that calling \texttt{SAPD} $T$ times, each with $\tilde{N}$ iterations, one can generate $x_\epsilon$ such that$$
 \mathbb{E}\big[\|\grad \phi_{\lambda}(x_\epsilon)\|\big] \leq \epsilon,
$$
where $T$ is given in \cref{cor:complexity-fg} and 
{\footnotesize
\begin{equation*}
\tilde{N} = \mathcal{O}\left(  
\frac{\max\{L_{xx},L_{yx}\}}{\gamma} + \frac{\max\{L_{yx},L_{xy}\}}{\sqrt{\gamma\mu_y}} + \frac{\max\{L_{yy},L_{yx}\}}{\mu_y}
+ \Big(
\frac{\delta_x^2}{\gamma} +
\frac{\delta_y^2}{\mu_y} \Big)
\frac{\gamma}{\epsilon^2}\right)\cdot\ln\Big( \frac{\max\{\gamma,\mu_y\}}{\epsilon}\Big)
\end{equation*}}%
Thus, considering the setting in \eqref{eq:uniform setting}, one can compute $x_\epsilon$ in practice requiring 
$T \tilde{N} = \cO\left(\frac{L\kappa_y\cG(x_0^0,y_0^0) }{\epsilon^2}\ln(1/\epsilon) + \frac{ L\kappa_y\delta^2\cG(x_0^0,y_0^0) }{\epsilon^4}\ln(1/\epsilon) \right)$ oracle calls; furthermore, $\ln(1/\epsilon)$ can be  removed by employing a restarting strategy as in~\cite{zhang2021robust}.}
\end{proof}
\section{Proof of Theorem~\ref{thm:unbounded-domains}}
\label{sec:Thm3_proof}
\sa{For completeness,} we provide a technical lemma below establishing Lipschitz continuity of the best response functions (see also~\cite[Lemma 2.5]{zhang2021robust} and \cite[Lemma B.2(a)]{lin-near-optimal}).

\begin{lemma}\sa{\cite[Proposition 1]{chen2021proximal}}
\label{lem:implicit_function}
\sa{Suppose Assumptions~\ref{ASPT: fg} and~\ref{ASPT: lipshiz gradient} hold.} {For any given $y\in\dom g$}, let $x^t_*(y)\triangleq\argmin_{x\in\cX}\cL^{t}(x,y)$; and {for any given $x\in\dom f$}, let $y_*(x)\triangleq\argmax_{y\in\cY}\cL^{t}(x,y)=\argmax_{y\in\cY}\cL(x,y)$. Then $x^t_*(\cdot)$ and \nsa{$y_*(\cdot)$} are Lipschitz maps on {$\dom g$ and $\dom f$}, with constants 
\sa{$\kappa_{xy}$ and $\kappa_{yx}$}, respectively, where $\kappa_{xy}\triangleq L_{xy}/\mu_x$ and $\kappa_{yx}\triangleq L_{yx}/\mu_y$.
\end{lemma}

\begin{lemma}\label{lemma: bounded Gap and disance}
 \sa{For any $t\geq 0$, let $z^t_*\triangleq (x^t_*,y^t_*)$ be the unique saddle point of $\cL^{t}$ defined in \cref{eq: SCSC problem chpt2}, and let $\{z_k^t\}_{k=0}^{N_t}$ be generated by running \texttt{SAPD} on $\min_{x\in\cX}\max_{y\in\cY}\cL^{t}(x,y)$ for $N_t\in\integers_+$ iterations, where $z^t_k\triangleq (x^t_k,y^t_k)$; and define $z_0^{t+1} \triangleq \tfrac{1}{ N_t }\sum_{i=0}^{N_t-1}z^t_{i+1}$.} Under the setting of \cref{lemma:bound-gap}, 
{\begin{equation}\label{eq: bounded Gap and disance}
     \max\Big\{\mE\left[\cG^{t}(z_0^{t+1})\right],~\mE\left[\norm{z_0^{t+1}-z_*^t}^2\right]\Big\}\leq \frac{1}{ N_t } C_{\tau,\sigma,\theta} \mE[\norm{z_0^{t}-z_*^t}^2] + C'_{\tau,\sigma,\theta}
\end{equation}}%
holds for all {$t\geq 0$ and $N_t\geq 1$}, for some \sa{positive constants} $C_{\tau,\sigma,\theta}$ and $C'_{\tau,\sigma,\theta}$.
\end{lemma}
\begin{proof}
For simplicity we assume $N_t=N$ for all $t\geq 0$ --the proof still holds for arbitrary $\{N_t\}_{t\geq 0}\subset\integers_+$. The proof mainly follows the proof of \cref{lemma:bound-gap}. 
\sa{We first show a bound for 
$\mE\left[\norm{z_0^{t+1}-z_*^t}^2\right]$ that is in the form of the rhs of \cref{eq: bounded Gap and disance};
then, 
we show it for $\mE\left[\cG^{t}(z_0^{t+1})\right]$. In addition, given $\{z_k^t\}^{N_t}_{k=0}$, we let $\bar{z}_{N_t}^t=(\bar{x}_{N_t}^t,\bar{y}_{N_t}^t)$, and $\bar{z}_{N_t}^t=z_0^{t+1}=\frac{1}{ N_t }\sum_{i=0}^{N_t-1}z^{t}_{i+1}$ for all $t\geq 0$ and $N_t\geq 1$.}

\sa{Now, we start with analyzing $\mE\left[\norm{z_0^{t+1}-z_*^t}^2\right]$. 
The analysis below mainly relies on the proof of \cref{lemma:bound-gap}. 
Indeed,} given $z^t_0=(x^t_0,y^t_0)$ for $t\geq0$, substituting $x=x^{t}_*$ and $y = y^{t}_*$ within \eqref{INEQ: ergordic gap-gap} and then using \cref{eq:U-bound-gap}, we obtain that
{\small
\begin{equation}
    \begin{aligned}
        \label{eq: custom gap bound}
\lefteqn{N  \mathbb{E}\big[\cL^{t}(\bar{x}^t_{N}, y^t_*)  - \cL^{t}(x^t_*, \bar{y}^t_{N}) \big] }\\
&\leq 
      \mathbb{E}\left[\left(\frac{1}{2\tau}+\frac{\eta_x}{2}\right)\|x^{t}_*-x^{t}_{0}\|^2
      +\left(\frac{1}{2\sigma}+\frac{\eta_y(1+2\theta)}{2}\right)\|y^t_{*}-y^{t}_{0}\|^2 
       +\sum_{k=0}^{N-1}{\rho}^{-k}(\tilde{P}^{t}_{k}+ \tilde{Q}^{t}_{k})\right]. 
    \end{aligned}
\end{equation}}%
Moreover, since $\cL^t(\cdot,y^t_*)$ is $\mu_x$-strongly convex and  $\cL^t(x^t_*,\cdot)$ is $\mu_y$-strongly concave, and $(x^t_*,y^t_*)$ is the unique saddle point of $\cL^t$, we have that
\begin{equation}
\label{eq:distance to gap}
    \frac{\mu_x}{2}\|\bar{x}^t_{N}-x^t_*\|^2 + \frac{\mu_y}{2}\|\bar{y}^t_{N}-y^t_*\|^2 \leq \cL^{t}(\bar{x}^t_{N}, y^t_*)  - \cL^{t}(x^t_*, \bar{y}^t_{N}) .
\end{equation}

If we let $\eta_x = \frac{1}{\tau}$ and $\eta_y = \frac{1}{\sigma}$, \sa{then
it follows from eqs.~(\ref{eq: custom gap bound}, \ref{eq:distance to gap},\ref{eq:variance-bound-gap}) and the fact that $\bar{z}^t_N=z^{t+1}_0$ that}
\begin{equation}
N\mathbb{E}\left[ \frac{\mu_x}{2}\|x^{t+1}_{0}-x^t_*\|^2 + \frac{\mu_y}{2}\|y^{t+1}_{0}-y^t_*\|^2\right] \leq \mathbb{E}\left[\overline{U}^t(x^t_*,y^t_*)\right] + \sa{N{\Xi}_{\tau,\sigma,\theta}},
\end{equation}
where \sa{${\Xi_{\tau,\sigma,\theta}}$} is defined in \cref{lemma:bound-gap} and \sa{for any $(x,y)\in\cX\times\cY$, we define}
\begin{align}
\label{eq:bar_U_t}
    \overline{U}^t(x,y) \triangleq \frac{1}{\tau}\|x-x^{t}_{0}\|^2
+\frac{1+\theta}{\sigma}\|y-y^{t}_{0}\|^2.
\end{align}

Therefore, we conclude that
\begin{equation}\label{eq: distance metric bound}
    \mathbb{E}\left[\|z^{t+1}_0 -z^t_*\|^2 \right]\leq \frac{1}{N}\overline{C}_{\tau,\sigma,\theta}\mathbb{E}\left[\|z^t_0-z^t_*\|^2\right] + \overline{C}'_{\tau,\sigma,\theta},
\end{equation}
where 
$$
\overline{C}_{\tau,\sigma,\theta} \triangleq \frac{2\max\{\frac{1}{\tau},\frac{1+\theta}{\sigma} \}}{\min\{\mu_x,\mu_y\}},\qquad \overline{C}'_{\tau,\sigma,\theta} \triangleq \frac{2}{\min\{\mu_x, \mu_y\}}\Xi_{\tau,\sigma,\theta}.
$$
\sa{This completes the first part of the proof. Next, we will bound $\mE[\cG^{t}(z^{t+1}_0)]$ using the bound on $\mE[\|z^{t+1}_{0}-z^{t}_{*}\|^2]$ we derived in the first part.}

Given $z^t_0$, using \cref{eq:gap_bound_fg_2-gap} and \cref{eq:variance-bound-gap} in the proof \cref{lemma:bound-gap} for $\eta_x = \frac{1}{\tau}$ and $\eta_y = \frac{1}{\sigma}$ as above, we obtain that
\begin{equation}\label{eq: simplified gap bound}
    \mE\left[\cG^{t}(z_0^{t+1})\right]
\leq \frac{1}{ N }\mathbb{E}\left[\overline{U}^t\Big(x^{t}_*(y^{t+1}_0), y_{*}(x^{t+1}_0)\Big)\right]
      +\Xi_{\tau,\sigma,\theta},
\end{equation}
where $\overline{U}^t(x,y)$ is defined in~\eqref{eq:bar_U_t} and \sa{${\Xi_{\tau,\sigma,\theta}}$} is defined in \cref{lemma:bound-gap}; furhermore, $x^t_*(\cdot)$ and $y_*(\cdot)$ are defined in \cref{eq: x* y*}. Next, we will use \cref{eq: distance metric bound} to \sa{derive} an upper bound for the right hand side of \cref{eq: simplified gap bound}.

\sa{Since $z_*^t$ is the unique saddle point for $\cL^t$, we have $x^t_*(y^t_*)=x^t_*$ and $y_*(x^t_*)=y^t_*$. Moreover,} according to 
\cref{lem:implicit_function}, $x^t_*(\cdot)$, $y_*(\cdot)$ is Lipschitz with \sa{constants} ${\kappa_{xy}=\frac{L_{xy}}{\mu_x}}$ and ${\kappa_{yx}=\frac{L_{yx}}{\mu_y}}$, respectively. 
Therefore, \cref{lem:implicit_function} and 
$$
\overline{U}^t(x^{t}_*(y^{t+1}_0), y_{*}(x^{t+1}_0)) \leq \frac{2}{\tau}\|x^t_*-x^{t}_*(y^{t+1}_0)\|^2
      +\frac{2+2\theta}{\sigma}\|y^t_*-y_{*}(x^{t+1}_0)\|^2 + 2\overline{U}^t(x^t_*,y^t_*),\quad \text{w.p.}~1,
$$
together imply that 
{\begin{equation*}
    \begin{aligned}
            \mathbb{E}& \left[\overline{U}^t(x^{t}_*(y^{t+1}_0), y_{*}(x^{t+1}_0))\right] 
            \\
            \leq&  \mathbb{E}\left[\frac{2\kappa_{xy}^2}{\tau}\|y^t_*-y^{t+1}_0\|^2
      +\frac{(2+2\theta)  \kappa_{yx}^2}{\sigma}\|x^t_*-x^{t+1}_0\|^2 + 2\overline{U}^t(x^t_*,y^t_*) \right]\\
      \leq &
      \mathbb{E}\left[ \max\Big\{
      \frac{2}{\tau},
      \frac{(2+2\theta)}{\sigma}
      \Big\} \Big(\max\{\kappa_{xy}^2,~\kappa_{yx}^2\}\|  z^{t+1}_0 -z^t_*\|^2  +\| z^{t}_0 -z^t_*\|^2 \Big)\right]
      \\
        \leq & \mathbb{E}\left[ \max\Big\{
      \frac{2}{\tau},
      \frac{(2+2\theta)}{\sigma}
      \Big\}\max\{\xzh{1},\kappa_{xy}^2,~\kappa_{yx}^2\} \left(\Big(\frac{1}{N}+1\Big)\overline{C}_{\tau,\sigma,\theta}\|z^t_0-z^t_*\|^2 + \overline{C}'_{\tau,\sigma,\theta}\right)\right],
    \end{aligned}
\end{equation*}}%
where we use \cref{eq: distance metric bound} for the last inequality. Then, if we use the above inequality within \cref{eq: simplified gap bound}, it follows that
{
\begin{equation*}
    \begin{aligned}
         \mE\left[\cG^{t}(z_0^{t+1})\right] & \leq\frac{1}{N} \overline{\overline{C}}_{\tau,\sigma,\theta}\mathbb{E}\left[\|z^t_0-z^t_*\|^2 \right]
          + \frac{1}{{ N }}\overline{\overline{C}}'_{\tau,\sigma,\theta}  +\Xi_{\tau,\sigma,\theta},
    \end{aligned}
\end{equation*}}%
where 
{\begin{align*}
&
\overline{\overline{C}}_{\tau,\sigma,\theta} \triangleq
    4~\max\Big\{
      \frac{1}{\tau},
      \frac{1+\theta}{\sigma}
      \Big\}\max\{\xzh{1},\kappa_{xy}^2,~\kappa_{yx}^2\} \overline{C}_{\tau,\sigma,\theta},
      \\
      & \overline{\overline{C}}'_{\tau,\sigma,\theta} \triangleq 2~\max\Big\{
      \frac{1}{\tau},
      \frac{1+\theta}{\sigma}
      \Big\}\max\{\xzh{1},\kappa_{xy}^2,~\kappa_{yx}^2\}\overline{C}'_{\tau,\sigma,\theta}.
\end{align*}
Thus, for $C_{\tau,\sigma,\theta} \triangleq \max\{\overline{C}_{\tau,\sigma,\theta},~ \overline{\overline{C}}_{\tau,\sigma,\theta}\}$ and $C'_{\tau,\sigma,\theta} \triangleq \max\{\overline{C}'_{\tau,\sigma,\theta},~\frac{1}{N}\overline{\overline{C}}'_{\tau,\sigma,\theta} + \Xi_{\tau,\sigma,\theta}\}$, we get the desired result in~\eqref{eq: bounded Gap and disance}.}
\end{proof}
\begin{lemma}
\label{lem:induction}
\sa{Under the premise of \cref{lemma:bound-gap}, $\mE[\norm{z_0^{t}-z_*^t}^2]$, $\mE\left[\cG^{t}(z_0^{t})\right]$ and $\mE\left[\cG^{t}(z_0^{t+1})\right]$ are finite for any $t\geq 0$ when either \cref{assump:compact} or \cref{assump:bounded-subdifferential} holds.} 
\end{lemma}
\begin{proof}
\sa{In~\cref{lemma: bounded Gap and disance}, we show that}
{\begin{align}
    \label{eq:induction-key}
    \max\Big\{\mE\left[\cG^{t}(z_0^{t+1})\right],~\mE\left[\norm{z_0^{t+1}-z_*^t}^2\right]\Big\}\leq \frac{1}{ N_t } C_{\tau,\sigma,\theta} \mE[\norm{z_0^{t}-z_*^t}^2] + C'_{\tau,\sigma,\theta},
\end{align}}%
for some $C_{\tau,\sigma,\theta},C'_{\tau,\sigma,\theta}\in\reals_+$ constants, dependent on the SAPD parameters. Next, we show that $\{\mE\left[\norm{z_0^t-z_*^t}^2\right]\}<\infty$ for all $t\geq 0$ by induction. This is trivially true for $t=0$, i.e., $\mE\left[\norm{z_0^0-z_*^0}^2\right]=\norm{z_0^0-z_*^0}^2<\infty$. Next, for some $t\geq 0$, suppose $\mE\left[\norm{z_0^t-z_*^t}^2\right]<\infty$, \eqref{eq:induction-key} implies that 
\begin{align}
    \mE\left[\norm{z_0^{t+1}-z_*^t}^2\right]<\infty.
    \label{eq:first-bound-induction}
\end{align}
\sa{The inductive assumption  $\mE\left[\norm{z_0^t-z_*^t}^2\right]<\infty$ and \eqref{eq:first-bound-induction} imply that 
\begin{align}
    \mE\left[\norm{z_0^{t+1}-z_0^t}^2\right]\leq 2\mE\left[\norm{z_0^{t+1}-z_*^t}^2\right]+2\mE\left[\norm{z_0^{t}-z_*^t}^2\right]<\infty.
    \label{eq:intermediate-induction}
\end{align}}%
For any $\mu_x>0$, fix $\lambda=(\mu_x+\gamma)^{-1}$; since we have $x_*^\ell=\prox{\lambda\phi}(x_0^\ell)$ for $\ell=t, t+1$ and $\prox{\lambda\phi}(\cdot)$ is non-expansive, we have
$\mE[\norm{x_*^{t+1}-x_*^t}^2]\leq \mE[\norm{x_0^{t+1}-x_0^t}^2]$. Moreover, 
\sa{
\cref{lem:implicit_function}  
implies that $\mE[\norm{y_*^{t+1}-y_*^t}^2]\leq \kappa_{yx}^2\mE[\norm{x_*^{t+1}-x_*^t}^2]$ for $\kappa_{yx}=\frac{L_{yx}}{\mu_y}$;} thus, \sa{using \eqref{eq:intermediate-induction}, we get}
\begin{align}
    \mE[\norm{z_*^{t+1}-z_*^t}^2]
\leq (\kappa_{yx}^2+1)\mE[\norm{x_0^{t+1}-x_0^t}^2]\sa{\leq(\kappa_{yx}^2+1)\mE[\norm{z_0^{t+1}-z_0^t}^2] <\infty.}
\label{eq:second-bound-induction}
\end{align} 
Therefore, we can conclude that $\mE[\norm{z_0^{t+1}-z_*^{t+1}}^2]\leq 2\mE[\norm{z_0^{t+1}-z_*^{t}}^2] + 2 \mE[\norm{z_*^{t+1}-z_*^{t}}^2]<\infty$, which follows from \eqref{eq:first-bound-induction} and \eqref{eq:second-bound-induction}. This completes induction, providing us with $\mE[\norm{z_0^{t}-z_*^t}^2]<\infty$ for all $t\geq 0$. Note that using this result together with the definition of $\cG^{t}$ and \eqref{eq:induction-key} implies that $0\leq \mE[\cG^{t}(z_0^{t+1})]<\infty$ for $t\geq 0$.


Next, we will argue that $\mE[\cG^{t}(z_0^{t})]<\infty$ for all $t\geq 0$ as well. Recall that $\cG^{t}(z_0^t)=\sup_{y\in\cY}\cL^{t}(x_0^t,y)-\inf_{x\in\cX}\cL^{t}(x,y_0^t)$; furthermore, note that $\cL^{t}(x_0^t,y)=\cL(x_0^t,y)$ for all $y\in\cY$, and given $z_0^t$, we have $\cL^{t}(\cdot,y_0^t)$ strongly convex with modulus $\mu_x$ and $\cL(x_0^t,\cdot)$ strongly concave with modulus $\mu_y$. Therefore, we have
{\small
\begin{align}
    \cL(x_0^t,y)&\leq \cL(x_0^t,y_0^t)+\fprod{\sa{\grad_y\Phi(x_0^t,y_0^t)-s_g(y_0^t)},~y-y_0^t}-\frac{\mu_y}{2}\norm{y-y_0^t}^2\nonumber\\
    &\leq \cL(x_0^t,y_0^t)+\frac{1}{2\mu_y}\norm{\sa{\grad_y\Phi(x_0^t,y_0^t)-s_g(y_0^t)}}^2,\label{eq:G-upper}\\
    \cL^{t}(x,y_0^t)&\geq \cL(x_0^t,y_0^t)+\fprod{\sa{\grad_x\Phi(x_0^t,y_0^t)+s_f(x_0^t)},~x-x_0^t}\sa{+}\frac{\mu_x}{2}\norm{x-x_0^t}^2\nonumber\\
    &\geq \cL(x_0^t,y_0^t)-\frac{1}{2\mu_x}\norm{\sa{\grad_x\Phi(x_0^t,y_0^t)+s_f(x_0^t)}}^2,\label{eq:G-lower}
\end{align}}%
where \sa{$s_f(x_0^t)\in\partial f(x_0^t)$ and $s_g(y_0^t)\in\partial g(y_0^t)$ such that $\norm{s_f(x_0^t)}\leq B_f$ and $\norm{s_g(y_0^t)}\leq B_g$ --see \cref{assump:bounded-subdifferential}; moreover,} we have used the fact that $\cL^{t}(x_0^t,y_0^t)=\cL(x_0^t,y_0^t)$ and \sa{$\partial_x\cL^{t}(x_0^t,y_0^t)=\grad_x\Phi(x_0^t,y_0^t)+\partial f(x_0^t)$}. Thus, \eqref{eq:G-upper} and \eqref{eq:G-lower} imply that
\begin{align*}
    \cG^{t}(z_0^t)
    &=\sup_{\substack{x\in\cX, y\in\cY}}\{\cL(x_0^t,y)-\cL^{t}(x,y_0^t)\}\leq \sa{\left(\norm{\grad\Phi(z_0^t)}^2+\norm{s_f(x_0^t)}^2+\norm{s_g(y_0^t)}^2\right)}/\min\{\mu_x,\mu_y\}\\
    &\leq \sa{\frac{1}{\mu}\left(\norm{\grad\Phi(z_0^t)-\grad\Phi(z_0^0)}^2+\norm{\grad\Phi(z_0^0)}^2+B_f^2+B_g^2\right)}\\
    &\leq \frac{L}{\mu}\norm{z_0^t-z_0^0}^2+\sa{\frac{1}{\mu}\left(\norm{\grad\Phi(z_0^0)}^2+B_f^2+B_g^2\right)},
\end{align*}
where $L=\max\{L_{xx},L_{yy},L_{yx},L_{xy}\}$ and $\mu=\min\{\mu_x,\mu_y\}$. \sa{Finally,  \eqref{eq:intermediate-induction} implies that $\mathbb{E}[\norm{z_0^t-z_0^0}^2]<\infty$; therefore, we can conclude that $\mathbb{E}[\cG^{t}(z_0^t)]<\infty$ for all $t\geq 0$.}
\end{proof}
\sa{Thus, \cref{lem:induction} implies that the analysis given in \cref{sec:pre_analysis_thm1} directly goes through if we replace \cref{assump:compact} with \cref{assump:bounded-subdifferential}, which does not require compactness of the problem domain.}
\section{\sa{Proof of Theorem~\ref{cor:complexity VR} and preliminary technical results}}
\label{sec:Thm4_proof}
\nsa{The general proof structure of Theorem~\ref{cor:complexity VR} is the same with \cref{cor:complexity-fg}'s. The main difference is the way we bound the variance, which is given in \cref{Lemma: SPIDER}.}
\label{sec:VR-proofs}
\subsection{Construction for the iteration complexity result}
\begin{lemma}\label{lemma: Gap bound with fg SPIDER}
Suppose Assumptions~\ref{ASPT: fg}, \ref{ASPT: unbiased noise assumption}, \ref{ASPT: lipshiz gradient VR} and \ref{ASPT: independence} hold.
\sa{Given $\{N_t\}_{t\geq 0}\subset\integers_+$, let $\{x^{t}_{0},y^{t}_{0}\}_{t\geq 0}$ be generated by \xz{\texttt{SAPD+}}, stated in~Algorithm~\ref{Alg: SAPD-plus}, when \texttt{VR-flag}=\textbf{true}, initialized from {$(x_0^0,y_0^0)\in\dom f\times\dom g$}
and using} $\tau,\sigma,\theta,\mu_x>0$ that satisfy 
{
\begin{equation}
    \label{eq: SCSC SAPD LMI SPIDER}
    \sa{G-\diag(g)\succeq 0,}
\end{equation}}%
{for some} $\alpha \in [0,  {\tfrac{1}{\sigma})}$, $\rho \in (0,1]$ and \sa{$\pi_x,\pi_y>0$,} \sa{where $G$ is defined in~\eqref{eq: SCSC SAPD LMI-dist}, $g\triangleq [\pi_x,\pi_y,L_x',L_y',0]^\top$} and 
\begin{align*}
L_x' \triangleq \sa{c(\rho)}~\big( \frac{ \xz{{L'_{xx}}^2}}{\xz{\pi_x}\xzrev{b'_{x}}} + \frac{\xz{2}(1+2\theta+2\theta^2)\rho^{-1} L^2_{yx}}{\xz{\pi_y}{\xzrev{b'_{y}}}}\big), \; 
L_y' \triangleq c(\rho)~\big(\frac{\rho L^2_{xy}}{\xz{\pi_x}\xzrev{b'_{x}}} +  \frac{\xz{2}(1+2\theta+2\theta^2)\rho^{-1} L^2_{yy}}{\xz{\pi_y}{\xzrev{b'_{y}}}}\big),
\end{align*}
such that $c(\rho)=\frac{2}{1-\rho}(\rho^{-q+1}-1)$ for $\rho\in(0,1)$ and $c(\rho)=2(q-1)$ for $\rho=1$, \nsa{where $L'_{xx}\triangleq L_{xx}+\mu_x+\gamma$}.
 Then for all $t\geq {0}$, \sa{it holds that}
\begin{equation}\label{eq:Gap bound with fg-VR}
{\small
\begin{aligned}
  \mathbb{E}\left[\cG^{t}(x^{t+1}_{0}, y^{t+1}_{0})\right] \leq &\frac{1}{K_{N_t}(\rho)}\left[
M^{\texttt{VR}}\Big(\frac{\mu_x}{4}\ED{x^{t}_{*}(y^{t+1}_{0}) - x^{t}_{0}} + \frac{\mu_y}{4}\ED{y_*(x^{t+1}_{0}) - y^{t}_{0}} \Big)+\xzf{\Xi^{\texttt{VR}}(b_0)}
\right]\\ 
&+ \xzf{\frac{K_{N_t}(\rho)-1}{K_{N_t}(\rho)}}\Xi^{\texttt{VR}}\xzf{(b)},
\end{aligned}}%
\end{equation}
where $K_{N_t}(\rho) = \sum_{k=0}^{N_t-1}\rho^{-k}$, $\Xi^{\texttt{VR}}\xzf(b)\triangleq \frac{\delta^2_x}{2\pi_x b}+(1+2\theta+2\theta^2)\frac{\delta^2_y}{\pi_y b}$ and $M^{\texttt{VR}}\triangleq\max\{\frac{2}{\mu_x}(\frac{1}{\tau}-\mu_x),\frac{2}{\mu_y\sigma}\}$.
\end{lemma}
\begin{proof}
For easier readability, we provide the proof in a separate subsection, see~\cref{sec:spider-proof}. 
\end{proof}

\begin{theorem}
\label{thm:WC_SP SPIDER}
\sa{Under the premise of Lemma~\ref{lemma: Gap bound with fg SPIDER}, given an arbitrary $\zeta>0$ and $T\in\integers_+$, suppose $N_t=N$ for all $t=0,\ldots T$ for some $N\in\integers_+$ such that $ N  \geq (1+\zeta) M^{\texttt{VR}}$, and \eqref{eq: additional stepsize condition WCSC}
has a solution for \xzf{$M_{\tau,\sigma,\theta}$ replaced by $M^{\texttt{VR}}$} and some $\beta_{1},\beta_{2}\in(0,1)$ and $p_1,p_2,p_3>0$ such that $p_1+p_2+p_3=1$. If either \cref{assump:compact} or \cref{assump:bounded-subdifferential} holds, then \eqref{eq: sketch convergence result} holds with ${\Xi}_{\tau,\sigma,\theta}=\xzf{\frac{N-1}{N}\Xi^{\texttt{VR}}(b) + \frac{1}{N}\Xi^{\texttt{VR}}(b_0)}$ for $\lambda=(\gamma+\mu_x)^{-1}$ and for all $T\geq 1$.}
\end{theorem}
\begin{proof}
\sa{The proof is omitted as it is essentially the same with the proof of \cref{thm:WC_SP}.}
\end{proof}

\subsection{Proof of Lemma~\ref{lemma: Gap bound with fg SPIDER} and preliminary technical results}
\label{sec:spider-proof}
\sa{In this section we prove \cref{lemma: Gap bound with fg SPIDER}.} We first state a technical lemma that will be used in our analysis.
\begin{lemma}\label{Lemma: SPIDER}
Suppose Assumptions~\ref{ASPT: fg}, \ref{ASPT: unbiased noise assumption},  \ref{ASPT: lipshiz gradient VR} and \ref{ASPT: independence} hold. Let $\{x_k^t,y_k^t\}_{k\geq 0}$ be \xz{\texttt{VR-SAPD}} iterates generated according to Algorithm~\ref{Alg: SAPD-VR} for solving $\min_x\max_y\cL^{t}(x,y)$.
\sa{Then, 
}
{\small
\begin{subequations}
\label{eq:variance-bound-gap-VR}
    \begin{align}
        &\ED{v_{k}^{t}- \nabla_x\Phi^t(x^{t}_{k},y^{t}_{k+1})}\leq \frac{\delta^2_x}{b}+ \sum_{i=(n_k-1)q+1}^{k}\frac{2\xz{{L'_{xx}}^2}}{\xz{\xzrev{b'_{x}}}}\ED{x^{t}_{i}-x^{t}_{i-1}}+ 
        \frac{2L^2_{xy}}{\xz{\xzrev{b'_{x}}}}\ED{y^{t}_{i+1}-y^{t}_{i}},
        \\
        & \ED{w_{k}^{t}- \nabla_y\Phi^t(x^{t}_{k},y^{t}_{k})}\leq  \frac{\delta^2_y}{b}+\sum_{i=(n_k-1)q+1}^{k}\frac{2L^2_{yx}}{\xz{{\xzrev{b'_{y}}}}}\ED{x^{t}_{i}-x^{t}_{i-1}}+ 
        \frac{2L^2_{yy}}{\xz{{\xzrev{b'_{y}}}}}\ED{y^{t}_{i}-y^{t}_{i-1}},\label{ineq: basic for wk}
    \end{align}
\end{subequations}}%
\na{for all $k>0$ such that $\mo(k,q)\neq 0$, 
where $n_k \triangleq \lceil k/q \rceil$; moreover, if $\mo(k,q)=0$ for any $k>0$,} then
\begin{equation}\label{eq:variance-bound-gap-VR-base}
        \ED{v_{k}^{t}- \nabla_x\Phi^t(x^{t}_{k},y^{t}_{k+1})}\leq   \frac{\delta^2_x}{b},\quad
         \ED{w_{k}^{t}- \nabla_y\Phi^t(x^{t}_{k},y^{t}_{k})}\leq \frac{\delta^2_y}{b}.
\end{equation}
\na{Finally, for $k=0$, 
\cref{eq:variance-bound-gap-VR} and \cref{eq:variance-bound-gap-VR-base} continue to hold after replacing $b$ with $b_0$.}
\end{lemma}
\begin{proof}
Recall that $\tilde\nabla_x\Phi_{I^x_k}^t(x^{t}_{k},y^{t}_{k+1}) \sa{\triangleq} \frac{1}{|I^x_k|} \sum_{\omega_{k}^{x,i}\in  I^x_k} \tilde\nabla_x\Phi^t(x^{t}_{k},y^{t}_{k+1};\omega_{k}^{x,i})$, where
$I_{k}^{x}=\{\omega_{k}^{x,i}\}_{i=1}^{\xzrev{b'_{x}}}$ is \sa{a randomly generated batch with $|I_{k}^{x}|=\xzrev{b'_{x}}$ independent elements which are also independent of $(x^t_{k-1},y^t_{k})$ and $(x^t_k,y^t_{k+1})$.}
According to the definition of \xz{$v_k$} in Algorithm~\ref{Alg: SAPD-VR}, for $\mo(k,q)>0$,
\begin{equation}
    v_{k}^{t}  =  v_{k-1}^{t} + \tilde\nabla_x\Phi_{I_{k}^{x}}^t(x^{t}_{k},y^{t}_{k+1})-\tilde\nabla_x\Phi_{I_{k}^{x}}^t(x^{t}_{k-1},y^{t}_{k}).
\end{equation}
\sa{Therefore,}
{\small
\begin{equation}\label{ineq: basic for vk}
    \begin{aligned}
        &\ED{v_{k}^{t}- \nabla_x\Phi^t(x^{t}_{k},y^{t}_{k+1})}
        \\
        &=  \ED{v_{k-1}^{t} + \tilde\nabla_x\Phi_{I_{k}^{x}}^t(x^{t}_{k},y^{t}_{k+1})-\tilde\nabla_x\Phi_{I_{k}^{x}}^t(x^{t}_{k-1},y^{t}_{k})- \nabla_x\Phi^t(x^{t}_{k},y^{t}_{k+1})} 
        \\
        &=  \ED{v_{k-1}^{t}- \nabla_x\Phi^t(x^{t}_{k-1},y^{t}_{k})  + \nabla_x\Phi^t(x^{t}_{k-1},y^{t}_{k}) - \tilde\nabla_x\Phi_{I_{k}^{x}}^t(x^{t}_{k-1},y^{t}_{k})+ \tilde\nabla_x\Phi_{I_{k}^{x}}^t(x^{t}_{k},y^{t}_{k+1})- \nabla_x\Phi^t(x^{t}_{k},y^{t}_{k+1})}
        \\
        &=  \ED{v_{k-1}^{t}- \nabla_x\Phi^t(x^{t}_{k-1},y^{t}_{k})}\\  
        &\quad+\ED{ \nabla_x\Phi^t(x^{t}_{k-1},y^{t}_{k}) - \tilde\nabla_x\Phi_{I_{k}^{x}}^t(x^{t}_{k-1},y^{t}_{k})+ \tilde\nabla_x\Phi_{I_{k}^{x}}^t(x^{t}_{k},y^{t}_{k+1})- \nabla_x\Phi^t(x^{t}_{k},y^{t}_{k+1})},
    \end{aligned}
\end{equation}}%
where for the last equality we used
{\small$$\mathbb{E}\left[\nabla_x\Phi^t(x^{t}_{k-1},y^{t}_{k}) - \tilde\nabla_x\Phi_{I_{k}^{x}}^t(x^{t}_{k-1},y^{t}_{k})+ \tilde\nabla_x\Phi_{I_{k}^{x}}^t(x^{t}_{k},y^{t}_{k+1})- \nabla_x\Phi^t(x^{t}_{k},y^{t}_{k+1})\right]=0.$$}%
Next, we bound the second expectation on the rhs of \eqref{ineq: basic for vk}. It follows that
{\small
\begin{equation}\label{ineq: basic for vk 2}
    \begin{aligned}
           & \ED{ \nabla_x\Phi^t(x^{t}_{k-1},y^{t}_{k}) - \tilde\nabla_x\Phi_{I_{k}^{x}}^t(x^{t}_{k-1},y^{t}_{k})+ \tilde\nabla_x\Phi_{I_{k}^{x}}^t(x^{t}_{k},y^{t}_{k+1})- \nabla_x\Phi^t(x^{t}_{k},y^{t}_{k+1})} 
           \\
           = & \frac{1}{{\xzrev{b'_{x}}}^2}\mathbb{E}\left[ 
           \|
           \sum_{i=1}^{\xzrev{b'_{x}}}\big( 
           \tilde\nabla_x\Phi^t(x^{t}_{k},y^{t}_{k+1};\omega_{k}^{x,i})- \tilde\nabla_x\Phi^t(x^{t}_{k-1},y^{t}_{k};\omega_{k}^{x,i}) -\nabla_x\Phi^t(x^{t}_{k},y^{t}_{k+1})
           + \nabla_x\Phi^t(x^{t}_{k-1},y^{t}_{k})
           \big)
           \|^2
           \right]
           \\
           = & \frac{1}{{\xzrev{b'_{x}}}^2}\sum_{i=1}^{{\xzrev{b'_{x}}}}\mathbb{E}\left[ 
           \|
           \tilde\nabla_x\Phi^t(x^{t}_{k},y^{t}_{k+1};\omega_{k}^{x,i})- \tilde\nabla_x\Phi^t(x^{t}_{k-1},y^{t}_{k};\omega_{k}^{x,i}) -\nabla_x\Phi^t(x^{t}_{k},y^{t}_{k+1})
           + \nabla_x\Phi^t(x^{t}_{k-1},y^{t}_{k})
           \|^2
           \right]      
           \\
           \leq & \frac{1}{{\xzrev{b'_{x}}}^2}\sum_{i=1}^{{\xzrev{b'_{x}}}}\big(\mathbb{E}\left[ 
           \|
           \tilde\nabla_x\Phi^t(x^{t}_{k},y^{t}_{k+1};\omega_{k}^{x,i})- \tilde\nabla_x\Phi^t(x^{t}_{k-1},y^{t}_{k};\omega_{k}^{x,i}) \|^2\right]
           \big)
           \\
            \leq &  \frac{2\xz{{L'_{xx}}^2}}{{\xzrev{b'_{x}}}}\ED{x^{t}_{k}-x^{t}_{k-1}}+\frac{2L^2_{xy}}{{\xzrev{b'_{x}}}}\ED{y^{t}_{k+1}-y^{t}_{k}},
     \end{aligned}
\end{equation}}%
where the second equality \sa{follows from the stochastic oracle being unbiased --see \cref{ASPT: unbiased noise assumption}, which implies $$\mathbb{E}\left[\nabla_x\Phi^t(x^{t}_{k-1},y^{t}_{k}) - \tilde\nabla_x\Phi^t(x^{t}_{k-1},y^{t}_{k};\omega_{k}^{x,i})+ \tilde\nabla_x\Phi^t(x^{t}_{k},y^{t}_{k+1};\omega_{k}^{x,i})- \nabla_x\Phi^t(x^{t}_{k},y^{t}_{k+1})\right]=0,$$ 
for all $i=1,\ldots,{\xzrev{b'_{x}}}$ and $\{\omega_i^k\}_{i=1}^{{\xzrev{b'_{x}}}}$ being independent}; the first inequality is because $\ED{\zeta-\mathbb{E}[\zeta]}\leq \ED{\zeta}$ for \sa{any} given random variable $\zeta$ with finite second order moment \sa{--we invoke this inequality for} $\zeta = \tilde\nabla_x\Phi^t(x^{t}_{k},y^{t}_{k+1};\omega_{k}^{x,i})-\tilde\nabla_x\Phi^t(x^{t}_{k-1},y^{t}_{k};\omega_{k}^{x,i})$; \sa{and finally, the last inequality follows from \cref{ASPT: lipshiz gradient VR} and the inequality $(a+b)^2\leq 2a^2+2b^2$ for any $a,b\in\reals$.}
Next, if we combine \cref{ineq: basic for vk} and \cref{ineq: basic for vk 2}, we get
\begin{equation*}
    \begin{aligned}
            &\ED{v_{k}^{t}- \nabla_x\Phi^t(x^{t}_{k},y^{t}_{k+1})}
            \\
          \leq & \ED{v_{k-1}^{t}- \nabla_x\Phi^t(x^{t}_{k-1},y^{t}_{k})} +\frac{2\xz{{L'_{xx}}^2}}{{\xzrev{b'_{x}}}}\ED{x^{t}_{k}-x^{t}_{k-1}}+\frac{2L^2_{xy}}{{\xzrev{b'_{x}}}}\ED{y^{t}_{k+1}-y^{t}_{k}}.
    \end{aligned}
\end{equation*}
\sa{Hence, if we sum the above inequality from $(n_k-1)q+1$ to $k$, we get a telescoping sum:}
\begin{equation}\label{ineq: basic for vk 3}
    \begin{aligned}
        &\ED{v_{k}^{t}- \nabla_x\Phi^t(x^{t}_{k},y^{t}_{k+1})}
        \\
        & \leq    \sum_{i=(n_k-1)q+1}^{k}\frac{2\xz{{L'_{xx}}^2}}{\xz{{\xzrev{b'_{x}}}}}\ED{x^{t}_{i}-x^{t}_{i-1}}+ \sum_{i=(n_k-1)q+1}^{k}\frac{2L^2_{xy}}{\xz{{\xzrev{b'_{x}}}}}\ED{y^{t}_{i+1}-y^{t}_{i}} 
        \\ 
        & +
        \ED{v_{(n_k-1)q}-\nabla_x\Phi^t(x^{t}_{(n_k-1)q}, y^{t}_{(n_k-1)q+1})}
        \\
        & \leq  \sum_{i=(n_k-1)q+1}^{k}\frac{2\xz{{L'_{xx}}^2}}{\xz{{\xzrev{b'_{x}}}}}\ED{x^{t}_{i}-x^{t}_{i-1}}+ \sum_{i=(n_k-1)q+1}^{k}\frac{2L^2_{xy}}{\xz{{\xzrev{b'_{x}}}}}\ED{y^{t}_{i+1}-y^{t}_{i}} + \frac{\delta^2_x}{b},
    \end{aligned}
\end{equation}
where \sa{the last inequality follows from \cref{ASPT: unbiased noise assumption} since $\mo((n_k-1) q,q)=0$ and for $\ell\in\integers_+$ such that $\mo(\ell,q)=0$, we have $v_{\ell} = \tilde\nabla_x\Phi^{t}_{\xzrev{\mathcal{B}^{x}_{\ell}}}(x^{t}_{\ell},y^{t}_{\ell+1})=\frac{1}{|\xzrev{\mathcal{B}^{x}_{\ell}}|} \sum_{\xzrev{\omega^{x,i}_{\ell}}\in  \xzrev{\mathcal{B}^{x}_{\ell}}} \tilde\nabla_x\Phi^t(x^{t}_{\ell},y^{t}_{\ell+1};\xzrev{\omega^{x,i}_{\ell}})$, where $\xzrev{\mathcal{B}^{x}_{\ell}}=\{\xzrev{\omega^{x,i}_{\ell}}\}$ is a randomly generated batch with $|\xzrev{\mathcal{B}^{x}_{\ell}}|=b$ independent elements which are also independent of $(x^t_{\ell},y^t_{\ell+1})$. This completes the proof of the case for $k$ such that $\mo(k,q)>0$.}%

When $\mo(k,q) = 0$, it follows from Algorithm~\ref{Alg: SAPD-VR} that $v_{k} = \tilde\nabla_x\Phi^{t}_{\mathcal{B}_k}(x^{t}_{k},y^{t}_{k+1})$. 
\sa{Hence, above discussion yields}
\begin{equation}
    \ED{v_{k}^{t}- \nabla_x\Phi^t(x^{t}_{k},y^{t}_{k+1})} \leq \frac{\delta^2_x}{b}.
\end{equation}
\sa{Finally, the second inequality in~\eqref{ineq: basic for wk} can be shown similarly.}
\end{proof}

Next, we will modify \cref{lem: basic lemma for SCSC} for \xz{\texttt{VR-SAPD}}, stated in Algorithm~\ref{Alg: SAPD-VR}. 
Specifically, instead of using the stochastic oracles $\tilde{\nabla}_x\Phi^t(x^{t}_{k},y^{t}_{k+1};\omega_k^x)$ and $\tilde{\nabla}_y\Phi^t(x^{t}_{k},y^{t}_{k};\omega_k^y)$ as in \cref{lem: basic lemma for SCSC}, we adopt $v_{k}^{t}$ and $w_{k}^{t}$ to estimate ${\nabla}_x\Phi^t(x^{t}_{k},y^{t}_{k+1})$ and ${\nabla}_y\Phi^t(x^{t}_{k},y^{t}_{k})$, respectively. 

\begin{lemma}\label{lem: basic lemma for SCSC VR}
\xz{Suppose Assumptions~\ref{ASPT: fg}, \ref{ASPT: unbiased noise assumption},  and \ref{ASPT: lipshiz gradient VR} hold.} Let $\{x_k^t,y_k^t\}_{k\geq 0}$ be \xz{\texttt{VR-SAPD}} iterates generated according to Algorithm~\ref{Alg: SAPD-VR} \sa{for solving $\min_x\max_y\cL^{t}(x,y)$}. 
Then 
for all $x\in\dom f\subset \cX$, $y\in\dom g\subset \cY$, and $k\geq 0$,
{\small
\begin{equation}\label{D1 VR}
    \begin{aligned}
       \cL^{t}( & x^{t}_{k+1}, y)  - \cL^{t}(x, y^{t}_{k+1})
        \\
    \leq
    &-\langle q^{t}_{k+1}, y^{t}_{k+1} - y \rangle + \theta \langle q^{t}_{k}, y^{t}_{k} - y \rangle  
    + \Lambda^{t}_k(x,y) - \Sigma^{t}_{k+1}(x,y)+ \Gamma^{t}_{k+1}+\sa{\varepsilon^{t,x}_{k}(x)+\varepsilon^{t,y}_{k}(y)},
\end{aligned}
\end{equation}}%
where $\sa{\varepsilon^{t,x}_k(x)}\triangleq 
       \langle v_{k}^{t} - \nabla_x \Phi^{t}(x^{t}_{k}, y^{t}_{k+1}),~x-x^{t}_{k+1} \rangle$ and $\sa{\varepsilon^{t,y}_{k}(y)}\triangleq\langle \tilde{s}^{t}_k -s^{t}_{k}, y^{t}_{k+1} - y \rangle$ for \sa{$\tilde{s}_k^t = (1+\theta) w_{k}^{t} - \theta w_{k-1}^t$} as defined in Algorithm~\ref{Alg: SAPD-VR},
       $q^{t}_{k}$ and $s^{t}_{k}$ are defined as in \eqref{eq:qksk}, and $\Lambda^{t}_{k}(x,y)$, $\Sigma^{t}_{k+1}(x,y)$, $\Gamma^{t}_{k+1}$ are the same with those in \cref{lem: basic lemma for SCSC}.
\end{lemma}
\begin{proof}
\sa{The proof uses the same arguments as the proof of \cref{lem: basic lemma for SCSC}. One only needs to replace $\tilde{\nabla}_x\Phi^t(x^{t}_{k},y^{t}_{k+1})$ and $\tilde{\nabla}_y\Phi^t(x^{t}_{k},y^{t}_{k})$ in the proof of \cref{lem: basic lemma for SCSC} with $v_{k}^{t},w_{k}^{t}$, respectively.}
\end{proof}

\subsection{Proof of Lemma~\ref{lemma: Gap bound with fg SPIDER}}\label{pf:gap-lemma SPIDER}
\begin{proof}
\sa{For simplifying the notation, let $N_t=N$ for some $N\in\integers_+$.}
\xzh{For arbitrary} $(x,y)\in\dom f\times \dom g$, 
 since $(x^{t}_{k+1},y^{t}_{k+1})\in\dom f\times \dom g$, using the concavity 
of $\cL^{t}(x^{t}_{k+1}, \cdot)$
and the convexity
of $\cL^{t}(\cdot, y^{t}_{k+1})$, \cref{lem: basic lemma for SCSC VR} and Jensen's lemma immediately implies that
{\small
\begin{align}
\label{eq:jensen bounded version SPIDER}
  K_N(\rho)   \left(\cL^{t}(\bar{x}^{t}_{N}, y)  - \cL^{t}(x, \bar{y}^{t}_{N}) \right) 
            \leq 
            \sum_{k=0}^{N-1}\rho^{-k} \left( \cL^{t}(x^{t}_{k+1}, y) - \cL^{t}(x, y^{t}_{k+1}) \right),
            \;\sa{\forall}\rho\in \sa{(0,1]},
\end{align}}%
where $\bar{x}^{t}_N=\frac{1}{ K_N(\rho) }\sum_{k=0}^{N-1}\sa{\rho^{-k}}x^{t}_{k+1},~\bar{y}^{t}_N=\frac{1}{ K_N(\rho) }\sum_{k=0}^{N-1}\sa{\rho^{-k}}y^{t}_{k+1}$, and $K_N(\rho)=\sum_{i=0}^{N-1}\rho^{-k}$. Thus, if we multiply both sides of 
\eqref{D1 VR} by $\rho^{-k}$ and sum the resulting inequality from $k=0$ to $N-1$, then using \eqref{eq:jensen bounded version SPIDER} we get
{\footnotesize
\begin{equation}\label{INEQ: difference of gap SPIDER}
    \begin{aligned}
            \sa{K_N(\rho)} & \left(\cL^{t}(\bar{x}^{t}_{N}, x)  - \mathcal{L}(x, \bar{y}^{t}_{N}) \right)  \\
            \leq & 
         \sum_{k=0}^{N-1}\rho^{-k} \Big( 
                -\langle q^{t}_{k+1}, y^{t}_{k+1} - x \rangle + \theta \langle q^{t}_{k}, y^{t}_{k} - x \rangle + \Lambda^{t}_{k} (x,y) - \Sigma^{t}_{k+1} (x,y)+ \Gamma^{t}_{k+1} 
            \\
             & \quad -
           \langle  v_{k}^{t} - \nabla_x \Phi^{t}(x^{t}_{k}, y^{t}_{k+1}) , x^{t}_{k+1} - x \rangle
            +
           \langle \tilde{s}^{t}_k -s^{t}_k, 
          y^{t}_{k+1} - x \rangle
          \Big)
          \\ 
        \leq & 
         \sum_{k=0}^{N-1}\rho^{-k} \Big( 
                \underbrace{-\langle q^{t}_{k+1}, y^{t}_{k+1} - x \rangle + \theta \langle q^{t}_{k}, y^{t}_{k} - x \rangle}_{\textbf{part 1}} + 
                \Lambda^{t}_{k} (x,y) - \Sigma^{t}_{k+1} (x,y) 
                + \Gamma^{t}_{k+1}
            \\
             & \quad  + \frac{\pi_x}{2} \|x^{t}_{k+1} - x \|^2 + \frac{\pi_y}{2}\|y^{t}_{k+1}-x\|^2 +
          \underbrace{\frac{1}{2\pi_x}\| v_{k}^{t} - \nabla_x \Phi^{t}(x^{t}_{k}, y^{t}_{k+1}) \|^2 
            + \frac{1}{2\pi_y}\|\tilde{s}^t_k-s^t_k\|^2 }_{\textbf{part 2}}
          \Big).
    \end{aligned}
\end{equation}}%
The second inequality follows from Young's inequality for some constants $\pi_x,\pi_y>0$. 

\sa{The following bound for \textbf{part1} can be obtained \xzh{from \eqref{eq:noisy-rate-part1-dsit}.} Indeed, for any $k \geq -1$, we get}
{\small
\begin{align}
\label{eq:noisy-rate-part1-gap SPIDER}
         \sum_{k=0}^{N-1}&\rho^{-k} 
                 ( \theta \langle q^{t}_k, y^{t}_{k} - y^{t}_{*} \rangle-\langle q^{t}_{k+1}, y^{t}_{k+1} - y^{t}_{*} \rangle )
             \leq \sum_{k=0}^{N-1}\rho^{-k}|1-{\frac{\theta}{\rho}} |~S^{t}_{k+1}\xzh{ (x,y)}
                 + \rho^{-N+1}{\frac{\theta}{\rho} }S^{t}_{N}\xzh{(x,y)}.
\end{align}}%
where \sa{$S_{k+1}^t\xzh{(x,y)}$ is defined in \eqref{eq:cauchy-S-dist}.} 

Next we consider \textbf{part 2}, recall that $n_k=\lceil k / q \rceil$ such that $(n_k-1)q+1\leq k \leq n_k q$, it follows from \cref{Lemma: SPIDER} that
\begin{equation}\label{eq:noisy-rate-part2-gap.1 SPIDER}
    \begin{aligned}
        &\sum_{k=\sa{0}}^{N-1}\rho^{-k}\ED{v_{k}^{t}- \nabla_x\Phi^t(x^{t}_{k},y^{t}_{k+1})}
        \\&\leq\sum_{\substack{k\in\{1,\ldots,N-1\}\\\text{s.t.}\ \mo(k,q)\neq 0}}\rho^{-k} \sum_{i=(n_k-1)q+1}^{k}\big(\frac{2\xz{{L'_{xx}}^2}}{\xz{{\xzrev{b'_{x}}}}}\ED{x^{t}_{i}-x^{t}_{i-1}}+ \frac{2L^2_{xy}}{\xz{{\xzrev{b'_{x}}}}}\ED{y^{t}_{i+1}-y^{t}_{i}}\big) + \frac{\delta^2_x}{b}\sum_{k=\xzf{1}}^{N-1}\rho^{-k}+\na{\frac{\delta_x^2}{b_0}}
        \\ &\sa{=} \sum_{k=1}^{N-1} \rho^{-k} \big(\frac{2\xz{{L'_{xx}}^2}}{\xz{{\xzrev{b'_{x}}}}}\ED{x^{t}_{k}-x^{t}_{k-1}}+  \frac{2L^2_{xy}}{\xz{{\xzrev{b'_{x}}}}}\ED{y^{t}_{k+1}-y^{t}_{k}} \big) \sa{\sum_{i=0}^{n_kq-k-1}\rho^{-i}}+ \frac{\delta^2_x}{b}\sum_{k=\xzf{1}}^{N-1}\rho^{-k}+\na{\frac{\delta_x^2}{b_0}}
        \\ &\sa{=} \sum_{k=0}^{N-1} \rho^{-k}\cdot\frac{\rho}{1-\rho}(\rho^{-n_kq+k}-1)\big(\frac{2\xz{{L'_{xx}}^2}}{\xz{{\xzrev{b'_{x}}}}}\ED{x^{t}_{k}-x^{t}_{k-1}}+ \frac{2L^2_{xy}}{\xz{{\xzrev{b'_{x}}}}}\ED{y^{t}_{k+1}-y^{t}_{k}}\big) + \frac{\delta^2_x}{b}\sum_{k=\xzf{1}}^{N-1} \rho^{-k}+\na{\frac{\delta_x^2}{b_0}}
        \\ &\leq \sum_{k=0}^{N-1} \rho^{-k}\cdot\frac{\rho}{1-\rho}(\rho^{-q+1}-1)\big(\frac{2\xz{{L'_{xx}}^2}}{\xz{{\xzrev{b'_{x}}}}}\ED{x^{t}_{k}-x^{t}_{k-1}}+ \frac{2L^2_{xy}}{\xz{{\xzrev{b'_{x}}}}}\ED{y^{t}_{k+1}-y^{t}_{k}}\big) + \frac{\delta^2_x}{b}\sum_{k=\xzf{1}}^{N-1} \rho^{-k} +\na{\frac{\delta_x^2}{b_0}}
    \end{aligned}
\end{equation}
\sa{where the first inequality \sa{follows from} \cref{Lemma: SPIDER}, the following equality is by rearranging terms, and for the last inequality we used the following bound: $n_k=\lceil k / q \rceil\leq k/q + (q-1)/q$; hence, $-n_k q+k\geq -q+1$.}
\sa{To bound \textbf{part\;2} in \cref{INEQ: difference of gap SPIDER}, we next consider $\|\tilde s^{t}_k-s^{t}_k\|^2$. {For $k> 0$},}
\begin{equation}\label{eq: bound of tilde s VR}
    \begin{aligned}
         & \|\tilde s^{t}_k-s^{t}_k\|^2 = \|(1+\theta)w_{k}^{t} - (1+\theta)\nabla_y\Phi^{t}(x^{t}_k,y^{t}_{k})  -\theta w_{k-1}^{t} +\theta\nabla_y\Phi^{t}(x^{t}_{k-1},y^{t}_{k-1})\|^2  \\
         &\leq \sa{2}(1+\theta)^2\|w_{k}^{t} - \nabla_y\Phi^{t}(x^{t}_k,y^{t}_{k}) \|^2+ \sa{2}\theta^2\|  w_{k-1}^{t} -\nabla_y\Phi^{t}(x^{t}_{k-1},y^{t}_{k-1})\|^2.
    \end{aligned}
\end{equation}
First, $x^{t}_{-1}=x^{t}_{0}$, $y^{t}_{-1}=y^{t}_{0}$ and
\eqref{eq:qksk} imply that $s^t_{0} = \nabla_y\Phi^{t}(x^{t}_0,y^{t}_{0})$, and recall that in Algorithm~\ref{Alg: SAPD-VR}, we set $\tilde{s}^{t}_{0} = w^t_0$; hence,
$$
\|\tilde s^{t}_0-s^{t}_0\|^2 = \|w_{0}^{t} - \nabla_y\Phi^{t}(x^{t}_0,y^{t}_{0}) \|^2,
$$
and \cref{eq: bound of tilde s VR} holds for $k\geq 0$ with \sa{$w_{-1}^t\triangleq\nabla_y\Phi^{t}(x^{t}_0,y^{t}_{0})$.}
Then, \cref{Lemma: SPIDER} implies that
\begin{equation}
    \begin{aligned}
           & \sum_{k=\sa{0}}^{N-1}\rho^{-k}\ED{\tilde s^{t}_{k}-  s^{t}_{k}} \\
           & \leq  \xz{2}(1+\theta)^2\sum_{\sa{\substack{k\in\{1,\ldots,N-1\}\\\text{s.t.}\ \mo(k,q)\neq 0}}}\rho^{-k}\sum_{i=(n_k-1)q+1}^{k}\big(\frac{2L^2_{yx}}{\xz{{\xzrev{b'_{y}}}}}\ED{x^{t}_{i}-x^{t}_{i-1}}+ \frac{2L^2_{yy}}{\xz{{\xzrev{b'_{y}}}}}\ED{y^{t}_{i}-y^{t}_{i-1}}\big)
           \\
           & \quad + \sa{\frac{\xz{2}\theta^2}{\rho}}\sum_{\sa{\substack{k\in\{1,\ldots,N-2\}\\\text{s.t.}\ \mo(k,q)\neq 0}}}\rho^{-k}\sum_{i=(\sa{n_{k}}-1)q+1}^{\sa{k}}\big(\frac{2L^2_{yx}}{\xz{{\xzrev{b'_{y}}}}}\ED{x^{t}_{i}-x^{t}_{i-1}}+ \frac{2L^2_{yy}}{\xz{{\xzrev{b'_{y}}}}}\ED{y^{t}_{i}-y^{t}_{i-1}} 
           \big)
           \\ 
           & \quad +  2(1+2\theta+2\theta^2)\frac{\delta^2_y}{b}\sum_{k=\xzf{1}}^{N-1}\rho^{-k} + \xzf{2(1+2\theta+2\theta^2)\frac{\delta^2_y}{b_0}}
           \\
           & \leq \xz{\frac{2}{\rho}}(1+2\theta+2\theta^2)\sum_{\sa{\substack{k\in\{1,\ldots,N-1\}\\\text{s.t.}\ \mo(k,q)\neq 0}}}\rho^{-k}\sum_{i=(n_k-1)q+1}^{k}\big(\frac{2L^2_{yx}}{\xz{{\xzrev{b'_{y}}}}}\ED{x^{t}_{i}-x^{t}_{i-1}}+ \frac{2L^2_{yy}}{\xz{{\xzrev{b'_{y}}}}}\ED{y^{t}_{i}-y^{t}_{i-1}}\big)
           \\ 
           & \quad +  \xz{2}(1+2\theta+2\theta^2)\frac{\delta^2_y}{b}\sum_{k=\xzf{1}}^{N-1}\rho^{-k}  + \xzf{2(1+2\theta+2\theta^2)\frac{\delta^2_y}{b_0}},
    \end{aligned}
\end{equation}
where the first inequality follows from \cref{Lemma: SPIDER}; 
in the last inequality we used $\rho\leq 1$ and combined the two sums. Next, as in \cref{eq:noisy-rate-part2-gap.1 SPIDER}, we can further obtain that
\begin{equation}\label{eq:noisy-rate-part2-gap.2 SPIDER}
    \begin{aligned}
           & \sum_{k=0}^{N-1}\rho^{-k}\ED{\tilde s^{t}_{k}-  s^{t}_{k}} \\
           & \leq \xz{2}(1+2\theta+2\theta^2)\sum_{k=0}^{N-1} \rho^{-k}\cdot\frac{1}{1-\rho}(\rho^{-q+1}-1)\big(\frac{2L^2_{yx}}{\xz{{\xzrev{b'_{y}}}}}\ED{x^{t}_{k}-x^{t}_{k-1}}+ \frac{2L^2_{yy}}{\xz{{\xzrev{b'_{y}}}}}\ED{y^{t}_{k}-y^{t}_{k-1}}\big)
           \\& +\xz{2}(1+2\theta+2\theta^2)\frac{\delta^2_y}{b}\sum_{k=\xzf{1}}^{N-1}\rho^{-k} + \xzf{2(1+2\theta+2\theta^2)\frac{\delta^2_y}{b_0}}.
    \end{aligned}
\end{equation}
Now we can bound \textbf{part\;2} in \cref{INEQ: difference of gap SPIDER} using \cref{eq:noisy-rate-part2-gap.1 SPIDER} and \cref{eq:noisy-rate-part2-gap.2 SPIDER}. In addition, Given {$(\bar{x}^{t}_N,\bar{y}^{t}_N)$}, the point $(x^{t}_{*}(\bar{y}^{t}_N),y_{*}(\bar{x}^{t}_N))\triangleq\argmax_{(x,y)\in\cX\times\cY}\cL^{t}(\bar{x}^{t}_{N},y)-\cL^{t}(x,\bar{y}^{t}_{N})$ uniquely exists. We will use the fact that
$$
\cG^t(\bar{x}^{t}_{N},\bar{y}^{t}_{N})=\sup_{(x,y)\in\cX\times\cY}\cL^{t}(\bar{x}^{t}_{N},y)-\cL^{t}(x,\bar{y}^{t}_{N})=\cL^{t}(\bar{x}^{t}_{N},y_{*}(\bar{x}^{t}_N))-\cL^{t}(x^{t}_{*}(\bar{y}^{t}_N),\bar{y}^{t}_{N})
$$
to complete the proof. 

  {Recall} that we defined $ D^{t}_{N}(x,y){=}\frac{1}{2\rho}(\frac{1}{\tau}-\mu_x) \|x^{t}_{N}-x\|^2 +  \frac{1}{2}{\left(\frac{1}{\rho\sigma} -  \alpha  \right)}\|y^{t}_{N}-y\|^2$; \nsa{first,} we substitute $(x,y) = (x^{t}_{*}(\bar{y}^{t}_N),y_{*}(\bar{x}^{t}_N))$ into \eqref{INEQ: difference of gap SPIDER}, and then
add $\rho^{-N+1} D^{t}_{N}\xzh{(x^{t}_{*}(\bar{y}^{t}_N),y_{*}(\bar{x}^{t}_N))}$ to both sides of \eqref{INEQ: difference of gap SPIDER}. \nsa{Finally,} taking the expectation of the new inequality, and then using \cref{eq:noisy-rate-part1-gap SPIDER}, \cref{eq:noisy-rate-part2-gap.1 SPIDER} and \cref{eq:noisy-rate-part2-gap.2 SPIDER} to bound \textbf{part\;1} and \textbf{part\;2}, we obtain
{
\begin{eqnarray}
\label{INEQ: ergordic gap  SPIDER}
\lefteqn{\mathbb{E}\big[K_N(\rho) \xzh{\cG^t(\bar{x}^{t}_{N},\bar{y}^{t}_{N})}+ {\rho}^{-N+1} D^{t}_{N}\xzh{(x^{t}_{*}(\bar{y}^{t}_N),y_{*}(\bar{x}^{t}_N))} \big]}\nonumber\\
&&\leq 
      \mathbb{E}\left[\hat{U}^{t}_{N}\xzh{(x^{t}_{*}(\bar{y}^{t}_N),y_{*}(\bar{x}^{t}_N))}\right] + \Big(\frac{\delta^2_x}{2\pi_x b}+(1+2\theta+2\theta^2)\frac{\delta^2_y}{\xz{\pi_y b}}\Big)\sum_{k=\xzf{1}}^{N-1}{\rho}^{-k} 
      \\
      &&+ \xzf{\Big(\frac{\delta^2_x}{2\pi_x b_0}+(1+2\theta+2\theta^2)\frac{\delta^2_y}{\xz{\pi_y b_0}}\Big)}, 
\end{eqnarray}}%
where $\hat{U}^{t}_{N}(x,y)$  is defined as
{
\small
\begin{equation}\label{eq:UPQ SPIDER}
\begin{aligned}
    & \hat{U}^{t}_{N}(x,y) \triangleq 
        \sum_{k=0}^{N-1}{\rho}^{-k}
      \Big(\Gamma^{t}_{k+1} + \Lambda^{t}_k{(x,y)} - \Sigma^{t}_{k+1} \sa{(x,y)} 
      \\
      & + |1-{\frac{\theta}{\rho}} |~S^{t}_{k+1}\xzh{(x,y)}
       + \frac{\pi_x}{2} \|x^{t}_{k+1} - x \|^2 + \frac{\pi_y}{2}\|y^{t}_{k+1}-y\|^2
      \Big) \\
      & +\sum_{k=0}^{N-1}\frac{\rho^{-k+1}}{1-\rho}(\rho^{-q+1}-1)\left(
      \left(\xz{\frac{ \xz{{L'_{xx}}^2}}{\xz{\pi_x}{\xzrev{b'_{x}}}} 
      +\frac{2(1+2\theta+2\theta^2)\rho^{-1}L^2_{yx}}{\pi_y {\xzrev{b'_{y}}}}}
      \right)
      \|x^{t}_{k}-x^{t}_{k-1}\|^2
      + 
      \frac{ L^2_{xy}}{\xz{\pi_x}\xz{{\xzrev{b'_{x}}}}}\|y^{t}_{k+1}-y^{t}_{k}\|^2 
      \right)
      \\&
      +  
      \sum_{k=0}^{N-1}\frac{\rho^{-k+1}}{1-\rho}(\rho^{-q+1}-1)\frac{\xz{2}(1+2\theta+2\theta^2)\rho^{-1} L^2_{yy}}{\pi_y\xz{{\xzrev{b'_{y}}}}}\|y^{t}_{k}-y^{t}_{k-1}\|^2 
      - 
      \rho^{-N+1}(- D^{t}_{N}(x,y) - {\frac{\theta}{\rho}} S^{t}_{N}\xzh{(x,y)}).
    \end{aligned}
\end{equation}}%
{The remaining part of the analysis directly follows from the arguments we used in the proof of \cref{lemma: dist bound with fg-general}.} 
We can analyze $\hat{U}^{t}_{N}\xzh{(x^{t}_{*}(\bar{y}^{t}_N),y_{*}(\bar{x}^{t}_N))}$ \sa{through writing it as a telescoping sum}. 
After adding and subtracting $\tfrac{\alpha}{2}\|y^{t}_{k+1}-y^{t}_{k}\|^2$,  and rearranging the terms, we get
{
\begin{equation}
\label{eq:U_def SPIDER}
    \begin{aligned}
           \hat{U}^{t}_{N}\xzh{(x^{t}_{*}(\bar{y}^{t}_N),y_{*}(\bar{x}^{t}_N))} 
           =  & 
           \frac{1}{2}\sum_{k=0}^{N-1}{\rho}^{-k}\Big({{\xi ^{*}_k}}^\top \hat{A}{\xi ^{*}_k}-{\xi ^{*}_{k+1}}^\top \hat{B} {\xi ^{*}_{k+1}}\Big)
           \\
           & -{\rho}^{-N+1}( -  D^{t}_{N}\xzh{(x^{t}_{*}(\bar{y}^{t}_N),y_{*}(\bar{x}^{t}_N))} - {\frac{\theta}{\rho}} \sa{S^{t}_{N}}\xzh{(x^{t}_{*}(\bar{y}^{t}_N),y_{*}(\bar{x}^{t}_N))})\\
           =  & 
           \frac{1}{2}{\xi^{*}_0}^\top \hat{A}\xi^{*}_0- \frac{1}{2}\sum_{k=1}^{N-1}{\rho}^{-k+1}[{{\xi ^{*}_k}}^\top{( \hat{B} 
           - \tfrac{1}{\rho}\hat{A})}{\xi ^{*}_k}]
           \\
           & -\rho^{-N + 1 }( \frac{1}{2}{\xi ^{*}_N}^\top   \hat{B} {\xi ^{*}_N}  -  D^{t}_{N}\xzh{(x^{t}_{*}(\bar{y}^{t}_N),y_{*}(\bar{x}^{t}_N))} -{\frac{\theta}{\rho}} S^{t}_{N}\xzh{(x^{t}_{*}(\bar{y}^{t}_N),y_{*}(\bar{x}^{t}_N))}),    
    \end{aligned}
\end{equation}}%
where ${\xi ^{*}_k}\in\reals^5$ is defined for $k\geq 0$ as follows:
{
${\xi ^{*}_k} \triangleq \left( 
    \begin{array}{*{20}{c}}
  \|x^{t}_{k} - \xzh{x^{t}_{*}(\bar{y}^{t}_N)}\| \\ 
  \|y^{t}_{k} - \xzh{y_{*}(\bar{y}^{t}_N)}\| \\ 
  \|x^{t}_{k} - x^{t}_{k-1}\| \\ 
  \|y^{t}_{k} - y^{t}_{k-1}\| \\
  \|y^{t}_{k+1} - y^{t}_{k}\| 
   \end{array}
   \right)$ {\normalsize such that $x^{t}_{-1} = x^{t}_{0}$, $y^{t}_{-1} = y^{t}_{0}$; and $\hat{A}, \hat{B} \in\reals^{5\times 5}$ are defined as}:
$$\hat{A} \triangleq   
  \begin{pmatrix}
    \frac{1}{\tau}-\mu_x & 0 & 0 & 0 & 0\\ 
  0 & \frac{1}{\sigma} & 0 & 0 & 0\\ 
  0 & 0 & \rho \sa{L_x'} & 0 & {\theta L_{yx}}\\ 
  0& 0 & 0 &  \rho \sa{L_y^+} & {\theta L_{yy}}\\
  0 & 0 & \theta L_{yx} & \theta L_{yy} & -\alpha 
\end{pmatrix},
$$
$$
   \hat{B} \triangleq   
  \begin{pmatrix}
    \frac{1}{\tau} -\pi_x& 0 & 0 & 0 & 0\\ 
  0 & \frac{1}{\sigma}+\mu_y-\pi_y & \sa{-|1-\frac{\theta}{\rho} |}~L_{yx} & \sa{-|1-\frac{\theta}{\rho} |}~L_{yy} & 0\\ 
  0 & \sa{-|1-\frac{\theta}{\rho} |}~L_{yx} & \tfrac{1}{\tau} - L'_{xx} & 0 & 0\\ 
  0& \sa{-|1-\frac{\theta}{\rho} |}~L_{yy} & 0 & \frac{1}{\sigma} - \alpha - \sa{L_y^-} & 0\\
  0 & 0 & 0 & 0 & 0
\end{pmatrix},$$}
with 
\begin{align*}
       & \sa{L_x'} \triangleq \frac{2}{1-\rho}(\rho^{-q+1}-1)
       \big(
       \frac{ \xz{{L'_{xx}}^2}}{\pi_x\xz{{\xzrev{b'_{x}}}}} 
       + 
       \frac{\xz{2}(1+2\theta+2\theta^2)\rho^{-1} L^2_{yx}}{\pi_y\xz{{\xzrev{b'_{y}}}}}
       \big),
    \\
   & \sa{L_y^+} \triangleq \frac{2}{1-\rho}(\rho^{-q+1}-1)\frac{\xz{2}(1+2\theta+2\theta^2)\rho^{-1} L^2_{yy}}{\pi_y\xz{{\xzrev{b'_{y}}}}},
   \\
   & \sa{L_y^-} \triangleq\frac{2\rho}{1-\rho}(\rho^{-q+1}-1)\frac{ L^2_{xy}}{\xz{\pi_x {\xzrev{b'_{x}}}}}.
\end{align*}
\sa{Using the same argument as in the proof of Lemma~\ref{lem:equivalent_systems}, and noticing that $L'_y$ in \cref{eq: SCSC SAPD LMI SPIDER} can be written as $L'_y=L_y^+ + L_y^-$, one can show that \cref{eq: SCSC SAPD LMI SPIDER} holds if and only if 
$\hat{B} -\tfrac{1}{\rho}\hat{A}\succeq 0$.} Therefore, it follows from \eqref{eq:U_def SPIDER} that 
{
\begin{align*}
    \hat{U}^{t}_{N}\xzh{(x^{t}_{*}(\bar{y}^{t}_N),y_{*}(\bar{x}^{t}_N))} \leq &  \sa{\frac{1}{2}{\xi^{*}_0}^\top \hat{A}\xi^{*}_0} \\
    & - {\rho}^{-N + 1 } \Big( \tfrac{1}{2}{\xi ^{*}_N}^\top  \hat{B} {\xi ^{*}_N}  -   D^{t}_{N}\xzh{(x^{t}_{*}(\bar{y}^{t}_N),y_{*}(\bar{x}^{t}_N))} -\xzh{\frac{\theta}{\rho}} \sa{S^{t}_{N}}\xzh{(x^{t}_{*}(\bar{y}^{t}_N),y_{*}(\bar{x}^{t}_N))}\Big), 
\end{align*}}
\text{holds w.p.~1. } Furthermore, define
$$
G''' \triangleq \begin{pmatrix}
     \frac{1}{\sigma}(1-\frac{1}{\rho})+\mu_y-\pi_y+ \alpha  & \big(\sa{-|1-\frac{\theta}{\rho} |}-\frac{\theta}{\rho} \big)L_{yx} & \big(\sa{-|1-\frac{\theta}{\rho} |}-\frac{\theta}{\rho} \big)L_{yy} \\ 
   \big(\sa{-|1-\frac{\theta}{\rho} |}-\frac{\theta}{\rho}\big)L_{yx} & \tfrac{1}{\tau} - L'_{xx} & 0 \\ 
   \big(\sa{-|1-\frac{\theta}{\rho} |}-\frac{\theta}{\rho} \big)L_{yy} & 0 & \frac{1}{\sigma} - \alpha - \sa{L_y^{-}}
\end{pmatrix},
$$
and recall that  $ D^{t}_{N}(x,y)=\frac{1}{2\rho}(\frac{1}{\tau}-\mu_x) \|x^{t}_{N}-x\|^2 +  \frac{1}{2}{\left(\frac{1}{\rho\sigma} -  \alpha  \right)}\|y^{t}_{N}-y\|^2$. \sa{Using a similar argument as in the proof of \cref{lemma: sub positive matrix}, we can show that \cref{eq: SCSC SAPD LMI SPIDER} implies}
{
\begin{eqnarray*}
    \lefteqn{\frac{1}{2}{\xi ^{*}_N}^\top \hat{B}  {\xi ^{*}_N}  -  D^{t}_{N}{\Big(x^{t}_{*}(\bar{y}^{t}_N),y_{*}(\bar{x}^{t}_N)\Big)} -\frac{\theta}{\rho} \sa{S^{t}_{N}}{\Big(x^{t}_{*}(\bar{y}^{t}_N),y_{*}(\bar{x}^{t}_N)\Big)}}\\
    &&=
\frac{1}{2}{\xi ^{*}_N}^\top  
     \begin{pmatrix}
       \frac{1}{\tau}(1-\frac{1}{\rho}) + \frac{\mu_x}{\rho}-\pi_x 
       & \mathbf{0}_{1\times3} & 0\\ 
       \mathbf{0}_{3\times 1} & G''' & \mathbf{0}_{3\times 1} \\
       0 & \mathbf{0}_{1\times3} & 0\\ 
     \end{pmatrix}
{\xi ^{*}_N}
  \geq 
0.
\end{eqnarray*}}%
Finally, \sa{since $x^{t}_{-1} = x^{t}_{0}$, $y^{t}_{-1} = y^{t}_{0}$, we have}
$$
\frac{1}{2}{\xi^{*}_0}^\top \hat{A}\xi^{*}_0\sa{\leq} \Big(\frac{1}{2\tau}-\frac{\mu_x}{2}\Big)\norm{x^{t}_{*}(\bar{y}^{t}_N)-x^{t}_{0}}^2+\frac{1}{2\sigma}\norm{y_{*}(\bar{x}^{t}_N)-y^{t}_{0}}^2.
$$ 
Therefore, we obtain that 
$$
\hat{U}^{t}_{N}(x^{t}_{*}(\bar{y}^{t}_N),y_{*}(\bar{x}^{t}_N))\leq \Big(\frac{1}{2\tau}-\frac{\mu_x}{2}\Big)\norm{x^{t}_{*}(\bar{y}^{t}_N)-x^{t}_{0}}^2+\frac{1}{2\sigma}\norm{y_{*}(\bar{x}^{t}_N)-y^{t}_{0}}^2,\text{ holds w.p.~1. }
$$
\sa{Now, we are ready to show the desired result of \cref{lemma: Gap bound with fg SPIDER}.} 
 Since $ D^{t}_{N}(x^{t}_{*}(\bar{y}^{t}_N),y_{*}(\bar{x}^{t}_N))\geq 0$, it follows from \eqref{INEQ: ergordic gap  SPIDER} that 
{\small
\begin{align*}
K_N(\rho)\mathbb{E}\big[ \cG^t(\bar{x}^{t}_{N},\bar{y}^{t}_{N})\big] \leq 
      &\mathbb{E}\Big[\left(\frac{1}{2\tau}-\frac{\mu_x}{2}\right)\|x^{t}_*(\bar{y}^{t}_N)-x^{t}_{0}\|^2
      +\frac{1}{2\sigma}\|y_{*}(\bar{x}^{t}_{N})-y^{t}_{0}\|^2 \Big]\\
       &\quad+\xzf{\rho^{-1}K_{N-1}(\rho)} \Big(\frac{\delta^2_x}{2\pi_x b}+(1+2\theta+2\theta^2)\frac{\delta^2_y}{\sa{\pi_y b}}\Big)+ \xzf{\Big(\frac{\delta^2_x}{2\pi_x b_0}+(1+2\theta+2\theta^2)\frac{\delta^2_y}{{\pi_y b_0}}\Big)}.
\end{align*}}%
Then dividing both side by $K_N(\rho)$ completes the proof.
\end{proof}
\subsection{A particular parameter choice}
We employ the 
matrix inequality~(MI) in~\cref{eq: SCSC SAPD LMI SPIDER} to describe the admissible set of \texttt{VR-SAPD} parameters that guarantee convergence.
\sa{Next, in \cref{LEMMA: sp parameter VR}, we  compute a particular solution to it 
 by exploiting its
 structure.} 
\begin{lemma}\label{LEMMA: sp parameter VR}
\sa{For any $\mu_x>0$, let $\xz{L'_{xx}=} L_{xx}+\gamma+\mu_x$. Let $\theta\in(0,1]$ and $\tau, \sigma>0$ \sa{be chosen as}
{\small
\begin{equation}
\label{eq:sufficient_cond_noisy_LMI-2}
\theta=1,\quad \tau = \frac{1}{L_{yx}+\sa{L'_{xx}}+\xz{L'_x}}, \quad \sigma = \frac{1}{2L_{yy}+L_{yx}+\xz{L'_y}},
\end{equation}}%
where $L'_x$ and $L'_y$ are defined in \cref{lemma: Gap bound with fg SPIDER}. Then $\{\tau, \sigma,\theta,\alpha,\rho, \pi_x,\pi_y\}$ is a solution to \eqref{eq: SCSC SAPD LMI SPIDER}  for $\rho=1$, $\pi_x=\mu_x$, $\pi_y=\mu_y$ and $\alpha =  L_{yx}+L_{yy}$.}
\end{lemma}
\begin{proof}
Define ${\everymath={\scriptstyle}
M_1 
 \triangleq 
 \begin{pmatrix}
 \frac{1}{\tau}  - \sa{L'_{xx}} -\xz{L'_x} & 0 & - L_{yx}\\ 
 0 & 0 & 0 \\
 - L_{yx} & 0 &L_{yx} 
\end{pmatrix}}$ and 
${\everymath={\scriptstyle}
M_2 \triangleq 
    \begin{pmatrix}
  0 & 0 & 0\\ 
 0 & \frac{1}{\sigma} \sa{-\alpha} -\xz{L'_y} & - L_{yy}\\
 0 & - L_{yy} & L_{yy} 
\end{pmatrix}}$.
\sa{Our choice of $\{\rho,\pi_x,\pi_y,\alpha\}$ implies that \eqref{eq: SCSC SAPD LMI SPIDER} holds whenever}
{\footnotesize
\begin{equation*}
\begin{aligned}
M_1 + M_2=
\begin{pmatrix}
  \tfrac{1}{\tau} - \sa{L'_{xx}}-\xz{L'_x} & 0 & - L_{yx}\\ 
 0 &\tfrac{1}{\sigma}- \alpha - \xz{L'_y}& - L_{yy}\\
 - L_{yx} & - L_{yy} & L_{yx} +L_{yy}
\end{pmatrix}\succeq \mathbf{0}.
\end{aligned}
\end{equation*}}%
\sa{Our choice of $\alpha=L_{yx}+L_{yy}$, and $\tau,\sigma>0$ as in \eqref{eq:sufficient_cond_noisy_LMI-2} implies that $M_1 
 \succeq 0$ and $M_2 \succeq 0$.} Thus, $M_1+M_2\succeq 0$. 
\end{proof}
Next, based on 
this lemma, we will give an explicit parameter choice for Algorithm~\ref{Alg: SAPD-VR}.
\subsection{Proof of Theorem~\ref{cor:complexity VR}}
\begin{proof}
\cref{LEMMA: sp parameter VR} implies \sa{that our choice of $\{\tau,\sigma,\theta,\alpha,\rho,\pi_x,\pi_y\}$ ensures that \cref{eq: SCSC SAPD LMI SPIDER}
holds. For the outer loop, if we set N as in~\eqref{eq: speical paramater VR innner loop} and}
\begin{equation}\label{eq: speical paramater VR outerloop}
    \begin{aligned}
        p_1 = \frac{1}{16},~p_2 = \frac{19}{32},~p_3=\frac{11}{32},~\beta_1 = \frac{4}{5},~\beta_2 = \frac{1}{2},~\zeta=32,
    \end{aligned}
\end{equation}
all assumptions of \cref{thm:WC_SP SPIDER} are satisfied, i.e., both the 
\sa{inequality system in~\eqref{eq: additional stepsize condition WCSC} \xzf{holds with $M_{\tau,\sigma,\theta}$ replaced by $M^{\texttt{VR}}$} and $N\geq (1+\zeta)M^{\texttt{VR}}$ holds. Specifically,  $M^{\texttt{VR}}=2\max\{\frac{1}{\gamma\tau}-1,~\frac{1}{\mu_y\sigma}\}$; thus, $N\geq (1+\zeta)M^{\texttt{VR}}$ trivially holds by our choice of $N$ in \eqref{eq: speical paramater VR innner loop}. The proof of \cref{eq: additional stepsize condition WCSC} holding 
for parameters in~\eqref{eq: speical paramater VR outerloop} follows directly from the proof of \cref{cor:complexity-fg}.}

Since all assumptions of \cref{thm:WC_SP SPIDER} are satisfied \sa{for $\mu_x = \gamma$, $\{\tau,\sigma,\theta\}$ as in \eqref{eq:sufficient_cond_noisy_LMI-2}, $N$ and $b$ as in \cref{eq: speical paramater VR innner loop} and other parameters chosen as in \cref{eq: speical paramater VR outerloop}, if
we substitute $\mu_x=\gamma$ and the specific parameter values given in~\cref{eq: speical paramater VR outerloop} into} \cref{eq: sketch convergence result} with \xzf{$\xzf{\Xi_{\tau,\sigma,\theta}=\frac{N-1}{N}\Xi^{\texttt{VR}}(b) + \frac{1}{N}\Xi^{\texttt{VR}}(b_0) }= \frac{N-1}{N}(\frac{\delta^2_x}{2\gamma b}+\sa{5}\frac{\delta^2_y}{\sa{\mu_y}b}) +
\frac{1}{N}(\frac{\delta^2_x}{2\gamma b_0}+\sa{5}\frac{\delta^2_y}{\sa{\mu_y}b_0})$}, it follows that 
\begin{equation}
    \begin{aligned}
        \frac{1}{T+1}\sum_{t=0}^{T}\mathbb{E}\left[\|\nabla\phi_{\lambda}(x^t_0) \|^2 \right]
        \leq
        &   
        48\gamma\left(\frac{1}{T+1}\cG(x_0^0,y_0^0)
        \xzf{
        +\frac{N-1}{N}(\frac{\delta^2_x}{2\gamma b}+\sa{5}\frac{\delta^2_y}{\sa{\mu_y}b}) +
\frac{1}{N}(\frac{\delta^2_x}{2\gamma b_0}+\sa{5}\frac{\delta^2_y}{\sa{\mu_y}b_0})
        } \right).
    \end{aligned}
\end{equation}
Thus, for any $\epsilon>0$,
the right side of the above inequality can be bounded by $\epsilon^2$ when
\begin{equation}\label{eq: split bound VR}
    \frac{48\gamma}{T+1}\cG(x_0^0,y_0^0)\leq \frac{\epsilon^2}{6},\quad \frac{24\delta_x^2}{ b}\leq\frac{\epsilon^2}{\xzf{12}},\quad \frac{240\gamma\delta_y^2}{ \mu_y b}\leq\frac{\epsilon^2}{\xzf{3}},\quad 
    \xzf{\frac{24\delta_x^2}{N b_0}\leq\frac{\epsilon^2}{12},}\quad 
    \xzf{\frac{240\gamma\delta_y^2}{ N\mu_y b_0}\leq\frac{\epsilon^2}{3}}.
\end{equation}
\sa{Our choice of $b$ \sa{in~\eqref{eq: speical paramater VR innner loop}} and $T\geq 288\cG(x_0^0,y_0^0)\frac{\gamma}{\epsilon^2}$ ensures that all the inequalities in~\eqref{eq: split bound VR} hold.}
\sa{Moreover, our choice of $N$ {and $\{\tau,\sigma,\theta\}$}in \eqref{eq: speical paramater VR innner loop}  and $\rho=1$ together with the definitions of $L_x'$ and $L_y'$ in \cref{lemma: Gap bound with fg SPIDER}  implies \eqref{eq:N-bound-VR}. Furthermore, it follows from the statement of Algorithm~\ref{Alg: SAPD-VR} that}
the total computation complexity is $T(\xzf{b_0}+Nb/q+N({\xzrev{b'_{x}}}+{\xzrev{b'_{y}}}))$, which completes the proof.
\end{proof}
\section{Proof of Theorem~\ref{Thm:WCMC} and preliminary technical results}
\label{sec:Thm5_proof}
\sa{Recall the definition of $\hat\cL$ given in \cref{eq:weakly convex-merely concave problem approx}.}
For any $x\in\cX$, define $\phi(x) \triangleq \max_{y\in\mathcal{Y}} \mathcal{L}(x,y)$ and $\hat{\phi}(x) \triangleq \max_{y\in\mathcal{Y}} \hat{\mathcal{L}}(x,y)$; moreover, let $\phi_\lambda(\cdot)$ and $\hat{\phi}_{\lambda}(\cdot)$ be respective Moreau envelopes for some $\lambda\in(0,\gamma^{-1})$.

\sa{We first show that one can obtain an $\epsilon$-stationary point for the WCMC problem in the form of \eqref{eq:main problem} such that $f(\cdot)=0$, $\mu_y=0$ and $\cD_y<\infty$ by computing an $\epsilon$-stationary point for \cref{eq:weakly convex-merely concave problem approx} with $\hat{\mu}_y=\Theta(\epsilon^2/(\gamma\cD_y^2))$.
Indeed, in \cref{lem:WCMC} below, we extend \cite[Corollary A.8]{lin-near-optimal} from $g$ being an indicator function of a closed convex set to a closed convex function.}

\begin{lemma}
\label{lem:WCMC}
Under the premise of \cref{Thm:WCMC},
for some fixed $\hat{\mu}_y=\Theta(\epsilon^2/(\gamma\cD_y^2))$, let $x_\epsilon\in \cX$ be such that $ \|\nabla\hat{\phi}(x_\epsilon)\|\leq \epsilon/(2\sqrt{6})$,
where $\hat{\phi}(x)\triangleq\max_{y\in\mathcal{Y}} \hat{\cL}(x,y)$. Then, $x_\epsilon$ is an $\epsilon$-stationary point of $\phi(\cdot)$, i.e., $\|\nabla\phi_{\lambda}(x_\epsilon)\|\leq\epsilon$ for $\lambda\in(0,\gamma^{-1})$, where $\phi(x) \triangleq \max_{y\in\cY}\cL(x,y)$.
\end{lemma}
\begin{proof}
Below We state some useful relations that will be used later in the proof. Since $f(\cdot)=0$, 
\cref{eq:weakly convex-merely concave problem approx} implies that for all $(x,y)\in\sa{\cX} \times \dom g$,
\begin{equation}\label{eq: gradient bound WCMC}
  \nabla_x \cL(x,y) = \nabla_x \hat{\cL}(x,y),\quad \|\nabla_y \Phi(x,y) - \nabla_y \hat{\Phi}(x,y) \|\leq \hat{\mu}_y\cD_y.
\end{equation}
We define $\hat{y}_{*}(\cdot)\triangleq \argmax_{y\in\cY} \hat{\cL}(\cdot,y)$.
It follows that from \cref{lemma: proximal point lemma v2} that
\begin{equation}
    \hat{y}_{*}(x_\epsilon)=\prox{\alpha g}\big(\hat{y}_{*}(x_\epsilon) + \alpha\nabla_y\hat{\Phi}(x_\epsilon,\hat{y}_{*}(x_\epsilon))\big).\label{eq: G_y}
\end{equation}
for \sa{any} $\alpha>0$. We are now ready for the proof of Lemma~\ref{lem:WCMC}.  

Let $y^+ \triangleq \prox{\alpha g}\big(\hat{y}_{*}(x_\epsilon) + \alpha\nabla_y\Phi(x_\epsilon,\hat{y}_{*}(x_\epsilon))\big)$, then we have 
\begin{equation}\label{eq: y bound WCMC}
    \begin{aligned}
            & \|y^+  - \hat{y}_{*}(x_\epsilon) \|
            \\
            = & \|\prox{\alpha g}\big(\hat{y}_{*}(x_\epsilon) + \alpha\nabla_y\Phi(x_\epsilon,\hat{y}_{*}(x_\epsilon))\big) - \prox{\alpha g}\big(\hat{y}_{*}(x_\epsilon) 
             + \alpha\nabla_y\hat{\Phi}(x_\epsilon,\hat{y}_{*}(x_\epsilon))\big) \|
             \\
             \leq & \alpha \|  \nabla_y\Phi(x_\epsilon,\hat{y}_{*}(x_\epsilon)) - \nabla_y\hat{\Phi}(x_\epsilon,\hat{y}_{*}(x_\epsilon))\| \leq \alpha\hat{\mu}_y\cD_y
    \end{aligned}
\end{equation}
where the first equality is by \cref{eq: G_y}; the second inequality is by $\|\prox{\alpha g}(y_1) - \prox{\alpha g}(y_2)\|\leq \| y_1 - y_2\|$ for all $y_1,y_2\in\cY$ and \cref{eq: gradient bound WCMC}. 
Moreover, using \cref{ASPT: lipshiz gradient} and  the above inequalities, we have 
\begin{equation*}
\begin{aligned}
             &\|\nabla_x \cL(x_\epsilon,y^+)\| \leq 
            \|\nabla_x \cL(x_\epsilon,y^+) 
             - \nabla\hat{\phi}(x_\epsilon)
             \|
             +
             \|
             \nabla\hat{\phi}(x_\epsilon)
             \|
             \\
             \leq &\|\nabla_x \cL(x_\epsilon,y^+) - \nabla_x \cL(x_\epsilon,\hat{y}_{*}(x_\epsilon)) \| 
             + \frac{\epsilon}{2\sqrt{6}} 
            \leq L_{xy}\alpha\hat{\mu}_y D_y +\frac{\epsilon}{2\sqrt{6}},
\end{aligned}
\end{equation*}
where 
\sa{the second inequality follows from Danskin's theorem and the fact that $\|\nabla\hat{\phi}(x_\epsilon)\|\leq \epsilon/(2\sqrt{6})$; finally, the last inequality use \cref{ASPT: lipshiz gradient} and \eqref{eq: y bound WCMC}. Thus, using $(a+b)^2\leq 2(a^2+b^2)$ for any $a,b\in\reals$, we get}
\begin{equation}\label{eq: x bound WCMC}
    \|\nabla_x \cL(x_\epsilon,y^+)\|^2  \leq \frac{\epsilon^2}{12} + 2L^2_{xy}\alpha^2\hat{\mu}_y^2 D_y^2.
\end{equation}
Later in the proof, \cref{eq: y bound WCMC} and \cref{eq: x bound WCMC} will be useful when we further analyze $y^+$.

\sa{Recall that our ultimate goal is to show that $\|\nabla \phi_{\lambda}(x_\epsilon)\|\leq\epsilon$. Now, for some arbitrary $\mu_x>0$, consider} $\prox{\lambda \phi}(x_\epsilon)=\argmin_{v\in\cX} \phi(v) + \frac{1}{2\lambda}\|v-x_\epsilon\|^{\sa{2}}$, where $\lambda = (\mu_x+\gamma)^{-1}$. It follows from \cref{Lemma: graident of ME}  that
$$\|\nabla \phi_{\lambda}(x_\epsilon)\|^2 = \frac{1}{\lambda^2} \|x_\epsilon - \prox{\lambda \phi}(x_\epsilon)\|^2
.$$
Since $\lambda = (\mu_x+\gamma)^{-1}$, $\phi(\cdot)+\frac{1}{2\lambda}\|\cdot - x_\epsilon\|^2$ is $\mu_x$-strongly convex with the unique minimizer $\prox{\lambda \phi}(x_\epsilon)$; therefore,
\begin{equation}
    \begin{aligned}\label{eq: lower bound 1}
            &\max_{y\in\cY} \cL(x_\epsilon,y) - \max_{y\in\cY} \cL(\prox{\lambda \phi}(x_\epsilon),y) - \frac{1}{2\lambda}\|\prox{\lambda \phi}(x_\epsilon)-x_\epsilon\|^2 \\
            & = \phi(x_\epsilon) - \phi(\prox{\lambda \phi}(x_\epsilon))  - \frac{1}{2\lambda}\|\prox{\lambda \phi}(x_\epsilon)-x_\epsilon\|^2\\ 
            &\geq \frac{\mu_x}{2}\| x_\epsilon - \prox{\lambda \phi}(x_\epsilon) \|^2 = \lambda^2\frac{\mu_x}{2}\|\nabla \phi_{\lambda}(x_\epsilon) \|^2.
    \end{aligned}
\end{equation}
\sa{In the following analysis, we will continue to polish the upper bound on $\|\nabla \phi_{\lambda}(x_\epsilon) \|^2$ on the left hand side of \cref{eq: lower bound 1}. Indeed,}
\begin{equation}
    \begin{aligned}\label{eq: upper bound 1}
            &\max_{y\in\cY} \cL(x_\epsilon,y) - \max_{y\in\cY} \cL(\prox{\lambda \phi}(x_\epsilon),y) - \frac{1}{2\lambda}\|\prox{\lambda \phi}(x_\epsilon)-x_\epsilon\|^2 \\
            & = \max_{y\in\cY} \cL(x_\epsilon,y) -\cL(x_\epsilon,y^+)+ \cL(x_\epsilon,y^+) - \max_{y\in\cY} \cL(\prox{\lambda \phi}(x_\epsilon),y) - \frac{1}{2\lambda}\|\prox{\lambda \phi}(x_\epsilon)-x_\epsilon\|^2
            \\
            & \leq \max_{y\in\cY} \cL(x_\epsilon,y) -\cL(x_\epsilon,y^+)+ \cL(x_\epsilon,y^+) - \cL(\prox{\lambda \phi}(x_\epsilon),y^+) - \frac{1}{2\lambda}\|\prox{\lambda \phi}(x_\epsilon)-x_\epsilon\|^2
            \\
            & \leq \max_{y\in\cY} \cL(x_\epsilon,y) -\cL(x_\epsilon,y^+)
            + \|x_\epsilon - \prox{\lambda \phi}(x_\epsilon) \|\|\nabla_x \cL(x_\epsilon, y^+)\| - \frac{\mu_x}{2}\| x_\epsilon - \prox{\lambda \phi}(x_\epsilon)\|^2
            \\
            & \leq \max_{y\in\cY} \cL(x_\epsilon,y) -\cL(x_\epsilon,y^+)
            + \frac{\|\nabla_x \cL(x_\epsilon, y^+)\|^2}{2\mu_x},
    \end{aligned}
\end{equation}
where the second inequality \sa{follows from} the $\mu_x$-strongly convexity of $\cL(\cdot,y^+) + \frac{1}{2\lambda}\|\cdot-x_\epsilon\|^2$ and \sa{Cauchy-Schwarz} inequality. \sa{Next, we continue to derive an appropriate upper bound on} $ \max_{y\in\cY} \cL(x_\epsilon,y) -\cL(x_\epsilon,y^+)$. Recall that $y^+ = \prox{\alpha g}\big(\hat{y}_{*}(x_\epsilon) + \alpha\sa{\nabla_y}\Phi(x_\epsilon,\hat{y}_{*}(x_\epsilon))\big)$; \sa{hence, the first-order optimality condition} yields that
$$
-\frac{1}{\alpha}\left(y^+ - \hat{y}_{*}(x_\epsilon) - \alpha\sa{\nabla_y}\Phi(x_\epsilon,\hat{y}_{*}(x_\epsilon))  \right) \in \partial g(y^+).
$$
Therefore, for any $y\in\cY$, we have that
$$
g(y) - g(y^+) \geq \langle y - y^+, -\frac{1}{\alpha}\left(y^+ - \hat{y}_{*}(x_\epsilon) - \alpha\sa{\nabla_y}\Phi(x_\epsilon,\hat{y}_{*}(x_\epsilon))  \right) \rangle ,
$$
which is equivalent to
\begin{equation}\label{eq:subgraident of g}
     g(y^+) - g(y)\leq \frac{1}{\alpha}\langle y - y^+, y^+ - \hat{y}_{*}(x_\epsilon)\rangle -\langle \sa{\nabla_y}\Phi(x_\epsilon,\hat{y}_{*}(x_\epsilon)),~ y - y^+  \rangle.
\end{equation}
Now, we ready to \sa{provide a useful} upper bound on  $ \max_{y\in\cY} \cL(x_\epsilon,y) -\cL(x_\epsilon,y^+)$. Indeed, given \sa{any} $\tilde{y}\in \argmax_{y\in\cY}\cL(x_\epsilon,y)$, we have
\begin{equation}\label{eq: upper bound 2}
    \begin{aligned}
             & \max_{y\in\cY} \cL(x_\epsilon,y) -\cL(x_\epsilon,y^+)
              =\cL(x_\epsilon,\tilde{y}) - \cL(x_\epsilon,\hat{y}_{*}(x_\epsilon))+ \cL(x_\epsilon,\hat{y}_{*}(x_\epsilon)) -\cL(x_\epsilon,y^+)
              \\
              & = \underbrace{\Phi(x_\epsilon,\tilde{y}) - \Phi(x_\epsilon,\hat{y}_{*}(x_\epsilon))}_{\textbf{part 1}}- g(\tilde{y}) + g(\hat{y}_{*}(x_\epsilon)) + 
              \underbrace{\Phi(x_\epsilon,\hat{y}_{*}(x_\epsilon)) -\Phi(x_\epsilon,y^+)}_{\textbf{part 2}}
              - g(\hat{y}_{*}(x_\epsilon)) + g(y^+)
              \\
              & \leq 
             \langle \nabla_y \Phi(x_\epsilon,\hat{y}_{*}(x_\epsilon)),\tilde{y} - \hat{y}_{*}(x_\epsilon)  \rangle \sa{-g(\tilde{y}) + g(y^+)}
             \\
             & \quad + \langle \nabla_y \Phi(x_\epsilon,\hat{y}_{*}(x_\epsilon)), \hat{y}_{*}(x_\epsilon) -y^+ \rangle + \frac{L_{yy}}{2}\|\hat{y}_{*}(x_\epsilon)-y^+\|^2
             \\
             & = \langle \nabla_y \Phi(x_\epsilon,\hat{y}_{*}(x_\epsilon)),\tilde{y} - y^+  \rangle - g(\tilde{y})+g(y^+) + \frac{L_{yy}}{2}\|\hat{y}_{*}(x_\epsilon)-y^+\|^2
            \\
              & \leq \frac{1}{\alpha}\langle \tilde{y} - y^+, y^+ - \hat{y}_{*}(x_\epsilon)\rangle + \frac{L_{yy}}{2}\|\hat{y}_{*}(x_\epsilon)-y^+\|^2 \\
              & =  -\frac{L_{yy}}{2}
              \| \hat{y}_{*}(x_\epsilon) -y^+ \|^2  + L_{yy} \langle \tilde{y}-\hat{y}_{*}(x_\epsilon), y^+-\hat{y}_{*}(x_\epsilon)\rangle
              \\
              &\leq L_{yy} \cD_{\cY}\|y^+-\hat{y}_{*}(x_\epsilon)\|, 
    \end{aligned}
\end{equation}
where in the first inequality, \sa{we use concavity and smoothness of $\Phi(x_\epsilon,\cdot)$ for \textbf{part 1} and \textbf{part 2}, respectively;} in the second inequality, we use \cref{eq:subgraident of g}; in the last equality, we set $\alpha=L_{yy}^{-1}$; and in the last inequality, we use Cauchy-Schwarz inequality and the fact that $\sup_{y_1,y_2\in\dom g}\|y_1-y_2\|\leq \cD_{\cY}$. Next, if we use \cref{eq: upper bound 2} within \cref{eq: upper bound 1}, it follows that
\begin{equation}
    \begin{aligned}\label{eq: bound 3}
            &\max_{y\in\cY} \cL(x_\epsilon,y) - \max_{y\in\cY} \cL(\prox{\lambda \phi}(x_\epsilon),y) - \frac{1}{2\lambda}\|\prox{\lambda \phi}(x_\epsilon)-x_\epsilon\|^2 \\
            & \leq L_{yy} \cD_{\cY}\|y^+-\hat{y}_{*}(x_\epsilon)\| 
            + \frac{\|\nabla_x \cL(x_\epsilon, y^+)\|^2}{2\mu_x}
            \\
            & \leq\hat{\mu}_y\cD_y^2  +\frac{\epsilon^2}{24\mu_x} + \frac{L^2_{xy}}{L^2_{yy}}\cdot\frac{\hat{\mu}_y^2}{\mu_x}\cdot D_y^2,
    \end{aligned}
\end{equation}
where the last inequality follows from \cref{eq: y bound WCMC} and \cref{eq: x bound WCMC} with $\alpha = L_{yy}^{-1}$. 

\sa{Finally,} if we use \cref{eq: bound 3} within \cref{eq: lower bound 1} and substitute $\lambda = (\gamma+\mu_x)^{-1}$, it follows that
\begin{equation}
    \frac{\mu_x}{2(\gamma+\mu_x)^2}\|\nabla \phi_{\lambda}(x_\epsilon) \|^2 \leq\sa{\hat{\mu}_y\cD_y^2}  +\frac{\epsilon^2}{24\mu_x} + \frac{L^2_{xy}}{L^2_{yy}}\cdot\frac{\hat{\mu}_y^2}{\mu_x}\cdot D_y^2.
\end{equation}
\sa{Thus, choosing the free parameter $\mu_x = \gamma$ implies that}
\begin{equation}
    \|\nabla \phi_{\lambda}(x_\epsilon) \|^2 \leq\sa{8\gamma\hat{\mu}_y\cD_y^2}  +\frac{\epsilon^2}{3} + 8\frac{L^2_{xy}}{L^2_{yy}}\cdot\hat{\mu}_y^2 \cdot D_y^2.
\end{equation}
\sa{Thus, we get $\|\nabla \phi_{\lambda}(x_\epsilon) \| \leq \epsilon$ for}
$\hat{\mu}_y = \min\left\{ \frac{\epsilon^2}{\sa{24}\gamma\cD_y^2},\;\frac{L_{yy}}{L_{xy}}\cdot\frac{\epsilon}{2\sqrt{6}\cD_y}\right\}$.
\end{proof}
\subsection{Proof of Theorem~\ref{Thm:WCMC}}
\begin{proof}
 To get a worst-case complexity, as in the previous sections, let
\begin{equation*}
   L \triangleq \max\{L_{xy},L_{yx},L_{xx},L_{yy}\},~\delta \triangleq \max\{\delta_x,\delta_y\},~\gamma = L. 
\end{equation*}
\cref{ASPT: lipshiz gradient} implies that $\grad_y \hat{\Phi}$ and $\grad_x \hat{\Phi}$ are Lipschitz such that for all $x,x'\in\cX$ and $y,y'\in\dom g$,
\begin{align*}
    \norm{\grad_y \hat{\Phi}(x,y)-\grad_y \hat{\Phi}(x',y')}\leq L_{yx}\norm{x-x'}+\hat{L}_{yy}\norm{y-y'},\\
    \norm{\grad_x \hat{\Phi}(x,y)-\grad_x \hat{\Phi}(x',y')}\leq L_{xx}\norm{x-x'}+L_{xy}\norm{y-y'},
\end{align*}
where $\hat{L}_{yy}= L_{yy}+\hat{\mu}_y$. Therefore,
the proof immediately follows from \cref{lem:WCMC} and \cref{cor:complexity-fg}, considering \texttt{SAPD+} with $\texttt{VR-flag}=\textbf{false}$ is applied on \eqref{eq:weakly convex-merely concave problem approx} with $\hat{\mu}_y = \min\left\{ \frac{\epsilon^2}{\sa{24}\gamma\cD_y^2},\;\frac{L_{yy}}{L_{xy}}\cdot\frac{\epsilon}{2\sqrt{6}\cD_y}\right\}$.
\end{proof}
\section{Details of fair classification example}\label{sec:experiement-details}
In the experiment of fair classification, $\{(\mathbf{a}_i,b_i)\}_{i=1}^n$ denotes the (data,label) pairs of the labeled image data set. 
$a_i\in\reals^{d_1\times d_2\times c}$, and $b_i$ is a label associated with one of the $K$-classes, i.e., $b_i\in \cC\triangleq \{\ C_j\}_{j=1}^K$ with $K\leq n$.
We employ the classifier  
$$
h(\cdot\;; \mathbf{x}): \mathbf{a}_i\in \mathbb{R}^{d_1\times d_2\times c} \rightarrow \mathbf{p}_i \in \mathbb{R}^{K},
$$
where $\mathbf{p}_i=(p_{ij})_{j=1}^K$ s.t. $\sum_{j=1}^{K}p_{ij} = 1$ and $p_{ij}\geq0$ for $j=1,2,..,K$, and $\mathbf{x}$ is the parameters of the classifier. Specifically, $h(\cdot;\mathbf{x})$ is a CNN with the structure as follows:
$$
[input]\rightarrow[conv-elu-maxpool]\times3\rightarrow [fc-elu]\times2\rightarrow [softmax]
$$
where \emph{exponential linear unit}~(elu) \cite{clevert2015fast} is the smoothed variant of \emph{rectified linear units}~(relu) activation function. Furthermore, given the input $\{(\mathbf{a}_i,b_i)\}_{i=1}^n$ and the output $\{\mathbf{p}_i\}_{i=1}^n$, the loss functions $\{l_j\}_{j=1}^K$ used in \cref{eq: experiment CNN} are
$$
\ell_j(\{(\mathbf{a}_i,b_i)\}_{i=1}^n;\mathbf{\mathbf{x}}) = -\frac{1}{N_j}\sum_{i=1}^{n}\log(p_{ij})\mathbf{1}_{C_j}(b_i)$$
where $N_j$ is the number of data with label $C_j$, i.e., $N_j=\sum_{i=1}^n\mathbf{1}_{C_j}(b_i)$ and
    $$\mathbf{1}_{C_j}(b_i) =
    \begin{cases}
    1&\text{ if } b_i=C_j
    \\
    0&\text{ o.w.}
    \end{cases}
    $$
and $p_{ij}$ is the $j$-th element of $\mathbf{p}_i$, and $\mathbf{p}_i = h(\mathbf{a}_i;\mathbf{\mathbf{x}})$.
\section{Additional analyses on the related work}
\label{sec:existing_work}
\sa{In some of the existing work on WCSC problems, particularly~\cite{huang2021efficient,huang2022accelerated,luo2020stochastic,xu2020enhanced}, except for $\kappa_y=L/\mu_y$, the individual effects of $L$ or $\mu_y$ are not explicitly stated in the final complexity bounds.}
\sa{\xzrev{To better} compare existing bounds with ours,} it is necessary to \xzrev{state the complexity bound dependence on $L$ and $\mu_y$}. For example, Huang~\emph{et al.}~\cite{huang2021efficient,huang2022accelerated} assume that {$\frac{1}{\mu_y}\leq L$}, that is equivalent to {$L\geq \sqrt{\kappa_y}$}; \sa{however, a constant factor depending on $L$ was ignored} in their oracle complexity result. Moreover, Huang \emph{et al.}\cite{huang2021efficient} employ a different convergence metric and claim that they obtain a competitive result. It turns out that their convergence metric is scaled by \sa{an algorithmic constant and when their results are converted into GNP metric, i.e., $\|\nabla\phi(\cdot)\|$, this constant adversely affects their complexity bounds}. \sa{A similar issue with the claimed complexity bounds also exists in~\cite{huang2022accelerated}, where the complexity bound are computed after the objective function is rescaled.} 
\sa{In this section, to provide a fair comparison,}
\begin{itemize}
    \item we give an explicit oracle complexity bound for the related works \sa{in~\cite{huang2021efficient,huang2022accelerated,luo2020stochastic,xu2020enhanced};}
    \item \sa{we discuss those parts in their analysis that are not convincing, and try our best to provide the corrected and optimized complexity bounds} based on their analysis. 
\end{itemize}
 Without loss of generality, \sa{for the sake of easier comparison, we consider the \emph{smooth minimax} problems,} i.e., $\min_{x\in\cX}\max_{y\in\cY}\cL(x,y)=\Phi(x,y)$.  
 \sa{We first fix \xzrev{the} notation to unify the discussion for the WCSC setting, \sa{i.e., $\cL(x,y)$ is weakly convex in $x$ and strongly concave in $y$.}} 
 
 Recall that $\phi(x)\triangleq\max_{y\in\cY}\cL(x,y)$; thus, $\phi(\cdot)$ is differentiable and we use $\|\nabla \phi(\cdot)\|$ as the convergence metric. In addition, we let $\phi_*\triangleq\inf_{x\in\cX}  \phi(x)$ and recall that $y_*(\cdot)=\argmax_{y\in\cY}\cL(\cdot,y)$.
 Moreover, for simplicity of the notation, we consider the  worst-case complexity bounds using $L$, i.e., 
\begin{equation}
\label{eq:worst-case-parameters}
   L = \max\{L_{xy},L_{yx},L_{xx},L_{yy}\},\quad\kappa_y = \frac{L}{\mu_y},\quad\delta= \max\{\delta_x,\delta_y\},\quad\gamma = L. 
\end{equation}
\subsection{Revisit of \cite[Theorem 1]{huang2021efficient}}
In this section, we provide the oracle complexity of Huang~\emph{et al.}\cite[Theorem 1]{huang2021efficient} using the metric $\|\nabla\phi(\cdot)\|$ for the Stochastic Mirror Descent Algorithm~(\texttt{SMDA}), stated in \cite[algorithm 1]{huang2021efficient}. Let $\tau,\sigma$ be the primal and dual stepsizes, respectively, $\eta$ be the momentum parameter, $b$ be the \xzrev{large} batchsize, and $u$ be convexity modulus of the Bregman distance generating function. We also list our notational convention in \cref{table:Important notations comparison NIPS} for reader's convenience.
\begin{table}[]
    \centering
    \begin{tabular}{c c c}
    \hline
    \textbf{Notation in \cite{huang2021efficient}} & \textbf{Notation in our paper} & \textbf{Meaning}\\
    \hline
    $\gamma$ & $\tau$  & primal stepsize \\
        $\lambda$ & $\sigma $ & dual stepsize \\
       $L_f$ & $L$ & Lipschitz constant as in~\eqref{eq:worst-case-parameters} \\
        $\mu$ & $\mu_y$ & concavity modulus of $\cL(x,\cdot)$\\
        $\kappa$ & $\kappa_y$ & condition number \\
        $\sigma$ & $\delta$ & variance bound for the SFO \\
        $b_1$ & $b'$ & small batch size for VR methods\\
        \hdashline
        $\rho$ & $u$ & \thead{convexity modulus of\\ Bregman distance generating function}\\
        \hline
    \end{tabular}
    \caption{Important notation for \cite{huang2021efficient} and this paper.\\ \textbf{Table notes.} (1) SFO: stochastic first-order oracle. (2) $u$ is only used in the analysis provided in this section.}
    \label{table:Important notations comparison NIPS}
\end{table}

Below, we restate the convergence result of \texttt{SMDA} for the class of Bregman distance functions such that $D_t(x,x')\xzrev{\triangleq}(x-x')^\top H_t (x-x')/2$ for some $H_t\succ 0$ --this class of Bregman functions are used for all the numerical experiments reported in~\cite{huang2021efficient}.
\begin{theorem}\label{NIPS THM1} \cite[Thoerem 1]{huang2021efficient} 
Suppose Assumptions~\ref{ASPT: fg},~\ref{ASPT: lipshiz gradient},~\ref{ASPT: unbiased noise assumption} hold with $f(\cdot)=g(\cdot)=0$. Let $\{x_t,y_t\}_{t=1}^{T}$ be generated by \texttt{SMDA}, stated in \cite[Algorithm 1]{huang2021efficient}, employing a stochastic first-order oracle to sample stochastic partial derivatives. For parameters chosen as $\eta\in(0,1]$, $\tau\in(0, \min\{ \tfrac{3u}{4L(1+\kappa_y)}, \tfrac{9\eta u \mu_y \sigma}{800\kappa_y^2}, \tfrac{2\eta \mu_y u \sigma}{25L^2}\}\Big]$ and $\sigma\in(0, 
\frac{1}{6L}]$, let $\eta_t=\eta$, $\tau_t=\tau$ and $\sigma_t=\sigma$ for $t\geq 0$. Then, for any given initial point $(x_0,y_0)$, \xzrev{it holds that}
\begin{equation}\label{eq: convergence result Huang NIPS 1}
\frac{1}{T}\sum_{t=1}^{T}\mE[\|\mathbf{G}^{t}\|] \leq \frac{4\sqrt{2( \phi(x_0)-\phi_*)}}{\sqrt{3T\tau u}} + \frac{4\sqrt{2}\Delta_0}{\sqrt{3T\tau u}} + \frac{10\delta}{\sqrt{3b} u} + \frac{20\delta\sqrt{\eta\sigma}}{3\sqrt{\tau u\mu_y b}},
\end{equation}
where $\phi(x) = \max_{y}\cL(x,y)$, $\phi_* = \inf_{x\in\cX} \phi(x)$,  $\Delta_0 = \|y_0 - y_*(x_0)\|$, $y_*(x_0) = \argmax_{y\in\cY} \cL(x_0,y)$, $\mathbf{G}^{t} =  H_t^{-1} \nabla\phi(x_t)$, and $H_t$ is a diagonal matrix such that $H_t\succeq u\mathbf{I}$ for \xzrev{all $t\geq 1$} and $u>0$.
\end{theorem}
\begin{remark}
When $f(\cdot)=g(\cdot)=0$, it follows from the update rules and the definition of Bregman distance function in  \cite[eq.(12-13),~eq.(22-23)]{huang2021efficient} that
$$
\mathbf{G}^{t} =  H_t^{-1} \nabla\phi(x_t),
$$
where $H_t$ is a diagonal matrix such that $H_t\succeq  u \mathbf{I}$. Note that 
$$
\mathbf{G}^{t} =  \nabla\phi(x_t) \iff H_t = \mathbf{I}.
$$
We noticed that the authors \xzrev{chose} the value of $u$ to improve their bounds; but, without addressing its effect on $\mathbf{G}^{t}$. More precisely, they still use $\|\mathbf{G}^{t}\|$ as the convergence metric and compare their complexity results with those papers using $\|\nabla\phi(x_t)\|$ as the convergence metric.
\end{remark}
In the following corollary, we will provide the optimal complexity for \texttt{SMDA} based the result in~\cref{eq: convergence result Huang NIPS 1}, i.e., \cite[Thoerem 1]{huang2021efficient}.
\begin{corollary}
Suppose Assumptions~\ref{ASPT: fg},~\ref{ASPT: lipshiz gradient},~\ref{ASPT: unbiased noise assumption} hold with $f(\cdot)=g(\cdot)=0$, and $\frac{1}{\mu_y}\leq L$ hold\footnote{The assumption $\frac{1}{\mu_y}\leq L$ is also made in \cite{huang2021efficient}.}. 
Consider the setting of \cref{NIPS THM1}, then \texttt{SMDA}~\cite[Algorithm 1]{huang2021efficient} can generate $x_\epsilon$ such that $\mathbb{E}\left[\|\nabla\phi(x_\epsilon)\| \right]\leq\epsilon$ by requiring at most $\mathcal{O}(\frac{\kappa_y^5\delta^2}{\mu_y ^2\epsilon^4})$ stochastic first-order oracle calls.
\end{corollary}
\begin{proof}
Recall that $H_t\succeq u\mathbf{I}$, $\mathbf{G}^{t} =  H_t^{-1} \nabla\phi(x_t)$ and $H_t$ is a diagonal matrix; therefore, we can obtain a tight upper bound on $\mE[\|\nabla\phi(x_t)\|]$ using \cref{eq: convergence result Huang NIPS 1} as follows:
\begin{equation}\label{eq: convergence result Huang NIPS 2}
\begin{aligned}
\frac{1}{T}\sum_{t=1}^{T} \mE[\|\nabla\phi(x_t)\|] \leq 
\frac{4\sqrt{2( \phi(x_0)-\phi_*)}}{\sqrt{3T}}\sqrt{\frac{ u}{\tau}} + \frac{4\sqrt{2}\Delta_0}{\sqrt{3T}}\sqrt{\frac{ u}{\tau}} + \frac{ 10\delta}{\sqrt{3b}} + \frac{ 20\delta\sqrt{\eta\sigma}}{3\sqrt{\mu_y  b}}\sqrt{\frac{ u}{\tau}}.
\end{aligned}
\end{equation}
If we use their parameter choices, i.e., $\eta\in(0,1]$,
\begin{equation}
\label{eq:Huang-param-choice}
    \sigma= \mathcal{O}\Big(\frac{1}{L}\Big),\quad \tau = u~\min\Big\{\frac{3 }{4L(1+\kappa_y)},~\frac{9\eta \mu_y\sigma}{800\kappa_y^2},~\frac{2\eta\mu_y  \sigma}{25L^2}\Big\},\quad u = \mathcal{O}(L^{\nu}),
\end{equation}
for some free parameter $\nu\geq0$, then we get
\begin{equation}
\label{eq:u-effect}
    \frac{ u}{\tau} = \max\Big\{\frac{4L(1+\kappa_y)}{3},\frac{800\kappa_y^2}{9\eta\mu_y\sigma},\frac{25L^2}{2\eta\mu_y\sigma}\Big\} =\Omega(\kappa_y^3),
\end{equation}
where the second term leads to $\kappa_y^3$. It is essential to note that $\tau$ choice in \eqref{eq:Huang-param-choice} implies that $u/\tau$ ratio is independent of $u$; hence, the parameter $u$ indeed does not affect the bound on the right-hand-side of \cref{eq: convergence result Huang NIPS 2}. Therefore, contrary to what is suggested in \cite{huang2021efficient}, choosing different values for $u$ through picking different $\nu\geq 0$ values indeed is not useful for proving tighter bounds in GNP metric $\norm{\grad\phi(x_k)}$ in this simple scenario using their parameter choices.

Note \cref{eq: convergence result Huang NIPS 2} can be simplified as
\begin{equation*}
    \frac{1}{T}\sum_{t=1}^{T}\mE[ \|\nabla\phi(x_t)\|] \leq \cO\Big( \sqrt{\frac{\kappa_y^3(\phi(x_0)-\phi_*)}{T}}+\frac{\delta}{\sqrt{b}} + \frac{\delta\kappa_y^2}{\sqrt{b}L}\Big).
\end{equation*}
Thus, for any $\epsilon>0$, to find point $x_t$ such that $\mE[\|\nabla\phi(x_t)\|]\leq\epsilon$, one should choose $t\geq T$ for
$$
T = \mathcal{O}\Big(\frac{\kappa_y^3}{\epsilon^2}(\phi(x_0)-\phi_*)\Big),\quad b = \mathcal{O}\Big(\frac{\kappa_y^4\delta^2}{L^2\epsilon^2}\Big),
$$
which leads to the oracles complexity of $$2bT=\mathcal{O}\Big(\frac{\kappa_y^7\delta^2}{L^2\epsilon^4}\Big)=\mathcal{O}\Big(\frac{\kappa_y^5\delta^2}{\mu_y ^2\epsilon^4}\Big).$$
\end{proof}

\subsection{Revisit of \cite[Theorem 3]{huang2021efficient}}
 
 In this section, we provide the oracle complexity of Huang~\emph{et al.}~\cite[Theorem 3]{huang2021efficient} using the metric $\|\nabla\phi(\cdot)\|$ for the Stochastic Mirror Descent Algorithm with variance reduction~(\texttt{SMDA-VR}), stated in \cite[algorithm 2]{huang2021efficient}. Let $\tau,\sigma$ be the primal and dual stepsizes, respectively, $\eta$ be the momentum parameter, $b$ be the large batchsize, $b'$ be the small batchsize, $q$ be the period for sampling large batch size (i.e., once every $q$ batches is large), and $u$ be the strongly-convex constant of the Bregman distance generating function. We also list our notational convention in \cref{table:Important notations comparison NIPS} for reader's convenience.
 
 Below, as we did in the previous section for \texttt{SMDA}, we restate the convergence result of \texttt{SMDA-VR} for the class of Bregman distance functions such that $D_t(x,x')=(x-x')^\top H_t (x-x')/2$ for some $H_t\succ 0$ --this class of Bregman functions are used for all the numerical experiments reported in~\cite{huang2021efficient}.
\begin{theorem}\label{NIPS THM2} \cite[Thoerem 3]{huang2021efficient} 
Suppose Assumptions~\ref{ASPT: fg},~\ref{ASPT: lipshiz gradient},~\ref{ASPT: unbiased noise assumption} hold with $f(\cdot)=g(\cdot)=0$. Let $\{x_t,y_t\}_{t=1}^{T}$ be generated by \texttt{SMDA-VR}, stated in \cite[Algorithm 2]{huang2021efficient}, employing a stochastic first-order oracle to sample stochastic partial derivatives. For parameters chosen as $\eta\in(0,1]$, $\tau=(0,~\min\Big\{ 
\tfrac{3u}{4L(1+\kappa_y)},
\tfrac{\eta \mu_y  \sigma u}{38L^2},
\frac{3u}{19L^2\eta},
\frac{u\eta}{8},
\tfrac{9 u \eta\mu_y \sigma}{400\kappa_y^2},
\Big\}\Big]$ and $\sigma\in(0,~\min\Big\{\frac{1}{6L}, \frac{9\mu_y}{100\eta^2L^2}\Big\}\Big]$, let $\eta_t=\eta$, $\tau_t=\tau$ and $\sigma_t=\sigma$ for $t\geq 0$ and $b'=q$. Then, for any given initial point $(x_0,y_0)$, we have
\begin{equation}\label{eq: convergence result Huang NIPS 3}
\frac{1}{T}\sum_{t=1}^{T}\mE[\|\mathbf{G}^{t}\|] 
\leq 
\frac{4\sqrt{2(\phi(x_0)-\phi_*)}}{\sqrt{3T\tau u}} 
+ \frac{4\sqrt{2}\Delta_0}{\sqrt{3T\tau u}} 
+ \frac{2\sqrt{2}\delta}{\sqrt{\tau u\eta b}L}.
\end{equation}
where $\phi(x) = \max_{y}\cL(x,y)$, $\phi_* = \inf_{x\in\cX} \phi(x)$,  $\Delta_0 = \|y_0 - y_*(x_0)\|$, $y_*(x_0) = \argmax_{y\in\cY} \cL(x_0,y)$,  $\mathbf{G}^{t} =  H_t^{-1} \nabla\phi(x_t)$, and $H_t$ is a diagonal matrix such that $H_t\succeq u\mathbf{I}$ for some $u>0$.
\end{theorem}
In the following corollary, we will provide the optimal complexity for \texttt{SMDA-VR} based the result in~\cref{eq: convergence result Huang NIPS 3}, i.e., \cite[Thoerem 3]{huang2021efficient}.
\begin{corollary}
Suppose Assumptions~\ref{ASPT: fg},~\ref{ASPT: lipshiz gradient},~\ref{ASPT: unbiased noise assumption} hold with $f(\cdot)=g(\cdot)=0$, and $\frac{1}{\mu_y}\leq L$ hold\footnote{The assumption $\frac{1}{\mu_y}\leq L$ is also made in \cite{huang2021efficient}.}.  Consider the setting of \cref{NIPS THM2}, then \texttt{SMDA-VR}~\cite[Algorithm 2]{huang2021efficient} can generate $x_\epsilon$ such that $\mathbb{E}\left[\|\nabla\phi(x_\epsilon)\| \right]\leq\epsilon$ by requiring at most $\mathcal{O}(\frac{\kappa_y^5\delta^2}{\mu_y\epsilon^3})$ stochastic first-order oracle calls.
\end{corollary}

\begin{proof}
Recall that $H_t\succeq u\mathbf{I}$, $\mathbf{G}^{t} =  H_t^{-1} \nabla\phi(x_t)$ and $H_t$ is a diagonal matrix; therefore, we can obtain a tight upper bound on $\mE[\|\nabla\phi(x_t)\|]$ using \cref{eq: convergence result Huang NIPS 3} as follows:
\begin{equation}\label{eq: convergence result Huang NIPS 4}
\begin{aligned}
\frac{1}{T}\sum_{t=1}^{T} \mE[\|\nabla\phi(x_t)\|]\leq \frac{4\sqrt{2(\phi(x_0)-\phi_*)}}{\sqrt{3T}}\sqrt{\frac{ u}{\tau}} + \frac{4\sqrt{2}\Delta_0}{\sqrt{3T}}\sqrt{\frac{ u}{\tau}}  + \frac{2\sqrt{2}\delta}{\sqrt{\eta b}L}\sqrt{\frac{ u}{\tau}}.
\end{aligned}
\end{equation}
If we use their parameter choices, i.e., $\eta\in(0,1]$,
\begin{equation}
\label{eq:Huang-param-choice-vr}
    \sigma = \mathcal{O}\Big(\frac{1}{\kappa_y L}\Big),~\tau = u~\min\Big\{\frac{3 }{4L(1+\kappa_y)},
    \frac{\eta\mu_y\sigma}{38L^2},
    \frac{3}{19L^2\eta},
    \frac{\eta}{8},
    \frac{9\eta\mu_y\sigma}{400\kappa_y^2}\Big\},~ u = \mathcal{O}(L^{1+\nu}),
\end{equation}
for some free design parameter $\nu\geq0$, then we get
\begin{equation*}
    \frac{ u}{\tau} = \max\Big\{\frac{4L(1+\kappa_y)}{3},\frac{38 L^2}{\eta\mu_y\sigma},
    \frac{19L^2\eta}{3},
    \frac{8}{\eta},
    \frac{400\kappa_y^2}{9\eta\mu_y\sigma}\Big\} =\Omega(\kappa_y^4),
\end{equation*}
 where the last term leads to $\kappa_y^4$. It is essential to note that $\tau$ choice in \eqref{eq:Huang-param-choice-vr} implies that $u/\tau$ ratio is independent of $u$; hence, the parameter $u$ indeed does not affect the bound on the right-hand-side of \cref{eq: convergence result Huang NIPS 4}. Therefore, contrary to what is suggested in \cite{huang2021efficient}, \xzrev{for the simple scenarios considered here choosing} different values for $u$ through picking different $\nu\geq 0$ values indeed is not useful for proving tighter bounds in GNP metric $\norm{\grad\phi(x_k)}$ \xzrev{with} their parameter choices.
 
 Note \xzrev{that} \cref{eq: convergence result Huang NIPS 4} can be simplified as
\begin{equation*}
    \frac{1}{T}\sum_{t=1}^{T} \mE[\|\nabla\phi(x_t)\|] \leq \cO\Big(\sqrt{\frac{\kappa_y^4(\phi(x_0)-\phi_*)}{T}} + \delta\frac{\kappa_y^2}{\sqrt{b}L}\Big).
\end{equation*}
Thus, for any $\epsilon>0$, to find point $x_t$ such that $\mE[\|\nabla\phi(x_t)\|]\leq\epsilon$, one should choose $t\geq T$ for
$$
T = \mathcal{O}\Big(\frac{\kappa_y^4(\phi(x_0)-\phi_*)}{\epsilon^2}\Big),\quad b = \mathcal{O}\Big(\frac{\kappa_y^4\delta^2}{L^2\epsilon^2}\Big),
$$
which leads to the oracle complexity of 
$$4b'T + 2bT/q=\mathcal{O}\Big(\frac{b'\kappa_y^4}{\epsilon^2} + \frac{\kappa_y^8\delta^2}{L^2\epsilon^4}/q\Big).$$
Since their parameter choice requires $b'=q$, to optimize the above bound, we let $b'=q=\mathcal{O}\left(\frac{\kappa_y^2}{L\epsilon}\right)$, which leads to
$$
\mathcal{O}\left(\frac{\kappa_y^6\delta^2}{L\epsilon^3}\right) =\mathcal{O}\left(\frac{\kappa_y^5\delta^2}{\mu_y\epsilon^3}\right).
$$
\end{proof}


 \subsection{
 Revisit of \cite[Theorem~1]{luo2020stochastic}
 }
Recall that $\phi(x)=\max_{y}\cL(x,y)$ and $\phi_* = \inf_{x\in\cX} \phi(x)$.
In this paper, the total oracle complexity to find point $x_\epsilon$ such that $\mE[\|\nabla\phi(x_\epsilon)\|]\leq\epsilon$  is given by
\begin{equation}\label{eq: complexity luo2020 1}
    \mathcal{O}(\kappa_y^2\epsilon^{-2}\log(\kappa_y/\epsilon)) + \mathcal{O}(T/q\cdot \sa{b}) + \mathcal{O}(T\cdot \sa{b'} \cdot m)
\end{equation}
where\footnote{In \cite{luo2020stochastic}, \xzrev{there is a typo in the choice} of $b = \lceil \frac{2250}{19}\delta^2\kappa_y^{-2}\epsilon^{2}\rceil$. Here, we provide the correct one.} 
\begin{equation}
\label{eq:SREDA-parameters}
    \begin{aligned}
            T = \Big\lceil \frac{100\kappa_y L\Delta_f}{9\epsilon^{2}}\Big\rceil,~q = \lceil \epsilon^{-1} \rceil,\ b = \lceil \frac{2250}{19}\delta^2\kappa_y^{2}\epsilon^{-2}\rceil,\ b' = \Big\lceil \frac{3687}{76}\kappa_y q \Big\rceil,\ m = \lceil 1024\kappa_y \rceil.
    \end{aligned}
\end{equation}
Given an arbitrary initial point $x_0$, let $y_0$ be obtained by inexactly solving $\max_y \cL(x_0,y)$, and they define $\Delta_f = \cL(x_0,y_0) - \frac{134\epsilon^2}{\kappa_yL} - \phi_*$.
In~\eqref{eq:SREDA-parameters}, the other parameters are defined as follows:
$b$ is the large batchsize, $b'$ is the small batchsize, $q$ is the period such that once every $q$ outer iterations, \texttt{SREDA} calls for a large batchsize,
$T$ is the number of the outer iterations and $m$ is the number of the inner iterations --each outer iteration requires $m$ inner iterations and each inner iteration calls for a small batchsize. Then \cref{eq: complexity luo2020 1} becomes $\mathcal{O}(\frac{L\kappa_y^3}{\epsilon^3})$. 

\subsection{
Revisit of \cite[Theorem~1]{xu2020enhanced}
}

Recall that $\phi(x)=\max_{y}\cL(x,y)$ and $\phi_* = \inf_{x\in\cX} \phi(x)$. In this paper, the precise parameter selection for \cite[Theorem~1]{xu2020enhanced} is provided in \cite[Theorem~3]{xu2020enhanced} of the supplementary material. Using these parameter choice implies that the total oracle complexity to find point $x_\epsilon$ such that $\mE[\|\nabla\phi(x_\epsilon)\|]\leq\epsilon$ is given by
\begin{equation}\label{eq: complexity xu2020 1}
    T\cdot b' \cdot m + \Big\lceil \frac{T}{q} \Big\rceil \cdot b + T_0,
\end{equation}
for an arbitrary initial point $x_0$, where the number of outer iterations, $T$,  the number of the inner iterations per each outer iteration, $m$, are set as follows:
\begin{equation*}
\begin{aligned}
    &T = \max\Big\{\frac{3345\kappa_y}{\epsilon^2},\ 6600(1+\kappa_y)L\frac{(\phi(x_0)-\phi_*)}{\epsilon^2}\Big\},\quad b = \frac{9366\delta^2\kappa_y^2}{\epsilon^2},
    \\&b' = \frac{\kappa_y}{\epsilon},\quad m = 52\kappa_y-1,\quad q = \frac{2}{13(1+\kappa_y)}\frac{\kappa_y}{\epsilon},\quad T_0=\cO(\kappa_y\log(\kappa_y)).
\end{aligned}
\end{equation*}
Above $b$ is the large batchsize, $b'$ is the small batchsize, $q$ is the period such that once every $q$ outer iterations, a large batch size is sampled rather than a small batch size. Then \cref{eq: complexity xu2020 1} becomes $\mathcal{O}(\frac{L\kappa_y^3}{\epsilon^3})$.

 \subsection{
 Revisit of 
\cite[Theorem~12]{huang2022accelerated}
 }
In this section, we provide the oracle complexity of \cite[Theorem 12]{huang2022accelerated} using the metric $\|\nabla\phi(\cdot)\|$ for the Accelerated first-order Momentum Descent Ascent~(\texttt{ACC-MDA}) algorithm, stated in \cite[algorithm 3]{huang2022accelerated}. Let $\tau,\sigma$ be the primal and dual stepsizes, respectively, $\{\eta_t\}$ be the momentum parameter sequence, and $b$ be the batchsize. We also list our notational convention in \cref{table:Important notations comparison JMLR} for reader's convenience.

\begin{table}[h]
    \centering
    \begin{tabular}{c c c}
    \hline
    Notations in \cite{huang2022accelerated} & Notations in our paper & Meaning\\
    \hline
    $\gamma$ & $\tau$  & primal stepsize \\
        $\lambda$ & $\sigma $ & dual stepsize \\
       $L_f$ & $L$ & Lipschitz constant as in~\eqref{eq:worst-case-parameters} 
       \\
       $L_g$ & $L(1+\kappa_y)$ &
       $L$-smooth constant of $\phi(x)$
       \\
        $\tau$ & $\mu_y$ &  concavity modulus of $\cL(x,\cdot)$\\
        \hline
    \end{tabular}
    \caption{Important notations for \cite{huang2022accelerated} and this paper.}
    \label{table:Important notations comparison JMLR}
\end{table}

Below, we restate the convergence result of \texttt{ACC-MDA} reported in \cite{huang2022accelerated}.
\begin{theorem}\label{JMRL THM12} \cite[Thoerem 12]{huang2022accelerated} 
Suppose Assumptions~\ref{ASPT: fg},~\ref{ASPT: lipshiz gradient},~\ref{ASPT: unbiased noise assumption} hold with $f(\cdot)=g(\cdot)=0$. Let  $\{x_t,y_t\}_{t=1}^{T}$ be generated by \texttt{ACC-MDA} algorithm, stated in~\cite[Algorithm 3]{huang2022accelerated}, when applied to the \emph{smooth minimax} problem $\min_{x\in\cX}\max_{y\in\cY}\cL(x,y)=\Phi(x,y)$. For some given $p>0$, let 
$\eta_t = \tfrac{p}{(\psi+t)^{1/3}}$ for all $t\geq 0$, $\tau\in (0,\ \min\{ \tfrac{\sigma\mu_y}{2L}\sqrt{\tfrac{2b}{8\sigma^2+75\kappa_y^2b}},\tfrac{\psi^{1/3}}{2L(1+\kappa_y)p}\}]$ and $\sigma\in(0, \min\{\tfrac{1}{6L}, \tfrac{27b\mu_y}{16}\}]$
such that $\psi\geq \max\{2, p^3,(c_1 p)^3, (c_2 p)^3\}$ for some $c_1\geq \tfrac{2}{3p^3} + \tfrac{9\mu_y^2}{4}$ 
and
$c_2\geq \tfrac{2}{3p^3} + \tfrac{75L^2}{2}$. Then for any given $x_0$, we have
\begin{equation}\label{eq: convergence result JMLR 1}
\frac{1}{T}\sum_{t=1}^{T}\mE[\|\nabla\phi(x_t)\|] 
\leq 
\frac{\sqrt{2M''}\psi^{1/6}}{T^{1/2}} + \frac{\sqrt{2M''}}{T^{1/3}},
\end{equation}
where $\phi(x) = \max_{y}\cL(x,y)$, $\Delta_0 = \|y_0 - y_*(x_0)\|^2$, $y_*(x_0) = \argmax_{y\in\cY} \cL(x_0,y)$, $M''=\frac{\phi(x_0)-\phi_*}{\tau p} + \frac{9L^2\Delta_0}{p\sigma\mu_y} + \frac{2\psi^{1/3}\delta^2}{b\mu_y^2 p^2} + \frac{2(c_1^2+c_2^2)\delta^2p^2}{b\mu_y^2}\ln(\psi+T)$, and $\phi_* = \inf_{x\in\cX} \phi(x)$.
\end{theorem}


\begin{remark}\cite[Remarks 13 and 14]{huang2022accelerated}
When $b=\cO\big(\kappa_y^{\nu}\big)$ for $\nu>0$ and $\kappa_y^{\nu}\leq \tfrac{8}{81 L\mu_y}$, they claim that they can obtain the gradient complexity of $\tilde{\mathcal{O}}(\kappa_y^3\epsilon^{-3})$ if $\nu=3$, and $\tilde{\mathcal{O}}(\kappa_y^{2.5}\epsilon^{-3})$ if $\nu=4$. However, for the assumption $\kappa_y^\nu\leq \frac{8}{81 L \mu_y}$ to hold in general, one needs to rescale the original objective function $\cL(x,y)$ with some $s\in(0,1]$ to define
\begin{align}
\label{eq:Ls}
\cL_s(x,y) \triangleq s\cdot\cL(x,y).
\end{align}
Then the Lipschitz constant of $\grad \cL_s$, strongly concavity modulus of $\cL(x,\cdot)$ for any $x\in\cX$ and the variance bound of the stochastic oracle for $\grad_x\cL_s$ and $\grad_y\cL_s$ can be written as $sL$, $s \mu_y$, and $s^2\delta^2$, respectively. We notice that the effect of scaling $\cL$ on the problem parameters is not discussed in \cite{huang2022accelerated} and \cref{eq: convergence result JMLR 1} is directly used to derive the convergence result assuming $\kappa_y^\nu\leq \tfrac{8}{81L\mu_y}$. As a consequence, the complexity results of $\tilde{\mathcal{O}}(\kappa_y^3\epsilon^{-3})$, $\tilde{\mathcal{O}}(\kappa_y^{2.5}\epsilon^{-3})$ do not hold without loss of generality unless the original function $\cL$ satisfies the restrictive assumption of $\kappa_y^\nu\leq \tfrac{8}{81L\mu_y}$.
\end{remark}
In the following discussion, we analyze the effect of scaling $\cL$ on the complexity bounds whenever $\kappa_y^\nu\leq \tfrac{8}{81L\mu_y}$ is not satisfied for the original objective function $\cL$, and we provide complexity bounds holding without loss of generality that are optimized by choosing $\nu>0$ properly. Now, consider implementing \texttt{ACC-MDA} on an appropriately scaled problem $\min_x\max_y\cL_s(x,y)$ where $\cL_s$ is defined in~\eqref{eq:Ls}. Let
\begin{equation}\label{eq: rescale constants}
    L_s \triangleq sL,\quad \mu_s \triangleq s\mu_y,\quad \delta_s \triangleq s\delta.
\end{equation}
Note that the condition numbers of $\cL_s$ and $\cL$ are the same, and are equal to $\kappa_y$, i.e., $\kappa_y=\frac{L}{\mu_y}=\frac{L_s}{\mu_s}$. In the upcoming discussion, suppose that $s\in(0,1]$ is chosen such that $\kappa_y^\nu\leq \tfrac{8}{81 L_s\mu_s}$.

To facilitate the complexity analysis and make the upcoming discussion easier, first we restate \cref{JMRL THM12} for the function  $\cL_s$, where we used the relation $\grad \phi$ and the derivative of $\max_y\cL_s(\cdot,y)$; indeed, the derivative of $\max_y\cL_s(\cdot,y)$ is equal to $s \grad \phi(\cdot)$, where $\phi(x) = \max_{y}\cL(x,y)$.
\begin{theorem}\label{JMRL THM12 v2} \cite[Thoerem 12]{huang2022accelerated} 
Suppose Assumptions~\ref{ASPT: fg},~\ref{ASPT: lipshiz gradient},~\ref{ASPT: unbiased noise assumption} hold with $f(\cdot)=g(\cdot)=0$. Let  $\{x_t,y_t\}_{t=1}^{T}$ be generated by \texttt{ACC-MDA} algorithm, stated in~\cite[Algorithm 3]{huang2022accelerated}, when applied to the \emph{smooth minimax} problem $\min_{x\in\cX}\max_{y\in\cY}\cL_s(x,y)=s\cdot\cL(x,y)$. 
For some given $p\geq0$,
let 
$\eta_t = \tfrac{p}{({\psi}+t)^{1/3}}$ for all $t\geq 0$, 
$\tau\in (0, \min\{ \tfrac{\sigma\mu_s}{2L_s}\sqrt{\tfrac{2b}{8\sigma^2+75\kappa_y^2b}},\tfrac{{\psi}^{1/3}}{2L_s(1+\kappa_y)p} \}]$
and
$\sigma\in(0, \min\{\tfrac{1}{6L_s}, \tfrac{27b\mu_s}{16}\}]$ such that
${\psi}\geq \max\{2, p^3,(c'_{1} p)^3, (c'_{2} p)^3 \}$ for some
$c'_{1}\geq \tfrac{2}{3p^3} + \tfrac{9\mu_s^2}{4}$ 
and
$c'_{2}\geq \tfrac{2}{3p^3} + \tfrac{75L_s^2}{2}$.
Then for any given $x_0$,
we have
\begin{equation}\label{eq: convergence result JMLR v2}
\frac{1}{T}\sum_{t=1}^{T}\mE[\|\nabla\phi(x_t)\|] 
\leq 
\frac{1}{s}\left(\frac{\sqrt{2M_s''}{\psi}^{1/6}}{T^{1/2}} + \frac{\sqrt{2M_s''}}{T^{1/3}}\right),
\end{equation}
where $\phi(x) = \max_{y}\cL(x,y)$, $\Delta_0 = \|y_0 - y_*(x_0)\|^2$, $y_*(x_0) = \argmax_{y\in\cY} \cL(x_0,y)$, $M_s''=\frac{s(\phi(x_0)-\phi_*)}{\tau p} + \frac{9L_s^2\Delta_0}{p\sigma\mu_s} + \frac{2{\psi}^{1/3}\delta_s^2}{b\mu_s^2 p^2} + \frac{2({c'_{1}}^2+{c'_{2}}^2)\delta_s^2p^2}{b\mu_s^2}\ln({\psi}+T)$, and $\phi_* = \inf_{x\in\cX} \phi(x)$.
\end{theorem}
Next, following the analysis in \cite[Remarks 10 and 13]{huang2022accelerated}, we provide a particular parameter choice for \texttt{ACC-MDA} so that it is applicable to the setting of \cref{JMRL THM12 v2}.
\begin{lemma}\label{eq: parameter JMLR}
Under the premise of \cref{JMRL THM12 v2}. Suppose  $\kappa_y^{\nu}\leq\frac{8}{81 L_s\mu_s}$, $b=\kappa_y^{\nu}$ for some $\nu>0$, and
\begin{equation}
    \label{eq: suggested choice}
    \frac{\sigma\mu_s}{2L_s}\sqrt{\frac{2b}{8\sigma^2+75\kappa_y^2b}}\leq\frac{{\psi}^{1/3}}{2L_s(1+\kappa_y)p}.
\end{equation}
If 
$\sigma = \min\{\tfrac{1}{6L_s}, \tfrac{27b\mu_s}{16}\}$
and 
$\tau = \min\{ \tfrac{\sigma\mu_s}{2L_s}\sqrt{\tfrac{2b}{8\sigma^2+75\kappa_y^2b}},\tfrac{{\psi}^{1/3}}{2L_s(1+\kappa_y)p} \}$, then ${\psi} = \Theta(\max\{1, L_s^6\})$ satisfies the condition in \cref{JMRL THM12 v2} and 
\begin{equation}
        \sigma=\Theta(b\mu_s),\quad
        \tau^{-1} = \Theta \Big(\frac{\kappa_y^3}{bL_s }\Big).
\end{equation}
\end{lemma}
\begin{remark}
The conditions $\kappa_y^{\nu}\leq \frac{8}{81L_s\mu_s}$ and \cref{eq: suggested choice},
and the choice $b=\kappa_y^{\nu}$ are as suggested in \cite{huang2022accelerated}.
\end{remark}
\begin{proof}
Since $\kappa_y^{\nu}\leq \frac{8}{81L_s\mu_s}$, we have $\sigma=\frac{27b\mu_s}{16}=\Theta(b\mu_s)$. Furthermore, \cref{eq: suggested choice} implies that we can simplify $\tau$ as 
$$
\tau= \frac{\sigma\mu_s}{2L_s}\sqrt{\frac{2b}{8\sigma^2+75\kappa_y^2b}}.
$$
Then it follows that,
\begin{equation*}
    \label{eq: tau bound JMLR during proof}
    \tau^{-1} = \Theta \Big(\frac{\kappa_y}{b\mu_s} \sqrt{b\mu_s^2 + \kappa_y^2} \Big) =\Theta \big(\frac{\kappa_y^2}{b L_s}(\kappa_y^{\nu/2}\mu_y +\kappa_y) \Big) = \Theta \Big(\frac{\kappa_y^2}{b }\Big(\kappa_y^{\nu/2-1} +\frac{\kappa_y}{L_s}\Big) \Big)= \Theta \Big(\frac{\kappa_y^3}{bL_s }\Big),
\end{equation*}
where we use the relation $\kappa^{\nu}_y\leq \frac{8}{81L_s\mu_s}\Rightarrow L_s^2\leq\frac{8}{81} \kappa_y^{1-\nu}$ for the last equality.
Next, from \cref{eq: suggested choice} and the requirement on $\psi$ in \cref{JMRL THM12 v2}, a sufficient condition $\psi$ is
\begin{equation}\label{eq: m bound 1}
    {\psi} \geq  \max\left\{2,\ p^3,\ (c'_{1} p)^3,\ (c'_{2} p)^3,\  \left(\sigma\mu_s(1+\kappa_y)p \sqrt{\frac{2b}{8\sigma^2+75\kappa_y^2b}}\right)^3 \right\},
\end{equation}
Now we consider the components of $\max$ operator in \eqref{eq: m bound 1}. Note positive constant $p$ can be chosen independent of other problem parameters, e.g., $p=1$. Furthermore, the requirement on $c_1',c_2'$ can be satisfied for
\begin{equation}
\label{eq:c-choice}
    \begin{aligned}
    c'_{1} = \Theta\big( \mu_s^2 \big),\quad c'_{2} = \Theta(L_s^2).
    \end{aligned}
\end{equation}
Finally, using $\sigma=\frac{27b\mu_s}{16}$ together with $b=\kappa_y^{\nu}$ yields that
\begin{equation*}
    \begin{aligned}
    & \sigma\mu_s(1+\kappa_y)p \sqrt{\frac{2b}{8\sigma^2+75\kappa_y^2b}}  = \Theta \Big(\kappa_y b\mu_s^2 \sqrt{\frac{1}{b\mu_s^2 + \kappa_y^2}}\Big) = \Theta \Big( L_s^2\kappa_y^{\nu-1}\sqrt{\frac{1}{\kappa_y^{\nu-2}L_s^2 + \kappa_y^2}}\Big)
    \\
    & = \Theta \Big( L_s^2\kappa_y^{\nu-1}\sqrt{\frac{1}{\kappa_y^2}}\Big) \leq \Theta(1),
    \end{aligned}
\end{equation*}
where we use the relation $\kappa^{\nu}_y\leq \frac{8}{81L_s\mu_s}\Rightarrow L_s^2\leq\frac{8}{81} \kappa_y^{1-\nu}$ for the last equality and the last inequality. Therefore, using the above relations within \cref{eq: m bound 1}, we observe that one can set
\begin{equation}
\label{eq:psi-choice}
    {\psi}^{1/3} = \Theta(\max\{1, L_s^2\}),
\end{equation}
which completes the proof.
\end{proof}
 
Next, we will use the parameters in \cref{eq: parameter JMLR} to provide an optimized complexity for \texttt{ACC-MDA} \cite[Algorithm 12]{huang2022accelerated} to generate $x_\epsilon$ such that $\mathbb{E}\left[\|\nabla\phi(x_\epsilon)\| \right]\leq\epsilon$.
\begin{corollary}
Suppose Assumptions~\ref{ASPT: fg},~\ref{ASPT: lipshiz gradient},~\ref{ASPT: unbiased noise assumption} hold with $f(\cdot)=g(\cdot)=0$, and  $\kappa_y^\nu> \tfrac{8}{81L\mu_y}$ for the original function $\cL$. Running \texttt{ACC-MDA} on $\min_x\max_y\cL_s(x,y)$ for  
\begin{equation}\label{eq: choice of s}
    s = \frac{2\sqrt{2}}{9}\frac{1}{L}\kappa_y^{(1-\nu)/2},
\end{equation}
and $b=\kappa_y^{\nu}$, one can generate $x_\epsilon$ such that $\mathbb{E}\left[\|\nabla\phi(x_\epsilon)\| \right]\leq\epsilon$ requiring at most  $\tilde{O}(\frac{L^{1.5}\kappa_y^{3.5}}{\epsilon^3})$ stochastic first-order oracle calls.
\end{corollary}

\begin{proof}
It follows from \cref{JMRL THM12 v2} that
that
\begin{equation}\label{eq: convergence result JMLR 2}
\frac{1}{T}\sum_{t=1}^{T}\mE[\|\nabla\phi(x_t)\| ]
\leq 
\frac{\sqrt{2M_s''}}{s}\Big(\frac{{\psi}^{1/6}}{T^{1/2}} + \frac{1}{T^{1/3}}\Big).
\end{equation}
Based on \cref{eq: parameter JMLR}, let ${\psi}=\Theta(\max\{1, L_s^6\})$; thus, $\frac{{\psi}^{1/6}}{T^{1/2}} \leq \frac{1}{T^{1/3}}$ when $T$ is large enough. Therefore, for all sufficiently small $\epsilon>0$, $\frac{1}{T^{1/3}}\leq \frac{s}{\sqrt{2M_s''}}\cdot \frac{\epsilon}{2}$ implies that $\frac{\psi^{1/6}}{T^{1/2}}\leq \frac{s}{\sqrt{2M_s''}}\cdot\frac{\epsilon}{2}$, and we get
\begin{equation}\label{eq: convergence eq JMLR}
    \frac{1}{s}\frac{\sqrt{2M_s''}}{T^{1/3}}\leq\frac{\epsilon}{2} \implies \min_{t\in\{0,\ldots,T\}}\mE[\|\nabla\phi(x_t)\|]=\epsilon.
\end{equation}
Moreover, note that \cref{eq: choice of s} implies $\kappa_y\leq \frac{8}{81 L_s\mu_s}$; thus, we can  choose $\tau,\sigma,b$ as in \cref{eq: parameter JMLR} which satisfy
\begin{equation*}
    \begin{aligned}
    &\sigma=\Theta(b\mu_s),\quad
        \tau^{-1} = \Theta \big(\frac{\kappa_y^3}{bL_s }\big), \quad
    b = \kappa_y^\nu.
\end{aligned}
\end{equation*}
Recall that $c'_{1}$, $c'_{2}$ chosen as in \eqref{eq:c-choice} and $\psi$ chosen as in \eqref{eq:psi-choice} satisfy all the required conditions in \cref{JMRL THM12 v2}; hence,
$M_s''=\frac{s(\phi(x_0)-\phi_*)}{\tau p} + \frac{9L_s^2\Delta_0}{p\sigma\mu_s} + \frac{2{\psi}^{1/3}\delta_s^2}{b\mu_s^2 p^2} + \frac{2({c'_{1}}^2+{c'_{2}}^2)\delta_s^2p^2}{b\mu_s^2}\ln({\psi}+T)$ implies that
\begin{equation}
\begin{aligned}
    \frac{1}{s^2}M_s''  &= \Theta \Big(
    \frac{\kappa_y^3}{sbL_s}
    +  
    \frac{\kappa_y^2}{s^2b} + \frac{\delta_s^2}{s^2b\mu_s^2}\max\{1, L_s^2\} +  \frac{\kappa_y^2L_s^2\delta_s^2}{s^2b}\ln({\psi}+T)
    \Big) \\
    & = \tilde{\Theta} \Big(
    \frac{\kappa_y^3}{s^2bL}
    +  
    \frac{\kappa_y^2}{s^2b} + \frac{\delta^2}{s^2b\mu_y^2}\max\{1, s^2L^2\} +  \frac{s^2\kappa_y^2L^2\delta^2}{b}
    \Big).
\end{aligned}
\end{equation}
Moreover, to satisfy \cref{eq: convergence eq JMLR}, one needs to choose $T\geq \Theta(\big(\frac{1}{s^2}M''_s\big)^{3/2} \frac{1}{\epsilon^3})$.  Since the total oracle complexity is $bT$, we obtain that
\begin{equation}\label{eq: bT bound with s JMLR}
\begin{aligned}
         & bT \geq\tilde{\Theta}\left(\frac{1}{\epsilon^3} \big(
    \frac{\kappa_y^{9/2-\nu/2}}{s^3L^{3/2}}
    +  
    \frac{\kappa_y^{3-\nu/2}}{s^3} + \frac{\delta^3\kappa^{-\nu/2}}{s^3\mu_y^3}\max\{1, s^3L^3\} +  s^3\kappa_y^{3-\nu/2}L^3\delta^{3}
    \big) \right).
\end{aligned}
\end{equation}
From \cref{eq: choice of s}, i.e., $s^2 = \frac{8}{81}\frac{1}{L^2}\kappa_y^{1-\nu}$, it follows that
\begin{equation}
    \label{eq: bT bound JMLR}
   bT \geq \tilde{\Theta}\left(\frac{1}{\epsilon^3} \big(
    L^{3/2}\kappa_y^{\nu+3}
    +  
    L^{3}\kappa_y^{\nu+3/2} + \delta^3\kappa^{\nu+3/2}\max\{1, \kappa_y^{\frac{3-3\nu}{2}}\} +  \kappa_y^{9/2-2\nu}\delta^{3}
    \big) \right).
\end{equation}
When $\nu\geq1$,  we have 
\begin{equation*}
    bT \geq \tilde{\Theta}\big(L^{3/2}\kappa_y^{\nu+3} 
    +  
    L^{3}\kappa_y^{\nu+3/2} + \delta^3\kappa_y^{3-\nu/2} +  \kappa_y^{9/2-2\nu}\delta^{3}
    \big),
\end{equation*}
the optimal value is achieved at $\nu=1$ and $bT \geq \Theta(\frac{L^{1.5}\kappa_y^4}{\epsilon^3})$; when $\nu<1$, we have 
\begin{equation*}
    bT \geq \tilde{\Theta}\big( L^{3/2}\kappa_y^{\nu+3} 
    +  
    L^{3}\kappa_y^{\nu+3/2} + \delta^3\kappa_y^{\nu+3/2} +  \kappa_y^{9/2-2\nu}\delta^{3}
    \big),
\end{equation*}
the optimal value is achieved at $\nu=\frac{1}{2}$ and $bT \geq\tilde{\Theta}(\frac{L^{1.5}\kappa_y^{3.5}}{\epsilon^3})$, which completes the proof.
\end{proof}
\begin{remark}
In \cite{huang2022accelerated}, Huang \emph{et al.} claims the oracle complexity of $\tilde{\mathcal{O}}(\kappa_y^3\epsilon^{-3})$ for $\nu=3$, and $\tilde{\mathcal{O}}(\kappa_y^{2.5}\epsilon^{-3})$ for $\nu=4$. However, our analysis leading to \cref{eq: bT bound JMLR} demonstrates that the complexities would be ${\tilde{\mathcal{O}}(\frac{L^{1.5}\kappa^6_y}{\epsilon^3})}$ for $\nu =3$ and  $\tilde{\mathcal{O}}(\frac{L^{1.5}\kappa^7_y}{\epsilon^3})$ for $\nu=4$.
\end{remark}

\end{document}